\documentclass[a4paper,reqno,11pt]{amsart}
\usepackage[T1]{fontenc}
\usepackage[utf8]{inputenc}
\usepackage[english]{babel}
\usepackage[dvipsnames,svgnames,table]{xcolor}
\usepackage{amssymb, amsmath, amsthm, graphicx, enumerate, color, mathtools, comment, caption, subcaption, float, lmodern}
\usepackage[foot]{amsaddr}
\usepackage[shortlabels]{enumitem}
\usepackage[unicode=true]{hyperref}
\usepackage[capitalise, noabbrev]{cleveref}
\usepackage{thmtools, thm-restate}
\newcommand{\q}[1]{``#1''}
\DeclarePairedDelimiter\set{\{}{\}}

\DeclareMathOperator{\Inc}{Inc}

\DeclareMathOperator{\se}{se}

\newcommand{\OR}{R}
\newcommand{\OL}{L}

\DeclareMathOperator{\Int}{int}
\DeclareMathOperator{\maxsw}{max-start-weight}
\DeclareMathOperator{\maxew}{max-end-weight}
\DeclareMathOperator{\maxsp}{max-start-path}
\DeclareMathOperator{\maxep}{max-end-path}
\newcommand{\maxpath}{\operatorname{max\text{-}path}}

\DeclareMathOperator{\colorin}{in}
\DeclareMathOperator{\colorout}{out}
\DeclareMathOperator{\Image}{Im}

\DeclareMathOperator{\gce}{gcpe}

\newcommand{\Wleft}{W_L}
\newcommand{\Wright}{W_R}
\newcommand{\sd}{\operatorname{sd}}
\newcommand{\shad}{\operatorname{shad}}
\newcommand{\shadz}{\operatorname{shad}}
\newcommand{\eno}{\operatorname{eno}} 
\newcommand{\HOO}{H_{\mathrm{OO}}}

\newcommand{\HIILR}{H_{\mathrm{IILR}}}

\newcommand{\HII}{H_{\mathrm{II}}}
\newcommand{\HIIL}{H_{\mathrm{IIL}}}
\newcommand{\HIIR}{H_{\mathrm{IIR}}}

\newcommand{\HIO}{H_{\mathrm{IO}}}
\newcommand{\HOI}{H_{\mathrm{OI}}}

\newcommand{\Hin}{H_{\colorin}}
\newcommand{\Hout}{H_{\colorout}}
\newcommand{\spine}{W^\star}

\let\le\leqslant
\let\ge\geqslant
\let\leq\leqslant
\let\geq\geqslant

\let\preceq\preccurlyeq

\let\subset\subseteq
\let\subsetneq\varsubsetneq

\let\epsilon\varepsilon

\newcommand{\WR}{W_{\OR}}
\newcommand{\WL}{W_{\OL}}

\newcommand{\bfs}{\textsc{bfs}}

\newcommand{\bbR}{\mathbb{R}}

\newcommand{\NN}{\mathbb{N}}

\newcommand{\calB}{\mathcal{B}} 
\newcommand{\calC}{\mathcal{C}}
\newcommand{\calD}{\mathcal{D}}

\newcommand{\calI}{\mathcal{I}}

\newcommand{\calL}{\mathcal{L}}
\newcommand{\calM}{\mathcal{M}}

\newcommand{\calP}{\mathcal{P}} 
 
\newcommand{\calR}{\mathcal{R}}
\newcommand{\calS}{\mathcal{S}}

\newcommand{\cgB}{\mathcal{B}} 

\newcommand{\cgD}{\mathcal{D}}

\newcommand{\Oh}{\mathcal{O}}
 
\newcommand{\cgR}{\mathcal{R}}
\newcommand{\cgS}{\mathcal{S}}

\makeatletter
\newcommand{\myitem}[1]{%
\item[#1]\protected@edef\@currentlabel{#1}%
}
\makeatother

\makeatletter
\newcommand{\leqnomode}{\tagsleft@true\let\veqno\@@leqno}
\newcommand{\reqnomode}{\tagsleft@false\let\veqno\@@eqno}
\makeatother

\renewenvironment{enumerate}{\begin{enumorig}[label=\textup{(\roman*)}, noitemsep, 
topsep=2pt plus 2pt, labelindent=.2em, leftmargin=*, widest=iii]}{\end{enumorig}}

\newenvironment{enumerateAlph}{\begin{enumorig}[label=\textup{(\alph*)}, noitemsep, 
topsep=2pt plus 2pt, labelindent=.2em, leftmargin=*, widest=10]}{\end{enumorig}}

\newenvironment{enumerateShadowBlock}{\begin{enumorig}[label=\textup{(sb\arabic*)}, noitemsep, 
topsep=2pt plus 2pt, labelindent=.2em, leftmargin=*, widest=sb3]}{\end{enumorig}}

\newenvironment{enumerateRegion}{\begin{enumorig}[label=\textup{(rg\arabic*)}, noitemsep, 
topsep=2pt plus 2pt, labelindent=.2em, leftmargin=*, widest=rg3]}{\end{enumorig}}

\newenvironment{enumerateNumI}{\begin{enumorig}[label=\textup{(I\arabic*)}, 
noitemsep, topsep=2pt plus 2pt, labelindent=.2em, leftmargin=*, widest=I10]}{\end{enumorig}}

\newenvironment{enumerateNumd}{\begin{enumorig}[label=\textup{(d\arabic*)}, 
noitemsep, topsep=2pt plus 2pt, labelindent=.2em, leftmargin=*, widest=d10]}{\end{enumorig}}

\newenvironment{enumerateNumr}{\begin{enumorig}[label=\textup{(r\arabic*)}, 
noitemsep, topsep=2pt plus 2pt, labelindent=.2em, leftmargin=*, widest=r10]}{\end{enumorig}}

\newenvironment{enumerateNumLRinout}{\begin{enumorig}[label=\textup{(L\arabic*)}, 
noitemsep, topsep=2pt plus 2pt, labelindent=.2em, leftmargin=*, widest=L12, align=left]}{\end{enumorig}}

\newenvironment{enumerateNumIIL}{\begin{enumorig}[label=\textup{(iil\arabic*)}, 
noitemsep, topsep=2pt plus 2pt, labelindent=.2em, leftmargin=*, widest=IIL3]}{\end{enumorig}}

\newenvironment{enumerateNumIIR}{\begin{enumorig}[label=\textup{(iir\arabic*)}, 
noitemsep, topsep=2pt plus 2pt, labelindent=.2em, leftmargin=*, widest=IIR3]}{\end{enumorig}}

\newenvironment{enumerateNumIILR}{\begin{enumorig}[label=\textup{(iilr\arabic*)}, 
noitemsep, topsep=2pt plus 2pt, labelindent=.2em, leftmargin=*, widest=IILR3]}{\end{enumorig}}

\newenvironment{enumerateNumIO}{\begin{enumorig}[label=\textup{(io\arabic*)}, 
noitemsep, topsep=2pt plus 2pt, labelindent=.2em, leftmargin=*, widest=io3]}{\end{enumorig}}

\newenvironment{enumerateNumOI}{\begin{enumorig}[label=\textup{(oi\arabic*)}, 
noitemsep, topsep=2pt plus 2pt, labelindent=.2em, leftmargin=*, widest=oi3]}{\end{enumorig}}

\renewenvironment{itemize}{\begin{itemorig}[label=\tiny\textbullet, noitemsep, 
topsep=2pt plus 2pt, labelindent=.2em, leftmargin=*, widest=ii]}{\end{itemorig}}
\makeatletter
\def\thm@space@setup{
  \thm@preskip=4mm
  \thm@postskip=0mm
}
\makeatother

\usepackage{mdframed}

\mdfdefinestyle{dontsplit}{
  hidealllines=true,
  nobreak=true,
  leftmargin=0pt,
  rightmargin=0pt,
  innerleftmargin=0pt,
  innerrightmargin=0pt,
}

\newmdtheoremenv[style=dontsplit]{theorem}{Theorem}
\newmdtheoremenv[style=dontsplit]{lemma}[theorem]{Lemma}
\newmdtheoremenv[style=dontsplit]{obs}[theorem]{Observation}
\newmdtheoremenv[style=dontsplit]{remark}[theorem]{Remark}
\newmdtheoremenv[style=dontsplit]{proposition}[theorem]{Proposition}
\newmdtheoremenv[style=dontsplit]{question}[theorem]{Question} 
\newmdtheoremenv[style=dontsplit]{corollary}[theorem]{Corollary} 
\newmdtheoremenv[style=dontsplit]{problem}[theorem]{Problem}
\newmdtheoremenv[style=dontsplit]{conjecture}[theorem]{Conjecture}
\newtheorem*{conjecture*}{Conjecture}

\theoremstyle{remark}
\newmdtheoremenv[style=dontsplit]{claim}[theorem]{Claim}
\crefname{claim}{Claim}{Claims}
\newmdtheoremenv[style=dontsplit]{example}[theorem]{Example}

\newenvironment{proofclaim}[1][]
	{\vspace{-\topsep}\begin{proof}[Proof] }{\end{proof}}

\crefformat{equation}{#2(#1)#3}
\let\eqref\cref
\crefformat{subsection}{Subsection #2#1#3}
\input{style.table-of-contents}

\hypersetup{ 
    colorlinks,
    linkcolor={RoyalBlue},
    citecolor={RubineRed},
    urlcolor={blue!80!black},
    pdftitle={Planarity and dimension I}
}

\newcommand{\piotr}[1]{{\color{orange} Piotr: #1}}
\newcommand{\jedrzej}[1]{\textcolor{ProcessBlue}{Jędrzej: #1}}
\newcommand{\michal}[1]{\textcolor{Bittersweet}{Michał: #1}}

\newcommand{\heather}[1]{\textcolor{LimeGreen}{Heather: #1}}

\newcommand{\fig}{\textcolor{magenta}{TODO: figure}}

\title{Planarity and dimension I}
\author[Blake]{Heather Smith Blake}
\address[Blake]{Mathematics \& Computer Science Department, Davidson College, Davidson, North Carolina 28035}
\email{hsblake@davidson.edu}

\author[Hodor]{Jędrzej Hodor}
\address[Hodor]{Theoretical Computer Science Department, 
Faculty of Mathematics and Computer Science and  Doctoral School of Exact and Natural Sciences, Jagiellonian University, Krak\'ow, Poland}
\email{jedrzej.hodor@gmail.com}

\author[Micek]{Piotr Micek}
\address[Micek]{Theoretical Computer Science Department, 
Faculty of Mathematics and Computer Science, Jagiellonian University, Krak\'ow, Poland}
\email{piotr.micek@uj.edu.pl}

\author[Seweryn]{Michał T.~Seweryn}
\address[Seweryn]{Computer Science Institute, Charles University, Prague, Czech Republic}
\email{seweryn@iuuk.mff.cuni.cz}

\author[Trotter]{William T. Trotter}
\address[Trotter]{School of Mathematics, Georgia Institute of Technology, Atlanta, Georgia 30332}
\email{trotter@math.gatech.edu}

\thanks{H.~S.~Blake is supported by a grant from the Simons Foundation. 
J.~Hodor and P.~Micek are supported by the National Science Center of Poland under grant UMO-2022/47/B/ST6/02837 within the OPUS 24 program. 
M.~T.~Seweryn is supported by ERC-CZ project LL2328 of the Ministry of Education of Czech Republic.
W.~T.~Trotter is supported by a grant from the Simons Foundation.}

\begin{document}

\begin{abstract}
The dimension of a partially ordered set $P$ (poset for short) is the least positive integer $d$ such that $P$ is isomorphic to a subposet of $\mathbb{R}^d$ with the natural product order.
Dimension is arguably the most widely studied measure of complexity for posets, and standard examples in posets are the canonical structure forcing dimension to be large. 
In many ways, dimension for posets is analogous to chromatic number for graphs 
with standard examples in posets playing the role of cliques in graphs. 
However, planar graphs have chromatic number at most four, while 
posets with planar diagrams may have arbitrarily large dimension. 
The key feature of all known constructions of such posets is that large dimension is forced by a large standard example.
The question of whether every poset of large dimension and with a planar cover graph contains a large standard example has been a critical challenge in posets theory since the early 1980s, with very little progress over the years.
We answer the question in the affirmative. 
Namely, we show that every poset $P$ with a planar cover graph has dimension $\mathcal{O}(s^8)$, where $s$ is the maximum order of a standard example in $P$.
\end{abstract}

\maketitle

\newpage
\parskip   1.5mm
\tableofcontents
\parskip   2mm

\newpage

\section{Introduction}
\label{sec:intro}

In this paper, we study finite partially ordered sets, called \emph{posets} for short.
The \emph{dimension} of a poset $P$, denoted $\dim(P)$, is the least positive integer $d$ 
such that
$P$ is isomorphic to a subposet of $\mathbb{R}^d$ equipped with the product order.\footnote{In the product order of $\mathbb{R}^d$, for $(x_1,\dots,x_d),(y_1,\dots,y_d) \in \mathbb{R}^d$, we have $(x_1,\dots,x_d) \leq (y_1,\dots,y_d)$ if and only if $x_i \leq y_i$ for every $i \in \{1,\dots,d\}$.}  
Dimension is arguably the most widely studied measure of complexity for posets.
It captures important concepts in graph theory such as planarity~\cite{S89} and nowhere denseness~\cite{JMOdMW19}.
The computational complexity of testing whether a poset has dimension at most $d$ appeared 
on the famous Garey-Johnson list of problems~\cite{GJ79}.
Nowadays, we know that dimension is \textsf{NP}-hard to compute~\cite{Y82,FMP17} and 
hard to approximate in a strong sense~\cite{CLN13}. 








Dimension was introduced in a foundational paper by Dushnik and Miller~\cite{DM41} in 1941. 
This paper also includes the canonical structure in posets forcing dimension to be large, 
namely, the family of standard examples.
For each integer $n$ with $n \geq 2$,
the \emph{standard example} of order $n$, denoted by $S_n$, is a poset on $2n$ elements 
$a_1,\dots,a_n,b_1,\dots,b_n$
such that 
$a_1,\dots,a_n$ are pairwise incomparable, 
$b_1,\dots,b_n$ are pairwise incomparable, and
for all $i,j \in \set{1,\ldots,n}$, 
we have $a_i < b_j$ in $S_n$ if and only if $i\neq j$.
See~\cref{fig:se}.
It is one of the first exercises in dimension theory to show that $\dim(S_n)=n$. 
Since dimension is a monotone parameter, 
$\dim(P)\geq n$ whenever $P$ contains a subposet isomorphic to $S_n$.

\begin{figure}[h]
  \centering
  \includegraphics{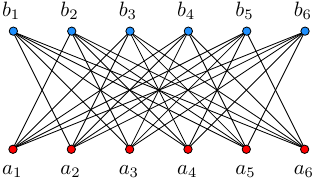}
  \caption{
        The standard example of order $6$. 
    }
   \label{fig:se}
\end{figure}

However, large standard examples are not the only way to drive dimension up. 
There are families of posets with arbitrarily large dimension such that for some integer $n$ with \(n \ge 2\), no poset in the family contains \(S_n\),
e.g., 
incidence posets of complete graphs (as proved by Dushnik and Miller~\cite{DM41}), interval orders (see a tight asymptotic bound on their dimension by 
Füredi, Hajnal, Rödl, and Trotter~\cite{FHRT91}), adjacency posets of triangle-free graphs with large chromatic number (as shown by Felsner and Trotter \cite{FT00}).
These results motivate the following definitions.
The \emph{standard example number} of a poset $P$, denoted $\se(P)$, 
is set to be $1$ if $P$ does not contain a subposet isomorphic to a standard example; 
otherwise, 
$\se(P)$ is the maximum integer $n$ such that $P$ contains a subposet isomorphic to $S_{n}$.
Clearly, for every poset $P$, we have $\se(P) \leq \dim(P)$. 
A class of posets $\calC$ is \emph{$\dim$-bounded} if there is a function $f$ 
such that $\dim(P) \leq f(\se(P))$ for every $P$ in $\calC$. 
As we discussed, the class of all posets is not $\dim$-bounded.


The dimension of a poset $P$ can be defined equivalently as the chromatic number of the hypergraph on the set of all incomparable pairs of $P$ with the edge set given by the set of all alternating cycles in~$P$ (see~\Cref{sec:dim} for the definition of an alternating cycle and~\Cref{prop:dim-alternating-cycles} for the equivalence).
This links the notion of dimension of posets with graph colorings.
The inequality $\se(P) \leq \dim(P)$ for all posets $P$ parallels the inequality $\omega(G) \leq \chi(G)$ for all graphs~$G$.\footnote{Here, for a graph $G$, $\chi(G)$ is its chromatic number and $\omega(G)$ is its clique number.}
Both inequalities are far from tight.
A class of graphs $\calC$ is \emph{$\chi$-bounded} if there is a function $f$ such that for every $G$ in $\calC$, we have $\chi(G) \leq f(\omega(G))$.
We refer readers to the recent survey by Scott and Seymour~\cite{SS20} on
the extensive body of research done on this topic.
The analogy breaks down with the celebrated Four Color Theorem, which states that planar graphs have chromatic number at most four, while \q{planar} posets may have arbitrarily large dimension as we now explain.

\begin{figure}[tp]
  \centering
  \includegraphics{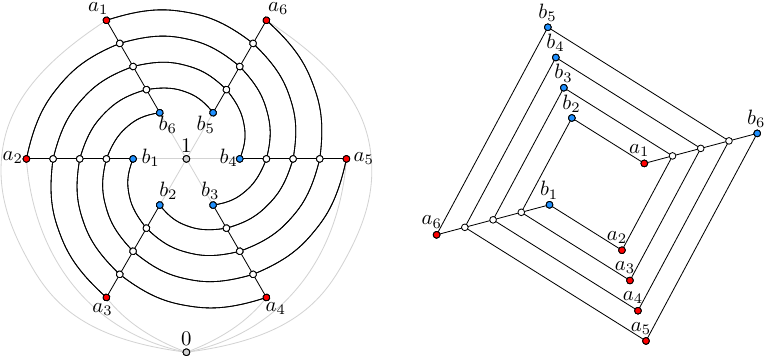}
  \caption{
  Left: The wheel of order $6$: it has a planar cover graph (the order relation goes inwards) and it contains a subposet isomorphic to the Kelly poset of order $6$.
  Right: The Kelly poset of order $6$: it has a planar diagram and contains a subposet isomorphic to $S_6$. 
}
   \label{fig:se_kelly_wheel}
\end{figure}

An element $y$ in a poset $P$ \emph{covers} an element $x$ in~$P$ if $x < y$ in~$P$ and there is no $z$ in~$P$ with $x<z<y$ in~$P$.
The \emph{cover graph} of a poset $P$ is the graph whose vertices are the
elements of $P$ and two elements are adjacent if one covers the other.
Somewhat unexpectedly, posets with planar cover graphs can have
arbitrarily large dimension, as we learned\footnote{
\emph{We learned} is a bit of an exaggeration as only one author of the manuscript in hand was alive in 1978.} 
in 1978~\cite{T78}, see \cref{fig:se_kelly_wheel}.  
A \emph{diagram} of $P$ is a drawing of the cover graph of $P$ in the plane such that whenever $xy$ is an edge in the cover graph and $x<y$ in~$P$, 
the relation is represented by a curve from $x$ to $y$ going upwards. 
In 1981,
Kelly~\cite{K81} published a seminal construction of posets with planar diagrams and arbitrarily large dimension, see~\cref{fig:se_kelly_wheel} again. 
A key remark is that all known constructions of planar posets with large dimension contain large standard examples. 
Thus, since the early 1980's, it remained a challenge and perhaps the most important problem in poset theory to settle the following.
\begin{conjecture*}\hfill
  \begin{enumerate}
  \item The class of posets with planar diagrams is $\dim$-bounded.
  \label{item:conj:diagram}
  \item The class of posets with planar cover graphs is $\dim$-bounded.
  \label{item:conj:cover-graph}
\end{enumerate}
\end{conjecture*}
We believe that the first published reference to Conjecture~\ref{item:conj:diagram} is
an informal comment on page 119 in~\cite{Tro-book} published
in 1992. However, the conjecture was circulating among researchers soon
after the constructions illustrated in \cref{fig:se_kelly_wheel} appeared.
Accordingly, Conjecture~\ref{item:conj:diagram} is more than~40 years old and obviously Conjecture~\ref{item:conj:cover-graph} is a stronger statement.
In this paper, we show that both statements are true and prove the following theorem.

\begin{theorem}\label{thm:cover-graph_se}
  For every poset \(P\) with a planar cover graph, 
  $\dim(P) \le 64s^6(s+3)^2 + 12$ where $s = \se(P)$.
\end{theorem}

The theorem has an important algorithmic aspect. 
Following our proof, one can design a polynomial time algorithm that given a poset $P$ with a planar cover graph on the input, 
returns an embedding of $P$ into $\mathbb{R}^d$, where $d$ is in $\Oh(\se(P)^8)$. 
Since $\se(P)\leq \dim(P)$, this constitutes an approximation algorithm for poset dimension in the planar setting. 
We believe that it is not known whether computing dimension of posets with planar cover graphs is \textsf{NP}-hard.

We remark that if $P$ has a planar diagram we can prove that 
$\dim(P)$ is bounded by a linear function of $\se(P)$. 
Since this requires separate proof techniques crafted for the planar diagram setup, we leave this line of research for a separate manuscript: Planarity and dimension~II.
Rephrasing~\cref{thm:cover-graph_se}, a large dimensional poset with a planar cover graph contains a large standard example. 
In another manuscript in preparation, Planarity and dimension III, we show that a poset with a planar cover graph containing a large standard example must contain 
a large Kelly poset.
Thus, large dimensional posets with planar cover graphs contain large Kelly posets.

It seems that for many years, there were no tools to force large standard examples in highly dimensional posets. 
Hence, there are very little results on dim-boundedness. 
Blake, Micek, and Trotter~\cite{BMT22}, 
proved that each family of posets with planar cover graphs and a unique minimal element (called a zero) is dim-bounded. 
Note that posets with planar cover graphs and a zero may have arbitrarily large dimension as they contain all the wheels, recall~\Cref{fig:se_kelly_wheel}.  
Recently, 
Joret, Micek, Pilipczuk, and Walczak~\cite{JMPW24} proved that 
posets with cover graphs of bounded treewidth 
(or even posets of bounded cliquewidth) are dim-bounded. 
Note that posets with cover graphs of treewidth at most $3$ may have arbitrarily large dimension as they contain all the Kelly posets, again recall~\Cref{fig:se_kelly_wheel}.  
The proof of the main result in~\cite{JMPW24} builds on 
 Colcombet’s deterministic version of Simon’s factorization theorem, which
is a fundamental tool in formal language and automata theory -- therefore, it is very different from the combinatorics in~\cite{BMT22} and this paper.

Since planar graphs exclude $K_5$ as a minor and graphs of treewidth less than $t$ exclude $K_{t+1}$ as a minor, the following statement generalizes both~\cref{thm:cover-graph_se} and the main result of~\cite{JMPW24}.
\begin{conjecture*}
    For every positive integer $t$, the class of posets with cover graphs excluding $K_t$ as a minor is $\dim$-bounded.
\end{conjecture*}
We remark that the class of posets with cover graphs excluding $K_t$ as a topological minor is not $\dim$-bounded~\cite[Section~6]{FMM20}.

In contrast to dim-boundedness, it is well understood for which minor-closed classes of graphs, posets with cover graphs in this class have dimension bounded by an absolute constant.
The very first result in this line is by Trotter and Moore~\cite{TM77} who showed that posets whose cover graphs are forests (so exactly those excluding $K_3$ as a minor) have dimension at most~$3$. 
Felsner, Trotter, and Wiechert~\cite{FTW14} proved that posets with outerplanar cover graphs have dimension at most $4$. 
Seweryn~\cite{S20} proved that posets with cover graphs excluding $K_4$ as a minor have dimension at most $12$. 
Note that posets with cover graphs excluding $K_5$ as a minor contain all Kelly posets, so they may have arbitrarily large dimension. 
Huynh, Joret, Micek, Seweryn, and Wollan~\cite{HJMSW21} proved that 
posets with cover graphs excluding a $2 \times t$ grid as a minor have dimension bounded by a function of $t$. 
All the above is qualitatively generalized by the following statement~\cite{JMPW24}: for a minor-closed class of graphs $\calC$ there exists a constant such that every poset with a cover graph in $\calC$ has dimension bounded by this constant if and only if $\calC$ excludes the cover graph of some Kelly poset.

Poset dimension was also studied with respect to other poset invariants.
Already in 1950, Dilworth~\cite{Dilworth1950} proved that 
dimension of a poset is always bounded by its width (i.e.\ the largest size of an antichain). 
This is a consequence of his famous chain decomposition result: 
every poset of width $w$ can be partitioned into $w$ chains. 
On the other hand, there are posets of height (i.e.\ the largest size of a chain) equal to $2$ and arbitrarily large dimension, e.g.\ standard examples.
This leads to a line of research initiated by Streib and Trotter~\cite{ST14} who proved that posets with planar cover graphs have dimension bounded by a function of their height.
In other words, large dimensional posets with planar cover graphs are not only wide but also tall.
This result was later improved and generalized in many directions~\cite{JMMTWW16,W17,MW17,JMW17,JMW18,KMT19,GS21}.
 
Finally, let us mention that in fact we prove a stronger result, which implies~\Cref{thm:cover-graph_se} (see~\Cref{thm:technical} and the necessary definitions in~\Cref{sec:dim}).
In particular, the stronger result implies that the class of posets that are subposets of posets with planar cover graphs is dim-bounded (see~\cref{cor:subposets-dim-bounded}).

\section{Outline of the proof}
Our proof of~\Cref{thm:cover-graph_se} is self-contained and does not rely on any external material beyond the fundamentals of graph theory and poset theory.
In this section, we discuss the main ideas and sketch the structure of the proof. 
The readers unfamiliar with the basics of poset dimension may want to skip this high-level overview and and proceed directly to~\Cref{sec:preliminaries}.
All the notation introduced in the outline is later introduced again. 

We begin with an alternative perspective on the dimension of posets, presented in detail in~\cref{sec:dim}, which is more amenable to combinatorial arguments.
Let $P$ be a poset.
We denote by $\Inc(P)$ all incomparable pairs of elements of $P$.
For an integer $k$ with $k \geq 2$, a sequence $((a_1,b_1), \dots, (a_k,b_k))$ of pairs in $\Inc(P)$ is a \emph{strict alternating cycle} of size $k$ in~$P$ 
if $a_i\leq b_{j}$ in~$P$ for all $i,j\in[k]$ if and only if $j = i +1$, cyclically (that is, $(a_{k+1},b_{k+1}) = (a_1,b_1)$).
Let $I \subset \Inc(P)$.
We say that such a strict alternating cycle is \emph{contained} in $I$ if $(a_i,b_i) \in I$ for all $i \in [k]$. 
The subset $I$ is \emph{reversible} if $I$ does not contain a strict alternating cycle. 
A family \(\calS\) of subsets of \(\Inc(P)\) \emph{covers $I$} if $I \subset \bigcup \calS$.
We define $\dim_P(I)$ as the minimum positive integer $d$ such that $I$ can be covered by $d$ reversible sets.
One can show that $\dim(P)=\dim_P(\Inc(P))$, see \Cref{prop:dim-alternating-cycles}. 

Let $P$ be a poset with a planar cover graph. 
The proof of the main theorem starts with an application of an unfolding (explained below) to $P$ in order to pinpoint a fragment of $P$ responsible for its dimension.  
The idea of an \q{unfolding of a poset} was introduced by Streib and Trotter~\cite{ST14} in 2014, and it is inspired by the following classical application of layerings to graph colorings. 
A \emph{layering} of a graph $G$ is a family $(Z_i : i\in \NN)$ of pairwise disjoint subsets of $V(G)$ such that for every edge $uv$ of $G$, there exists $i \in \NN$ with $\{u,v\} \subset Z_i \cup Z_{i+1}$.
It is well-known that
given a graph $G$ and a
layering of $G$, 
if every layer induces a $k$-colorable graph, 
then $G$ is \(2k\)-colorable (one can use two disjoint
palettes of \(k\) colors each, one for even layers, and one for odd layers).
The counterpart of the above for posets is formulated in terms of the unfolding.
An \emph{unfolding} of a poset $P$ is a layering of the comparability graph of $P$.
If the union of any two consecutive sets in the unfolding induces a subposet of dimension at most \(d\), then the dimension of \(P\) is at most \(2d\), see~\cref{prop:unfolding}.
We use a specific unfolding of $P$ that emerges from a \textsc{bfs}-layering of the comparability graph of $P$ that starts in a minimal element of $P$. 
Given such an unfolding, we localize two consecutive layers whose union maximizes dimension, and we contract all the elements of $P$ that lie in layers before, see~\cref{lem:PlanarCoverGraphReduction}.
In this way, paying a multiplicative factor of $2$, we reduce the problem of bounding $\dim(P)$ to a simpler setting.
Namely, we have
\[
\dim(P) \leq 2\cdot\dim_Q(I),
\]
where $Q$ is a poset and $I \subset \Inc(Q)$ such that: $Q$ has a planar cover graph; 
$Q$ has a distinguished minimal element $x_0$; 
the cover graph $G$ of $Q$ has a fixed planar drawing with $x_0$ in the exterior face;
for every $(a,b) \in I$, we have $x_0 < b$ in $Q$;
and $\se(Q) \leq \se(P)$.
This is encapsulated in the notion of an \emph{instance}, see \ref{item:instance:planar_cover_graph}--\ref{item:instance:I_singly_constrained} in~\Cref{sec:topology_instance} and~\Cref{cor:poset-to-instance}.
An instance is a tuple $(Q,x_0,G,e_{-\infty},I)$, where $Q$, $x_0$, $G$, and $I$ are as above, and additionally, $e_{-\infty}$ is a simple curve in the plane that is contained in the exterior face of the drawing and has one endpoint in $x_0$.
This will be an \q{anchor} determining \q{directions} in the drawing, as we explain below.

Once an instance $(P,x_0,G,e_{-\infty},I)$ is fixed, we discuss the topology of the drawing, see~\Cref{sec:topology_instance}. 
This part of the material draws from concepts introduced by Blake, Micek, and Trotter~\cite{BMT22}.
Given an element $u$ in~$P$ and an edge (or the anchor) $e$ incident to $u$, all the edges incident to $u$ are ordered from left to right (i.e.\ clockwise) around $u$ starting from $e$ (see more details and the notation in~\Cref{ssec:ordering-edges}).
A \emph{witnessing path} in~$P$ is a path in the cover graph of $P$ of ascending elements in~$P$.
Let $B$ be the set of all elements $b$ in~$P$ with $x_0 \leq b$ in~$P$.
Given $b \in B$, we consider all the witnessing paths from $x_0$ to $b$ in~$P$.
Using the ordering of edges described above, we order these paths.
Namely, given two such witnessing paths, we consider their longest common prefix starting from $x_0$.
Next, we look at the last edge of this prefix (or $e_{-\infty}$ if the prefix is trivial), and according to this edge, we compare the next edges of the paths to decide which one is \emph{left} and which one is \emph{right} of the other.
This ordering is in fact a linear ordering of all witnessing paths from $x_0$ to $b$, see~\Cref{obs:ue_ordering_is_usually_linear}. 
This way, we obtain the \emph{leftmost} witnessing path from $x_0$ to $b$, denoted by $W_L(b)$, and the \emph{rightmost} witnessing path from $x_0$ to $b$, denoted by $W_R(b)$, see~\Cref{sec:leftmost-rightmost} and~\Cref{fig:leftmost-rightmost}.

Given $b\in B$, the paths $W_L(b)$ and $W_R(b)$ start from $x_0$ and then they may split and rejoin multiple times before they both end in~$b$. 
Let $z_0,\ldots,z_{n}$ be all the common elements of $W_L(b)$ and $W_R(b)$ in the natural order. 
Thus, $z_0=x_0$ and $z_n=b$.  
For $i\in[n]$, consider a region enclosed by $W_L(b)$ and $W_R(b)$ between $z_{i-1}$ and $z_{i}$.
Note that, in case $i\neq n$ both paths may continue outside this region or inside this region. 
In the latter case, $z_i$ is called a \emph{reversing} element of $b$. 
Let $x_1,\ldots,x_d$ be all the reversing elements of $b$ and let $x_{d+1}=b$ (recall that $x_0$ is already defined).
The region enclosed by $W_L(b)$ and $W_R(b)$ from $x_i$ up to $x_{i+1}$ is denoted by $\shad_i(b)$ for each $i \in \{0,\dots,d\}$. 
We additionally set $\shad_{d+1}(b) = \emptyset$.
We say that $x_i$ is the \emph{initial} element of $\shad_i(b)$. 
See~\Cref{ssec:shadows} and~\Cref{fig:shadows}.
In the next step of the proof, we reduce the problem to an instance in which any two pairs $(a, b)$ and $(a', b')$ belonging to one strict alternating cycle in $I$ satisfy $b \notin \shad_0(b')$, and $b' \notin \shad_0(b)$. 
Therefore, we use the leftmost and rightmost witnessing paths and shadows to topologically order the elements in~$B$, see~\Cref{sec:ordering_elements_in_B}.
Namely, for $b,b' \in B$, we say that $b$ is \emph{left} of $b'$ if $W_L(b)$ is left of $W_L(b')$, $W_R(b)$ is left of $W_R(b')$, $b \notin \shad_0(b')$, and $b' \notin \shad_0(b)$. 

Within the next step of the argument, we remove twice from $I$ a subset of pairs of dimension at most $2$ (non-risky pairs -- see~\Cref{ssec:risky}, and non-dangerous pairs -- see~\cref{ssec:dangerous}).
Let us skip discussing this detail here and just mention that this is where the additive terms \q{$+2$} in the inequality displayed below come from.
To each pair $(a,b) \in I$, we assign an \emph{address} $(j,x)$, where $j$ is the minimum nonnegative integer such that $a \notin \shad_j(b)$ and $x$ is the initial element of $\shad_j(b)$.
We split the pairs in $I$ into $I_0$ and $I_1$ depending on the parity of the first coordinate of their address: 
$I_{\theta}=\set{(a,b)\in I: \textrm{$(j,x)$ is the address of $(a,b)$ and $j\equiv\theta\bmod 2$}}$. 
Note that $\dim_P(I) \leq \dim_P(I_0) + \dim_P(I_1)$.
Let $\theta \in \{0,1\}$.
We prove that all pairs in a strict alternating cycle contained in $I_\theta$ have the same address, see~\cref{lem:cgI-comprehensive}. 
This property, together with an elementary fact from poset dimension theory, see~\cref{prop:dim_of_sum_incomparable_sets}, gives $\dim_P(I_\theta) = \max \dim_P(I(j,x))$, where the maximum goes over every possible addresses $(j,x)$ with \mbox{$j\equiv\theta\bmod2$} and $I(j,x)$ is the set of all pairs in $I$ with the address $(j,x)$.
We fix an address $(j,x)$ such that $\dim_P(I_\theta) = \dim_P(I(j,x))$.
We define $P'$ as the poset obtained from $P$ by removing all elements strictly less than $x$ in~$P$. 
Clearly, all pairs of $I(j,x)$ survive in $P'$. 
Moreover, $x$ lies on the exterior face of the inherited drawing of the cover graph of $P'$. 
Now, the element $x$ in $P'$ plays the role of $x_0$ in~$P$.
In particular, $\shad_j$ in~$P$ becomes $\shad_0$ in $P'$, and every pair in $I(j,x)$ has address $(0,x)$ in the new setting.

In other words, again paying a multiplicative factor of $2$ (splitting $I$ into $I_0$ and $I_1$), we reduce the problem to a simpler setting.
Namely, we obtain an instance $(P',x_0',G',e_{-\infty}',I')$ with
\[
\dim_P(I) \leq 2\cdot(\dim_{P'}(I')+2)+2,
\]
and satisfying the following properties.
First, $a \notin \shad_0(b)$ for every $(a,b) \in I'$. 
Next, for every strict alternating cycle $((a_1,b_1),\ldots,(a_k,b_k))$ contained in $I'$, the elements $\set{b_1,\ldots,b_k}$ are linearly ordered by the \q{left of} relation, see the second assertion in~\Cref{lem:cgI-comprehensive}.
Finally, all the pairs in $I'$ are dangerous, which is a detail that we omit in the outline.
An instance is \emph{good} if it satisfies these additional properties, see~\ref{item:instance:not_in_shadow}--\ref{item:instance:dangerous} in~\cref{ssec:interface} and~\Cref{cor:interface}. 
For technical reasons, we insist that a good instance that we fix is also \emph{maximal} in a certain sense that we also omit in the outline, see~\ref{item:instance:maximal} in~\cref{ssec:interface} and~\Cref{prop:good_instance_to_maximal_good_instance}.

Given a maximal good instance $(P,x_0,G,e_{-\infty},I)$, we split $I$ into reversible sets along the following plan. 
We devise six auxiliary graphs $H_1,\dots,H_6$ on the vertex set $I$ such that for every strict alternating cycle $((a_1,b_1),\dots,(a_k,b_k))$ contained in $I$, there exist distinct $i,j \in [k]$ such that $(a_i,b_i)$ and $(a_j,b_j)$ are adjacent in one of the graphs, see~\Cref{lem:coloring}.
Suppose that we are given a proper vertex coloring $c_i$ of $H_i$, for each $i \in [6]$. 
Consider the product coloring of $I$, i.e., define $\kappa((a,b)) = (c_1((a,b)), \dots, c_6((a,b)))$ 
for each $(a,b) \in I$. 
It follows that each color class of $\kappa$ yields a reversible subset of $I$.
Namely, for each color \(\xi\) in the coloring \(\kappa\), 
the set $I_\xi = \{(a,b) \in I : \kappa((a,b)) = \xi\}$ is reversible.
In particular,
\[\dim_P(I) \leq \chi(H_1) \cdots \chi(H_6).
\]

The construction of suitable graphs $H_1,\dots,H_6$ that fit the framework above is an essential part of this paper. 
All six graphs share some fundamental properties.
E.g., they all admit the following orientation.
Let $H\in\set{H_1,\ldots,H_6}$. 
When $(a,b)$ are adjacent $(a',b')$ in $H$, then either $b$ is left of $b'$ or $b'$ is left of $b$. 
Hence, we can orient all edges of $H$ so that when $((a,b),(a',b'))$ is an edge in $H$, then $b$ is left of $b'$.
In particular, this orientation is acyclic.
The natural way of coloring a graph $H$ with a fixed acyclic orientation is to assign to each vertex $v$ of $H$, the value $\maxsp(H,v)$, i.e., the maximum order of a directed path in $H$ starting in $v$.
This coloring witnesses $\chi(H) \leq \maxpath(H)$, where $\maxpath(H)$ is the maximum order of a directed path in $H$.

Four of the six auxiliary graphs admit a suitable upper bound on $\maxpath$. 
Namely, for each $H \in \{H_1,\dots H_4\}$, we have $\maxpath(H) \leq \se(P)$, see~\Cref{prop:HOO,prop:HIIL,prop:HIIR,prop:HIILR} proved in~\Cref{sec:aux-proofs}, and so, $\chi(H) \leq \se(P)$.
The remaining two graphs must be treated with more care.
Let $H = H_5$.
When $((a,b),(a',b'))$ is an edge in $H$, then
\[\maxsp(H,(a,b)) > \maxsp(H,(a',b')) > \maxsp(H,(a,b)) - m,\]
where $m = 2\se(P)\cdot(2\se(P)+6)$.
The first inequality is trivial, while proving the second one is a substantial portion of the whole proof. 
It follows that 
\[\maxsp(H,(a,b)) \not\equiv \maxsp(H,(a',b')) \bmod m.\]
We color each vertex $(a,b)$ of $H$ with the value $\maxsp(H,(a,b)) \bmod m$.
Therefore, two adjacent vertices in $H$ are assigned distinct colors, and so, the coloring is proper.
In conclusion, $\chi(H_5) \leq m$. 
The argument for $H_6$ is symmetric 
(roughly, we replace each occurrence of $\maxsp$ by $\maxep$). 
Eventually, we also obtain $\chi(H_6) \leq m$. 
In this outline, we omit a detail that the edges in $H_5$ and $H_6$ are weighted with values from $\set{0,1}$, see~\Cref{lemma:HIO} for a precise statement. 
Altogether, we obtain
\[\dim_P(I) \leq \chi(H_1) \cdots \chi(H_6) \leq \se(P)^4\cdot \left(\se(P)\cdot(2\se(P)+6)\right)^2.
\]

The definitions of the six auxiliary graphs are based on the topology of the drawing. 
We associate each strict alternating cycle $((a_1,b_1),(a_2,b_2))$ contained in $I$ with two regions $\calR_1$ and $\calR_2$ in the plane, see~\Cref{ssec:classification}. 
Next, we classify these cycles 
according to the statements \q{$a_1 \in \calR_2$} and \q{$a_2 \in \calR_1$} being satisfied or not. 
As a result, each strict alternating cycle of size $2$ in $I$ is of one of the four types: In-In, In-Out, Out-In, and Out-Out.
This classification inspires the definitions of the six auxiliary graphs: $\HOO$, $\HIIL$, $\HIIR$, $\HIILR$, $\HIO$, and $\HOI$, see~\Cref{sec:aux-definitions}. 
It also lays the foundations for the key property discussed above: 
every strict alternating cycle (of arbitrary size) contained in $I$ contains two incomparable pairs adjacent in one of the auxiliary graphs, see~\Cref{lem:coloring}.

\section{Preliminaries}
\label{sec:preliminaries}
We denote by $\mathbb{R}$ the set of real numbers and by $\NN$ the set of nonnegative integers.
For a positive integer $k$, we write $[k]$ as a
compact form of $\{1,\dots,k\}$.

\subsection{Graphs}
We consider simple and finite graphs.  
Unless we say otherwise, all graphs are undirected. 
The vertex set of a graph $G$ is denoted by \(V(G)\) and the edge set of $G$ is denoted by \(E(G)\).
Let $G_1,G_2$ be two graphs. 
The \emph{union} of $G_1$ and $G_2$, denoted by $G_1 \cup G_2$ is the graph with the vertex set $V(G_1)\cup V(G_2)$ and the edge set $E(G_1)\cup E(G_2)$.
The \emph{intersection} of $G_1$ and $G_2$, denoted by $G_1 \cap G_2$ is the graph with vertex set $V(G_1)\cap V(G_2)$ and edge set $E(G_1)\cap E(G_2)$.

Let $k$ be a nonnegative integer.
A \emph{path} is a graph 
with the vertex set $\set{v_0,\ldots,v_k}$ and the 
edges $\set{v_{i-1}v_{i} : i \in [k]}$, where $v_0,\dots,v_k$ are pairwise distinct. 
We often refer to a path by a natural sequence of its vertices, 
writing, say, $v_0\cdots v_k$ or equivalently $v_k \cdots v_0$.
Writing a path as $v_0\cdots v_k$ fixes an underlying orientation of the path, where $v_0$ is the first vertex and $v_k$ is the last vertex.
We say that $v_0$ and $v_k$ are the \emph{endpoints} of $v_0 \cdots v_k$.
We also say that $v_0 \cdots v_k$ \emph{starts} in $v_0$ and \emph{ends} in $v_k$.
A \emph{cycle} is a graph with at least three vertices such that removing each of its edges gives a path. 
A \emph{forest} is a graph with no cycles.
A \emph{tree} is a connected forest.
Let $T$ be a tree, and let $u,v$ be two vertices in $T$.
We write $u[T]v$ to denote the unique path in $T$ with endpoints $u$ and $v$ and with an orientation from $u$ to $v$. 

We use the following convenient notation for manipulating paths in graphs. 
Let $G$ be a graph.
Let \(v_0, \ldots, v_k\) be vertices of \(G\), and let \(T_1, \ldots, T_k\) be trees in \(G\) such that each \(T_i\) contains the vertices \(v_{i-1}\) and \(v_i\). Then we denote by \(v_0 [T_1] v_1 [T_2] \cdots [T_k] v_k\) the union of the paths \(v_{i-1} [T_i] v_i\). See \cref{fig:concatenating_paths}.
When $T_i = v_{i-1}v_i$, then we may omit \q{\([T_i]\)} in this notation. 
When the resulting graph is a path, we consider it with the orientation from $v_0$ to $v_k$.
For example, for a path \(W = v_0 \cdots v_5\), we have \(v_0[W]v_2 v_3 [W] v_5 = W\).



\subsection{Orders and posets}
Let $X$ be a set. 
A \emph{partial order} on $X$ is a binary relation $R$ on \(X\)
that is \emph{reflexive} (for every \(x \in X\), \((x,x) \in R\)), \emph{antisymmetric} (for all \(x,y\in X\), if \((x,y) \in R\) and \((y,x) \in R\), then \(x=y\)), and \emph{transitive} (for all \(x,y,z \in X\), if \((x,y) \in R\) and \((y,z) \in R\), then \((x,z) \in R\)). 
A partial order $R$ is \emph{linear} if
for all $x,y \in X$, we have $(x,y) \in R$ or $(y,x) \in R$.
When \(R\) is a partial order (resp.\ linear order) on \(X\), we say that \(R\) \emph{partially orders} (resp.\ \emph{linearly orders}) \(X\).
When $R$ partially orders (resp.\ linearly orders) $X$, we say that $P = (X, R)$ is a \emph{partially ordered set}, or \emph{poset} for short (resp.\ \emph{linearly ordered set}).
We refer to the set \(X\) as the \emph{ground set} of \(P\), and to the members of \(X\) as the \emph{elements} of \(P\).
Note that $(X,R^{-1})$ is also a poset, we call it the \emph{dual} of~$P$ and denote it by~$P^{-1}$.

Let $x$ and $y$ be elements of~$P$.
We say that $x$ and $y$ are \emph{comparable} in~$P$ (or in $R$)  if \((x,y) \in R\) or \((y, x) \in R\).
We use the notation \q{$x \leq y$ in~$P$} whenever $(x,y) \in R$, and \q{$x < y$ in~$P$} whenever $(x,y) \in R$ and $x \neq y$.
Furthermore, we say that $x$ and $y$ are \emph{incomparable} in~$P$ (or in $R$) if they are not comparable in $R$. 
In this case, we write \q{$x \parallel y$ in~$P$}.

An element $x$ of~$P$ is \emph{minimal} in~$P$ if for every element $y$ of~$P$, $y \leq x$ in~$P$ implies $y = x$.
Symmetrically, an element $x$ of~$P$ is \emph{maximal} in~$P$ if for every element $y$ of~$P$, $x \leq y$ in~$P$ implies $y = x$.
A \emph{chain} in~$P$ is a set of elements in~$P$ such that every two elements are comparable in~$P$, 
and an \emph{antichain} in \(P\) is a set of elements in~$P$ such that every two distinct elements are incomparable in~$P$. 

\begin{figure}[tp]
  \begin{center}
    \includegraphics{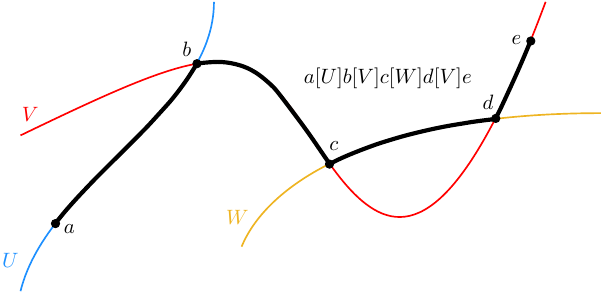}
  \end{center}
  \caption{
    A sample application of the notation for path concatenation.
  }
  \label{fig:concatenating_paths}
\end{figure}


The \emph{comparability graph} of \(P\) is a graph on the ground set of \(P\) in which two distinct vertices are adjacent if and only if they are comparable in \(P\). 
An element \(x\) of~$P$ is \emph{covered} by an element \(y\) of~$P$ (and \(y\) \emph{covers} \(x\)) in \(P\) if
\(x < y\) in~$P$ and there is no element \(z\) of~$P$ with \(x < z < y\) in \(P\).
The \emph{cover graph} of \(P\)
is a subgraph of the comparability graph in which two vertices are adjacent if either one of them covers the other.
Observe that the dual \(P^{-1}\) has the same comparability graph and cover graph as \(P\).

Two posets $P$ and $Q$ are \emph{isomorphic} if there exists a bijection $f$ from the ground set of~$P$ to the ground set of $Q$ such that for all elements $x$ and $y$ of~$P$, we have $x \leq y$ in~$P$ if and only if $f(x) \leq f(y)$ in $Q$.
For a subset $Y$ of the ground set of~$P$, the \emph{subposet induced} by $Y$ in~$P$ is a poset $Q$ with the ground set $Y$ such that for all $x,y \in Y$, we have $x \leq y$ in $Q$ if and only if $x \leq y$ in~$P$.
A poset $Q$ is a \emph{subposet} of a poset \(P\) if it is a subposet induced by some subset of the ground set of~$P$ in~$P$.
A subposet \(Q\) of a poset \(P\) is \emph{convex} if for all elements \(x,y,z\) with \(x < y < z\) in \(P\), if \(x\) and \(z\) are elements of \(Q\), then also \(y\) is an element of \(Q\). 
Note that if \(Q\) is a convex subposet of \(P\), then the cover graph of \(Q\) is a subgraph of the cover graph of \(P\).

For an element \(x\) of \(P\), we denote by \(P - x\)
the subposet of \(P\) induced by the set of all elements of \(P\) except \(x\).

A \emph{finite} poset is a poset with a finite ground set.
All posets in this paper are finite, unless stated otherwise. 
A \emph{component} of a poset is a subposet induced by a connected component of its cover graph.
A poset is \emph{connected} if it has exactly one component.

\subsection{Dimension of posets}\label{sec:dim}
In the introduction, we presented a concise geometric definition of
dimension of posets. 
However, we (and most other researchers) work with a combinatorial equivalent. 
In this section, we discuss equivalent definitions of dimension and prove its basic properties.

Let \(P\) be a poset.
A \emph{linear extension} of \(P\) is a linearly ordered set $L$ on the same ground set as $P$ such that $x\leq y$ in~$P$ implies $x\leq y$ in $L$ for all elements $x$ and $y$ of~$P$.
Every poset has a linear extension, and furthermore, for all incomparable elements \(x\) and \(y\) of~$P$ there exists a linear extension $L$ with \(x < y\) in $L$
(we prove an even stronger statement in \cref{prop:dim-alternating-cycles}). 
Therefore, if $\calL$ is the set of all linear extensions of~$P$ then for all elements $x$ and $y$ of~$P$, we have
\begin{equation}
\label{eq:all-linear-extensions}
\textrm{$x\leq y$ in~$P$ if and only if  $x\leq y$ in $L$ for each $L\in\calL$.} 
\end{equation}

Recall that the dimension of a poset \(P\) is the least positive integer \(d\) such that \(P\) is isomorphic to
a subposet of \(\mathbb{R}^d\) (ordered by the product order).
If \(Q\) is a finite subposet of \(\mathbb{R}^d\), then by slightly perturbing the coordinates of the points we can obtain an isomorphic subposet \(Q'\) of \(\mathbb{R}^d\) such that for each \(i \in [d]\), the \(i\)th coordinates of the elements of \(Q'\) are pairwise distinct. 
Then for each \(i \in [d]\), the linear order on the ground set of \(Q'\) given by the increasing \(i\)th coordinate yields a linear extension of \(Q'\).
This justifies an equivalent, combinatorial definition of dimension.

The dimension of \(P\) is the least positive integer \(d\) for which there exist \(d\) linear extensions
\(L_1, \ldots, L_d\) of \(P\) such that for all elements $x$ and $y$ of \(P\), we have
\begin{align*}
\textrm{$x\leq y$ in~$P$ if and only if  $x\leq y$ in $L_i$ for each $i\in[d]$.} 
\end{align*}
Indeed, a linear extension \(L_i\) can be seen as the order of elements in the \(i\)th coordinate in \(\mathbb{R}^d\),
and conversely, the order of elements in the \(i\)th coordinate in \(\mathbb{R}^d\) forms a linear extension. 
By~\eqref{eq:all-linear-extensions}, the dimension is well-defined.
In fact, Hiraguchi~\cite{H55} showed that an \(n\)-element poset with \(n \ge 4\), has dimension at most \( n/2 \). 

For a poset \(P\), let \(\Inc(P)\) denote the set of all ordered pairs \((x, y)\) of elements of~$P$ with \(x \parallel y\) in \(P\).
A subset $I \subseteq \Inc(P)$ is \emph{reversible} in \(P\) if there exists a linear extension $L$ of~$P$ with
$y < x$ in $L$ for all $(x,y)\in I$.
A family \(\calS\) of subsets of \(\Inc(P)\) \emph{covers $I$} if $I \subset \bigcup \calS$.
In this case, sometimes we say that $\calS$ is a \emph{covering} of $I$.
We state arguably the most useful equivalent definition of dimension of posets.


\begin{obs} \label{prop:redefine_dimension}
    For every poset \(P\), the dimension of~$P$ is the least positive integer $d$ for which 
          $\Inc(P)$ can be covered by $d$ reversible sets.
\end{obs}


Let $P$ be a poset.
Motivated by the above, for every \(I \subseteq \Inc(P)\), we define the \emph{dimension} of $I$ in~$P$, denoted by $\dim_P(I)$, as the minimum positive integer $d$ such that $I$ can be covered by $d$ reversible sets.
Note that by \cref{prop:redefine_dimension}, $\dim(P)=\dim_P(\Inc(P))$. 

Similarly, we define a restricted version of the standard example number.
We say that a set $J\subseteq \Inc(P)$ \emph{induces} a standard example in~$P$ if $|J| \geq 2$ and for all distinct $(a,b), (a',b')\in J$, 
we have $a < b'$ and $a'<b$ in~$P$.  
For every $I\subseteq\Inc(P)$, 
let $\se_P(I)$ be defined as $1$ when there is no subset of $I$ inducing a standard example in~$P$; otherwise, $\se_P(I)$ is the maximum size of a subset of $I$ that induces a standard example in~$P$.
Note that $\se(P) = \se_P(\Inc(P))$.

It turns out that there is a relatively simple criterion to verify if a given $I \subset \Inc(P)$ is reversible.
For an integer $k$ with $k \geq 2$, a sequence $((a_1,b_1), \dots, (a_k,b_k))$ of pairs in $\Inc(P)$ is an \emph{alternating cycle} of size $k$ in~$P$ 
if $a_i\leq b_{i+1}$ in~$P$ for all $i\in[k]$, cyclically 
(that is, $b_{k+1} = b_1$).
We say that $I \subset \Inc(P)$ \emph{contains} an alternating cycle \(((a_1, b_1), \ldots, (a_k, b_k))\) if all pairs \((a_i, b_i)\) belong to  $I$.
An alternating cycle $((a_1,b_1),\dots,(a_k,b_k))$ is \emph{strict} if for all $i,j \in [k]$, we have
$a_i\le b_{j}$ in~$P$ if and only if $j=i+1$ (cyclically).
Note that in this case, $\{a_1, \dots, a_k\}$ and $\{b_1, \dots, b_k\}$ are $k$-element antichains in~$P$.
Note also that in alternating cycles, we allow that $a_i=b_{i+1}$ for some or even all values of $i$.
Trotter and Moore~\cite{TM77} made the following elementary observation that has proven over time to be far-reaching in nature.
\begin{proposition}\label{prop:dim-alternating-cycles}
Let $P$ be a poset and let $I \subset \Inc(P)$.
The following conditions are equivalent:
\begin{enumerate}
    \item\label{item:I-reversible} $I$ is reversible,
    \item\label{item:I-with-no-alternating-cycle} $I$ does not contain an alternating cycle,
    \item\label{item:I-with-no-strict-alternating-cycle} $I$ does not contain a strict alternating cycle.
\end{enumerate}
\end{proposition}
\begin{proof}
For the implication \ref{item:I-reversible} to \ref{item:I-with-no-alternating-cycle}, suppose that $I\subseteq \Inc(P)$ is a reversible set which contains an alternating cycle $((a_1,b_1), \ldots, (a_k, b_k))$. 
Since $I$ is reversible, there is a linear extension $L$ of~$P$ with $b_i < a_i$ in $L$ and $a_i \leq b_{i+1}$ in $L$ (cyclically) for all $i\in [k]$. 
Therefore, \[b_1 < a_1 \leq b_2 < a_2 \leq\cdots < a_k \leq b_1 \text{ in $L$}\] which is a contradiction. 

For the implication \ref{item:I-with-no-alternating-cycle} to \ref{item:I-reversible}, consider a set $I\subseteq \Inc(P)$ which does not contain an alternating cycle.  
Let $X$ be the ground set of~$P$. 
Define an auxiliary binary relation $R$ on $X$ where $(x,y) \in R$ if and only if $x\leq y$ in~$P$ or there is a sequence $((a_1, b_1), \ldots, (a_k,b_k))$ of elements from $I$ where $a_i \leq b_{i+1}$ in~$P$ for each $i\in [k-1]$, $x\leq b_1$ in~$P$, and $a_k \leq y$ in~$P$. 
We show that $R$ is a partial order.
Clearly $R$ is reflexive. 

Now suppose $(x,y) \in R$ and $(y,z) \in R$ for some elements $x,y,z$ in~$P$. 
If $x\leq y$ in~$P$ or $y\leq z$ in~$P$, it is quick to verify $(x,z) \in R$.
Otherwise, the concatenation of the witnessing sequences for $(x,y) \in R$ and $(y,z) \in R$ is a sequence that shows $(x,z) \in R$. 
Therefore, $R$ is transitive.

Next, we prove that $R$ is antisymmetric. 
Consider $x,y$ in~$P$ with $(x,y) \in R$ and $(y,x) \in R$. 
In the case $x< y$ in~$P$, then $y\not\leq x$ in~$P$ and so there is a sequence $((a_1, b_1), \ldots, (a_k,b_k))$ witnessing $(y,x) \in R$. 
Since $a_k \leq x < y \leq b_1$ in~$P$, the sequence is actually an alternating cycle with all of its pairs in $I$, a contradiction.
Thus, $x\not< y$ in~$P$ and similar reasoning shows that $y\not< x$ in~$P$.
As a result, we assume that $x \parallel y$ in~$P$.
So each of $(x,y) \in R$ and $(y,x) \in R$ is witnessed by a sequence of pairs in $I$.
The concatenation of these sequences is again an alternating cycle in $I$, a contradiction. 
Therefore, $x=y$ and $R$ is antisymmetric.

We have verified that $R$ partially orders $X$.
Let $Q = (X,R)$ and let $L$ be a linear extension of $Q$. 
Note that if $x\leq y$ in~$P$, then $x \leq y$ in $Q$.
Hence, $L$ is also a linear extension of~$P$. 
Furthermore, for each $(a,b)\in I$, the sequence $((a,b))$ witnesses that $b \leq a$ in $Q$, thus, $L$ reverses all pairs in $I$. 
Therefore, $I$ is reversible as desired.

The implication \ref{item:I-with-no-alternating-cycle} to \ref{item:I-with-no-strict-alternating-cycle} is clear.
To prove the implication \ref{item:I-with-no-strict-alternating-cycle} to \ref{item:I-with-no-alternating-cycle}, consider $I\subseteq \Inc(P)$ which contains an alternating cycle. 
Let $C = ((a_1, b_1), \ldots, (a_k,b_k))$ be such a cycle with minimum size. 
Suppose $C$ is not a strict alternating cycle. 
So there exist $i,j \in [k]$ such that $a_{i}$ is comparable to $b_j$ and $j \neq i+1$ (cyclically).
Without loss of generality, by rotating the cycle, we can assume that $i = 1$, and so, $j \notin \{1,2\}$.
If $a_1 \leq b_j$ in~$P$, then $((a_1,b_1), (a_j, b_j), \ldots, (a_k, b_k))$ is an alternating cycle in $I$ shorter than $C$. 
If instead $b_j \leq a_1$ in~$P$, then $a_{j-1} \leq b_{2}$ in~$P$, and so $2<j-1$. 
Thus $((a_{2},b_{2}), \ldots, (a_{j-1},b_{j-1}))$ is an alternating cycle shorter than $C$. 
In each case, we have a contradiction, so $C$ must be a strict alternating cycle in $I$. 
\end{proof}



Next, we develop some easy properties of dimension.
For every poset $P$ and for each $I \subset \Inc(P)$, recall that $I^{-1} = \{(b,a) : (a,b) \in I\}$.
Note that if $I \subset \Inc(P)$, then $I^{-1} \subset \Inc(P^{-1})$.
The following observation is straightforward.

\begin{obs}
    For every poset $P$ and $I \subset \Inc(P)$, we have $\dim_{P}(I) = \dim_{P^{-1}}(I^{-1})$.
\end{obs}

Let $P$ be a poset, $I\subseteq \Inc(P)$, 
and let $\{I_i : i \in [s]\}$ be a covering of $I$. 
The following inequalities are trivial but still very useful
\[
\textstyle\dim_P(I) \leq \dim_P\left(\bigcup_{i\in[s]} I_i \right) \leq \sum_{i\in[s]} \dim_P(I_i).
\]
The next proposition describes a situation when we get a stronger bound, i.e.\ $\dim_P(I)\leq \max\{\dim_P(I_i): i\in[s]\}$ (so somehow the problem becomes local with respect to the covering).

\begin{proposition}\label{prop:dim_of_sum_incomparable_sets}
    Let $P$ be a poset, let $I \subseteq \Inc(P)$, and let $I_1,\dots,I_s \subseteq I$ be pairwise disjoint.
    Assume that for every strict alternating cycle in~$P$ contained in $I$, there is $i \in [s]$ such that all the pairs in the strict alternating cycle are in $I_i$.
    Then, $\dim_P(I) = \max\{ \dim_P(I_i): i \in [s]\}$.
\end{proposition}
\begin{proof}
    Let $d = \max  \{\dim_P(I_i) : i \in [s]\} $. 
    Clearly, $\dim_P(I) \geq d$, thus, we focus on the other inequality. 
    Since for each $i \in [s]$, $\dim_P(I_{i}) \leq d$, there exists a covering $\{I_{i,j} : j\in [d]\}$ of $I_{i}$ consisting of reversible sets in~$P$.
    For each $j \in [d]$, let $J_j = \bigcup_{i=1}^s I_{i,j}$.
    Fix $j\in[d]$. 
    We claim that $J_j$ is reversible.
    Suppose to the contrary that there is a strict alternating cycle in~$P$ contained in $J_j$. 
    Since $I_1,\dots,I_s$ are pairwise disjoint, for every $i \in [s]$, we have $J_j\cap I_i \subset I_{i,j}$.
    Now, by the assumption, there is $i \in [s]$ such that all the pairs from the cycle are in $I_i$. 
    Therefore, all pairs of the cycle are in $I_{i,j}$. 
    This contradicts the fact that $I_{i,j}$ is reversible and completes the proof.
\end{proof}

Back in the 1950s, Hiraguchi~\cite{H55} proved that if a poset \(P\) is the union of
disjoint chains, then \(\dim(P) \le 2\), and otherwise
$\dim(P)$ equals the maximum dimension of a component of~$P$. 
We will use a slightly refined version of this statement devised for subsets of incomparable pairs.


\begin{proposition}\label{prop:dim-components}
Let $P$ be a poset, and let $I\subseteq \Inc(P)$ with $\dim_P(I) \geq 3$. Then there exists a component \(C\) of \(P\)
such that $\dim_P(I) = \dim_C(I_C)$,
where \(I_C = I \cap \Inc(C)\).
\end{proposition}
\begin{proof}
    Let $C_1,\dots,C_s$ be components of~$P$, 
    let $I_{C_i} = I \cap \Inc(C_i)$, and
    $d_i = \dim_{C_i}(I_{C_i})$, for each $i \in [s]$.
    Let $d = \max  (\{d_i : i \in [s]\} \cup \{2\})$. 
    Since for each $i \in [s]$, $\dim_P(I_{C_i}) \leq d$, there exists a covering $\{I_{i,j} : j \in [d]\}$ of $I_{C_i}$ consisting of $d$ reversible sets in $C_i$.
    Next, we define $J_1 = \{(a,b) \in I : a \text{ in } C_i, \ b \text{ in } C_j, \ \text{and} \ i < j\}$ and $J_2 = \{(a,b) \in I : a \text{ in } C_i, b \text{ in } C_j, \ \text{and} \ i > j\}$.
    Note that $\{J_1,J_2\} \cup \{I_{i,j} : i \in [s], j \in [d]\}$ is a covering of $I$.

    We need the following abstract observation. 
    Consider $K \subset \Inc(P)$ such that 
    \begin{enumerate}
    \item if $(a,b) \in K$ with $a \text{ in } C_i$, $b \text{ in } C_j$, then $i \leq j$;
    \item $K \cap \Inc(C_i)$ is reversible in~$P$, for each $i \in [s]$.
    \end{enumerate}
    We claim that $K$ is reversible in~$P$.
    Let $((a_1,b_1),\dots,(a_k,b_k))$ be an alternating cycle in~$P$ contained in $K$. 
    Let $i_{\alpha}\in[s]$ be such that $a_{\alpha}\text{ is in } C_{i_\alpha}$, for each $\alpha\in[k]$. 
    Fix $\alpha \in [k]$.
    Since $a_{\alpha}\leq b_{\alpha+1}$ in~$P$, it follows that $b_{\alpha + 1} \text{ is in } C_{i_\alpha}$, and therefore, $i_{\alpha+1} \leq i_\alpha$ (cyclically).
    Thus, $i_1 = \dots = i_k$. However, this is a contradiction with the fact that $K \cap \Inc(C_{i_\alpha})$ is reversible.
    This completes the proof that $K$ is reversible in~$P$.

    For each $j \in [d]$, let
    \[
        I_j = 
        \begin{cases}
            \bigcup_{i=1}^s I_{i,j} \cup J_j \ &\text{if} \ j \in \{1,2\},\\
            \bigcup_{i=1}^s I_{i,j} \ &\text{otherwise}.
        \end{cases}
    \]
    Note that $\set{I_j: j\in[d]}$ is a covering of $I$. 
    By the observation above, $I_j$ is reversible in~$P$ for each $j \in [d]$, and so $\dim_P(I) \leq d$.
    It follows that $3 \leq \dim_P(I) \leq d$, and so, $d \geq 3$.
    In particular, $d = d_i$ for some $i \in [s]$.
    Since $\dim_P(I) \geq d_i$ for every $i \in [s]$ and $\dim_P(I) \geq 2$, we have $\dim_P(I) \geq d$.
    Altogether, $\dim_P(I) = d$, and thus, $\dim_P(I) = d_i = \dim_{C_i}(I_{C_i})$ for some $i \in [s]$.
\end{proof}

\subsection{Topology and planarity}
For a set $S$ of points in the plane, we write $\partial S$ for the topological boundary of $S$ and $\Int S$ for the topological interior of $S$.
Additionally, we define the \emph{exterior} of $S$ as the interior of the complement of $S$ in the plane.

A \emph{simple curve} $\gamma$ in the plane is the image of an injective continuous map of a closed segment into the plane. 
In this case, the \emph{endpoints} of $\gamma$ are the images of the endpoints of the segment, while the \emph{interior} of $\gamma$ is the set of non-endpoint points in \(\gamma\).
We say that $\gamma$ \emph{connects} its endpoints.
A \emph{simple closed curve} in the plane is the image of an injective continuous map of a circle into the plane. 
Since all combinatorial objects considered in this paper are finite, we always assume that curves are finite unions of segments in the plane.

Let $\gamma$ be a simple closed curve.
The Jordan Curve Theorem\footnote{Since in this paper we only consider curves that are finite unions of segments, we only need the Jordan Curve Theorem for Polygons~\cite[Chapter~4]{Diestel-book}.} states that the complement of~$\gamma$ in the plane consists of two arc-connected components, one bounded $B$ and one unbounded~$U$. 
Moreover, the boundary of each of the components is equal to $\gamma$.
We define the \emph{region} of $\gamma$ as the union of $\gamma$ and $B$. 
In particular, if $\calR$ is the region of $\gamma$, then $\Int \calR = B$, $\partial\calR = \gamma$, and the exterior of $\calR$ is $U$.

Let us state and prove one simple topological observation.
\begin{proposition}\label{obs:region_containment}
    Let $\gamma_1$ and $\gamma_2$ be simple closed curves and let $\Gamma_1$ and $\Gamma_2$ be the regions of $\gamma_1$ and $\gamma_2$ respectively.
    If $\partial \Gamma_1 \subset \Gamma_2$, then $\Gamma_1 \subset \Gamma_2$.
\end{proposition}

\begin{proof}
    Since $\partial \Gamma_1 \subset \Gamma_2$, the exterior of \(\Gamma_2\) is disjoint from \(\partial \Gamma_1\). Since the exterior of \(\Gamma_2\) is connected, it must be contained in some component of the complement of $\partial \Gamma_1$ in the plane. 
    Since the exterior of \(\Gamma_2\) is unbounded, it must be contained in the exterior of \(\Gamma_1\). 
    Thus, \(\Gamma_1 \subseteq \Gamma_2\).
\end{proof}

A \emph{drawing} of a graph $G$ is a function $f$ assigning a subset of the plane to each vertex and each edge of $G$ such that $f(u)$ is a point in the plane for every vertex $u$ in $G$ and no two vertices are assigned to the same point; for every edge $uv$ in $G$, $f(uv)$ is a curve in the plane connecting $f(u)$ and $f(v)$ disjoint from $f(w)$ for every vertex $w$ of $G$ distinct from $u$ and $v$.
We say that a drawing $f$ of a graph $G$ is \emph{planar} if the interiors of the images of edges of $G$ under $f$ are pairwise disjoint.
A graph \(G\) is \emph{planar} if it admits a planar drawing.
A \emph{plane graph} is a planar graph with a fixed planar drawing.
In plane graphs, we usually identify vertices with the assigned points in the plane and edges with the assigned curves in the plane.
More generally, we identify subgraphs of plane graphs with unions of respective vertices and edges in the plane.
For example, every path in a plane graph is a simple curve, and every cycle in a plane graph is a simple closed curve.
The complement of a plane graph $G$ in the plane is a union of arc-connected components.
The topological closure of such a component is a
\emph{face} of $G$.
The only unbounded face is called the \emph{exterior face} of $G$.

\subsection{Ordering edges in a plane graph} \label{ssec:ordering-edges}

Let $G$ be a plane graph. 
Let $u$ be a vertex of $G$.
There is a natural clockwise cyclic ordering of 
the edges of $G$ incident to $u$.
When $e_0,e_1,\dots,e_\ell$ are edges (not necessarily distinct) 
incident to $u$, we will write \emph{$e_0\preccurlyeq e_1\preccurlyeq \dots \preccurlyeq e_\ell$ in the $u$-ordering}
if starting with $e_0$ and proceeding in the clockwise manner around $u$, stopping at $e_\ell$ (i.e.\ at the first occurrence of $e_{\ell}$ or at the first return to $e_0$ when $e_\ell = e_0$), we have visited the edges $e_1,\dots,e_{\ell-1}$ in this order.
We can replace some of the $\preccurlyeq$ symbols with~$\prec$ when the corresponding edges are distinct.


Sometimes, it is handy to cut the $u$-ordering.
That is, for a given edge $e_0$ incident to $u$, there is a clockwise linear order on the edges incident to $u$ with $e_0$ the least element.
When $e$ and $e'$ are edges incident to $u$ 
, we say that \emph{$e \preccurlyeq e'$ in the $(u,e_0)$-ordering} if either $e_0 = e$ or $e_0\prec e\preccurlyeq e' \prec e_0$ in the $u$-ordering.
We can replace $\preccurlyeq$ with $\prec$ when the edges $e,e'$ are distinct.
See an example on the left side of \cref{fig:ue-orderings}.

\begin{figure}[tp]
     \centering
     \begin{subfigure}[t]{0.49\textwidth}
         \centering
         \includegraphics{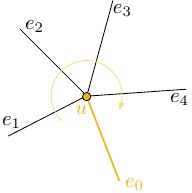}
     \end{subfigure}
     \begin{subfigure}[t]{0.49\textwidth}
         \centering
         \includegraphics{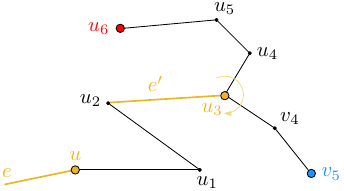}
     \end{subfigure}
        \caption{
            Left: We have $e_0 \prec e_1 \prec e_2 \prec e_3 \prec e_4$ in the $u$-ordering as well as in the $(u,e_0)$-ordering.
            On the other hand, $e_4 \prec e_0 \prec e_1$ in the $u$-ordering, which is not the case in the $(u,e_0)$-ordering.
            Right: Let $U=uu_1u_2u_3u_4u_5u_6$ and let $V = uu_1u_2u_3v_4v_5$. 
            Both paths start in $u$, and neither is a prefix of the other.
            We have, $u[U]u_3 = u[V]u_3$.
            Let $e' = u_2u_3$.
            Then, $u_3u_4 \prec u_3v_4$ in the $(u_3,e')$-ordering, hence, $U \prec V$ in the $(u,e)$-ordering.
            Moreover, the sets $\{u_4,u_5,u_6\}$ and $\{v_4,v_5\}$ are disjoint, thus, $U$ and $V$ are $u$-consistent.
        }
        \label{fig:ue-orderings}
\end{figure}

Let $u$ be a vertex of $G$, and let $e$ be 
an edge 
incident to $u$.
The $(u,e)$-ordering can be extended to the family $\calP$ of paths in $G$ starting in $u$ 
in the following way.
Let $U,V \in \calP$ with $U=u_0 \cdots u_\ell$, $V = v_0\cdots v_{m}$.
Let $i$ be the maximum nonnegative integer such that $u_0[U]u_i = v_0[V]v_i$, and let $e_0 = e$ if $i = 0$ and $e_0 = u_{i-1}u_i$ otherwise.
We say that \emph{$U \preccurlyeq V$ in the $(u,e)$-ordering} if 
\[\text{$U = V$ \ \ or \ \ $i < \ell$, $i < m$, and $u_iu_{i+1} \prec u_iv_{i+1}$ in the $(u_i,e_0)$-ordering.}\]
In the latter case, we write $U \prec V$ in the $(u,e)$-ordering.
Note that if one of $U$ and $V$ is a proper prefix of the other, then neither of $U \prec V$ and $V \prec U$ holds in the $(u,e)$-ordering.
Moreover, if $\{u_{i+1},\dots,u_\ell\}$ is disjoint from $\{v_{i+1},\dots,v_m\}$, then we say that $U$ and $V$ are \emph{$u$-consistent}. 
See an illustration of these notions on the right side of \cref{fig:ue-orderings}.

\begin{proposition}\label{prop:sandwiched-paths-basic}
    Let $G$ be a plane graph, let $u$ be a vertex of $G$, let $e$ be an edge incident to~$u$.
    Let $U,V,W$ be paths in $G$ starting in $u$ such that $U \prec V$ and  $V \prec W$ in the $(u,e)$-ordering.
    If $v$ is a common vertex of $U$ and $W$ with $u[U]v = u[W]v$, then $v$ is a vertex of $V$ and $u[U]v = u[V]v =u[W]v$.
\end{proposition}
\begin{proof}
    Let $U = u_0\cdots u_\ell$, $V = v_0 \cdots v_m$, and let $i$ be such that $u_i = v$.
    Let $j$ be the maximal nonnegative integer such that $u[U]u_j = u[V]u_j$.
    If $i \leq j$, the assertion holds.  
    Thus, suppose that $j < i$.
    In this case $u_ju_{j+1}$ lies in both $U$ and $W$. 
    Let $e_0 = e$ if $j = 0$ and $e_0 = u_{j-1}u_j$ otherwise.
    Since $U \prec V$ in the $(u,e)$-ordering, $j<m$ and 
    $u_ju_{j+1} \prec u_jv_{j+1}$ in the $(u_j,e_0)$-ordering.
    Since $V \prec W$ in the $(u,e)$-ordering, $u_jv_{j+1} \prec  u_ju_{j+1} $ in the $(u_j,e_0)$-ordering. 
    This contradiction completes the proof.
\end{proof}

\begin{proposition}\label{prop:u_e_is_poset}
    Let $G$ be a plane graph, let $u$ be a vertex of $G$, let $e$ be 
    an edge 
    incident to~$u$, and let $\calP$ be a family of paths in $G$ starting in $u$.
    Then, the $(u,e)$-ordering partially orders $\calP$.
\end{proposition}
\begin{proof}
    Let $U,V,W \in \calP$.
    The $(u,e)$-ordering is clearly reflexive and antisymmetric on $\calP$.
    It suffices to show that it is transitive, namely, if $U \prec V$ and $V \prec W$ in the $(u,e)$-ordering, then  $U \prec W$ in the $(u,e)$-ordering.
    Let $U = u_0\cdots u_\ell, V = v_0\cdots v_m, W = w_0\cdots w_n$, and let $k$ be the maximum nonnegative integer such that $u_0[U]u_k = w_0[W]w_k$.
    By~\Cref{prop:sandwiched-paths-basic}, $u_k$ is a vertex of $V$ and $u_0[U]u_k = v_0[V]v_k = w_0[W]w_k$.
    Since $U \prec V$ and $V \prec W$ in the $(u,e)$-ordering, $k < \ell$, $k < m$, and $k < n$.
    Let $e_0 = e$ if $k = 0$ and $e_0 = u_{k-1}u_k$ otherwise.
    We have $u_ku_{k+1} \preccurlyeq v_k v_{k+1}$ and $v_kv_{k+1} \preccurlyeq w_kw_{k+1}$ in the $(u_k,e_0)$-ordering.
    Therefore, $u_ku_{k+1} \preccurlyeq w_kw_{k+1}$ in the $(u_k,e_0)$-ordering, and so, by the definition of $k$, $u_ku_{k+1} \prec w_kw_{k+1}$ in the $(u_k,e_0)$-ordering.
    This yields $U \prec W$ in the $(u,e)$-ordering as desired.    
\end{proof}


For emphasis, we state the following straightforward observation.

\begin{obs} \label{obs:ue_ordering_is_usually_linear}
    Let $G$ be a plane graph, let $u$ be a vertex of $G$, and let $e$ be 
    an edge 
    incident to $u$. 
    Every pair of paths in $G$ starting in $u$ is comparable in the $(u,e)$-ordering unless one is a subpath of the other.
    In particular, if $\calP$ is a family of paths in $G$ starting in $u$ such that no path is a subpath of the other, then the $(u,e)$-ordering linearly orders $\calP$.
\end{obs}

\subsection{Unfolding}
\label{ssec:unfolding}

Given a poset \(Q\) and an element \(x\) of \(Q\), we say that a set \(J \subseteq \Inc(Q)\) is \emph{singly constrained} by \(x\) in $Q$ if for every \((a, b) \in J\) we have \(x \le b\) in \(Q\).
In 2014, Streib and Trotter \cite{ST14} introduced a concept that is now known as ``poset unfolding'', 
which allows reducing the problem of bounding the dimension of a poset to bounding the dimension
of a singly constrained set of pairs in a poset whose cover graph is a minor of the cover graph of the original poset. 
In particular, they showed that for every poset \(P\) with a planar cover graph, there exist a poset \(Q\) with a planar cover graph, an element \(x\) in $Q$ such that \(Q - x\) is a subposet of \(P\) or \(P^{-1}\), and a set \(J \subseteq \Inc(Q)\) that is singly constrained by $x$ in $Q$ with 
    \[\dim(P) \le 2\dim_Q(J).\]
In this subsection, we introduce all the necessary tools, and we reprove (with a slightly adjusted setup) the reduction of Streib and Trotter.


Let $P$ be a poset. 
An \emph{upset} in \(P\) is a subset \(U\) of elements of \(P\) such that for
all elements \(x, y\) with \(x \le y\) in \(P\), if \(x \in U\), then \(y \in U\).
For each element \(x\) of \(P\), we denote by \(U_P[x]\) the upset in \(P\) consisting
of all elements \(y\) such that \(x \le y\) in \(P\).
Dually, a \emph{downset} in \(P\) is a subset \(D\) of elements of \(P\) such that for
all elements \(x, y\) with \(x \le y\) in \(P\), if \(y \in D\), then \(x \in D\).
For each element \(y\) of \(P\), we denote by \(D_P[y]\) the downset in \(P\) consisting
of all elements \(x\) such that \(x \le y\) in \(P\).
For a set of elements \(Z\) in \(P\), we denote by \(D_P[Z]\) and \(U_P[Z]\) the
unions \(\bigcup_{z \in Z} D_P[z]\) and \(\bigcup_{z \in Z} U_P[z]\), respectively.
Note that every upset and downset in a poset induces a convex subposet.

An \emph{unfolding} of a poset \(P\) is a family \((Z_i : i \in \NN)\) of pairwise disjoint subsets of the ground set of \(P\) such that for all nonnegative integers $i,j$ and elements \(x \in Z_i, y \in Z_j\) with \(x \le y\) in \(P\), either
\(i = j\), or \(|i-j| = 1\) and \(i\) is even (and thus \(j\) is odd).
For convenience, for every nonnegative integer $k$, we write $Z_{\geq k} = \bigcup_{i \geq k} Z_i$ and $Z_{< k} = \bigcup_{0 \leq i < k} Z_i$ where the unions go over integers.
Note that for each \(k\), the set \(Z_k\) is a downset if \(k\) is even, and
an upset if \(k\) is odd. 
Moreover,
$Z_{\geq k}$ and $Z_{< k+1}$ are downsets in~$P$ whenever $k$ is even, and they are upsets in~$P$ whenever $k$ is odd.
In particular, $Z_{\geq k}$ and $Z_{< k+1}$ induce convex subposets of~$P$.

An unfolding of a poset is somewhat a parallel of a layering of a graph.\footnote{A \emph{layering} of a graph $G$ is a family $(Z_i : i\in \NN)$ of pairwise disjoint subsets of $V(G)$ such that for every edge $uv$ of $G$, there exists $i \in \NN$ such that $\{u,v\} \subset Z_i \cup Z_{i+1}$.
Given a connected graph \(G\) and \(u,v \in V(G)\), the \emph{distance} between $u$ and $v$ is the minimum number of edges in a path in \(G\) with endpoints \(u\)
and \(v\).
Now, given a connected graph $G$ and a vertex $v$ of $G$, the \bfs\emph{-layering} of $G$ from $v$ is the sequence
\((Z_i : i \in \NN)\), where \(Z_i\)
is the set of all vertices at distance \(i\) from \(v\) in \(G\) for every nonnegative integer $i$.
} 
Next, we discuss a parallel of \textsc{bfs}-layerings for posets.
Let \(P\) be a connected poset, and let \(z_0\) be a minimal
element of \(P\). 
Let \((Z_i : i\in \NN)\) be the \textsc{bfs}-layering of the comparability 
graph of \(P\) from \(z_0\), that is, \(Z_0 = \{z_0\}\), and for each positive integer $k$,
\[
Z_k =
\begin{cases}
  U_P[Z_{k-1}] \setminus (Z_0 \cup \cdots \cup Z_{k-1})&\textrm{if \(k\) is odd,}\\
  D_P[Z_{k-1}] \setminus (Z_0 \cup \cdots \cup Z_{k-1})&\textrm{if \(k\) is even.}  
\end{cases}
\]
Hence, \((Z_i : i\in \NN)\) is an unfolding of~$P$, and we call it the unfolding of \(P\) \emph{from} \(z_0\).
See an example of an unfolding from a minimal element in \cref{fig:unfolding}.

\begin{figure}[tp]
    \centering 
    \includegraphics[scale=1]{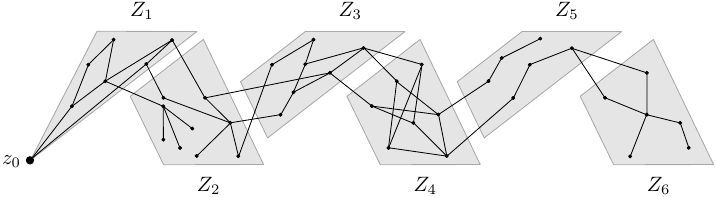} 
    \caption{The unfolding of a poset from $z_0$.
    The drawing is a diagram, that is, for two elements connected by an edge, the lower is less than the higher in the poset.
    In all the following figures, when a segment has no direction and it is not stated otherwise, the comparability is diagram-like (go upwards).
    } 
    \label{fig:unfolding} 
\end{figure} 

Let \((Z_i : i \in \NN)\) be an unfolding of~$P$.
A set \(I \subseteq \Inc(P)\) is
\emph{supported} by
\((Z_i : i \in \NN)\) if there exists a nonnegative integer \(k\) such that 
\[
    \bigcup_{(a,b) \in I}\set{a,b} \subseteq Z_{\geq k}\ \ \ \textrm{and}\ \ \ \begin{cases}\set{b : (a,b) \in I} \subseteq Z_k & \ \ \text{if} \ k \ \text{is odd},\\
        \set{a : (a,b) \in I} \subseteq Z_k & \ \ \text{if} \ k \ \text{is even}.\end{cases}
\]
If $k$ is odd, then we say that \(I\) is \emph{supported
from below} by $(Z_i : i \in \NN)$, if $k$ is even, then we say that \(I\) is \emph{supported
from above} by $(Z_i : i \in \NN)$.
In both cases we say that \(k\) \emph{witnesses} that \(I\) is supported by $(Z_i : i \in \NN)$.
See the set $J'$ in~\Cref{fig:unfolding-contracting}, which is supported from below, witnessed by $k = 3$.


It is well-known that
given a graph $G$ and a
layering of $G$, 
if every layer induces a $k$-colorable graph, 
then $G$ is \(2k\)-colorable (one can use two disjoint
palettes of \(k\) colors each, one for even layers, and one for odd layers).
The counterpart of the above for posets is formulated in terms of the unfolding.
Given a connected poset \(P\) and an unfolding of \(P\),
if the union of any two consecutive sets in the unfolding induces a subposet of dimension at most \(d\),
then the dimension of \(P\) is at most \(2d\).
We adapt the proof of this fact to reduce the problem of bounding \(\dim_P(I)\) for any \(I \subseteq \Inc(P)\)
to bounding \(\dim_P(I')\) for some \(I' \subseteq I\) which is supported by the unfolding of \(P\).

\begin{proposition}\label{prop:unfolding}
    Let \(P\) be a poset, 
    and let \((Z_i : i \in \NN)\) be an
    unfolding of \(P\).
    Then every set \(I \subseteq \Inc(P)\) 
    contains subsets \(I_0,I_1 \subseteq I\), where
    $I_0$ is supported from above by \((Z_i : i \in \NN)\) in~$P$,  
    $I_1$ is supported from below by \((Z_i : i \in \NN)\) in~$P$, and
    \[
      \dim_P(I) \le \dim_P(I_0)+\dim_P(I_1).
    \]
\end{proposition}

\begin{proof}
  Define
  \begin{align*}
      I_0'&=\{(a, b) \in I:\textrm{\(a\in Z_i,b\in Z_j\) with \(i < j\) or (\(i=j\) and \(i\) is even)}\}\\
      I_1'&=\{(a, b) \in I:\textrm{\(a\in Z_i,b\in Z_j\) with \(i > j\) or (\(i=j\) and \(i\) is odd)}\}.
  \end{align*}
    Note that $I=I_0'\cup I_1'$. 
    We will find 
    $I_0\subset I_0'$ supported from above by \((Z_i : i \in \NN)\) in~$P$ and
    $I_1 \subset I_1'$ supported from below by \((Z_i : i \in \NN)\) in~$P$,
    such that 
    $\dim_P(I_0') = \dim_P(I_0)$ and $\dim_P(I_1') = \dim_P(I_1)$. 
    This way, we will obtain
    \[
        \dim_P(I) \le \dim_P(I_0') + \dim_P(I_1') = \dim_P(I_0) + \dim_P(I_1),
    \]
    as desired.

    We start with the construction of $I_1\subset I_1'$.
    For each odd positive integer \(j\), define $I_{1,j} = \set{(a,b) \in I_1' : b \in Z_j}$.
    We show that each alternating cycle in \(P\) that has all pairs in \(I_1'\) is contained
    in \(I_{1,j}\) for some odd positive integer \(j\).
    Let $((a_1,b_1),\dots,(a_k,b_k))$ be an alternating cycle in \(P\) with all pairs in $I_1'$.
    For each $\alpha \in [k]$, let $i_{\alpha}, j_{\alpha} \in \NN$ be so that $a_\alpha \in Z_{i_{\alpha}}$ and $b_\alpha \in Z_{j_\alpha}$.
    We claim that for each $\alpha \in [k]$, we have \(j_\alpha \le j_{\alpha+1}\) cyclically.
    Since \((a_\alpha, b_\alpha) \in I_1'\), either \(i_\alpha > j_\alpha\), or
    \(i_\alpha = j_\alpha\) and \(i_\alpha\) is odd.
    Since \(a_\alpha \le b_{\alpha+1}\) in \(P\), we have \(|i_\alpha - j_{\alpha+1}| \le 1\),
    so if \(i_\alpha > j_\alpha\), then \(j_\alpha \le j_{\alpha+1}\).
    On the other hand, if \(i_\alpha = j_\alpha\) and \(i_\alpha\) is odd,
    then \(Z_{i_\alpha}\) is an upset in \(P\), so
    \(a_\alpha \le b_{\alpha+1}\) in \(P\) implies that
    \(b_{\alpha+1} \in Z_{i_\alpha}\), and thus, \(j_\alpha = i_\alpha = j_{\alpha+1}\).
    In both cases, \(j_\alpha \le j_{\alpha+1}\).
    This holds cyclically for all \(\alpha \in [k]\),
    so \(j_1 = \cdots = j_k\).
    In order to conclude the claim, it suffices to show that $j_1$ is odd.
    Suppose to the contrary that \(j_1\) is even.
    Since \(Z_{j_1}\) is a downset in~$P$ and $a_k\leq b_1$ in~$P$, 
    we have \(a_k \in Z_{j_1}\), so
    \(i_k = j_1 = j_k\).
    As \(j_1\) is even, this contradicts \((a_k, b_k) \in I_1'\).
    Hence, \(j_1\) is odd and \((a_\alpha, b_\alpha) \in I_{1,j_1}\) for each \(\alpha \in [k]\).
    By \cref{prop:dim_of_sum_incomparable_sets}, there exists
    an odd positive integer \(j\) such that \(\dim_P(I_{1,j}) = \dim_P(I_1')\). 
    Let $I_1 = I_{1,j}$.
    For each \((a, b) \in I_1\) we have \(b \in Z_j\) and \(a \in Z_{\geq j}\),
    so \(j\) witnesses that \(I_1\) is supported from below by \((Z_i : i \in \NN)\) in~$P$.
    
    Next, we construct $I_0 \subseteq I_0'$.
    Note that the construction is symmetric.
    For each even positive integer \(i\), define $I_{0,i} = \set{(a,b) \in I_0 : a \in Z_i}$.
    We show that each alternating cycle in \(P\) that has all pairs in \(I_0'\) is contained
    in \(I_{0,i}\) for some even positive integer \(i\).
    Let $((a_1,b_1),\dots,(a_k,b_k))$ be an alternating cycle in \(P\) with all pairs in $I_0'$.
    For each $\alpha \in [k]$, let $i_{\alpha}, j_{\alpha} \in \NN$ be so that $a_\alpha \in Z_{i_{\alpha}}$ and $b_\alpha \in Z_{j_\alpha}$.
    We claim that for each $\alpha \in [k]$, we have \(i_\alpha \le i_{\alpha-1}\) cyclically.
    Since \((a_\alpha, b_\alpha) \in I_0'\), either \(i_\alpha < j_\alpha\), or
    \(i_\alpha = j_\alpha\) and \(i_\alpha\) is even.
    Since \(a_{\alpha-1} \le b_{\alpha}\) in \(P\), we have \(|i_{\alpha-1} - j_{\alpha}| \le 1\),
    so if \(i_\alpha < j_\alpha\), then \(i_\alpha \le i_{\alpha-1}\).
    On the other hand, if \(i_\alpha = j_\alpha\) and \(j_\alpha\) is even,
    then \(Z_{j_\alpha}\) is a downset in \(P\), so
    \(a_{\alpha-1} \le b_{\alpha}\) in \(P\) implies that
    \(a_{\alpha-1} \in Z_{j_\alpha}\), and thus, \(i_\alpha = j_\alpha = i_{\alpha-1}\).
    In both cases, \(i_\alpha \le i_{\alpha-1}\).
    This holds cyclically for all \(\alpha \in [k]\),
    so \(i_1 = \cdots = i_k\).
    In order to conclude the claim, it suffices to show that $i_1$ is even.
    Suppose to the contrary that \(i_1\) is odd.
    Since \(Z_{i_1}\) is an upset in~$P$ and $a_1\leq b_2$ in~$P$, 
    we have \(b_2 \in Z_{i_1}\), so
    \(j_2 = i_1 = i_2\).
    As \(i_1\) is odd, this contradicts \((a_2, b_2) \in I_0'\).
    Hence, \(i_1\) is even and \((a_\alpha, b_\alpha) \in I_{0,i_1}\) for each \(\alpha \in [k]\).
    By \cref{prop:dim_of_sum_incomparable_sets}, there exists
    an even positive integer \(i\) such that \(\dim_P(I_{0,i}) = \dim_P(I_0')\). 
    Let $I_0 = I_{0,i}$.
    For each \((a, b) \in I_0\) we have \(a \in Z_i\) and \(b \in Z_{\geq i}\),
    so \(i\) witnesses that \(I_0\) is supported from above by \((Z_i : i \in \NN)\) in~$P$.
%
%
%
%
\end{proof}

Finally, we apply \cref{prop:unfolding} to posets with planar cover graphs.
Application of the following result is the initial step in the proof of the main result of this paper: \cref{thm:cover-graph_se}.

\begin{figure}[tp]
    \centering 
    \includegraphics{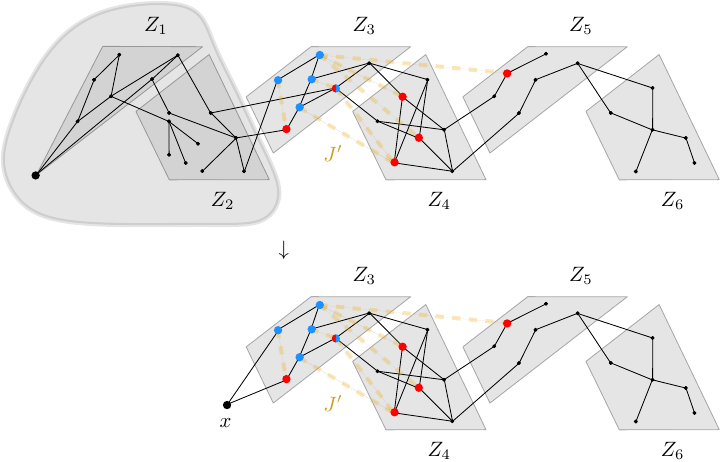} 
    \caption{
    With yellow lines, we denote the incomparable pairs in $J'$, which is supported from below by $(Z_i : i \in \NN)$. 
    We contract $Z_0 \cup Z_1 \cup Z_2$ into a vertex $x$. The cover graph is still planar, and $J'$ is singly constrained by $x$.
    } 
    \label{fig:unfolding-contracting} 
\end{figure} 

\begin{lemma}\label{lem:PlanarCoverGraphReduction}
    Let \(P\) be a poset with a planar cover graph and let $I \subset \Inc(P)$.
    Then there exist a poset \(Q\),
    a set \(J \subseteq \Inc(Q)\), and a minimal element \(x\)
    of \(Q\)
    such that
    \begin{enumerate}
        \item the cover graph of $Q$ is planar,
        \item \(Q - x\) is a convex subposet of \(P\) and $J \subset I$ or $Q-x$ is a convex subposet of \(P^{-1}\) and $J \subset I^{-1}$,
        \item \(J\) is singly constrained by \(x\) in $Q$, and
        \item \(\dim_P(I) \le 2\dim_Q(J)\).
    \end{enumerate}
\end{lemma}
\begin{proof}
  If \(\dim_P(I) \le 2\), then the lemma is satisfied by a poset \(Q\)
  consisting of one element \(x\) and \(J = \emptyset\).
  Hence, we assume that $\dim_P(I) \geq 3$.
  By \cref{prop:dim-components}, there is a component $C$ of~$P$ with $\dim_P(I) = \dim_C(I_C)$, where $I_C = \Inc(C) \cap I$.
  Thus, without loss of generality, we assume that $P$ is connected.
  Let \(z_0\) be a minimal element of \(P\), 
  and let \((Z_i : i \in \NN)\) be the unfolding of \(P\) from \(z_0\).
  By \cref{prop:unfolding} applied to \(I\),
  there exists a set \(J' \subseteq I\)
  supported by \((Z_i : i \in \NN)\), which is witnessed by the integer $k$ such that
  \[
    \dim_P(I) \le 2\dim_P(J').
  \]

  If $k = 0$, then for each $(a,b) \in J'$, we have $a = z_0$.
  Therefore, $J'$ is reversible in~$P$ as there is no strict alternating cycle in~$P$ containing only pairs in $J'$.
  It follows that $\dim_{P}(J') \leq 1$, hence, $\dim_{P}(I) \leq 2 \dim_P(J') \leq 2$, which contradicts $\dim_P(I)\geq3$.

  Next, we assume that $k > 0$. 
  Let $Y=Z_{< k}$. 
  Since $k > 0$, the subgraph of the cover graph of~$P$ 
  induced by $Y$ is nonempty.
  We claim that it is also connected.
  Indeed, let $P_Y$ be the subposet of~$P$ induced by $Y$. 
  By the properties of the unfolding from $z_0$, the comparability graph of $P_Y$ is connected, and $P_Y$ is a convex subposet of~$P$. 
  Since $P_Y$ is a convex subposet of~$P$, the cover graph of $P_Y$ coincides with the subgraph of the cover graph of~$P$ induced by~$Y$. 
  Since $P_Y$ is connected, its cover graph is connected as well, and thus, the subgraph of the cover graph of~$P$ induced by $Y$ is connected, as claimed.

  We obtain $Q'$ from $P$ by ``contracting'' the set \(Y\) into a single element \(x\).
  Formally, let the ground set of $Q'$ consist of all the elements of $Z_{\geq k}$ and a new element $x$.
  The order relations within $Z_{\geq k}$ are the same as in~$P$.
  If $k$ is odd, then let $x < z$ in $Q'$ for all $z\in Z_k$ and if $k$ is even, let $z < x$ in $Q'$ for all $z \in Z_k$, while $x\parallel z$ in $Q'$ for all $z\in Z_{\geq k+1}$ in both cases.
  See \cref{fig:unfolding-contracting}.

  Note that the neighborhood of $x$ in the cover graph of $Q'$ is the set of all minimal elements in the subposet of~$P$ induced by $Z_k$ if $k$ is odd and the set of all maximal elements in the subposet of~$P$ induced by $Z_{k}$ if $k$ is even.
  Moreover, since $Z_{\geq k}$ induces a convex subposet of~$P$, 
  the cover graph of $Q'$ is a subgraph of the graph $H$ obtained from the cover graph $G$ of~$P$ by contracting $Y$ into $x$ (note that it can be a proper subgraph as in~\Cref{fig:unfolding-contracting}).
  Since the subgraph of $G$ induced by $Y$ is nonempty and connected, $H$ is planar.
  In particular, the cover graph of $Q'$ is planar.
  
  Each $(a,b) \in J'$ remains an incomparable pair in $Q'$, moreover, the subposets induced by $\bigcup_{(a,b) \in J'} \{a,b\}$ are the same in~$P$ and $Q'$, and so, $\dim_P(J') = \dim_{Q'}(J')$.
  Finally, in the case, where $k$ is odd, $x \leq b$ in $Q'$ for all $(a,b) \in J'$, and we take $Q = Q', J = J'$.
  In the case, where $k$ is even, $a \leq x$ in $Q'$ for all $(a,b) \in J'$.
  Let $Q = (Q')^{-1}$ and let $J = (J')^{-1}$. 
  In both cases, $J$ is singly constrained by $x$ in $Q$ and $\dim_P(I) \leq 2 \dim_{Q'}(J') = 2\dim_Q(J)$.
  Note that $x$ is a minimal element in $Q$.
  %
  %
\end{proof}

\section{Topology of an instance}
\label{sec:topology_instance}

We say that $(P,x_0,G,e_{-\infty},I)$ is an \emph{instance} if
\begin{enumerateNumI}
    \item $P$ is a poset with a planar cover graph,\label{item:instance:planar_cover_graph}
    \item $x_0$ is a minimal element in~$P$,\label{item:instance:x_0_minimal}
    \item $G$ is a plane graph obtained from the cover graph of $P$ by fixing a planar drawing of $G$ with $x_0$ in the exterior face,\label{item:instance:plane-graph}
    \item $e_{-\infty}$ is a curve in the plane contained in the exterior face of $G$ such that the only common point of the curve and $G$ is $x_0$,\label{item:instance:e_infty}
    \item $I \subset \Inc(P)$ and $I$ is singly constrained by $x_0$ in~$P$.\label{item:instance:I_singly_constrained}
\end{enumerateNumI}

The notion of unfolding introduced in~\Cref{ssec:unfolding} and~\Cref{lem:PlanarCoverGraphReduction} allow us to reduce the main problem to studying instances as captured in the following statement.

\begin{corollary}\label{cor:poset-to-instance}
    For every poset $P$ with a planar cover graph and $I \subset \Inc(P)$, there exists an instance $(P',x_0,G,e_{-\infty},I')$ such that
    \[
    \dim_P(I) \leq 2\dim_{P'}(I').
    \]
    Moreover, $P' - x_0$ is a convex subposet of $P$ and $I' \subset I$, or $P' - x_0$ is a convex subposet of $P^{-1}$ and $I' \subset I^{-1}$.
\end{corollary}
\begin{proof}
    Let $P$ be a poset with a planar cover graph and $I \subset \Inc(P)$.
    By \cref{lem:PlanarCoverGraphReduction}, there exists a poset $P'$, a set $I' \subset \Inc(P')$, and a minimal element $x_0$ of $P'$ such that the cover graph of $P'$ is planar, \(P' - x_0\) is a subposet of \(P\) and $I' \subset I$ or $P'-x_0$ is a subposet of \(P^{-1}\) and $I' \subset I^{-1}$, $I'$ is singly constrained by $x_0$ in $P'$, and $\dim_P(I) \leq 2 \dim_{P'}(I')$.
    Let $G$ be a plane graph obtained from the cover graph of $P$ by fixing a planar drawing of $G$ with $x_0$ in the exterior face, and let $e_{-\infty}$ be a curve in the plane contained in the exterior face of $G$ such that the only common point of the curve and $G$ is $x_0$.
    It follows that $(P',x_0,G,e_{-\infty},I')$ is a desired instance.
\end{proof}

In this section, we develop a basic understanding of what an instance looks like.
We discuss the theory of shadows and formulate many useful properties of instances.
To this end, we fix an instance $(P,x_0,G,e_{-\infty},I)$ for the remainder of this section.

\subsection{Orientations and regions}\label{sec:orienting}
Let $G_{-\infty}$ be the plane graph obtained from $G$ by adding a new vertex $x_{-\infty}$ and a new edge $x_{-\infty}x_0$, which is identified with $e_{-\infty}$.
We call $G_{-\infty}$ the \emph{rooted graph} of the instance $(P,x_0,G,e_{-\infty},I)$.
When an instance is fixed (like in this section), every $(v,e)$-ordering of edges or paths (as defined in~\Cref{ssec:ordering-edges}) is considered in the rooted graph of the instance.

Let $\gamma$ be a cycle in $G$ and let $\Gamma$ be the region of $\gamma$.
An \emph{orientation} of $\gamma$ is a cyclic ordering of the edges of $\gamma$ such that incident edges are consecutive.
In particular, $\gamma$ has exactly two orientations: \emph{clockwise} and \emph{counterclockwise}.
If $e_1,\dots,e_n$ (cyclically) is an orientation of $\gamma$, then we say that $e_{i+1}$ \emph{follows} $e_i$ in this orientation and $e_{i-1}$ \emph{precedes} $e_i$ in this orientation (indices are considered cyclically in $[n]$).
%
The main feature of the counterclockwise orientation is enclosed in the following observation.

\begin{obs}\label{obs:when_in_gamma}
    Let $\gamma$ be a cycle in $G$ and let $\Gamma$ be the region of $\gamma$.
    Let $w$ be a vertex of $\gamma$ and let $e^-$ and $e^+$ be the edges of $\gamma$ incident to $w$ such that $e^+$ follows $e^-$ in the counterclockwise orientation in $\gamma$.
    Then, for every edge $e$ incident to $w$,
    \begin{align*}
        e \text{ is contained in } \Gamma \ \text{ if and only if } \ e^- \preccurlyeq e \preccurlyeq e^+  \text{ in the $w$-ordering}.
    \end{align*}
\end{obs}

\subsection{Leftmost and rightmost witnessing paths}\label{sec:leftmost-rightmost}
For two elements $u$ and $v$ of \(P\), a \emph{witnessing path} from \(u\) to \(v\) in~$P$ is a path \(w_0 \cdots w_n\) in \(G\) such that $w_0 = u$, $w_n = v$, and
\(w_{i-1}\) is covered by \(w_{i}\) for every \(i \in [n]\).
Observe that there exists a witnessing path from \(u\) to \(v\) in \(P\) if and only if \(u \le v\) in \(P\).

Let $U$ and $V$ be two paths in $G$ starting in $x_0$. 
We say that $U$ is \emph{left} of $V$ if $U \prec V$ in the $(x_0,e_{-\infty})$-ordering.
If $U$ is left of $V$, then we say that $V$ is \emph{right} of $U$.
Moreover, we define $\gce(U,V)$ to be the element $w$ in $U \cap V$, where $x_0[U]w = x_0[V]w$ and this prefix is the longest possible.
The abbreviation stands for \q{greatest common prefix-element}.

Let 
    \[B = U_P[x_0] =  \set{b \text{ in } P : x_0 \leq b \text{ in } P}.\]
Let $b \in B$ and let $\calP$ be the family of all witnessing paths from $x_0$ to $b$ in~$P$.
Note that $\calP$ is nonempty.
By \cref{obs:ue_ordering_is_usually_linear}, $\calP$ has a unique minimal element and a unique maximal element in the $(x_0,e_{- \infty})$-ordering. 
We call the minimal element the \emph{leftmost} witnessing path from $x_0$ to $b$ in~$P$ and we call the maximal element the \emph{rightmost} witnessing path from $x_0$ to $b$ in~$P$.
We denote the former $W_L(b)$ and the latter $W_R(b)$.
See \cref{fig:leftmost-rightmost}.
Note that by definition, for all witnessing paths $W$ from $x_0$ to $b$, either $W = W_L(b)$ or $W_L(b)$ is left of $W$.
Similarly, for all witnessing paths $W$ from $x_0$ to $b$, either $W = W_R(b)$ or $W_R(b)$ is right of $W$.
Next, we develop some properties of the leftmost and rightmost witnessing paths.

\begin{figure}[tp]
     \centering
     \begin{subfigure}[t]{1\textwidth}
         \centering
         \includegraphics{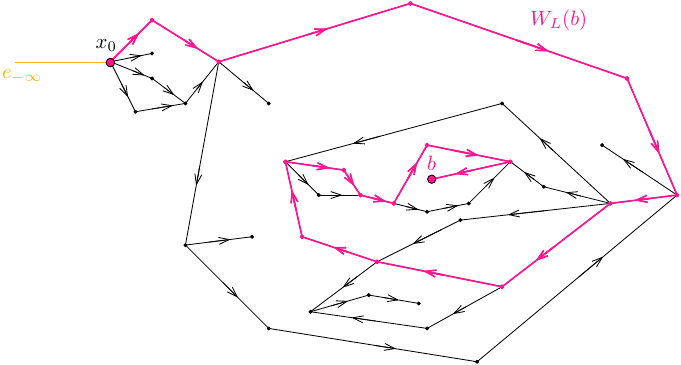}
     \end{subfigure}
     \begin{subfigure}[t]{1\textwidth}
         \centering
         \includegraphics{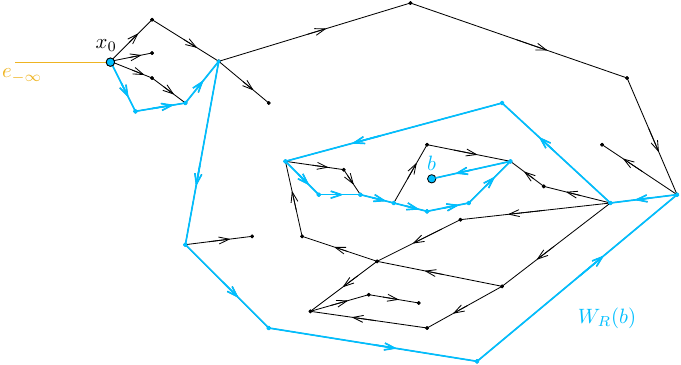}
     \end{subfigure}
        \caption{
            Recall that we fixed an instance $(P,x_0,G,e_{-\infty},I)$. 
            The figure illustrates a part of the drawing of $G$ induced by elements in $B = U_P[x_0]$.
            In the top part of the figure, we mark the edges and vertices of the leftmost witnessing path from $x_0$ to $b$, and in the bottom part, we mark the edges and vertices of the rightmost witnessing path from $x_0$ to $b$.
            We will use the poset depicted in this figure several times to illustrate various notions.
        }
     \label{fig:leftmost-rightmost}
\end{figure}

\begin{proposition}\label{prop:Tleft-and-Tright} \label{prop:W-consistent}
  Let $b,b'\in B$. 
  \begin{enumerate}
      \myitem{$(L)$} The paths $\Wleft(b)$ and $\Wleft(b')$ are $x_0$-consistent. \label{prop:W-consistent:left}
      \myitem{$(R)$} The paths $\Wright(b)$ and $\Wright(b')$ are $x_0$-consistent. \label{prop:W-consistent:right}
  \end{enumerate}
\end{proposition}
\begin{proof} 
    We prove that $W_L(b)$ and $W_L(b')$ are $x_0$-consistent.
    The proof for $W_R(b)$ and $W_R(b')$ is symmetric. 
    Suppose to the contrary that $W_L(b)$ and $W_L(b')$ are not $x_0$-consistent.
    Let $u = \gce(W_L(b),W_L(b'))$. 
    Since the paths are not $x_0$-consistent there exists an element $v$ distinct from $u$ in the intersection of $u[W_L(b)]b$ and $u[W_L(b')]b'$.
    Let $e_0 = e_{-\infty}$ when $u = x_0$ and let $e_0$ be the last edge in $x_0[W_L(b)]u$ when $u \neq x_0$.
    Let $e$ be the first edge of $u[W_L(b)]b$ and let $e'$ be the first edge of $u[W_L(b')]b'$.
    It follows that $e \neq e'$.
    Without loss of generality, assume that $e \prec e'$ in the $(u,e_0)$-ordering. 
    Consider the path $W = x_0[W_L(b)]v[W_L(b')]b'$.
    By definition, $W$ is left of $W_L(b')$, which is a contradiction.
\end{proof}
By \cref{prop:Tleft-and-Tright}, it follows that for all $u,b\in B$ if $u$ lies in $W_L(b)$, then $W_L(u)$ is a prefix of $W_L(b)$. 
Similarly, if $u$ lies in $W_R(b)$, then $W_R(u)$ is a prefix of $W_R(b)$. 
We will use this property implicitly many times.
The next statement is an immediate corollary of~\Cref{prop:sandwiched-paths-basic}.

\begin{corollary}\label{prop:sandwiched_paths}
    Let $W_1,W_2,W_3$ be pairwise $x_0$-consistent witnessing paths in~$P$ such that $W_1$ is left of $W_2$ and  $W_2$ is left of $W_3$.
    If an element $u$ in~$P$ lies in both $W_1$ and $W_3$, then $u$ lies in $W_2$.
\end{corollary}

\vbox{
\begin{proposition}\label{prop:shortcuts} 
  Let $b,b'\in B$.
  \begin{enumerate}
      \myitem{$(L)$} If $W_L(b)$ is left of $W_L(b')$, then $b \not\leq b'$ in~$P$. \label{prop:item:shortcuts-left-tree} 
      \myitem{$(R)$} If $W_R(b)$ is right of $W_R(b')$, then $b \not\leq b'$ in~$P$. \label{prop:item:shortcuts-right-tree} 
  \end{enumerate}
\end{proposition}
}
\begin{proof}
  We prove statement~\ref{prop:item:shortcuts-left-tree}.  
  The argument for statement~\ref{prop:item:shortcuts-right-tree} is symmetric.
  Suppose to the contrary that $W_L(b)$ is left of $W_L(b')$ and $b \leq b'$ in~$P$.
  Let $W$ be a witnessing path from $b$ to $b'$.
  Since $W_L(b)$ is left of $W_L(b')$, we have $x_0[W_L(b)]b[W]b'$ left of $W_L(b')$, which is a contradiction.
\end{proof}

We finish the subsection on witnessing paths with a statement which is useful in a definition of some specific regions in $G$ that we study in~\cref{ssec:regions}.

\begin{proposition}\label{prop:W_L-left-of-W_R}
Let $u,v\in B$ such that $W_L(u)$ is left of $W_L(v)$. 
Let $q$ be a common element of $W_L(u)$ and $W_R(v)$, 
let $e$ be the first edge of $q[W_L(u)]u$, and let $f$ be the first edge of $q[W_R(v)]v$.
Let $e^-$ be the last edge of $x_0[W_L(u)]q$ if $q\neq x_0$, 
and $e^-=e_{-\infty}$ if $q=x_0$. 
Then,
    \[e \preccurlyeq f \text{ in the $(q,e^-)$-ordering}.\]
\end{proposition}
\begin{proof}
    Suppose to the contrary that \(f \prec e\) in the \((q, e^-)\)-ordering. 
    Consequently, the witnessing path \(x_0[W_L(u)]q[W_R(v)]v\) is left of the witnessing path \(x_0[W_L(u)]q[W_L(u)]u = W_L(u)\).
    Hence \(W_L(v)\), which is the leftmost witnessing path from \(x_0\) to \(v\) in~$P$, is also left of \(W_L(u)\), a contradiction.
\end{proof}

\subsection{Blocks and shadows}
\label{ssec:shadows}
Let $x,y$ be elements of $P$ with $x < y$ in~$P$. 
Let $U,V$ be two witnessing paths from $x$ to $y$ in~$P$ such that the only common elements of $U$ and $V$ are $x$ and $y$.  
Note that there are two possibilities:
(a) $x$ is covered by $y$ in~$P$ and $U = V = xy$; and (b) each of $U$ and $V$ has some internal elements, and $U \cup V$ is a simple closed curve\footnote{Recall that witnessing paths are identified with curves in the drawing.} contained in $G$.
The \emph{block} enclosed by $U$ and $V$ is a subset $\cgB$ of the plane defined as $\cgB=U$ if (a) holds; 
$\cgB$ is the region of $U \cup V$ if (b) holds.
In the case, where (a) holds, we say that $\calB$ is \emph{degenerate}, and if (b) holds, we say that $\calB$ is \emph{non-degenerate}.
We will refer to the element $x$ as $\min\cgB$, while
the element $y$ will be $\max\cgB$.  
Since $\partial\calB = U \cup V$, for every element $w \in \partial \calB$, we have $\min\calB \leq w \leq \max\calB$ in~$P$.
In particular, the following holds.

\begin{obs}\label{obs:paths_and_blocks}
Let $\cgB$ be a block, let $u$ and $v$ be elements of $P$ with $u \leq v$ in~$P$ such that 
one of $u$ and $v$ is in $\cgB$ and the other is not in $\Int \cgB$. 
Then $\min\cgB \leq v$ in~$P$ and $u\leq \max\cgB$ in~$P$.
\end{obs}

Let $\cgB$ be a non-degenerate block enclosed by $U$ and $V$ that are witnessing paths from 
$\min\cgB$ to $\max\cgB$.  
We will designate one of $U$ and $V$ to be the \emph{left side} of $\cgB$, while the other path will be the \emph{right side} of $\cgB$. 
The distinction is made using the following convention.
If $e$ and $f$ are the edges of $U \cup V$ incident to $\min \cgB$ such that $f$ follows $e$ in the counterclockwise orientation of $U \cup V$, then the path (one out of $U$ and $V$) containing $e$ is the left side and the path containing $f$ is the right side.
The elements $\min\cgB$ and $\max\cgB$ belong to
both sides.  
An element $u$ on the left side of $\cgB$, with $\min\cgB< u< \max\cgB$ in~$P$
is said to be \emph{strictly} on the left side of $\cgB$.  
Symmetrically, an element $u$ on the right side of $\cgB$, with $\min\cgB< u< \max\cgB$ in~$P$ 
is said to be \emph{strictly} on the right side of $\cgB$.  
When $\cgB$ is a degenerate block, we consider the whole block to be both the left side and the right side of 
$\cgB$, and there are no elements that are strictly on one of the two sides.

For the discussion to follow, we fix an element $b \in B$. 
We will introduce notation that defines sequences, paths, blocks, and elements of $P$, all depending on the choice of $b$.  
The elements of $P$ in $W_L(b)\cap W_R(b)$ form a chain containing
$x_0$ and $b$.
The elements of this chain can be listed
sequentially as $(z_0,\dots,z_m)$ such that $z_0 < \cdots < z_m$ in~$P$.  We will refer to
$(z_0,\dots,z_m)$ as the \emph{sequence of common points} of $b$.  Note that
$z_0=x_0$ and $z_m=b$.  
For all $i\in [m]$, 
we define $\cgB_i$ to be the block enclosed by 
$z_{i-1}[W_L(b)]z_i$ and $z_{i-1}[W_R(b)]z_i$. 
We will refer to the blocks obtained in such a way as \emph{shadow blocks} of~$b$. 
A \emph{shadow block} is a shadow block of an element in $B$. 
The sequence $(\cgB_1,\dots,\cgB_m)$ is called the \emph{sequence of blocks} of $b$.  
See \cref{fig:shadow-blocks}.
We prove the following seemingly obvious statement.

\begin{proposition}
\label{prop:seemingly-obvious}
    Let $\calB$ be a shadow block, 
    let $x=\min \calB$ and $y = \max \calB$. 
    Then $x[W_L(y)]y$ is the left side of $\calB$ while $x[W_R(y)]y$ is the right side of $\calB$.
\end{proposition}
\begin{proof}
    The statement holds when $\calB$ is degenerate, thus, assume otherwise.
    Let $e$ be the edge of $x[W_L(y)]y$ incident to $x$, and let $f$ be the edge of $x[W_R(y)]y$ incident to $x$.
    Since $x\leq z$ in~$P$ for every element $z$ in $\partial\calB$, 
    we get that the only common element of $W_L(x) = x_0[W_L(y)]x$ and $\partial\calB$ is $x$.
    Let $e^- = e_{-\infty}$ if $x = x_0$ and let $e^-$ be the last edge of $W_L(x)$ otherwise.
    By~\ref{item:instance:e_infty}, $e^-$ is contained in the exterior of $\calB$.
    Therefore, by~\Cref{obs:when_in_gamma}, $e^- \prec e \prec f$ in the $x$-ordering. In particular, $f$ follows $e$ in the counterclockwise orientation of $\partial \calB$, so $x[W_L(y)]y$ is the left side of $\calB$ and $x[W_R(y)]y$ is the right side of $\calB$.
\end{proof}

\begin{figure}[tp]
  \begin{center}
    \includegraphics{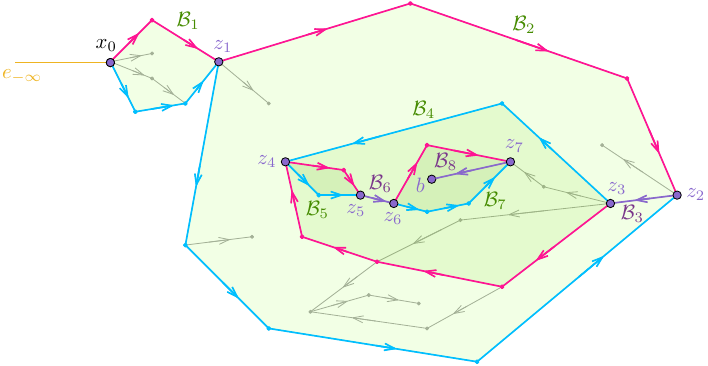}
  \end{center}
  \caption{
    We mark in purple the common edges and vertices of the paths $W_L(b)$ and $W_R(b)$, and we mark in red (resp.\ blue) the remaining edges and vertices of the path $W_L(b)$ (resp.\ $W_R(b)$).
    The purple vertices $z_1,\dots,z_7$ along with $z_0 = x_0$ and $z_8 = b_1$ form the sequence of common points of $b$.
    The shadow block $\calB_1$ is a non-degenerate block with the initial element $x_0$ and the terminal element $z_1$, the shadow block $\calB_2$ is a non-degenerate block with the initial element $z_1$ and the terminal element $z_2$, the shadow block $\calB_3$ is a degenerate block consisting of the edge $z_2z_3$, and so on.
    We obtain non-degenerate shadow blocks $\calB_1,\calB_2,\calB_4,\calB_5,\calB_7$ marked in green, and degenerate shadow blocks $\calB_3,\calB_6,\calB_8$ marked in purple as common edges of $W_L(b)$ and~$W_R(b)$.
  }
  \label{fig:shadow-blocks}
\end{figure}

Let $\calB$ be a shadow block of $b$, let $x = \min\calB$, and let $y = \max\calB$.
Let $u$ be an element of $P$ in $\partial\calB$.
If $u$ is on the left side of $\cgB$, let $e^+_L$ and $e^-_L$ be, respectively, the edges (provided they exist) immediately after and immediately before $u$ on the path $\Wleft(b)$. 
Also, if $u$ is on the right side of $\cgB$, let $e^+_R$ and $e^-_R$ be, respectively, the edges (provided they exist) immediately after and immediately before $u$ on $\Wright(b)$. 
Note that if $u \in \{x,y\}$, then $u$ is on both sides of~$\cgB$. 
By~\cref{obs:when_in_gamma}, an edge $e$ incident to $u$ lies in $\calB$ if and only if one of the following holds (see \cref{fig:characterize-block}):
\begin{enumerateShadowBlock}
    \item $u = x$ and $e^+_L\preccurlyeq e \preccurlyeq e^+_R$ in the $x$-ordering;\label{items:leaving_shadows:x}
    \item $u = y$ and $e^-_R\preccurlyeq e \preccurlyeq e^-_L$ in the $y$-ordering;\label{items:leaving_shadows:y}
    \item $u$ is strictly on the left side and $e^+_L \preccurlyeq e \preccurlyeq e^-_L$ in the $u$-ordering;\label{items:leaving_shadows:left}
    \item $u$ is strictly on the right side and $e^-_R \preccurlyeq e \preccurlyeq e^+_R$ in the $u$-ordering.\label{items:leaving_shadows:right}
\end{enumerateShadowBlock}

\begin{figure}[tp]
  \begin{center}
    \includegraphics{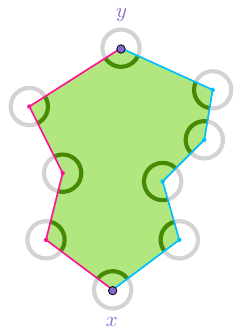}
  \end{center}
  \caption{
    For each vertex on the boundary of a non-degenerate block, we mark the angle, where the incident edges are in the interior of the block (dark shade) and in the exterior of the block (light shade).
    This illustrates, where the items \ref{items:leaving_shadows:x}--\ref{items:leaving_shadows:right} come from.
  }
  \label{fig:characterize-block}
\end{figure}

Since $x_0$ is in the exterior face of $G$, we have $x_0 \notin \Int\calB_i$ for all $i\in[m]$.
Furthermore, let $i,j \in [m]$ with $i < j$.
The boundaries of $\cgB_i$ and $\cgB_j$ have at most one common point.
In fact, their only common point can be $z_i$, and this is the case if and only if $j = i+1$.
In particular, if $j \neq i + 1$, then $\calB_i$ and $\calB_j$ are either disjoint or one is contained in the other.
Finally, we claim that $\calB_i$ is not contained in $\calB_j$.
Suppose to the contrary that $\calB_i \subset \calB_j$.
Since $x_0 \notin \Int \calB_j$ and $z_{i-1} \in \calB_i$, by~\Cref{obs:paths_and_blocks}, $z_{j-1} = \min \calB_j \leq z_{i-1}$ in~$P$, which is a contradiction as $z_{i-1} < z_{j-1}$ in~$P$ by definition.
We conclude that either $\calB_i$ and $\calB_j$ are disjoint or $\calB_j \subset \calB_i$.


Let $i \in [m-1]$.
Let $e^+_L$ and $e^-_L$ be, respectively, the edges immediately after and immediately before $z_{i}$ on the path $\Wleft(b)$ and let $e^+_R$ and $e^-_R$ be, respectively, the edges immediately after and immediately before $z_{i}$ on $\Wright(b)$.
Since $i<m$, we know $z_i\neq b$, and therefore, $z_i< b$ in~$P$.
We have two cases (illustrated in~\cref{fig:blocks_cases}):

Case 1:  $b$ is in the exterior of $\cgB_i$.

We claim that in this case, the interiors of $\calB_{i}$ and $\calB_{i+1}$ are disjoint.
Since we know that $z_i$ is the only common point of the boundaries of $\calB_{i}$ and $\calB_{i+1}$, in order to prove the claim, we have to exclude the situation where $\calB_{i+1} \subset \calB_{i}$.
If $\calB_{i+1} \subset \calB_{i}$, then $z_{i+1} \in \Int \calB_i$.
Since $b \notin \calB_i$, by~\Cref{obs:paths_and_blocks}, $z_{i+1} \leq \max \calB_i = z_i$ in~$P$, which is a contradiction.
Therefore, as claimed, the interiors of $\calB_{i}$ and $\calB_{i+1}$ are disjoint.
Moreover, by simple induction, we obtain that $\calB_i$ and $\calB_j$ are disjoint for every $j$ with $i + 1 < j \leq m$.

Furthermore, we claim that 
$e^-_L \prec e^+_L \preccurlyeq e_R^+ \prec e^-_R$
in the $z_i$-ordering.
Indeed, by~\ref{items:leaving_shadows:y} for $\calB_i$ and each $e\in\set{e^+_L,e^+_R}$, we obtain $ e^-_L \prec e \prec e^-_R$ in the $z_i$-ordering.
Moreover, by~\ref{items:leaving_shadows:x} for $\calB_{i+1}$ and $e^-_L$, we obtain $e^+_R \prec e_L^- \prec e^+_L$ in the $z_i$-ordering, in other words, $e^-_L \prec e^+_L \preccurlyeq e_R^+$ in the $z_i$-ordering.
Combining the above, we obtain $e^-_L \prec e^+_L \preccurlyeq e_R^+ \prec e^-_R$ in the $z_i$-ordering as desired.


Case 2:  $b$ is in the interior of $\cgB_i$.

We claim that in this case, the interiors of $\calB_{i}$ and $\calB_{i+1}$ are not disjoint, and thus, $\calB_{i+1} \subset \calB_i$.
If $\calB_{i}$ and $\calB_{i+1}$ have disjoint interiors, then $z_{i+1} \notin \Int \calB_i$.
Since $b \in \calB_i$, by~\Cref{obs:paths_and_blocks}, $z_{i+1} \leq \max \calB_i = z_i$ in~$P$, which is a contradiction.
Therefore, as claimed, $\calB_{i+1} \subset \calB_i$.
Moreover, by simple induction, we obtain that $\calB_j \subset \calB_i$ contains for every $j$ with $i + 1 < j \leq m$.
Furthermore, again by~\ref{items:leaving_shadows:x} and~\ref{items:leaving_shadows:y}, $e^-_R \prec e^+_L \preccurlyeq e^+_R \prec e^-_L$ in the $z_i$-ordering.

\begin{figure}[tp]
  \begin{center}
    \includegraphics{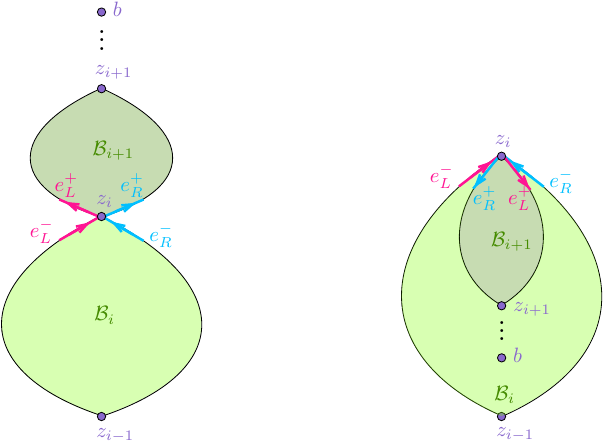}
  \end{center}
  \caption{
    In some figures, for simplicity, instead of drawing all elements and edges, we depict witnessing paths as curves and mark only some \q{important} elements and edges.
    Here, on the left-hand side, we depict the case when $b$ is in the exterior of $\calB_i$, and on the right-hand side, we depict the case when $b$ is in the interior of $\calB_i$.
    In the latter case, $z_i$ is a reversing element.
}
  \label{fig:blocks_cases}
\end{figure}


When the second case holds, i.e., $b$ is in the interior of $\cgB_i$, we call the
element $z_i$ a \emph{reversing element} of $b$.  The 
number of reversing elements in the sequence of common elements
of $b$ is the \emph{shadow depth} of $b$, denoted $\sd(b)$.
Note that $b$ may have no reversing elements and then, the shadow depth of $b$ is $0$.

Next, when $\sd(b)=r$, we define the sequence $(\shad_0(b),\dots,\shad_r(b))$
of sets called the \emph{shadow sequence} of $b$.
First, let $(i_1,\dots,i_r)$ be the subsequence of $(1,\dots,m-1)$ determined by
the subscripts of the reversing elements of $b$.
We expand this sequence by adding $i_0=0$ at the beginning and adding $i_{r+1}=m$ at the end.
For each $j \in \{0,\dots,r\}$, we then set
\[
  \shad_j(b) = \bigcup_{\ \ \, i_{j}\, <\, i\, \le\, i_{j+1}} \cgB_i.
\] 
Also, we refer to $(\cgB_{i_{j}+1},\ldots, \cgB_{i_{j+1}})$ as the \emph{sequence of blocks} of $\shad_j(b)$
.
We call $\cgB_{i_{j+1}}$ the \emph{terminal block} of $\shad_j(b)$.
Also, we call the element
call $z_{i_{j}}$ the \emph{initial element} of $\shad_j(b)$,
and we call $z_{i_{j+1}}$ the \emph{terminal element} of $\shad_j(b)$.
Additionally, we define $\shad_{r+1}(b) = \emptyset$.
See an example in \cref{fig:shadows}.

\begin{figure}[tp]
  \begin{center}
    \includegraphics{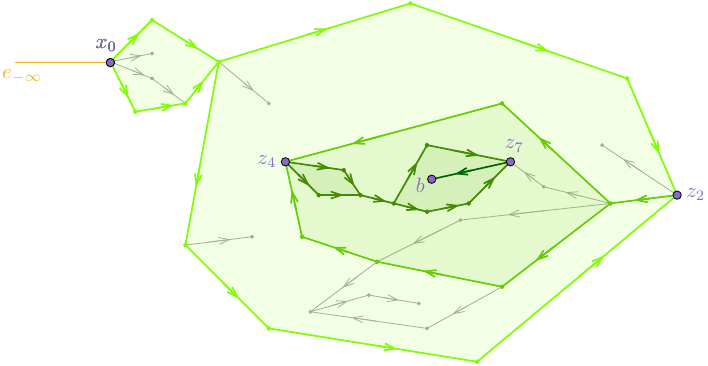}
  \end{center}
  \caption{
    There are three reversing elements of $b$, namely, $z_2,z_4,z_7$. 
    Hence, $\sd(b) = 3$.
    $\set{\shad_j(b)}_{j\in[3]}$ are depicted with various shades of green.
    Recall that in~\Cref{fig:shadow-blocks}, we discussed the shadow blocks of $b$.
    In particular, $\shad_0(b) = \calB_1 \cup \calB_2$, $\shad_1(b) = \calB_3 \cup \calB_4$, $\shad_2(b) = \calB_5 \cup \calB_6 \cup \calB_7$, and $\shad_3(b) = \calB_8$, and finally, $\shad_4(b) = \emptyset$.
  }
  \label{fig:shadows}
\end{figure}




We continue with a series of statements concerning the combinatorics of shadow blocks and shadows.
\Cref{obs:paths_and_blocks} immediately gives the following facts.

\begin{obs}\label{prop:paths_stay_in_blocks}
     Let $\cgB$ be a shadow block and let $u$ and $v$ be elements of $P$ with $u \leq v$ in~$P$ and $u \not\leq \max \calB$ in~$P$.
     Then, $u \in \cgB$ if and only if $v \in \cgB$.
     More precisely, either $u,v \in \Int \calB$ or $u,v \notin \calB$.
\end{obs}

\begin{obs}\label{obs:equivalence_for_shadows}
    Let $b \in B$, let $u$ and $v$ be elements of $P$, and let $j$ be a nonnegative integer.
    Assume that $u \not\leq b$ and $u \leq v$ in~$P$.
    Then, $u \in \shad_j(b)$ if and only if $v \in \shad_j(b)$.
    More precisely, either $u,v \in \Int \shad_j(b)$ or $u,v \notin \shad_j(b)$.
\end{obs}

The next proposition is an abstraction of two corollaries stated right after the proof.
Namely, we show that if a witnessing path has both endpoints in a shadow block (or in $\shad_j(b)$ for some $b\in B$ and $j\in\set{0,\ldots,\sd(b)}$), then the path lies entirely in the shadow block (or $\shad_j(b)$, respectively).
\begin{proposition}\label{prop:witnessing-paths-in-shadow-blocks-abstract}
  Let $b \in B$ and let $j$ be a nonnegative integer.
  Let $(\calD_1,\dots,\calD_m)$ be a subsequence of consecutive elements of the sequence of blocks of $\shad_j(b)$ and let $\calD = \bigcup_{i \in [m]} \calD_i$.
  Let $W$ be a witnessing path in~$P$ with both endpoints in $\calD$.
  Then all edges of $W$ lie in $\cgD$. 
\end{proposition}
\begin{proof}
    We argue by contradiction.  
    Let $u$ and $v$ be elements in the intersection of $W$ and $\partial\calD$ with $u \leq v$ in~$P$ such that $u[W]v$ is in the exterior of $\calD$ except for the elements $u$ and $v$.
    For each $i \in [m]$, let $x_i = \min \calD_i$.
    Additionally, let $x=x_1$ and $y=\max\cgD_m$.
    Then $x\le u< v\le y$ in~$P$.  
    In particular, $u \neq y$.
    Let $W'$ be a witnessing path from $v$ to $y$ in~$P$.
    For convenience, let $U = u[W]v[W']y$.
    Let $e$ be the first edge of $u[W]v$.
    Let $i \in [m]$ be the least integer such that $u \in \calD_i$ and let $\calB = \calD_i$.
  
  By~\Cref{prop:seemingly-obvious}, the left side of $\cgB$ is contained in 
  $\Wleft(y)$ and the right side of $\cgB$ is contained in $\Wright(y)$. 
  If $u$ is on the left side of $\cgB$, let
  $e^+_L$ and $e^-_L$ be, respectively (provided they exist), the edges immediately after and
  immediately before $u$ on the path $\Wleft(y)$. 
  Of course, in the case $u=x_0$ we set $e^-_L=e_{-\infty}$. 
  Also, if $u$ is on the
  right side of $\cgB$, let $e^+_R$ and
  $e^-_R$ be, respectively (provided they exist), the edges immediately after and immediately
  before $u$ on $\Wright(y)$. 
  Again, if $u=x_0$ we set $e^-_R=e_{-\infty}$. 
  Note that if $u=x_i$, then $u$ is on both
  sides of $\cgB$. 

    First, assume that 
    \begin{equation} \label{eq:on-the-left}
        \text{$u$ is on the left side of $\calB$ and $e_L^- \prec e \prec e_L^+$ in the $u$-ordering.}
    \end{equation}
    In particular, $e \prec e^+_L$ in the $(u,e^-_L)$-ordering. 
    It follows that $x_0[W_L(y)]u[U]y$ is left of $W_L(y)$, which is a contradiction.

    Next, symmetrically assume that 
    \begin{equation} \label{eq:on-the-right}
        \text{$u$ is on the right side of $\calB$ and $e_R^+ \prec e \prec e_R^-$ in the $u$-ordering.}
    \end{equation}
    In particular, $e^+_R \prec e$ in the $(u,e^-_R)$-ordering.
    It follows that $x_0[W_R(y)]u[U]y$ is right of $W_R(y)$, which is a contradiction. 

    Now, assume that neither~\eqref{eq:on-the-left} nor~\eqref{eq:on-the-right} hold.
    If $u$ was strictly on the left side of $\cgB$, then since $e$ is not in $\cgB$, by \ref{items:leaving_shadows:left},~\eqref{eq:on-the-left} would hold.
    If $u$ was strictly on the right side of $\cgB$, then since $e$ is not in $\cgB$, by \ref{items:leaving_shadows:right},~\eqref{eq:on-the-right} would hold.
    If $i \neq 1$ and $u = x_i$, then~\ref{items:leaving_shadows:x} applied to $\calD_i$ and~\ref{items:leaving_shadows:y} applied to $\calD_{i-1}$, one of~\eqref{eq:on-the-left} and~\eqref{eq:on-the-right} would hold.
    If $i = 1$ and $u = x_1 = x$ and $x$ was a reversing element, then $e^-_R \prec e^+_L \prec e^+_R \prec e^-_L$ in the $x$-ordering, and since $e$ is not in $\calB$, by \ref{items:leaving_shadows:x}, $e_R^+ \prec e \prec e_L^+$ in the $x$-ordering; altogether,~\eqref{eq:on-the-left} or~\eqref{eq:on-the-right} or both hold.
    Therefore, we may assume that $i = 1$, $u = x$, and $x$ is not a reversing element.
    In this case, $e_L^- \prec e^+_L\prec e^+_R\prec e^-_R$ in the $x$-ordering.
    Since $e$ is not in $\calB$, by \ref{items:leaving_shadows:x}, we obtain that $e_R^+ \prec e \prec e_L^+$.
    It follows that one of the three holds:~\eqref{eq:on-the-left} or~\eqref{eq:on-the-right} or $e^-_R\prec e \prec e^-_L \prec e_R^-$ in the $x$-ordering.
    Thus, we assume the latter.
    In particular, $\calB$ is not the first block in the sequence of blocks of $y$, and the preceding block $\calB'$ is non-degenerate.
    By~\ref{items:leaving_shadows:y}, $e$ is in the interior of $\calB'$.
    Since $x$ is not reversing, the interiors of $\calB$ and $\calB'$ are disjoint.
    Therefore, $u[W]v$ must intersect the boundary of $\calB'$.
    This produces a cycle in~$P$, which is a contradiction.
\qedhere

\end{proof}

As indicated before,~\Cref{prop:witnessing-paths-in-shadow-blocks-abstract} yields the following two statements.

\begin{corollary}\label{cor:path-in-block}
    Let $\cgB$ be a shadow block, and let $W$ be a witnessing path in~$P$ with both endpoints in $\cgB$.
    Then all edges of $W$ lie in $\cgB$. 
\end{corollary}

\begin{corollary}\label{cor:path-in-shadow}
    Let $b \in B$ and let $j$ be a nonnegative integer.
  Let $W$ be a witnessing path in~$P$ with both endpoints in $\shad_j(b)$.
  Then all edges of $W$ lie in $\shad_j(b)$. 
\end{corollary}

When $\calB$ is a shadow block and $u \in B$, then we somehow control the behavior of $W_L(u)$ and $W_R(u)$ outside $\calB$.
In~\Cref{prop:uWLWR}, we show that $\min \calB$ lies in both of these paths.
In particular, by~\Cref{prop:W-consistent}, 
$W_L(\min \calB)$ is a prefix of $W_L(u)$ and 
$W_R(\min \calB)$ is a prefix of $W_R(u)$.
Moreover, in~\Cref{prop:uWLWR_max}, we show that if we also assume $\max \calB < u$ in~$P$, then $\max \calB$ lies in both $W_L(u)$ and $W_R(u)$.
\Cref{prop:top-shadows} is an adjustment of this statement to shadows.

\begin{proposition}\label{prop:uWLWR}
  Let $\cgB$ be a shadow block and let $u \in B$. 
  Assume that $u \in \cgB$. 
  Then $\min\cgB$ is in both $\Wleft(u)$ and $\Wright(u)$.
\end{proposition}
\begin{proof}
  We set $x=\min\cgB$ and $y=\max\cgB$.  
  We show that $x$ is in $\Wleft(u)$.
  The proof that $x$ is in $\Wright(u)$ is symmetric.
  Since $x_0$ is in the exterior face of $G$ and $u$ is in $\cgB$, 
  the path $\Wleft(u)$ intersects $\partial\cgB$. 
  Recall that $\partial\cgB$ is the union of $x[W_L(y)]y$ and $x[W_R(y)]y$.
  In particular, every element $w$ in $\partial\cgB$ satisfies $x \leq w \leq y$ in~$P$.
  If $\Wleft(u)$ intersects $x[W_L(y)]y$ in an element $w$, then by the fact that $W_L(u)$ and $W_L(y)$ are $x_0$-consistent (\cref{prop:Tleft-and-Tright}.\ref{prop:W-consistent:left}), we obtain $x_0[W_L(u)]w = x_0[W_L(y)]w$, and so, $x$ lies in $W_L(u)$.

  Therefore, for the remainder of the proof we assume that $\Wleft(u)$ does not intersect $x[W_L(y)]y$, and so, let $w$ be an element in the intersection of $W_L(u)$ and $x[W_R(y)]y$ with $w \notin \{x,y\}$.
  We will show that this leads to a contradiction.
  Note that in this setting, $W_L(w)$ is a prefix of $W_L(u)$.
  Since $x$ is not in $W_L(w)$, the path $W_L(x)$ is not a subpath of $W_L(w)$.
  Therefore, since $x < w$ in~$P$ and by \cref{prop:shortcuts}.\ref{prop:item:shortcuts-left-tree}, $W_L(w)$ is left of $W_L(x)$.
  Since $W_L(w)$ is left of $W_L(x)$ and $W_L(x)$ is a subpath of $W_L(y)$, we obtain that $W_L(w)$ is left of $W_L(y)$.
  However, $w < y$ in~$P$, thus for every witnessing path $W$ from $w$ to $y$ in~$P$, we have $x_0[W_L(w)]w[W]y$ left of $W_L(y)$, which is a desired contradiction.
\end{proof}

\begin{proposition}\label{prop:uWLWR_max}
  Let $\cgB$ be a shadow block and let $u \in B$. 
  Assume that $u \in \cgB$ and $\max\calB < u$ in~$P$.
  Then $\max\cgB$ belongs to both $\Wleft(u)$ and $\Wright(u)$.
\end{proposition}
\begin{proof}
    We set $x=\min\cgB$, and $y=\max\cgB$.
    By \cref{prop:uWLWR}, $x$ lies in $W_L(u)$. 
    Therefore, $x_0[W_L(u)]x = x_0[W_L(y)]x$.
    Note that $u$ does not lie on $W_L(y)$ as this yields $u\leq y$ in~$P$ contrary to the assumption.
    Let $w = \gce(W_L(u),W_L(y))$. 
    Thus, $x\leq w$ in~$P$ and we claim that $y\leq w$ in~$P$.
    Suppose otherwise that $x\leq w< y$ in~$P$.
    By \cref{cor:path-in-block}, 
    the whole path $x[W_L(u)]u$ lies in $\calB$. 
    Let $W$ be an arbitrary witnessing path from $y$ to $u$ in~$P$.
    If $x=w$ then by \ref{items:leaving_shadows:x}, $x_0[W_L(y)]y[W]u$ is left of $W_L(u)$ which is a contradiction.
    If $x<w<y$ in~$P$ then $w$ is strictly on the left side, and so, by \ref{items:leaving_shadows:left}, again $x_0[W_L(y)]y[W]u$ is left of $W_L(u)$, which is a contradiction.
    This shows that $y$ belongs to $W_L(u)$.
    The proof that $y$ belongs to $W_R(u)$ is symmetric.
\end{proof}

\begin{proposition}\label{prop:top-shadows}
    Let $u,y \in B$ and let $j$ be a nonnegative integer with $\sd(y) = j$.
    Assume that $u \in \shad_j(y)$ and $y < u$ in~$P$.
    Then, $u$ lies in the terminal block of $\shad_j(y)$, and $y$ lies in both $W_L(u)$ and $W_R(u)$.
\end{proposition}
\begin{proof}
    Let $\calB$ be the block of $\shad_j(b)$ such that $u \in \calB$.
    Suppose to the contrary that $\calB$ is not the terminal block of $\shad_j(y)$.
    By the construction of $\shad_j(b)$, $y$ lies only in the terminal block of $\shad_j(b)$, and so, $y \notin \calB$.
    Therefore, a witnessing path from $y$ to $u$ in~$P$ intersects $\partial \calB$.
    This yields a directed cycle in~$P$, which is a contradiction.
    Thus, $\calB$ is the terminal block of $\shad_j(u)$.
    This gives the first part of the assertion.
    To get the second part, it suffices to apply~\Cref{prop:uWLWR_max}.
\end{proof}

We conclude this subsection with an exhaustive description of what we know about $\sd(u)$ and $\shad_i(u)$ for $i\in\set{0,\ldots,\sd(u)}$ given that 
$u$ lies in a shadow block of an element $b\in B$.

\begin{proposition}\label{prop:shadow-comp}
  Let $u,b\in B$, let $j \in \{0,\dots,\sd(b)\}$, let $\cgB$ be a block of $\shad_j(b)$.
  If $u \in \calB$, then the following statements hold:
  \begin{enumerate}
    \item if $u$ is the initial element of $\shad_j(b)$, then $\sd(u)=j-1$, unless $j=0$ and $u=x_0$; \label{prop:shadow-comp:item:sd_of_base} 
    \item if $u$ is not the initial element of $\shad_j(b)$, then $\sd(u)\ge j$; \label{prop:shadow-comp:item:sd_of_non_base}
    \item if $u$ is not the initial element of $\shad_j(b)$ and $u$ lies in $\partial\calB$, then $\sd(u) = j$;
    \label{prop:shadow-comp:item:sd_boundary_of_block}
    \item $\shad_i(u) = \shad_i(b)$ for all $i\in\set{0,\ldots,j-1}$;  \label{prop:shadow-comp:item:equality_of_small_shadows}
    \item $\shad_j(\min\calB) \subseteq \shad_j(u) \subseteq \shad_j(\max\calB) \subseteq \shad_j(b)$;  \label{prop:shadow-comp:item:inclusions_of_j_shadows}
    \item $\shad_j(u) = \shad_j(\max\calB)$ if
      and only if $\max\calB \leq u$ in~$P$.  \label{prop:shadow-comp:item:hanging_element}
  \end{enumerate}
\end{proposition}
\begin{proof}
    Let $x=\min\cgB$ and $y=\max\cgB$, let $(z_0,\ldots,z_m)$ be the sequence of common points of~$b$.
    By definition, $x$ is a common point of $b$, so let $i$ be an integer such that $z_i = x$.
    In particular, $x$ is an element of $W_L(b) \cap W_R(b)$, hence, $W_L(x)$ is a prefix of $W_L(b)$ and $W_R(x)$ is a prefix of $W_R(b)$.
    It follows that $(z_0,\ldots,z_i)$ is the sequence of common points of $x$.
    Since $\calB$ is a block of $\shad_j(b)$ and $x = \min\calB$, exactly $j$ elements of $(z_0,\ldots,z_i)$ are reversing elements of $b$.
    For every reversing element $z$ of $b$ that is contained in $(z_0,\ldots, z_{i})$, either 
    $z$ is a reversing element of $x$, or $z=z_i=x$ and this element coincides with the initial element of $\shad_j(b)$.
    If $x$ is the initial element of $\shad_j(b)$, then either $x = x_0$ (and $j=0$) or $x$ has $j-1$ reversing elements, thus, $\sd(x)=j-1$.
    If $x$ is not the initial element of $\shad_j(b)$, then $x$ has $j$ reversing elements, so $\sd(x)=j$.
    
    \cref{prop:uWLWR} implies that $x$ is a vertex of $W_L(u) \cap W_R(u)$, thus, $x$ is also a common point of $u$, and moreover, $(z_0,\ldots,z_i)$ is a prefix of the sequence of common points of $u$.
    Therefore, every block of the sequence of shadow blocks of $b$ that precedes $\calB$ is also a shadow block of $x$, $u$, and $y$.

    If $u$ is the initial element of $\shad_j(b)$, then $u = x$, and $x$ is a reversing element of $b$, hence, $\sd(x) = \sd(u) = j-1$ (unless $j = 0$ and $x = u = x_0$), which proves~\ref{prop:shadow-comp:item:sd_of_base}.
    If $u$ is not the initial element of $\shad_j(b)$, then either $u \neq x$ or $x$ is not the initial element of $\shad_j(b)$, hence, $(z_0,\ldots,z_i)$ contains $j$ reversing elements of $u$, and so, $\sd(u) \geq j$, which proves~\ref{prop:shadow-comp:item:sd_of_non_base}.

    Next, we prove \ref{prop:shadow-comp:item:sd_boundary_of_block}. Suppose that $u$ is not the initial element of $\shad_j(b)$, but $u$ lies in $\partial\calB$.
    If $u = x$ or $u = y$, then $(z_0,\dots,z_i)$ or resp.\ $(z_0,\dots,z_i,y)$ is exactly the sequence of common elements of $u$, hence, $\sd(u) = j$.
    Now, assume that $u$ is strictly on the left side of $\calB$ (the proof in the case, where $u$ is strictly on the right side of $\calB$ is symmetric).
    We know that $\sd(u) \geq j$ and we want to prove that $\sd(u) \leq j$.
    To this end, we need to show that there are no reversing elements $z$ of $u$ with $z_i = x < z < u$ in~$P$.
    Note that $W_L(u)$ is a subpath of $W_L(b)$.
    In particular, $x[W_L(u)]u$ is a segment of $\partial\calB$.
    Since $W_R(u)$ contains $x$, by \cref{cor:path-in-block}, $x[W_R(u)]u$ is contained in $\calB$.
    Let $z$ be an element of $W_L(u) \cap W_R(u)$ with $x < z < u$ in~$P$.
    Let $e_L^+$ and $e_L^-$ be, respectively, the edges immediately after and immediately before $z$ on $W_L(u)$, and let  $e_R^+$ and $e_R^-$ be, respectively, the edges immediately after and immediately before $z$ on $W_R(u)$.
    Since $z$ lies strictly on the left side of $\cgB$ and both $e_R^-$ and $e_R^+$ lie in $\calB$, by~\ref{items:leaving_shadows:left}, $e^+_L \preccurlyeq e_R^- \preccurlyeq e_L^-$ and $e^+_L \preccurlyeq e_R^+ \preccurlyeq e_L^-$ in the $z$-ordering.
    Recall that if $z$ was a reversing element of $u$, then $e^-_R \prec e^+_L \preccurlyeq e^+_R \prec e^-_L$ in the $z$-ordering.
    These two sets of inequalities can not hold simultaneously, which shows that $z$ is not a reversing element of $u$, and so, completes the proof of \ref{prop:shadow-comp:item:sd_boundary_of_block}.

    Since the sequence of common points of $b$ and the sequence of common points of $u$ agree up to at least $x$, and thus, up to at least the initial element of $\shad_j(b)$, we obtain~\ref{prop:shadow-comp:item:equality_of_small_shadows}.

    The union of all shadow blocks of $\shad_j(b)$ that precede $\calB$ is equal to $\shad_j(x)$, hence, $\shad_j(x) \subseteq \shad_j(u)$.
    Moreover, $\shad_j(x) \cup \calB = \shad_j(y)$, thus, $\shad_j(y) \subseteq \shad_j(b)$.
    In order to prove~\ref{prop:shadow-comp:item:inclusions_of_j_shadows}, it remains to prove that $\shad_j(u) \subseteq \shad_j(y)$.
    Consider a block $\calD$ of $\shad_j(u)$ such that $x \leq \min\calD$.
    We have, $x\le \min\cgD\le \max\cgD\le u$ in~$P$.
    Both sides of $\calD$ are subpaths of witnessing paths from $x$ to $u$.
    It follows from Proposition~\ref{cor:path-in-block} that all edges and vertices of the boundary of $\cgD$ are in $\cgB$.  
    Thus by~\Cref{obs:region_containment}, $\cgD\subseteq\cgB\subseteq\shad_j(y)$, which yields $\shad_j(u) \subseteq \shad_j(y)$, and completes the proof of~\ref{prop:shadow-comp:item:inclusions_of_j_shadows}.

    Finally, we prove~\ref{prop:shadow-comp:item:hanging_element}.
    If $\shad_j(y) = \shad_j(u)$, then since $y$ is on the boundary of $\shad_j(u)$, we have $y \leq u$ in~$P$.
    In order to prove the reverse implication, suppose that $y \leq u$ in~$P$.
    It suffices to justify that $y$ is a common point of $u$ and that $y$ is a reversing element of $u$.  The former follows from \cref{prop:uWLWR_max} and the latter follows from the fact that $u$ is in $\cgB$.
\end{proof}

\subsection{Ordering elements in \texorpdfstring{$B$}{B}}\label{sec:ordering_elements_in_B}

Given two elements in $B$, 
we inspect the interaction of their leftmost and rightmost witnessing paths. 
This establishes four scenarios, and we relate each case with the comparability status of the elements and their shadows as defined in~\Cref{ssec:shadows}.
Let $b_1,b_2 \in B$.
We say that $(b_1,b_2)$ is 
\begin{itemize}
    \item an \emph{inside pair} whenever $W_L(b_2)$ is left of $W_L(b_1)$ and $W_R(b_1)$ is left of $W_R(b_2)$,
    \item an \emph{outside pair} whenever $W_L(b_1)$ is left of $W_L(b_2)$ and $W_R(b_2)$ is left of $W_R(b_1)$,
    \item a \emph{left pair} whenever $W_L(b_1)$ is left of $W_L(b_2)$ and $W_R(b_1)$ is left of $W_R(b_2)$,
    \item a \emph{right pair} whenever $W_L(b_2)$ is left of $W_L(b_1)$ and $W_R(b_2)$ is left of $W_R(b_1)$.
\end{itemize}
See \cref{fig:pairs} for some examples.
Note that a pair $(b_1,b_2)$ is an outside pair whenever $(b_2,b_1)$ is an inside pair and $(b_1,b_2)$ is a right pair whenever $(b_2,b_1)$ is a left pair. 
The above distinction is in fact a classification of all incomparable pairs of elements in $B$.
That is, if $b_1\parallel b_2$ in~$P$, the pair $(b_1,b_2)$ is always of one of the four types: inside, outside, left, right.
On the other hand, a comparable pair of elements in $B$ is not necessarily of one of the four types.
However, all left and right pairs are incomparable by~\cref{prop:shortcuts}.

\begin{figure}[tp]
     \centering
     \begin{subfigure}[t]{1\textwidth}
         \centering
         \includegraphics{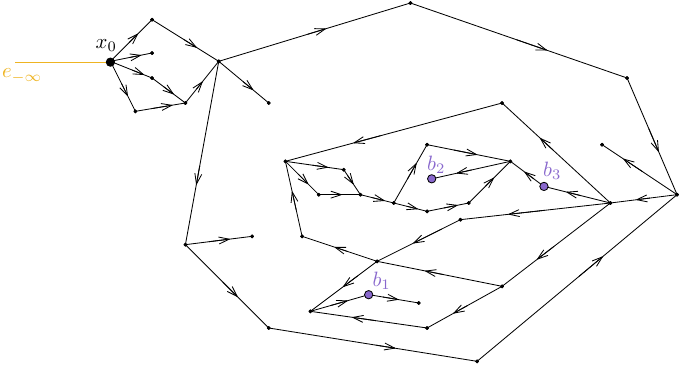}
     \end{subfigure}
     \begin{subfigure}[t]{1\textwidth}
         \centering
         \includegraphics{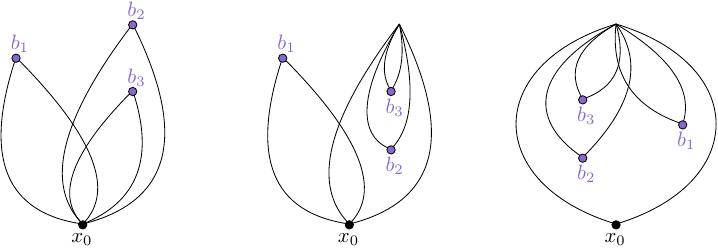}
     \end{subfigure}
  \caption{
    In the poset on the top of the figure, the pairs $(b_1,b_2)$ and $(b_1,b_3)$ are left pairs while $(b_2,b_3)$ is an outside pair.
    It follows that $(b_2,b_1)$ and $(b_3,b_1)$ are right pairs and $(b_3,b_2)$ is an inside pair.
    Below, we present more schematic drawings in which we draw only the witnessing paths forming boundaries of shadows of respective elements in $B$.
    For readability, we omit the arrows in such drawings.
    We also do not mark intersection of curves, although each such intersection must be an element of the poset.
    Still, $(b_1,b_2)$ and $(b_1,b_3)$ are left pairs while $(b_2,b_3)$ is an outside pair. 
  }
  \label{fig:pairs}
\end{figure}

Within the next two propositions, we show that under some mild assumptions, two elements $b_1,b_2 \in B$ form an inside pair $(b_1,b_2)$ if and only if there is a nonnegative integer $j$ such that $\shad_j(b_1) \subsetneq \shad_j(b_2)$.
Next in~\Cref{prop:hard_left_right_characterization}, we give some necessary conditions on shadows of elements $b_1,b_2 \in B$ forcing them to form a left or right pair.

\begin{proposition}
\label{prop:paths_directions_in_shadows}
    Let $j$ be a nonnegative integer, let $u,y\in B$ with $\sd(y) = j$ and $u \in \shad_j(y)$.
    \begin{enumerate}
        \myitem{$(L)$} $W_L(u)$ is not left of $W_L(y)$. \label{prop:paths_directions_in_shadows:left}
        \myitem{$(R)$} $W_R(u)$ is not right of $W_R(y)$. \label{prop:paths_directions_in_shadows:right}
    \end{enumerate}
    Moreover, if $u$ is in the interior of $\shad_j(y)$, then either $y < u$ in~$P$ or $(u,y)$ is an inside pair.
\end{proposition}
\begin{proof}
    We prove statement~\ref{prop:paths_directions_in_shadows:left}. 
    The argument for statement~\ref{prop:paths_directions_in_shadows:right} is symmetric.
    If $u$ lies in $W_L(y)$, then the assertion is clear since $W_L(u)$ is a subpath of $W_L(y)$.
    If $y < u$ in~$P$, then by \cref{prop:top-shadows}, $y$ lies in $W_L(u)$, hence, $W_L(y)$ is a subpath of $W_L(u)$ and the assertion holds as well.
    Therefore, we can assume that $u$ is not an element of $W_L(y)$ and $y \not< u$ in~$P$.
    In particular, either $W_L(y)$ is left of $W_L(u)$ or $W_L(u)$ is left of $W_L(y)$.
    In order to finish the proof we have to prove that $W_L(y)$ is left of $W_L(u)$.
    
    Let $\calB'$ be a shadow block of $\shad_j(y)$ containing $u$, and let $x',y'$ be the initial and terminal elements of $\calB'$, respectively.
    By \cref{prop:uWLWR}, $x'$ belongs to $W_L(u)$.
    First, assume that $y' < u$ in~$P$.
    Note that this implies that $\calB'$ is not the terminal block of $\shad_j(y)$.
    By~\Cref{prop:uWLWR_max}, $y'$ lies in $W_L(u)$.
    In particular, $y'$ lies in both $W_L(y)$ and $W_L(u)$.
    In this case, $u \in \Int\calB'$ and $y \notin \calB'$.
    Moreover, $y'[W_L(u)]u$ is contained in $\calB'$ and $y'[W_L(y)]y$ intersects $\calB'$ only in $y'$.
    Therefore, by \ref{items:leaving_shadows:y}, $W_L(y)$ is left of $W_L(u)$, as desired. 
    Next, assume that $y' \not< u$ in~$P$.
    By \cref{cor:path-in-block}, $x'[W_L(u)]u$ is contained in $\calB'$. 
    Observe that neither of the paths $W_L(y')$ and $W_L(u)$ is a subpath of the other.
    Let $w = \gce(W_L(y'),W_L(u))$.
    We have $x' \leq w < y'$ in~$P$ and $w[W_L(u)]u$ lies in $\calB'$.
    If $w=x'$, then by \ref{items:leaving_shadows:x}, $W_L(y')$ is left of $W_L(u)$.
    Similarly, if $w$ is strictly on the left side, then by \ref{items:leaving_shadows:left}, $W_L(y')$ is left of $W_L(u)$.
    Since $W_L(y')$ is a subpath of $W_L(y)$, we conclude the proof of~\ref{prop:paths_directions_in_shadows:left}.

    Finally, let us prove the \q{moreover} part.
    Assume that $u$ lies in the interior of $\shad_j(y)$, and that $y \not< u$ in~$P$.
    To complete the proof, we will show that $(u,y)$ is an inside pair.
    Since $y \not< u$ in~$P$, $W_L(y)$ is not a subpath of $W_L(u)$ and $W_R(y)$ is not a subpath of $W_R(u)$.
    On the other hand, since $u$ lies in the interior of $\shad_j(y)$, $W_L(u)$ is not a subpath of $W_L(y)$ and $W_R(u)$ is not a subpath of $W_R(y)$.
    It follows that $(u,y)$ is a pair of one of the four types.
    Furthermore, by \ref{prop:paths_directions_in_shadows:left} and \ref{prop:paths_directions_in_shadows:right}, $(u,y)$ is an inside pair.
\end{proof}

\begin{proposition}\label{prop:inside_pair_implies_containment}
    Let $j$ be a nonnegative integer, let $u,b\in B$ with $\sd(u),\sd(b) \geq j$, and let $x$ be the initial element of $\shad_j(b)$.
    Assume that $x \leq u$ in~$P$.
    If $(u,b)$ is an inside pair, then $u \in \shad_j(b)$.
\end{proposition}
\begin{proof}
    Suppose to the contrary that $(u,b)$ is an inside pair and $u \notin \shad_j(b)$.
    Let $y$ be the terminal element of $\shad_j(b)$.
    Let $W$ be a witnessing path from $x$ to $u$, and let $w$ be the last element of $W$ that is in $\shad_j(b)$.
    In particular, $w$ lies either in $x[W_L(b)]b$ or $x[W_R(b)]b$.
    Note that $w \neq u$.
    If $w$ lies in $W_L(b)$, then let $W_1 = x_0[W_L(b)]w[W]u$, and if $w$ lies in $W_R(b)$, then let $W_2 = x_0[W_R(b)]w[W]u$.    
    Recall that $W_L(u)$ is either equal or left of any witnessing path from $x_0$ to $u$, and $W_R(u)$ is either equal or right of any witnessing path from $x_0$ to $u$.
    Moreover, the relation of being left/right is transitive (by \cref{prop:u_e_is_poset}).
    In particular, since $(u,b)$ is an inside pair, $W_1$ (if defined) can not be left of $W_L(b)$ and $W_2$ (if defined) can not be right of $W_R(b)$.

    Remember that $w \neq u$, and so, let $e$ be the edge following $w$ in $W$.
    By definition, $e$ is in the exterior of $\shad_j(b)$.
    If $w$ is in $W_L(b)$, let $e^+_L$ and $e^-_L$ be, respectively (provided they exist), the edges immediately after and immediately before $w$ on the path $W_L(b)$. 
    Of course, in the case $w=x_0$ we set $e^-_L=e_{-\infty}$. 
    Also, if $w$ is in $W_R(b)$, let $e^+_R$ and $e^-_R$ be, respectively (provided they exist), the edges immediately after and immediately before $w$ on $W_R(b)$. 
    Again, if $w=x_0$ we set $e^-_R=e_{-\infty}$. 

    We split the reasoning into cases depending on where $w$ lies.
    If $w$ lies strictly on the left (resp.\ right) side of some shadow block of $\shad_j(b)$, then by \ref{items:leaving_shadows:left} (resp.\ \ref{items:leaving_shadows:right}), we obtain $e \prec e_L^+$ in the $(w,e_L^-)$-ordering (resp.\ $e_R^+ \prec e$ in the $(w,e_R^-)$-ordering).
    This yields that $W_1$ is left of $W_L(b)$ (resp.\ $W_2$ is right of $W_R(b)$), which is a contradiction. 
    
    Next, assume that  $w$ is an element in $W_L(b) \cap W_R(b)$ with $w \notin\{x,y\}$.
    Let $\calB'$ and $\calB$ be the blocks of $\shad_j(b)$ such that $w = \max \calB'$ and $w = \min \calB$.
    We have $e_L^- \prec e_L^+ \preccurlyeq e_R^+ \prec e_R^-$ in the $w$-ordering.
    By~\ref{items:leaving_shadows:x} applied to $\calB$ and \ref{items:leaving_shadows:y} applied to $\calB'$, either $e^-_L \prec e \prec e_L^+$ in the $w$-ordering or $e_R^+ \prec e \prec e_R^-$ in the $w$-ordering.
    In other words, either $e \prec e_L^+$ in the $(w,e_L^-)$-ordering or $e_R^+ \prec e$ in the $(w,e_R^-)$-ordering.
    In the former case, we obtain that $W_1$ is left of $W_L(b)$ and in the latter case, we obtain that $W_2$ is right of $W_R(b)$.
    Each of these outcomes leads to a contradiction.

    Now, consider the case where $w = x$.
    If $j = 0$, then $w = x = x_0$.
    Note that $e_L^- = e_R^- = e_{-\infty}$.
    By~\ref{items:leaving_shadows:x}, either $e \prec e_L^+$ in the $(w,e_L^-)$-ordering or $e_R^+ \prec e$ in the $(w,e_R^-)$-ordering.
    Similarly as in the other cases, this gives a contradiction.
    Thus, we assume that $j > 0$.
    In particular, $x$ is a reversing element of $b$.
    We have $e_R^+ \prec e_L^- \prec e_R^- \prec e_L^+$ in the $w$-ordering.
    By~\ref{items:leaving_shadows:x} applied to the first block of $\shad_j(b)$, we have $e_R^+ \prec e \prec e_L^-$ in the $w$-ordering.
    We obtain that either $e_R^+ \prec e \prec e_R^-$ in $w$-ordering or $e_L^- \prec e \prec e_L^+$ in the $w$-ordering.
    In other words, either $e \prec e_L^+$ in the $(w,e_L^-)$-ordering or $e_R^+ \prec e$ in the $(w,e_R^-)$-ordering.
    Again, this is a contradiction.
    
    Finally, assume that $w = y$.
    Note that $y \neq b$ as otherwise $b \leq u$ in~$P$, which contradicts~\cref{prop:shortcuts}.\ref{prop:item:shortcuts-left-tree} (since $W_L(b)$ is left of $W_L(u)$).
    In particular, $y$ is a reversing element of $b$, and so, $e_R^+ \prec e_L^- \prec e_R^- \prec e_R^+$ in the $w$-ordering.
    By~\ref{items:leaving_shadows:y}, $e_L^- \prec e \prec e_R^-$ in the $w$-ordering.
    We obtain that $e \prec e_L^+$ in the $(w,e_L^-)$-ordering, and so, $W_1$ is left of $W_L(b)$, which is a contradiction.
    This completes the proof.
    \qedhere

\end{proof}

\begin{proposition}\label{prop:hard_left_right_characterization}
    Let $j$ be a nonnegative integer, let $b_1,b_2\in B$ with $\sd(b_1),\sd(b_2) \geq j$, and assume that $\shad_j(b_1)$ and $\shad_j(b_2)$ have the same initial element. 
    If $b_1 \notin \shad_j(b_2)$ and $b_2 \notin \shad_j(b_1)$, then either $(b_1,b_2)$ is a left pair or $(b_1,b_2)$ is a right pair.
\end{proposition}
\begin{proof}
    If $W_L(b_1)$ was a subpath of $W_L(b_2)$, then $b_1 \in \shad_{j}(b_2)$, hence, this is not the case.
    By reversing the roles of $b_1$ and $b_2$ and/or replacing $W_L$ with $W_R$ in the above sentence, we obtain that neither of $W_L(b_1)$ and $W_L(b_2)$ is a subpath of the other and neither of $W_R(b_1)$ and $W_R(b_2)$ is a subpath of the other.
    Therefore, $(b_1,b_2)$ is a pair of one of the four types: inside, outside, left, right.
    By \cref{prop:inside_pair_implies_containment}, if $(b_1,b_2)$ is an inside pair, then $b_1 \in \shad_j(b_2)$, and if $(b_1,b_2)$ is an outside pair, then $b_2 \in \shad_j(b_1)$.
    It follows that neither of the above holds, thus, $(b_1,b_2)$ is either a left pair or a right pair.
\end{proof}

In~\Cref{sec:good_instance}, we give a reduction that allow us to restrict our attention to instances $(P,x_0,G,e_{-\infty},I)$ such that for every $(a,b)\in I$ we have $a\notin\shad_0(b)$. 
Therefore, we develop a number of statements just for $\shad_0(b)$ for $b\in B$. 
Also, to simplify the notation, for every $b \in B$, we write
    \[\shadz(b) = \shad_0(b).\]

The next proposition is a very intuitive fact about shadows that is quite technical to prove directly but follows nicely from the classification of incomparable pairs.

\begin{proposition}\label{prop:comparability_implies_shadow_containment}
    Let $b_1,b_2 \in B$.
    If $b_1 \leq b_2$ in~$P$, then $\shadz(b_1) \subset \shadz(b_2)$.
\end{proposition}
\begin{proof}
    If $W_L(b_1)$ is a subpath of $W_L(b_2)$ or $W_R(b_1)$ is a subpath of $W_R(b_2)$, then $b_1 \in \shadz(b_2)$, and the inclusion $\shadz(b_1) \subset \shadz(b_2)$ follows from \cref{prop:shadow-comp}.\ref{prop:shadow-comp:item:inclusions_of_j_shadows}.
    Otherwise, $(b_1,b_2)$ is of one of the four types: left, right, inside, outside.
    By \cref{prop:shortcuts}, $W_L(b_1)$ is not left of $W_L(b_2)$ and $W_R(b_1)$ is not right of $W_R(b_2)$, hence, $(b_1,b_2)$ is neither left nor right.
    Therefore, $(b_1,b_2)$ is either inside or outside.
    Hence, by \cref{prop:inside_pair_implies_containment}, either $b_1 \in \shadz(b_2)$ or $b_2 \in \shadz(b_1)$. 
    If $b_1 \in \shadz(b_2)$ then by \cref{prop:shadow-comp}.\ref{prop:shadow-comp:item:inclusions_of_j_shadows} $\shadz(b_1)\subseteq\shadz(b_2)$ as desired. 
    So we assume that $b_2 \in \shadz(b_1)$.
    Let $y$ be the terminal element of $\shadz(b_1)$.
    In particular, $b_2 \in \shadz(y)$ and $y \leq b_1 \leq b_2$ in~$P$.
    Thus, by~\Cref{prop:top-shadows}, $b_2$ lies in the terminal block of $\shadz(y) = \shadz(b_1)$.
    Finally, by \cref{prop:shadow-comp}.\ref{prop:shadow-comp:item:hanging_element}, $\shadz(b_1) = \shadz(y) = \shadz(b_2)$, which concludes the proof.
\end{proof}

Let $b_1,b_2 \in B$.
We say that $b_1$ is \emph{left} of $b_2$ if $(b_1,b_2)$ is a left pair, and $b_1 \notin \shadz(b_2)$, $b_2 \notin \shadz(b_1)$.
We say that $b_1$ is \emph{right} of $b_2$ if $b_2$ is left of $b_1$.
In the bottom part of~\Cref{fig:pairs}, $b_1$ is left of $b_2$ in the examples on the left and in the middle, however, it is not the case in the example on the right.

Next, we prove simple facts about the left of and right of notions.
The goal (achieved in \cref{prop:left_porders_bs}) is to show that the relation of being left of (right of) partially orders $B$.
We start with a straightforward corollary of \cref{prop:hard_left_right_characterization} stated for emphasis.

\begin{corollary}\label{cor:notin_shads_implies_left_or_right}
    Let $b_1,b_2 \in B$.
    If $b_1 \notin \shadz(b_2)$ and $b_2 \notin \shadz(b_1)$, then either $b_1$ is left $b_2$ or $b_1$ is right of $b_2$. 
\end{corollary}

Note that the above corollary allows us to employ the following reasoning.
Say that we know that $b_1 \notin \shadz(b_2)$ and $b_2 \notin \shadz(b_1)$ for some $b_1,b_2 \in B$ and we want to prove that $b_1$ is left of $b_2$.
Then, by \cref{cor:notin_shads_implies_left_or_right}, either $b_1$ is left $b_2$ or $b_1$ is right of $b_2$.
This means that $(b_1,b_2)$ is either a left pair or a right pair.
At this point, it suffices to prove only one of the conditions required for a pair to be a left pair (or a right pair).
That is, the fact that $W_L(b_1)$ is left of $W_L(b_2)$ (or that $W_R(b_1)$ is left of $W_R(b_2)$) implies that $b_1$ is left of $b_2$.

\begin{proposition}\label{prop:charaterization_left_right_by_ys}
    Let $b_1,b_2 \in B$. 
    Let $y_1$ be the terminal element of $\shadz(b_1)$, and let $y_2$ be the terminal element of $\shadz(b_2)$.
  \begin{enumerate}
      \myitem{$(L)$} $b_1$ is left of $b_2$ if and only if $y_1$ is left of $b_2$. \label{prop:charaterization_left_right_by_ys:left} 
      \myitem{$(R)$} $b_1$ is left of $b_2$ if and only if $b_1$ is left of $y_2$.\label{prop:charaterization_left_right_by_ys:right} 
  \end{enumerate}
  Note that the above statements imply that $b_1$ is left of $b_2$ if and only if $y_1$ is left of $y_2$.
\end{proposition}
\begin{proof}
    We prove statement \ref{prop:in_one_shad_but_not_other_then_left:left}. 
    The argument for statement \ref{prop:in_one_shad_but_not_other_then_left:right} is symmetric.
    Note that $\shadz(b_1) = \shadz(y_1)$.
    It follows that $b_2 \notin \shadz(b_1)$ if and only if $b_2 \notin \shadz(y_1)$.
    Additionally, we claim that $b_1 \notin \shadz(b_2)$ if and only if $y_1 \notin \shadz(b_2)$.
    Indeed, by~\cref{prop:shadow-comp}.\ref{prop:shadow-comp:item:inclusions_of_j_shadows}, if $y_1 \in \shadz(b_2)$, then $\shadz(y_1) \subset \shadz(b_2)$ (by~\Cref{prop:comparability_implies_shadow_containment}), and so, $b_1 \in \shadz(b_2)$.
    On the other hand, if $b_1 \in \shadz(b_2)$, then $\shadz(b_1) \subseteq \shadz(b_2)$, and so, $y_1 \in \shadz(b_2)$.
    Summarizing, $b_1 \notin \shadz(b_2)$ and $b_2 \notin \shadz(b_1)$ if and only if $y_1 \notin \shadz(b_2)$ and $b_2 \notin \shadz(y_1)$.

    Assume that $y_1$ is left of $b_2$.
    By the above, to prove that $b_1$ is left of $b_2$, it suffices to show that $(b_1,b_2)$ is a left pair.
    Since $(y_1,b_2)$ is a left pair, $W_L(y_1)$ is a subpath of $W_L(b_1)$, and $W_R(y_1)$ is a subpath of $W_R(b_1)$, we obtain that $(b_1,b_2)$ is a left pair.
    
    Finally, assume that $b_1$ is left of $b_2$. 
    We will prove that $y_1$ is left of $b_2$.
    Again by the equivalence above, $y_1 \notin \shadz(b_2)$ and $b_2 \notin \shadz(y_1)$.
    Hence, by \cref{cor:notin_shads_implies_left_or_right}, either $y_1$ is left $b_2$ or $y_1$ is right of $b_2$.
    It suffices to prove that $W_L(y_1)$ is left of $W_L(b_2)$.
    Let $w = \gce(W_L(b_1),W_L(b_2))$.
    We claim that $w < y_1$ in~$P$.
    Indeed, if $y_1 \leq w$ in~$P$, then since $y_1$ lies in $W_L(b_1)$, we obtain that $y_1$ lies in $W_L(b_2)$.
    However, this implies $\shadz(y_1) \subset \shadz(b_2)$ (by~\cref{prop:comparability_implies_shadow_containment}).
    This is false, hence, indeed $w < y_1$ in~$P$.
    Since $w < y_1$ in~$P$ and $W_L(y_1)$ is a subpath of $W_L(b_1)$, we obtain that $W_L(y_1)$ left of $W_L(b_2)$.
    This completes the proof.
\end{proof}

\begin{proposition} \label{prop:in_one_shad_but_not_other_then_left}
    Let $b_1,b_2,d \in B$ with $b_1$ left of $b_2$.
  \begin{enumerate}
      \myitem{$(L)$} If $d \in \shadz(b_1)$ and $d \notin \shadz(b_2)$, then $d$ is left of $b_2$. \label{prop:in_one_shad_but_not_other_then_left:left} 
      \myitem{$(R)$} If $d \in \shadz(b_2)$ and $d \notin \shadz(b_1)$, then $d$ is right of $b_1$. \label{prop:in_one_shad_but_not_other_then_left:right} 
  \end{enumerate}
\end{proposition}
\begin{proof}
    We prove statement \ref{prop:in_one_shad_but_not_other_then_left:left} -- see~\Cref{fig:prop_shadows_d}. 
    The argument for statement \ref{prop:in_one_shad_but_not_other_then_left:right} is symmetric.
    Assume that $d \in \shadz(b_1)$ and $d \notin \shadz(b_2)$.
    Note that by \cref{prop:shadow-comp}.\ref{prop:shadow-comp:item:inclusions_of_j_shadows}, $\shadz(d) \subset \shadz(b_1)$.
    Since $b_1$ is left of $b_2$, we have $b_2 \notin \shadz(d)$, and so, $b_2 \notin \shadz(d)$.
    By \cref{cor:notin_shads_implies_left_or_right}, either $d$ is left $b_2$ or $d$ is right of $b_2$.

\begin{figure}[tp]
  \begin{center}
    \includegraphics{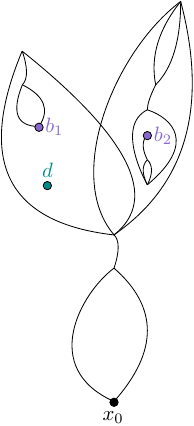}
  \end{center}
  \caption{
    An illustration of the statement of \cref{prop:in_one_shad_but_not_other_then_left}.\ref{prop:in_one_shad_but_not_other_then_left:left}. 
    We only draw the paths forming boundaries of shadows of $b_1$ and $b_2$.
  }
  \label{fig:prop_shadows_d}
\end{figure}

    First, assume that $d$ is on the boundary of $\shadz(b_1)$.
    If $d$ lies in $W_L(b_1)$, then $W_L(d)$ is a subpath of $W_L(b_1)$.
    Moreover, $\gce(W_L(b_1),W_L(b_2))$ is strictly less than $d$ in~$P$, since otherwise $d \in \shadz(b_2)$.
    Since $W_L(b_1)$ is left of $W_L(b_2)$, we obtain $W_L(d)$ left of $W_L(b_2)$.
    Symmetrically, if $d$ lies in $W_R(b_1)$, we obtain $W_R(d)$ left of $W_R(b_2)$.
    In both cases, we conclude that $d$ is left of $b_2$.


    Next, assume that $d$ is in the interior of $\shadz(b_1)$.
    Let $y_1$ be the terminal element of $\shadz(b_1)$.
    Since $b_1$ is left of $b_2$, by \cref{prop:charaterization_left_right_by_ys}.\ref{prop:charaterization_left_right_by_ys:left}, $y_1$ is left of $b_2$.
    By \cref{prop:paths_directions_in_shadows}, either $y_1 < d$ in~$P$ or $(d,y_1)$ is an inside pair.
    In the case $y_1<d$, \cref{prop:shadow-comp}.\ref{prop:shadow-comp:item:hanging_element} implies $\shadz(y_1) = \shadz(d)$, and so, $y_1$ is the terminal element of $\shadz(d)$. 
    Then by \cref{prop:charaterization_left_right_by_ys}.\ref{prop:charaterization_left_right_by_ys:left}, $d$ is left of $b_2$ as desired. Otherwise, assume $(d,y_1)$ is an inside pair. Therefore $W_R(d)$ is left of $W_R(y_1)$. And since $y_1$ is left of $b_2$, $W_R(y_1)$ is left of $W_R(b_2)$. By transitivity (\cref{prop:u_e_is_poset}), we obtain $W_R(d)$ left of $W_R(b_2)$, which completes the proof. 
\end{proof}

Next, we argue that the relation of being left (and as follows right as well) is transitive, and henceforth, its reflexive closure partially orders $B$.

\begin{proposition}\label{prop:left_porders_bs}
    For all $b_1,b_2,b_3 \in B$, if $b_1$ is left of $b_2$ and $b_2$ is left of $b_3$, then $b_1$ is left of~$b_3$.
\end{proposition}
\begin{proof}
    Assume that $b_1$ is left of $b_2$ and $b_2$ is left of $b_3$.
    By \cref{prop:u_e_is_poset}, we immediately obtain that $(b_1,b_3)$ is a left pair.
    It suffices to argue that $b_1 \notin \shadz(b_3)$ and $b_3 \notin \shadz(b_1)$.
    Suppose to the contrary that $b_1 \in \shadz(b_3)$.
    Since $b_1 \notin \shadz(b_2)$ and $b_2$ is left of $b_3$, by \cref{prop:in_one_shad_but_not_other_then_left}.\ref{prop:in_one_shad_but_not_other_then_left:right}, we obtain that $b_1$ is right of $b_2$, which is a contradiction.
    Similarly, if $b_3 \in \shadz(b_1)$, then we obtain that $b_3$ is left of $b_2$, which is again a contradiction and ends the proof. 
\end{proof}



\section{Reduction to a maximal good instance}
\label{sec:good_instance}
In~\Cref{sec:topology_instance}, we defined instances and developed a small theory describing them.
In~\cref{cor:poset-to-instance}, we showed that the main difficulty of~\cref{thm:cover-graph_se} is captured by the notion of instances.
The goal of this section is a further reduction from instances to 
good instances, defined in~\Cref{ssec:interface}.
Given an instance $(P,x_0,G,e_{-\infty},I)$, we focus only on pairs in $I$ that we call \q{risky} as two reversible sets can cover all the non-risky pairs, see~\Cref{prop:risky_low_dim}.
In~\Cref{ssec:escape-address}, for a risky pair $(a,b) \in I$, we define its escape number to be the least nonnegative integer such that $a \notin \shad_j(b)$.
Next, we split $I$ into $I_0$ and $I_1$ depending on the parity of their escape number.
The main technical statement of this section is~\Cref{lem:cgI-comprehensive}, which describes the structure of strict alternating cycles with all pairs in $I_{\theta}$ for a fixed $\theta\in\set{0,1}$. 
E.g., all pairs in such an alternating cycle must have the same escape number.
This structure allows us eventually to reduce the problem to instances that carry additional properties \ref{item:instance:not_in_shadow}--\ref{item:instance:dangerous} and also~\ref{item:instance:maximal}. 
The final outcome of this section is~\Cref{cor:interface}.

We fix an instance $(P,x_0,G,e_{-\infty},I)$ within Subsections~\ref{ssec:risky} to \ref{ssec:dangerous}.

\subsection{Risky pairs}\label{ssec:risky}
\vbox{We say that $(a,b) \in I$ is a \emph{risky} pair if
\begin{enumerateNumr}
    \item there exists $b' \in B$ with $a \leq b'$ in~$P$ such that $W_L(b')$ is left of $W_L(b)$, and \label{items:risky_b'}
    \item there exists $b'' \in B$ with $a \leq b''$ in~$P$ such that $W_R(b)$ is left of $W_R(b'')$. \label{items:risky_b''}
\end{enumerateNumr}
}
See \Cref{fig:risky}.
In fact, the \q{interesting} pairs are risky pairs, which is formally proved in the next proposition.
Later, in~\Cref{prop:property_of_risky_pairs}, we establish a key property of risky pairs.

\begin{figure}[tp]
  \begin{center}
    \includegraphics{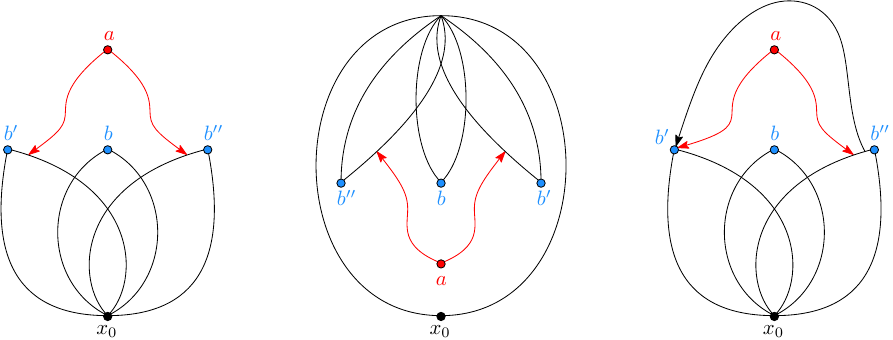}
  \end{center}
  \caption{
    In all three drawings the pair $(a,b)$ is a risky pair witnessed by $b'$ and $b''$.
    Note that $b'$ is not necessarily left of $b''$ (the middle drawing), in fact, $(b',b)$ is not even necessarily a left pair (the right drawing).
    Note that in the right drawing, the element $b'$ also can play a role of $b''$, that is, $W_R(b)$ is left of $W_R(b'')$.
    The situation in the left drawing is visibly the \q{cleanest} one.
    Further in the paper, we will introduce the notion of \q{dangerous} pairs, the left drawing is the only one among the three depicting a dangerous pair.
    Note that $a$ is not necessarily in $B$.
    We tend to mark such elements in red and the elements in $B$ in blue.
  }
  \label{fig:risky}
\end{figure}

\begin{proposition}\label{prop:risky_low_dim} 
    Let $I'$ be the set of all non-risky pairs in $I$.
    Then $\dim_P(I') \leq 2$.
\end{proposition}
\begin{proof}
    Let $I_1'$ be all the pairs in $I$ that do not satisfy~\ref{items:risky_b'} and let $I_2'$ be all the pairs in $I$ that do not satisfy~\ref{items:risky_b''}.
    Note that $I' = I_1' \cup I_2'$.
    We claim that $I_1'$ and $I_2'$ are reversible.
    We will prove that $I_1'$ is reversible, the proof for $I_2'$ is symmetric.
    Suppose to the contrary that $I_1'$ is not reversible, and so by~\cref{prop:dim-alternating-cycles}, $I_1'$ contains a strict alternating cycle $((a_1,b_1),\dots,(a_k,b_k))$.
    Let $i \in [k]$.
    Since $(a_i,b_i)$ does not satisfy~\ref{items:risky_b'}, $a_i \leq b_{i+1}$ in~$P$ (we consider the indices cyclically, e.g.\ $b_{k+1} = b_1$), and $b_i \parallel b_{i+1}$ in~$P$ we have $W_L(b_{i+1})$ right of $W_L(b_{i})$. 
    However, this cannot hold cyclically for all $i\in[k]$. 
    The contradiction completes the proof.
\end{proof}

\begin{proposition}\label{prop:property_of_risky_pairs}
    Let $(a,b) \in I$ be a risky pair, let $b_0 \in B$ with $a \parallel b_0$ in~$P$, and suppose that for some nonnegative integer $j$, $\shad_j(b_0)$ and $\shad_j(b)$ have the same initial element.
    If $a \in \shad_j(b_0)$, then $b \in \shad_j(b_0)$.
\end{proposition}
\begin{proof}
    See~\Cref{fig:risky_trapping}.
    Since $(a,b)$ is a risky pair, there exist $b',b'' \in B$ as in \ref{items:risky_b'} and \ref{items:risky_b''}.
    By \cref{obs:equivalence_for_shadows}, we obtain that
    $b'$ and $b''$ lie in the interior of $\shad_j(b_0)$.

\begin{figure}[tp]
  \begin{center}
    \includegraphics{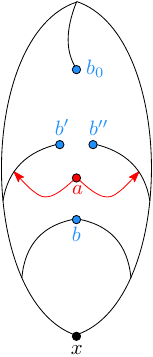}
  \end{center}
  \caption{
    Illustration of the statement of~\cref{prop:property_of_risky_pairs}.
    The element $x$ is the initial element of $\shad_j(b_0)$ and $\shad_j(b)$.
    The brief intuition of why the proposition holds is the following.
    Since $a \in \shad_j(b_0)$, the elements $b'$ and $b''$ witnessing that $(a,b)$ is a risky pair are in $\shad_j(b_0)$ and in turn \q{trap} $b$ in $\shad_j(b_0)$.
  }
  \label{fig:risky_trapping}
\end{figure}

    Let $x$ and $y$ be the initial and terminal elements of $\shad_j(b_0)$, respectively.
    By \cref{prop:paths_directions_in_shadows} applied to $b'$, either $W_L(y)$ is left of $W_L(b')$ (as $(b',y)$ is an inside pair) or $y<b'$ in~$P$.
    In the case where $y<b'$ in~$P$, by~\cref{prop:top-shadows}, $W_L(y)$ is a subpath of $W_L(b')$.
    Similarly, by \cref{prop:paths_directions_in_shadows} applied to $b''$, either $W_R(b'')$ is left of $W_R(y)$ or $W_R(y)$ is a subpath of $W_R(b'')$.
    We claim that
    \begin{equation}
    \begin{split}
    &\textrm{no witnessing path from $x_0$ to $b$ in~$P$ is left of $W_L(y)$, and}\\
    &\textrm{no witnessing path from $x_0$ to $b$ in~$P$ is right of $W_R(y)$.}
    \end{split}
    \label{eq:no-witnessing-paths-on-wrong-side-of-y}
    \end{equation}
    Indeed, if $U$ is a witnessing path from $x_0$ to $b$ in~$P$ and $U$ is left of $W_L(y)$, then $U$ is left of $W_L(b')$.
    However, this implies that $W_L(b)$ is left of $W_L(b')$, a contradiction. 
    The proof of the statement for rightmost paths goes analogously.

    First, suppose that neither of $W_L(b),W_L(y)$ is a subpath of the other and that neither of $W_R(b),W_R(y)$ is a subpath of the other.
    By \eqref{eq:no-witnessing-paths-on-wrong-side-of-y}, it follows that $W_L(y)$ is left of $W_L(b)$ and $W_R(b)$ is left of $W_R(y)$.
    Therefore, $(b,y)$ is an inside pair, and by \cref{prop:inside_pair_implies_containment}, $b \in \shad_j(y) = \shad_j(b_0)$.

    If $W_L(b)$ is a subpath of $W_L(y)$, then $b \in \shad_j(b_0)$.
    Similarly, if $W_R(b)$ is a subpath of $W_R(y)$, then $b \in \shad_j(b_0)$.

    Finally, it remains to consider the case, where $W_L(y)$ is a subpath of $W_L(b)$ or $W_R(y)$ is a subpath of $W_R(b)$.
    Assume that $W_L(y)$ is a subpath of $W_L(b)$, the proof in the other case is symmetric.
    It follows that $y \leq b$ in~$P$.
    Recall that either $W_L(y)$ is left of $W_L(b')$ or $W_L(y)$ is a subpath of $W_L(b')$.
    If $W_L(y)$ is left of $W_L(b')$, then $W_L(b)$ is left of $W_L(b')$, which is false as $b'$ witnesses \ref{items:risky_b'} for $(a,b)$.
    Hence, $W_L(y)$ is a subpath of $W_L(b')$.
    Since $b'$ is in the interior of $\shad_j(y)$, by \cref{cor:path-in-shadow} the path $y[W_L(b')]b'$ lies in $\shad_j(y)$.
    Suppose to the contrary that $b \notin \shad_j(y)$ and let $W$ be a witnessing path from $y$ to $b$ in~$P$.
    Note that the only element of $W$ in $\shad_j(y)$ is $y$.
    By \ref{items:leaving_shadows:y}, this yields that $x_0[W_L(y)]y[W]b$ is left of $W_L(b')$, which is left of $W_L(b)$ (by \ref{items:risky_b'}), which is a contradiction that ends the proof.
\end{proof}

\subsection{Escape number and address.}\label{ssec:escape-address}
Let $(a,b) \in I$.
We define the \emph{escape number} of $(a,b)$ to be the least nonnegative
integer $j$ such that $a\notin\shad_j(b)$.  
Note that this is well-defined as if $j = \sd(b) + 1$, then $\shad_j(b) = \emptyset$.
For each $\theta\in\{0,1\}$, we let 
    \[I_\theta = \{(a,b) \in I : \ \text{the escape number of $(a,b)$ is } j \text{ with } j\equiv\theta\text{ mod }2 \}.\]
Note that $I$ is partitioned into $I_0$ and $I_1$.

We define the \emph{address} of $(a,b)$ to be the pair $(j,x)$, where $j$ is the escape number of $(a,b)$ and $x$ is the initial element of $\shad_j(b)$.
See \Cref{fig:address}.

\begin{figure}[tp]
  \begin{center}
    \includegraphics{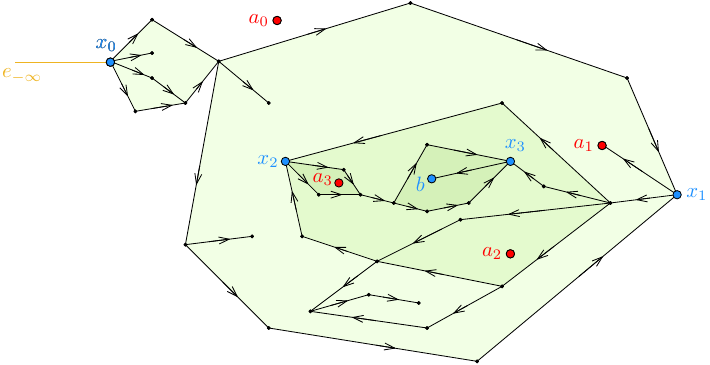}
  \end{center}
  \caption{
    For every $j \in \{0,1,2,3\}$, the escape number of $(a_j,b)$ is equal to $j$, and the address is equal to $(j,x_j)$.
  }
  \label{fig:address}
\end{figure}

We prove now a comprehensive lemma that provides 
structural information about strict alternating cycles with all pairs being risky and in $I_\theta$ for a fixed $\theta\in\set{0,1}$. 
See~\Cref{fig:comprehensive_lemma}.

\begin{figure}[tp]
  \begin{center}
    \includegraphics{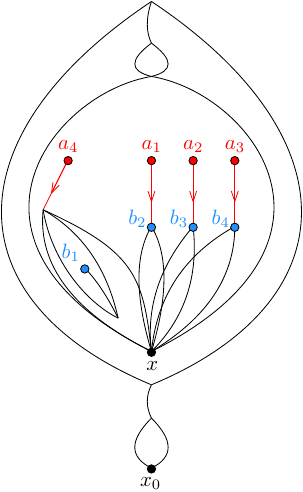}
  \end{center}
  \caption{
    An example of a strict alternating cycle with all the pairs in $I_0$ (the address of $(a_i,b_i)$ for each $i \in [4]$ is equal to $(2,x)$).
    In particular, in order to study the cycle, it suffices to study the terminal block of $\shad_2(x)$.
    \Cref{lem:cgI-comprehensive} shows that every strict alternating cycle with all the pairs in $I_\theta$ for some $\theta \in \{0,1\}$ \q{looks like this}.
    Note, however, that the shadow depth of each $b_i$ may be much higher than the number in the common address of the pairs in the cycle (as it indeed is for $b_1$).
  }
  \label{fig:comprehensive_lemma}
\end{figure}

\begin{lemma}\label{lem:cgI-comprehensive}
  Let $\theta\in\{0,1\}$, and let $((a_1,b_1),\dots,(a_k,b_k))$ be a strict
  alternating cycle in~$P$ with all the pairs being risky pairs in $I_\theta$. 
  Then there is a nonnegative integer $j$ and an element $x$ of~$P$ such that for every $\alpha\in[k]$, the address of $(a_\alpha,b_\alpha)$ is $(j,x)$; and for all distinct $\alpha,\beta \in [k]$, $b_\beta \notin \shad_j(b_\alpha)$.
\end{lemma}

\begin{proof}
  The indices in $[k]$ are considered cyclically, e.g.\ $a_{k+1} = a_{1}$.
  For each $\alpha\in[k]$, let $(j_\alpha,x_{\alpha})$ be the address of $(a_\alpha,b_\alpha)$. 
  Set $j=\min\{j_\alpha:\alpha\in[k]\}$. 
  Fix $\gamma\in[k]$ such that $j=j_{\gamma}$ and let $x=x_{\gamma}$.
  Eventually, we will show that $(j,x)$ is the common address of $(a_{\alpha},b_{\alpha})$ for all $\alpha\in[k]$.
  The proof is split into three intermediate statements encapsulated by the following claims.

\begin{claim}\label{claim:statements_for_minimum_j}
For all $\alpha\in[k]$, if $j>0$, then $x$ is the terminal element of $\shad_{j-1}(b_{\alpha})$.
\end{claim}
\begin{proofclaim}
If $j=0$, then the claim holds vacuously, thus, suppose that $j>0$ and let $\alpha \in [k]$.
Since $j$ is at most the escape number of $(a_\alpha,b_\alpha)$, the element $a_\alpha$ lies in $\shad_{j-1}(b_\alpha)$.
Let $\cgB_{\alpha}$ be the block of $\shad_{j-1}(b_{\alpha})$ such that $a_{\alpha}$ lies in $\cgB_{\alpha}$.
Since $a_{\alpha}\leq b_{\alpha+1}$ and $a_{\alpha}\parallel b_{\alpha}$ in~$P$, by \cref{prop:paths_stay_in_blocks}, $b_{\alpha+1}$ lies in the interior of~$\cgB_{\alpha}$. 
By~\cref{prop:shadow-comp}.\ref{prop:shadow-comp:item:inclusions_of_j_shadows},
\[
\shad_{j-1}(b_{\alpha+1})
\subseteq \shad_{j-1}(\max\cgB_{\alpha})
\subseteq \shad_{j-1}(b_{\alpha}).
\]
Since these inclusions hold for all $\alpha\in[k]$ cyclically, we conclude that for all $\alpha\in[k]$,
\[
\shad_{j-1}(b_{\alpha+1})
= \shad_{j-1}(\max\cgB_{\alpha})
= \shad_{j-1}(b_{\alpha}).
\]
Since $\shad_{j-1}(b_{\alpha})$ are the same for all $\alpha\in[k]$, 
their terminal blocks also coincide. 
Let $\cgB$ be the common terminal block. 
Recall that for every $\alpha \in [k]$, $b_{\alpha+1}$ lies in the terminal block of $\shad_{j-1}(b_{\alpha+1})$, thus, $\calB = \calB_{\alpha}$.
Since $x=x_{\gamma}$ is the terminal element of $\shad_{j-1}(b_{\gamma})$, we conclude that $x=\max \cgB$ and the assertion holds.
\end{proofclaim}

\begin{claim}\label{claim:walking_down}
    Let $\alpha,\delta \in [k]$ with $\delta \neq \alpha - 1$ and such that $a_{\delta}$ lies in the interior of $\shad_j(b_\alpha)$.
    Then, $a_\alpha$ lies in the interior of $\shad_j(b_\alpha)$. 
\end{claim}
\begin{proofclaim}
    Without loss of generality, assume that $\alpha = 1$.    Note that $\delta \neq k$. 
    We show that if $\beta\in[k-1]$ and 
    $a_\beta$ lies in the interior of $\shad_j(b_1)$, then 
    either $\beta=1$ or $a_{\beta-1}$ is also in the interior of $\shad_j(b_1)$.
    The statement of the claim immediately follows.

    Suppose that $\beta\in\set{2,\ldots,k-1}$ and $a_\beta$ lies in the interior of $\shad_j(b_1)$. 
    Recall that by~\cref{claim:statements_for_minimum_j}, $\shad_j(b_1)$ and $\shad_j(b_\beta)$ have a common initial element (namely $x$).
    Since the pair $(a_\beta,b_\beta)$ is a risky pair, $a_\beta \parallel b_1$ in~$P$, and $a_\beta$ is in the interior of $\shad_j(b_1)$, by \cref{prop:property_of_risky_pairs}, $b_\beta$ lies in the interior of $\shad_j(b_1)$.
    Since $a_{\beta-1} \leq b_\beta$ in~$P$ and $1 \neq \beta$, by \cref{obs:equivalence_for_shadows}, we obtain that $a_{\beta-1}$ is in the interior of $\shad_j(b_1)$, as desired.
\end{proofclaim}

\begin{figure}[tp]
  \begin{center}
    \includegraphics{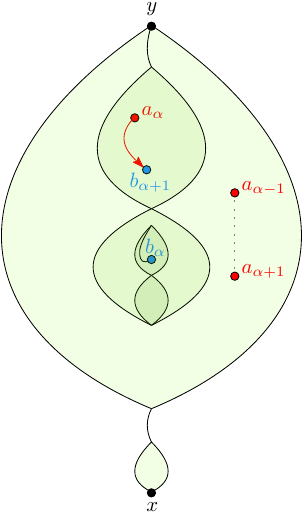}
  \end{center}
  \caption{
    An illustration of the proof of~\Cref{claim:j_is_constant}.
    We show $\shad_j(b_\alpha)$ (equal to $\shad_{j}(b_{\alpha+1})$).
    We arrive at $a_{\alpha+1} \in \shad_j(b_{\alpha+1})$, which is a contradiction. 
  }
  \label{fig:comprehensive_lemma_claim}
\end{figure}

\begin{claim}\label{claim:j_is_constant}
For all $\alpha\in[k]$, we have $j_\alpha = j$.
\end{claim}
\begin{proofclaim}
    Suppose to the contrary that there is some $\alpha \in [k]$ with $j_\alpha \neq j$.
    Since we consider only pairs in $I_\theta$ and $j$ is the minimum value of the escape number among all the pairs in the strict alternating cycle, we have $j_\alpha \geq j + 2$.
    Without loss of generality, we can assume that $j_{\alpha +1} = j$.

    The argument is illustrated in~\Cref{fig:comprehensive_lemma_claim}.
    By definition of the escape number, the element $a_\alpha$ lies in the interior of $\shad_{j+1}(b_\alpha)$.
    Since $a_\alpha \leq b_{\alpha+1}$ in~$P$, by \cref{obs:equivalence_for_shadows}, $b_{\alpha+1}$ lies in the interior of $\shad_{j+1}(b_\alpha)$.
    By \cref{prop:shadow-comp}.\ref{prop:shadow-comp:item:equality_of_small_shadows}, we have $\shad_j(b_\alpha) = \shad_j(b_{\alpha + 1})$. 

    Let $y$ be the terminal element of $\shad_j(b_\alpha) = \shad_j(b_{\alpha + 1})$.
    It follows that $y \leq b_\alpha$ and $y \leq b_{\alpha + 1}$ in~$P$, hence, by definition of strict alternating cycles, for every $\beta \in [k]$, we have $a_\beta \parallel y$ in~$P$.
    In particular, $a_{\alpha-1} \parallel y$ in~$P$.
    Since $a_{\alpha-1} \leq b_{\alpha}$ in~$P$, by \cref{obs:equivalence_for_shadows}, the element $a_{\alpha-1}$ lies in the interior of $\shad_j(b_\alpha) = \shad_j(b_{\alpha + 1})$.
    Therefore, by \cref{claim:walking_down}, $a_{\alpha+1}$ lies in the interior of $\shad_{j}(b_{\alpha+1})$, which contradicts $j_{\alpha + 1} = j$ and ends the proof of the claim.
\end{proofclaim}

To complete the proof of the lemma, it suffices to prove that for all distinct $\alpha,\beta \in [k]$, we have $b_\beta \notin \shad_j(b_\alpha)$.
Suppose to the contrary that there are distinct $\alpha,\beta\in[k]$ such that $b_\beta \in \shad_j(b_\alpha)$.
Since $b_\alpha \parallel b_\beta$ in~$P$, we obtain a stronger statement that $b_\beta$ lies in the interior of $\shad_j(b_\alpha)$.
Since $\alpha \neq \beta$, we have $a_{\beta-1} \parallel b_\alpha$ in~$P$.
This combined with the fact that $a_{\beta-1} \leq b_\beta$ in~$P$ and \cref{prop:paths_stay_in_blocks} implies that $a_{\beta-1}$ lies in the interior of $\shad_j(b_\alpha)$.
Therefore, by \cref{claim:walking_down}, we obtain that $a_\alpha$ lies in the interior of $\shad_j(b_\alpha)$, which yields $j_\alpha > j$, which contradicts \cref{claim:j_is_constant}.
This completes the proof of the lemma.
\end{proof}

\subsection{Dangerous pairs}\label{ssec:dangerous}
In this subsection, we refine the notion of risky pairs to the so-called dangerous pairs.
We say that $(a,b) \in I$ is a \emph{dangerous} pair if
\begin{enumerateNumd}
    \item there exists $b' \in B$ with $a \leq b'$ in~$P$ such that $b'$ is left of $b$, and \label{items:dangerous_b'}
    \item there exists $b'' \in B$ with $a \leq b''$ in~$P$ such that $b$ is left of $b''$. \label{items:dangerous_b''}
\end{enumerateNumd}

See examples of dangerous and non-dangerous pairs in \Cref{fig:risky}.
Dangerous pairs are the only \q{interesting} pairs in the following sense. 
Later on, we will apply \cref{lem:cgI-comprehensive}, to reduce the problem only to pairs so that in every strict alternating cycle, the elements $b$ are ordered left-to-right in the sense introduced in \cref{sec:ordering_elements_in_B}. 
We show that, assuming this property, the non-dangerous pairs have low dimension.

\begin{proposition}\label{prop:dangerous_low_dim}
    Let $I' \subset I$ be a set of pairs that are not dangerous such that for every strict alternating cycle $((a_1,b_1),\dots,(a_k,b_k))$ in~$P$ contained in $I'$, for all $\alpha,\beta \in [k]$ with $\alpha \neq \beta$, either $b_\alpha$ left of $b_{\beta}$ or $b_{\alpha}$ right of $b_\beta$.
    Then, $\dim_P(I') \leq 2$.
\end{proposition}
\begin{proof}
    Let $I_1'$ be all the pairs in $I'$ that do not satisfy~\ref{items:dangerous_b'} and let $I_2'$ be all the pairs in $I'$ that do not satisfy~\ref{items:dangerous_b''}.
    Note that $I' = I_1' \cup I_2'$.
    We claim that $I_1'$ and $I_2'$ are reversible.
    We will prove that $I_1'$ is reversible, the proof for $I_2'$ is symmetric.
    Suppose to the contrary that $I_1'$ is not reversible, and so by~\cref{prop:dim-alternating-cycles}, $I_1'$ contains a strict alternating cycle $((a_1,b_1),\dots,(a_k,b_k))$.
    Let $\alpha \in [k]$.
    By the assumption, either $b_\alpha$ is left of $b_{\alpha+1}$ or $b_{\alpha+1}$ is left of $b_\alpha$ (we consider the indices cyclically, e.g.\ $b_{k+1} = b_1$).
    Since $(a_\alpha,b_\alpha)$ does not satisfy~\ref{items:dangerous_b'}, $a_\alpha \leq b_{\alpha+1}$ in~$P$ we have that $b_{\alpha+1}$ is right of $b_\alpha$. 
    However, this cannot hold cyclically for all $\alpha\in[k]$. 
    The contradiction completes the proof.
\end{proof}


\begin{proposition}\label{prop:d_not_in_shadow_when_dangerous}
    Let $(a,b) \in I$ be a dangerous pair and let $d \in B$ such that either $d$ is left of $b$ or $b$ is left of $d$.
    Then $a \notin \shadz(d)$.
\end{proposition}
\begin{proof}
    Assume that $d$ is left of $b$.
    When $b$ is left of $d$, the proof is symmetric.
    Let $b'' \in B$ witness \ref{items:dangerous_b''} for $(a,b)$.
    Since $b'' \notin \shadz(b)$ and $a \leq b''$ in~$P$, by~\Cref{obs:equivalence_for_shadows}, we obtain $a \notin \shadz(b)$.

    Suppose to the contrary that $a \in \shadz(d)$.
    Since $a \leq b''$ in~$P$, there is a witnessing path $W$ from $a$ to $b''$ in~$P$. Let $u$ be the first element in $B$ on $W$.
    It follows that $u \in \shadz(d)$.
    Moreover, since $a \notin \shadz(b)$ and $a \leq u$ in~$P$, by~\Cref{obs:equivalence_for_shadows}, we obtain $u \notin \shadz(b)$.
    Since $d$ left of $b$, by \cref{prop:in_one_shad_but_not_other_then_left}.\ref{prop:in_one_shad_but_not_other_then_left:left}, $u$ is left of $b$.
    Since $u$ is left of $b$ and $b$ is left of $b''$, $u$ is left of $b''$ (by~\cref{prop:left_porders_bs}).
    In particular, $u \notin \shadz(b'')$, however, $u \leq b''$ in~$P$, which is a contradiction with \cref{prop:comparability_implies_shadow_containment}, hence, indeed, $a \notin \shadz(d)$.
\end{proof}

\begin{corollary}\label{prop:dangerous_a_not_in_shadows_b'_b''}\label{prop:dangerous_a_not_in_B}
    Let $(a,b) \in I$ be a dangerous pair, let \ref{items:dangerous_b'} be witnessed by $b' \in B$, and let \ref{items:dangerous_b''} be witnessed by $b'' \in B$. Then
        \[a \notin \shadz(b'), \ \ a \notin \shadz(b''),  \ \ \text{and} \ \ a \notin B.\]    
\end{corollary}
\begin{proof}
    The first two items are direct corollaries of~\cref{prop:d_not_in_shadow_when_dangerous}.
    For the last item, if $a \in B$, then  since $a \leq b'$ in~$P$, by \cref{prop:comparability_implies_shadow_containment}, $\shadz(a) \subset \shadz(b')$, which contradicts $a \notin \shadz(b')$, and proves that $a \notin B$.
\end{proof}

\subsection{Interface}\label{ssec:interface}

In order to capture the consequences of the theory developed in this section, we introduce the following notion. 
We say that an instance $(P,x_0,G,e_{-\infty},I)$ is a \emph{good instance} if
\begin{enumerateNumI}
    \setcounter{enumi}{5}
    \item $a \notin \shadz(b)$ for every $(a,b) \in I$;\label{item:instance:not_in_shadow}
    \item for every strict alternating cycle $((a_1,b_1),\dots,(a_k,b_k))$ in~$P$ contained in $I$, for all distinct  $\alpha,\beta \in [k]$, either $b_\alpha$ is left of $b_{\beta}$ or $b_{\alpha}$ is right of $b_\beta$;\label{item:instance:sacs} 
    \item every pair in $I$ is dangerous.\label{item:instance:dangerous}
\end{enumerateNumI}

Next, we show that for every instance, there is a good instance such that the dimension does not vary much.

\begin{corollary}\label{cor:interface}
    For every instance $(P,x_0,G,e_{-\infty},I)$, there exists a good instance $(P',x_0',G',e_{-\infty}',I')$ such that 
    $P'$ is a convex subposet of $P$, $I' \subset I$, and
    \[\dim_P(I) \leq 2\dim_{P'}(I') + 6.\]
\end{corollary}
\begin{proof}
    Let $\calI = (P,x_0,G,e_{-\infty},I)$ be an instance. 
    For every pair $(j,x)$ where $j$ is a nonnegative integer and $x$ is an element of $P$, 
    let 
    \begin{align*}
    I(j,x)&=\set{(a,b)\in I : \textrm{$(a,b)$ is risky and the address of $(a,b)$ is $(j,x)$}}.
    \intertext{We define}
    I_\theta &= \bigcup\{I(j,x) : \ j \text{ with } j\equiv\theta\text{ mod }2 \text{ and } x \text{ in } P \}, \textrm{ for each $\theta\in\set{0,1}$ and}\\
    I_2 &= \set{(a,b)\in I : \textrm{ $(a,b)$ is not risky}}.
    \end{align*}
    Note that $I_0,I_1,I_2$ are pairwise disjoint and their union is $I$. 
    Hence, applying \cref{prop:risky_low_dim}, we fix $\theta \in \{0,1\}$ such that
        \[\dim_P(I) \leq \dim_P(I_0) + \dim_P(I_1) + \dim_P(I_{2}) \leq 2\dim_P(I_\theta) + 2.\]
    By \cref{lem:cgI-comprehensive}, for every strict alternating cycle in~$P$ contained in~$I_\theta$, there exists $(j,x)$ such that all the pairs in the alternating cycle are in $I(j,x)$.
    Therefore, applying \cref{prop:dim_of_sum_incomparable_sets}, we fix a pair $(j,x)$ such that $\dim_P(I_\theta) = \dim_P(I(j,x))$. 
    Let $J = I(j,x)$.

    Let $P'$ be the subposet of $P$ induced by all elements $p$ in $P$ with $p \not< x$ in~$P$.
    Note that $P'$ is a convex subposet of $P$.
    Indeed, if $p,q,r$ in~$P$ with $p \leq q \leq r$ in~$P$ and $p,r$ in $P'$, then $r \not<x$ in~$P$ implies $q \not< x$ in~$P$, and so, $q$ is in $P'$.
    Since $P'$ is a convex subposet of $P$, we can set $G'$ to be the plane graph isomorphic to the cover graph of $P'$ inherited from $G$.
    Since $x_0$ lies in the exterior face of $G$ and the only element of $W_L(x)$ in $P$ that is in $P'$ is $x$, the element $x$ lies in the exterior face of $G'$.
    We set $x_0' = x$.
    Finally, let $e_{-\infty}' = e_{-\infty}$ when $x_0 = x$ and let $e_{-\infty}$ be the last edge of $W_L(x)$ otherwise.

    We claim that $\calI' = (P',x_0',G',e_{-\infty}',J)$ satisfies \ref{item:instance:planar_cover_graph}--\ref{item:instance:sacs}.
    Items \ref{item:instance:planar_cover_graph}, \ref{item:instance:x_0_minimal}, \ref{item:instance:plane-graph}, and \ref{item:instance:e_infty} are clear from the definitions.
    Note that $J \subset \Inc(P')$.
    Indeed, for every $(a,b) \in J$, the address of $(a,b)$ is $(j,x)$, and so, $x < b$ in~$P$; it follows that $a \not< x$ in~$P$, and finally both $a$ and $b$ are in~$P'$.
    Since $x < b$ in~$P$ for every $(a,b) \in J$, we find that $J$ is singly constrained by $x$ in~$P$. 
    This implies~\ref{item:instance:I_singly_constrained}.
    Let $b$ be an element of $P'$ with $x < b$ in~$P$.
    Note that $\shad_j(b)$ with respect to $\calI$ is equal to $\shad_0(b)$ with respect to $\calI'$.
    Since for every $(a,b) \in J$, the escape number with respect to $\calI$ is $j$, we have $a \notin \shad_j(b)$ with respect to $\calI$, and so, $a \notin \shad_0(b)$ with respect to $\calI'$.
    This proves \ref{item:instance:not_in_shadow}.
    Finally,~\ref{item:instance:sacs} follows from \cref{lem:cgI-comprehensive} and~\Cref{cor:notin_shads_implies_left_or_right}.

    In the last step, we enforce the condition \ref{item:instance:dangerous} on a slightly smaller instance.
    Let $I'$ be the set of all the pairs in~$J$ that are dangerous with respect to $\calI'$.
    By \ref{item:instance:sacs} and \cref{prop:dangerous_low_dim}, we have $\dim_{P'}(J \setminus I') \leq 2$.
    Note that replacing the set of pairs in an instance by its subset preserves conditions \ref{item:instance:planar_cover_graph}--\ref{item:instance:sacs}.
    It follows that $(P',x_0',G',e_{-\infty}',I')$ is a good instance.
    Moreover, $P'$ is a convex subposet of $P$, $I' \subset I$, and 
    \begin{align*}
        \dim_P(I) \leq 2\dim_P(I_\theta)+2 = 2\dim_{P'}(J)+2 &\leq 2(\dim_{P'}(I') + \dim_{P'}(J\setminus I')) + 2\\
        &\leq 2\dim_{P'}(I') + 6. \qedhere
    \end{align*}
\end{proof}


For a given set $X$ and $I \subset X \times X$, we define
\begin{align*}
    \pi_1(I) &= \{a \in X : \text{ there exists $b' \in X$ with $(a,b') \in I$} \} \text{ and }\\
    \pi_2(I) &= \{b \in X : \text{ there exists $a' \in X$ with $(a',b) \in I$} \}.
\end{align*}
We say that a good instance $(P,x_0,G,e_{-\infty},I)$ is \emph{maximal} if
\begin{enumerateNumI}
    \setcounter{enumi}{8}
    \item $I$ is the set of all pairs $(a,b)\in\Inc(P)$ such that $a \in \pi_1(I)$, $b \in \pi_2(I)$, $a\not\in\shadz(b)$, and $(a,b)$ is dangerous. \label{item:instance:maximal}
\end{enumerateNumI}



Note that it is not immediate if a good instance is \q{contained} in some maximal good instance. 
This issue is resolved in the next proposition.

\begin{proposition}\label{prop:good_instance_to_maximal_good_instance}
    For every good instance $(P,x_0,G,e_{-\infty},I)$, there exists $I^+ \subset \Inc(P)$ such that $I \subset I^+ \subset \Inc(P) \cap (\pi_1(I) \times \pi_2(I))$ and $(P,x_0,G,e_{-\infty},I^+)$ is a maximal good instance.
\end{proposition}
\begin{proof}
    Let $\calI = (P,x_0,G,e_{-\infty},I)$ be a good instance and let $I^+$ be the set of all pairs $(a,b) \in \Inc(P)$ such that there exist elements $a'$, $b'$ in~$P$ with $(a',b),(a,b')\in I$, $a\not\in\shadz(b)$, and $(a,b)$ is dangerous.
    Since $\calI$ satisfies \ref{item:instance:not_in_shadow} and \ref{item:instance:dangerous}, we have $I \subset I^+$.
    Let $\calI^+ = (P,x_0,G,e_{-\infty},I^+)$.
    It is clear that $\calI^+$ satisfies \ref{item:instance:planar_cover_graph}--\ref{item:instance:not_in_shadow} and~\ref{item:instance:dangerous}.
    By definition, $\calI^+$ also satisfies~\ref{item:instance:maximal}.
    Thus, it suffices to argue that $\calI^+$ satisfies~\ref{item:instance:sacs}.
    Let $((a_1,b_1),\dots,(a_k,b_k))$ be a strict alternating cycle in~$P$ contained in $I^+$.
    We plan to apply~\cref{lem:cgI-comprehensive}.
    To this end, note that since $a \notin \shadz(b)$ for every $(a,b) \in I^+$, the escape number of each pair $(a,b) \in I^+$ is $0$.
    Therefore, by~\cref{lem:cgI-comprehensive} applied in the context of $\calI^+$, for all distinct $\alpha,\beta \in [k]$, either $b_\alpha$ is left of $b_\beta$ or $b_\alpha$ is right of $b_\beta$.
    This shows that $\calI^+$ satisfies~\ref{item:instance:sacs}, and thus, ends the proof.
\end{proof}

\section{Topology of a maximal good instance}
\label{sec:topology-good-instance}
The tools built so far give ground for the first two steps of the proof of~\Cref{thm:cover-graph_se}.
The unfolding technique presented in the preliminaries (see~\cref{lem:PlanarCoverGraphReduction})
allows to reduce the statement of~\Cref{thm:cover-graph_se} to a statement on dimension of an instance (see~\Cref{cor:poset-to-instance}). 
In~\Cref{sec:good_instance}, we saw how to further reduce the problem to a statement on a good instance (see~\Cref{cor:interface}) and a maximal good instance (see~\Cref{prop:good_instance_to_maximal_good_instance}).
It remains to prove the following bound on dimension in maximal good instances. 

\begin{theorem}\label{thm:maximal_instance_imply_dim_boundedness}
    For every maximal good instance $(P,x_0,G,e_{-\infty},I)$, we have 
        \[\dim_P(I) \leq 16\se_P(I)^6 \cdot (\se_P(I)+3)^2.\]
\end{theorem}

Before we prove~\Cref{thm:maximal_instance_imply_dim_boundedness}, we need some more theory.
In this section, we work with a fixed maximal good instance $(P,x_0,G,e_{-\infty},I)$.

\subsection{Exposed paths}
Let 
    \begin{align*}
    B &= \{b \text{ in } P : x_0 \leq b \text{ in $P$} \} \text{ and}\\
    A &= \{a \text{ in } P : \textrm{there exists $(a,b)\in I$}\}.
    \end{align*}
Since every $(a,b)\in I$ is dangerous (see~\ref{item:instance:dangerous}), by \cref{prop:dangerous_a_not_in_B}, $A \cap B = \emptyset$.
Let $a \in A$ and $b \in B$ with $a < b$ in $P$.
We say that a witnessing path $W$ from $a$ to $b$ in $P$ is \emph{exposed} if $b$ is the only element of $W$ in $B$.
For every $a \in A$, we define
    \begin{align*}
    Y(a) &= \{y \in B : a < y \text{ in } P \text{ and } a \notin \shadz(y)\} \text{ and}\\
    Z(a) &= \{z \in Y(a) :\text{there exists an exposed witnessing path from $a$ to $z$ in $P$}\}.
    \end{align*}
Observe that every element in $Y(a)$ generates at least one element in $Z(a)$.
Indeed, let $y \in Y(a)$ and let $W$ be a witnessing path from $a$ to $y$ in $P$.
Let $z$ be the first element in $B$ in $W$.
Since $z \leq y$ in $P$, we have $\shadz(z) \subset \shadz(y)$ (by~\Cref{prop:comparability_implies_shadow_containment}), and so, $z \in Y(a)$.
Moreover, $a[W]z$ is an exposed witnessing path from $a$ to $z$ in $P$, and so, $z \in Z(a)$.
Note that when 
$z\in Z(a)$, there can be a witnessing path from $a$ to $z$ that is not 
exposed.
In particular, the set $Z(a)$ is not necessarily an antichain. 
See~\Cref{fig:exposed_paths}.

Before, we proceed with the material on exposed paths and elements in $Z(a)$, we state a technical corollary of~\Cref{prop:dangerous_a_not_in_shadows_b'_b''}.
It follows since (by~\ref{item:instance:dangerous}) all the pairs in $I$ are dangerous.
We will use it extensively.

\begin{corollary}\label{prop:dangerous-implies-in-Y}
    Let $(a,b) \in I$ and let $d \in B$ such that either $b$ is left of $d$ or $d$ is left of $b$.
    Then, $a \notin \shadz(d)$.
    Moreover, if $a < d$ in $P$, then $d \in Y(a)$.
\end{corollary}


\begin{figure}[tp]
  \begin{center}
    \includegraphics{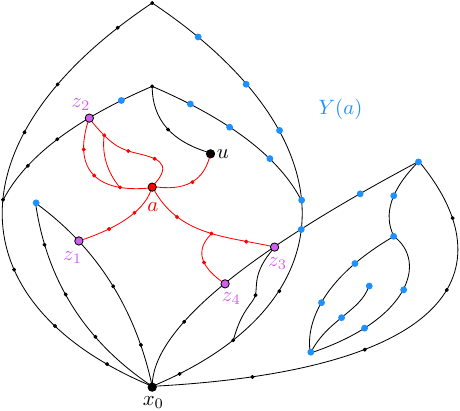}
  \end{center}
  \caption{
        $Z(a) = \{z_1,z_2,z_3,z_4\}$.
        Note that $u \notin Y(a)$ as $a \in \shadz(u)$.
        Note also that $Z(a)$ is not an antichain in $P$.
        We mark the elements in $Y(a)$ with blue.
  }
  \label{fig:exposed_paths}
\end{figure}

\begin{proposition}\label{prop:properties_of_Z}
    Let $a \in A$, $b \in B$, $y \in Y(a)$, and $z,z' \in Z(a)$. Then,
    \begin{enumerate}
        \item if $a \notin \shadz(b)$, then $z \notin \Int \shadz(b)$; \label{prop:properties_of_Z:item:not_interior_of_y} 
        \item if $z\leq y$ in $P$, then $z \in \partial\shadz(y)$; \label{prop:properties_of_Z:item:boundary_of_y}
        \item $\sd(z) = 0$; \label{prop:properties_of_Z:item:sd_of_z}
        \item if $z \parallel z'$ in $P$, then either $z$ is left of $z'$ or $z$ is right of $z'$. \label{prop:properties_of_Z:left_right}
    \end{enumerate}
\end{proposition}
\begin{proof}
    In order to prove~\ref{prop:properties_of_Z:item:not_interior_of_y}, suppose  that $a\notin \shadz(b)$ and $z\in\shadz(b)$. 
    We argue that $z$ must be in the boundary of $\shadz(b)$.
    Since $z \in Z(a)$, there is an exposed witnessing path $W$ from $a$ to $z$ in $P$.
    Note that $W$ intersects $\partial \shadz(b)$.
    The only element of $W$ in $B$ is $z$, and all the elements of $\partial\shadz(b)$ are in $B$, hence, $z \in \partial\shadz(b)$. This completes the proof of \ref{prop:properties_of_Z:item:not_interior_of_y}.

    In order to prove~\ref{prop:properties_of_Z:item:boundary_of_y} suppose that $z \leq y$ in $P$.
    By~\cref{prop:comparability_implies_shadow_containment}, we have $\shadz(z) \subset \shadz(y)$.
    In particular, $z \in \shadz(y)$.
    By~\ref{prop:properties_of_Z:item:not_interior_of_y}, $z$ lies on the boundary of $\shadz(y)$, which proves~\ref{prop:properties_of_Z:item:boundary_of_y}.
    Item \ref{prop:properties_of_Z:item:sd_of_z} follows from \ref{prop:properties_of_Z:item:boundary_of_y} and \cref{prop:shadow-comp}.\ref{prop:shadow-comp:item:sd_boundary_of_block}.
    Finally, suppose that $z \parallel z'$ in $P$.
    Since $z' \in Z(a) \subset Y(a)$, $a \notin \shadz(z)$.
    If $z' \in \shadz(z)$, then $z' \in \Int \shadz(z)$, which is a contradiction with~\ref{prop:properties_of_Z:item:not_interior_of_y}.
    It follows that $z' \notin \shadz(z)$, and symmetrically, we obtain $z \notin \shadz(z')$.
    Therefore, by \cref{cor:notin_shads_implies_left_or_right}, either $z$ is left of $z'$ or $z$ is right of $z'$, which proves~\ref{prop:properties_of_Z:left_right}.
\end{proof}

\begin{proposition}\label{prop:left_is_preserved_from_y_to_z}
  Let $a \in A$, $b \in B$ with $a \parallel b$ in $P$ and $a \notin \shadz(b)$.
  Let $y,y' \in Y(a)$ with $y' \leq y$ in $P$. 
  \begin{enumerate}
      \myitem{$(L)$} If $y$ is left of $b$, then $y'$ is left of $b$.
      \label{prop:left_is_preserved_from_y_to_z:left} 
      \myitem{$(R)$} If $y$ is right of $b$, then $y'$ is right of $b$. 
      \label{prop:left_is_preserved_from_y_to_z:right} 
  \end{enumerate}
\end{proposition}
\begin{proof}
    We prove only \ref{prop:left_is_preserved_from_y_to_z:left} as the proof of \ref{prop:left_is_preserved_from_y_to_z:right} is symmetric.
    Assume that $y$ is left of $b$.
    Since $a \notin \shadz(b)$, by \cref{obs:equivalence_for_shadows}, we have $y' \notin \shadz(b)$.
    Furthermore, since $y' \leq y$ in $P$, we have, $y' \in \shadz(y)$ (by \cref{prop:comparability_implies_shadow_containment}).
    Finally, by \cref{prop:in_one_shad_but_not_other_then_left}.\ref{prop:in_one_shad_but_not_other_then_left:left}, since $y$ is left of $b$, we obtain that $y'$ is left of~$b$.
\end{proof}

\begin{corollary}\label{cor:left_is_preserved_from_y_to_z}
    Let $(a,b) \in I$ and let $y,y' \in B$ with $a < y' \leq y$ in $P$. 
  \begin{enumerate}
      \myitem{$(L)$} If $y$ is left of $b$, then $y'$ is left of $b$.
      \label{cor:left_is_preserved_from_y_to_z:left} 
      \myitem{$(R)$} If $y$ is right of $b$, then $y'$ is right of $b$. 
      \label{cor:left_is_preserved_from_y_to_z:right} 
  \end{enumerate}
\end{corollary}
\begin{proof}
    We prove only \ref{prop:left_is_preserved_from_y_to_z:left} as the proof of \ref{prop:left_is_preserved_from_y_to_z:right} is symmetric.
    Assume that $y$ is left of $b$.
    By~\ref{item:instance:dangerous}, $(a,b)$ is dangerous, and so by~\cref{prop:d_not_in_shadow_when_dangerous}, $a \notin \shadz(y)$.
    Since $y' \leq y$ in $P$, we have $\shadz(y') \subset \shadz(y)$ (by~\cref{prop:comparability_implies_shadow_containment}). 
    It follows that $a \notin \shadz(y')$.
    In particular, $y,y' \in Y(a)$.
    By~\ref{item:instance:not_in_shadow}, $a \notin \shadz(b)$.
    Finally, we apply~\cref{prop:left_is_preserved_from_y_to_z}.\ref{prop:left_is_preserved_from_y_to_z:left} to obtain that $y'$ is left of~$b$. 
\end{proof}


For every $a \in A$, let $\calM(a)$ be the family of all the paths of the form $M = x_0[U]z[V]a$, where
\begin{itemize}
    \item $z \in Z(a)$ (we call the element $z$ the \emph{peak} of $M$),
    \item $U$ is a witnessing path from $x_0$ to $z$ in $P$,
    \item $V$ is an exposed witnessing path from $a$ to $z$ in $P$.
\end{itemize}
Since every path in $\calM(a)$ starts in $x_0$ and ends in $a$, no path in $\calM(a)$ contains another path from $\calM(a)$ as a subpath.
Therefore, by \cref{obs:ue_ordering_is_usually_linear}, the relation of being left linearly orders $\calM(a)$.
Note also that $\calM(a)$ is nonempty as there exists $b\in B$ such that $(a,b)\in I$,  so by \ref{item:instance:dangerous} and \ref{items:dangerous_b'}, there exists $b'\in B$ with $a< b'$ in $P$, thus, $b' \in Y(a)$, and therefore, $Z(a)$ is nonempty.
In particular, $\calM(a)$ has a minimum and a maximum element.
Let $M_L(a) \in \calM(a)$ be such that $M_L(a)$ is left of $M$ for every $M \in \calM(a)$ with $M\neq M_L(a)$, and let $M_R(a) \in \calM(a)$ be such that $M_R(a)$ is right of $M$ for every $M \in \calM(a)$ with $M\neq M_R(a)$.
Moreover, denote by $z_L(a)$ the peak of $M_L(a)$ and by $z_R(a)$ the peak of $M_R(a)$.
See~\Cref{fig:paths_M(a)}.

We finish this subsection with a series of simple and intuitive propositions on properties of $M_L(a)$ and $z_L(a)$ (resp.\ $M_R(a)$ and $z_R(a)$) for $a\in A$. 
\Cref{prop:paths_M_consistent} says that $M_L(a)$ contains $W_L(z_L(a))$. 
\Cref{prop:ML_consistent} gives $x_0$-consistency of paths in $\set{M_L(a) : a\in A}$. 
Propositions~\ref{prop:order_on_Ms_induces_order_in_zs} and \ref{prop:z_L_b_z_R} describe the position of $z_L(a)$ within $Z(a)$ and with respect to $b$ when $(a,b)\in I$.

\begin{figure}[tp]
  \begin{center}
    \includegraphics{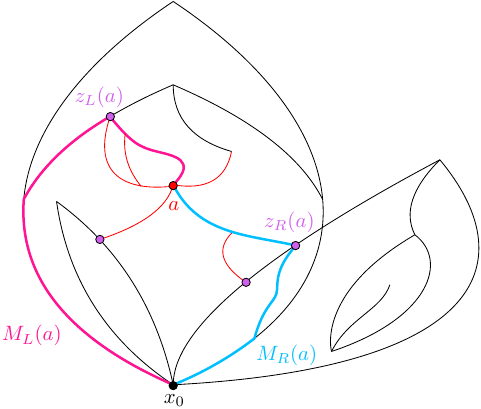}
  \end{center}
  \caption{
    The paths $M_L(a)$ and $M_R(a)$ drawn in the same poset as in~\Cref{fig:exposed_paths}.
    }
  \label{fig:paths_M(a)}
\end{figure}

\begin{proposition}\label{prop:paths_M_consistent}
    Let $a\in A$. Then
    \[x_0[M_L(a)]z_L(a) = W_L(z_L(a)) \ \ \text{ and } \ \ x_0[M_R(a)]z_R(a) = W_R(z_R(a)).\]
\end{proposition}
\begin{proof}
    We prove only the first identity as the proof of the second one is symmetric.
    Suppose to the contrary that $x_0[M_L(a)]z_L(a) \neq W_L(z_L(a))$. 
    Note that both paths are witnessing paths from $x_0$ to $z_L(a)$. 
    Thus, $W_L(z_L(a))$ is left of $x_0[M_L(a)]z_L(a)$.
    Consider the path $M = x_0[W_L(z_L(a))]z_L(a)[M_L(a)]a$.
    It is immediate to see that $M \in \calM(a)$ and $M$ is left of $M_L(a)$, which is a contradiction.
\end{proof}

\begin{proposition}\label{prop:ML_consistent}
    Let $a_1,a_2 \in A$.
  \begin{enumerate}
      \myitem{$(L)$} The paths $M_L(a_1)$ and $M_L(a_2)$ are $x_0$-consistent. \label{prop:ML_consistent:left}
      \myitem{$(R)$} The paths $M_R(a_1)$ and $M_R(a_2)$ are $x_0$-consistent.\label{prop:ML_consistent:right}
  \end{enumerate}
\end{proposition}
\begin{proof}
    We prove only~\ref{prop:ML_consistent:left} as the proof of~\ref{prop:ML_consistent:right} is symmetric.
    Let $z_i=z_L(a_i)$ and $M_i=M_L(a_i)$ for each $i\in[2]$. 
    If one of the paths $M_1,M_2$ is a subpath of the other, then the assertion follows.
    Hence, assume without loss of generality that $M_1$ is left of $M_2$.
    Let $u = \gce(M_1,M_2)$.
    Suppose to the contrary that the paths are not $x_0$-consistent.
    Then, there exists an element $v$ common to both paths with $v$ not in $x_0[M_1]u=x_0[M_2]u$.
    Note that by~\Cref{prop:paths_M_consistent} and~\Cref{prop:W-consistent}.\ref{prop:W-consistent:left}, $v \notin B$.
    The path $M = x_0[M_1]v[M_2]a_2$ is left of $M_2$.
    We claim that $M \in \calM(a_2)$, which contradicts the definition of $M_2$.
    The path $a_1[M_L(a_1)]v[M_L(a_2)]a_2$ does not have elements in $B$, hence, $a_1 \notin \shadz(z_1)$ implies $a_2 \notin \shadz(z_1)$.
    Since additionally the path $z_1[M_1]v[M_2]a_2$ is an exposed witnessing path from $a_2$ to $z_1$ in $P$, we have $z_1 \in Z(a_2)$.
    Moreover, $M = x_0[M_1]z_1[M_1]v[M_2]a_2$, and so, $M \in \calM(a_2)$.
\end{proof}

\begin{proposition}\label{prop:order_on_Ms_induces_order_in_zs}
    Let $a \in A$ and let $z \in Z(a)$.
  \begin{enumerate}
      \myitem{$(L)$} Either $z_L(a)$ is left of $z$ or $z_L(a)$ and $z$ are comparable in $P$. \label{prop:order_on_Ms_induces_order_in_zs:left}
      \myitem{$(R)$} Either $z_R(a)$ is right of $z$ or $z_R(a)$ and $z$ are comparable in $P$. \label{prop:order_on_Ms_induces_order_in_zs:right}
  \end{enumerate}
\end{proposition}
\begin{proof}
    We prove only \ref{prop:order_on_Ms_induces_order_in_zs:left} as the proof of \ref{prop:order_on_Ms_induces_order_in_zs:right} is symmetric.
    Assuming that $z\parallel z_L(a)$ in $P$, it suffices to prove that $z_L(a)$ is left of $z$.
    By~\cref{prop:properties_of_Z}.\ref{prop:properties_of_Z:left_right}, either $z_L(a)$ is left of $z$ or $z$ is left of $z_L(a)$, hence, it is enough to prove that $W_L(z_L(a))$ is left of $W_L(z)$.
    Let $M = x_0[W_L(z)]z[V]a$, where $V$ is an exposed witnessing path from $a$ to $z$ in $P$.
    Note that $M \in \calM(a)$ and $z$ is the peak of $M$.
    If $M_L(a) = M$, then $z_L(a) = z$, which contradicts the assumption that $z \parallel z_L(a)$ in $P$. 
    Thus, $M_L(a)$ is left of $M$.
    Let $w = \gce(M_L(a),M)$.
    If $w \notin B$, then the peaks of $M_L(a)$ and $M$ coincide, that is, $z_L(a) = z$, which is a contradiction again, hence, $w \in B$.
    Therefore, $W_L(z_L(a))$ is left of $W_L(z)$, 
    which ends the proof.
\end{proof}

\begin{proposition}\label{prop:z_L_b_z_R}
    Let $(a,b) \in I$. Then $z_L(a)$ is left of $b$, and $z_R(a)$ is right of $b$.
\end{proposition}
\begin{proof}
    We will show that $z_L(a)$ is left of $b$.
    The proof that $z_R(a)$ is right of $b$ is symmetric.
    The pair $(a,b)$ is a dangerous pair by \ref{item:instance:dangerous}.
    By \ref{items:dangerous_b'}, there exists $b' \in B$ with $a < b'$ in $P$ such that $b'$ is left of $b$.
    By \cref{prop:dangerous_a_not_in_shadows_b'_b''}, $a \notin \shadz(b')$.
    It follows that $b' \in Y(a)$.
    Let $z' \in Z(a)$ with $z' \leq b'$ in $P$.
    By~\cref{cor:left_is_preserved_from_y_to_z}, $z'$ is left of $b$.
    By~\ref{item:instance:not_in_shadow}, $a \notin \shadz(b)$.
    Since $a < z_L(a)$ in $P$, by~\cref{obs:equivalence_for_shadows}, we have 
    \begin{equation}\label{eq:z_L-notin-b}
        z_L(a) \notin \shadz(b).
    \end{equation}
    By \cref{prop:order_on_Ms_induces_order_in_zs}.\ref{prop:order_on_Ms_induces_order_in_zs:left}, either $z_L(a)$ is left of $z'$ or $z_L(a)$ and $z'$ are comparable in $P$.
    If $z_L(a)$ is left of $z'$, then by transitivity (\cref{prop:left_porders_bs}), $z_L(a)$ is left of $b$.
    Thus, we assume that $z_L(a)$ and $z'$ are comparable in $P$.
    First, assume that $z_L(a) \leq z'$ in $P$.
    Since $(a,b) \in I$, $a \leq z_L(a) \leq z' \leq b'$ in $P$, and $b'$ is left of $b$, by~\cref{cor:left_is_preserved_from_y_to_z}.\ref{cor:left_is_preserved_from_y_to_z:left}, we obtain that $z_L(a)$ is left of~$b$.
    
    Finally, assume that $z' < z_L(a)$ in $P$.
    In particular, $z' \in \shadz(z_L(a))$ (by \cref{prop:comparability_implies_shadow_containment}), however, by \cref{prop:properties_of_Z}.\ref{prop:properties_of_Z:item:not_interior_of_y}, $z'$ does not lie in the interior of $\shadz(z_L(a))$, thus, $z'$ is an element of $W_L(z_L(a)) \cup W_R(z_L(a))$.
    Next, we claim that 
    \begin{equation}\label{eq:b-notinz_L}
        b \notin \shadz(z_L(a)).
    \end{equation}
    Suppose to the contrary that $b \in \shadz(z_L(a))$.
    By \cref{prop:paths_directions_in_shadows}.\ref{prop:paths_directions_in_shadows:right}, $W_R(b)$ is not right of $W_R(z_L(a))$.
    If $z'$ is an element of $W_R(z_L(a))$, then $W_R(z')$ is a subpath of $W_R(z_L(a))$, although, since $z'$ is left of $b$, this implies that $W_R(b)$ is right of $W_R(z_L(a))$, which is a contradiction.
    Therefore, $z'$ is an element of $W_L(z_L(a))$ and it is not an element of $W_R(z_L(a))$.
    In other words, $z'$ is strictly on the left side of some shadow block $\calB$ of $\shadz(z_L(a))$.
    The path $M_L(a)$ is left of $M = x_0[W_L(z')]z'[W]a$, where $W$ is an exposed witnessing path from $a$ to $z'$ in $P$.
    By~\Cref{prop:paths_M_consistent}, $x_0[M_L(a)]z_L(a) = W_L(z_L(a))$.
    It follows that $x_0[M_L(a)]z' = x_0[M]z'$ and the edges following $z'$ in $M_L(a)$ and $M$ are distinct.
    Moreover, by \ref{items:leaving_shadows:left}, the first edge of $z'[M]a = W$ is in the interior of $\calB$. Since the path $W$ is exposed, all the vertices of $W$ except $z'$ are in the interior of $\calB$, and so, in $\shadz(z_L(a))$.
    In particular, $a \in \shadz(z_L(a))$, which is a contradiction.
    This contradiction completes the proof of~\eqref{eq:b-notinz_L}, that is, $b \notin \shadz(z_L(a))$.
    Recall that also $z_L(a) \notin \shadz(b)$ (by~\eqref{eq:z_L-notin-b}).
    Therefore, by \cref{cor:notin_shads_implies_left_or_right}, either $z_L(a)$ is left of $b$ or $b$ is left of $z_L(a)$.
    Recall that $z'$ is an element of $W_L(z_L(a)) \cup W_R(z_L(a))$.
    If $z'$ lies on $W_L(z_L(a))$, then $W_L(z')$ is a subpath of $W_L(z_L(a))$, and since $z'$ is left of $b$, $W_L(z_L(a)))$ is left of $W_L(b)$, which implies that $z_L(a)$ is left of $b$.
    Similarly, if $z'$ lies on $W_R(z_L(a))$, then $W_R(z')$ is a subpath of $W_R(z_L(a))$, and since $z'$ is left of $b$, $W_R(z_L(a)))$ is left of $W_R(b)$, which implies that $z_L(a)$ is left of $b$.
\end{proof}



\subsection{Regions}\label{ssec:regions}

Let $a \in A$, $u,v \in Z(a)$ with $u$ left of $v$, let $U$ be an exposed  witnessing path from $a$ to $u$, let $V$ be an exposed witnessing path from $a$ to $v$, 
and let
\begin{align*}
q&=\text{the maximal common element of $\WL(u)$ and $\WR(v)$ in $P$},\\
m&=\text{the maximal common element of $U$ and $V$ in $P$},\\
\gamma&=q[\WR(v)]v[V]m[U]u[\WL(u)]q.
\end{align*}
Note that $\gamma$ is a simple closed curve.
See~\Cref{fig:mixed_regions}.

Since $u$ is left of $v$, we have $q \notin \{u,v\}$, and in particular, both $q[W_L(u)]u$ and $q[W_R(v)]v$ contain at least one edge.
Let $e$ be the first edge of $q[W_L(u)]u$ and 
$f$ be the first edge of $q[W_R(v)]v$. 
Note that the only common element of $\gamma$ and  $x_0[W_L(u)]q$ is $q$.
By~\cref{prop:W_L-left-of-W_R}, $f$ follows $e$ in the counterclockwise orientation of $\gamma$.
In other words, the cyclic ordering of the edges in $q[\WR(v)]v[V]m[U]u[\WL(u)]q$ is the counterclockwise orientation of $\gamma$.

We define $\calR(a,u,v,U,V)$ to be the region of $\gamma$.
Let $\calR = \calR(a,u,v,U,V)$.
We say that $q$ is the \emph{lower-min} of $\calR$ and $m$ is the \emph{upper-min} of $\calR$.
We associate with $\calR$ two paths $\gamma_L = x_0[W_L(u)]u[U]m$ and $\gamma_R = x_0[W_R(v)]v[V]m$.
We say that the elements of $q[\gamma_L]m$  and $q[\gamma_R]m$ are on the \emph{left side} of $\calR$ and on the \emph{right side} of $\calR$ respectively.
All elements on the left (right) side of $\calR$ except $q$ and $m$ are said to be \emph{strictly} on the left (right) side of $\calR$.

For convenience, we use the following notation.
A tuple $(a,u,v,U,V,\calR,q,m,\gamma_L,\gamma_R)$ is a \emph{region tuple} if
$a \in A$, $u,v \in Z(a)$ with $u$ left of $v$, $U$ and $V$ are exposed witnessing paths in $P$ from $a$ to $u$ and $a$ to $v$ respectively, $\calR = \calR(a,u,v,U,V)$, $q$ is the lower-min of $\calR$, $m$ is the upper-min of $\calR$, $\gamma_L = x_0[W_L(u)]u[U]m$, and $\gamma_R = x_0[W_R(v)]v[V]m$.

\begin{figure}[tp]
  \begin{center}
    \includegraphics{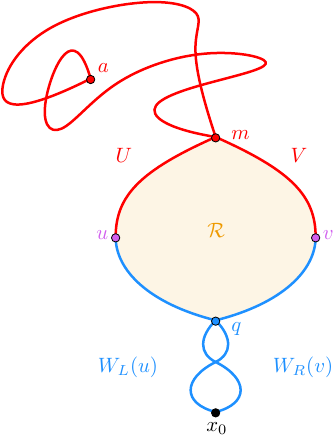}
  \end{center}
  \caption{
    The region $\calR = \calR(a,z_1,z_2,W_1,W_2)$.
    Note that $a$ can be in the interior of $\calR$, in the exterior of $\calR$, or equal to $m$.
  }
  \label{fig:mixed_regions}
\end{figure}

Let $(a,u,v,U,V,\calR,q,m,\gamma_L,\gamma_R)$ be a region tuple.
Given an edge incident to the boundary of a region, we characterize when it lies in the region.
Note similarities to \ref{items:leaving_shadows:x}--\ref{items:leaving_shadows:right} and recall~\Cref{fig:characterize-block}.

Let $w$ be an element in $\partial \calR$.
If $w$ is on the left side of $\cgR$, let $e^+_L$ and $e^-_L$ be, respectively, the edges (provided they exist) immediately after and immediately before $w$ on the path $\gamma_L$.
Also, if $w$ is on the right side of $\cgR$, let $e^+_R$ and $e^-_R$ be, respectively, the edges (provided they exist) immediately after and immediately before $w$ on $\gamma_R$.
By~\cref{obs:when_in_gamma}, an edge $e$ incident to $w$ lies in $\calR$ if and only if one of the following holds 
\begin{enumerateRegion}
    \item $w = q$ and $e^+_L\preccurlyeq e \preccurlyeq e^+_R$ in the $q$-ordering;\label{items:leaving_regions:x}
    \item $w = m$ and $e^-_R\preccurlyeq e \preccurlyeq e^-_L$ in the $m$-ordering;\label{items:leaving_regions:y}
    \item $w$ is strictly on the left side and $e^+_L \preccurlyeq e \preccurlyeq e^-_L$ in the $w$-ordering;\label{items:leaving_regions:left}
    \item $w$ is strictly on the right side and $e^-_R \preccurlyeq e \preccurlyeq e^+_R$ in the $w$-ordering.\label{items:leaving_regions:right}
\end{enumerateRegion}


Assume that $w = q$.
If $u = x_0$, we set $e_L^- = e_R^- = e_{-\infty}$.
Otherwise, we set $e_L^-$ and $e_R^-$ to be the edges preceding $w$ in $W_L(u)$ and $W_R(v)$ respectively. 
We claim that
\begin{enumerateRegion}
\setcounter{enumi}{4}
    \item $e_R^+ \prec e_R^- \preccurlyeq e_L^- \prec e_L^+$ in the $q$-ordering.
    \label{item:ordering-around-q}
\end{enumerateRegion}
When $q = x_0$, by~\ref{item:instance:e_infty}, we know that $e_{-\infty}$ does not lie in $\calR$, and so,~\ref{item:ordering-around-q} holds.
Next, assume that $q \neq x_0$.
Let $W$ be a witnessing path from $x_0$ to $q$ in $P$.
This path can intersect $\partial\calR$ only in $q$.
Recall that $x_0 \notin \Int \calR$.
Therefore, the last edge $e$ of $W$ does not lie in $\calR$.
Applying the above to $W = x_0[W_L(u)]q = W_L(q)$ and $W = x_0[W_R(v)]q = W_R(q)$, we obtain that $e_L^-$ and $e_R^-$ do not lie in $\calR$.
In particular, by~\ref{items:leaving_regions:x}, $e_R^+ \prec e_L^- \prec e_L^+$ and $e_R^+ \prec e_R^- \prec e_L^+$ in the $q$-ordering.
These two combined give two possibilities: $e_R^+ \prec e_R^- \preccurlyeq e_L^- \prec e_L^+$ or $e_R^+ \prec e_L^- \prec e_R^- \prec e_L^+ $ in the $q$-ordering.
The former coincides with~\ref{item:ordering-around-q}, thus, suppose to the contrary that the latter holds.
Reordering the edges, we obtain $e_R^- \prec e_L^+ \prec e_L^-$ in the $q$-ordering.
This, by~\ref{items:leaving_shadows:y}, yields that $e_L^+$ lies in the terminal block $\calB$ of $\shadz(q)$.
Since $q[W_L(u)]u$ and $\partial\calB$ intersect only in $q$, and since the first edge of $q[W_L(u)]u$ lies in $\calB$, we conclude that $q[W_L(u)]u$ lies in $\calB$. 
In particular, $u \in \Int \calB$,
Thus, by~\Cref{prop:shadow-comp}.\ref{prop:shadow-comp:item:hanging_element}, we have $\shadz(q) = \shadz(u)$.
Since $\shadz(q) = \shadz(u)$ and $u \in \Int \shadz(q)$, 
we obtain that $\sd(u) > 0$.
However, $u \in Z(a)$, so $\sd(u) > 0$ contradicts~\Cref{prop:properties_of_Z}.\ref{prop:properties_of_Z:item:sd_of_z}. 
Thus,~\ref{item:ordering-around-q} follows.

Within the remaining statements of this subsection, we characterize where an element $b \in B$ lies relative to $\calR$. 
In~\Cref{prop:shad-disjoint-from-calR}, we show that if $b \in \shadz(q) \setminus \{q\}$, then $b \notin \calR$.
\Cref{prop:path_does_not_leave_regions} is an analog of~\cref{cor:path-in-block} for regions defined as above: we show that a witnessing path with both endpoints in $\calR$ lies entirely in $\calR$.
In~\Cref{prop:b_relative_to_region}, we give equivalent conditions for being in $\calR$.
\Cref{cor:sandwitch_b_in_region} is a simpler although less precise statement inferred from~\Cref{prop:b_relative_to_region}.
See~\cref{fig:characterization-regions}.


\begin{proposition}\label{prop:shad-disjoint-from-calR}
    Let $(a,u,v,U,V,\calR,q,m,\gamma_L,\gamma_R)$ be a region tuple and let $b \in B$.
    If $b \in \calR \cap \shadz(q)$, then $b = q$.
\end{proposition}
\begin{proof}
    Suppose that $b \in \calR$ and $b \neq q$.
    We will prove that $b \notin \shadz(q)$.
    Let $W$ be a witnessing path from $x_0$ to $b$ in $P$.
    We have $b \in \calR$ and $x_0 \notin \Int \calR$.
    Thus, $W$ intersects $\partial \calR$.
    Let $w$ be the last element of $W$ in $\partial\calR$.
    That is, $w[W]b \subset \calR$.
    Since $w \in B$, we get that $w$ lies in $q[W_L(u)]u \cup q[W_R(v)]v$.
    If $w$ lies in $q[W_L(u)]u$, let $W' = q[W_L(u)]w[W]b$ and if $w$ lies in $w[W_R(v)]v$, let $W' = q[W_R(v)]w[W]b$. 
    Since $b\neq q$, the path $W'$ contains at least one edge. 
    Let $e$ be the first edge of $W'$.
    Recall that $e$ lies in $\calR$.

    Let $e^+_L$ and $e^+_R$ be, respectively, the edges following $q$ in $W_L(u)$ and $W_R(v)$.
    If $q = x_0$, we set $e_L^- = e_R^- = e_{-\infty}$.
    Otherwise, we set $e_L^-$ and $e_R^-$ to be, respectively, the edges preceding $q$ in $W_L(u)$ and $W_R(v)$. 
    By~\ref{item:ordering-around-q} and~\ref{items:leaving_regions:x},
    \[e_L^- \prec e_L^+ \preccurlyeq e \preccurlyeq e_R^+ \prec e_R^- \text{ in the $q$-ordering}.\]
    By~\ref{items:leaving_shadows:y}, $e$ does not lie in $\shadz(q)$.
    Since all elements of $W'$ are greater than $q$ in $p$, the element $q$ is the only common element of $W'$ and $\shadz(q)$.
    In particular, $b \notin \shadz(q)$, which ends the proof.
\end{proof}

\begin{proposition}\label{prop:path_does_not_leave_regions}
    Let $(a,u,v,U,V,\calR,q,m,\gamma_L,\gamma_R)$ be a region tuple.
    Let $W$ be a witnessing path in $P$ with all elements in $B$.
    If both endpoints of $W$ are in $\calR$, then $W \subset \calR$. 
\end{proposition}
\begin{proof}
    Suppose to the contrary that the endpoints of $W$ are in $\calR$ but $W$ is not contained in $\calR$.
    Let $r$ and $s$ be elements in the intersection of $W$ and $\partial \calR$ with $r \leq s$ in $P$ such that $r[W]s$ is in the exterior of $\calR$ except for the elements $r$ and $s$.
    Let $e$ be the first edge of $r[W]s$.
    In particular, $e$ does not lie in $\calR$. 
    Since all elements of $W$ are in $B$, we have that
    $r$ and $s$ lie in $u[W_L(u)]q[W_R(v)]v$.
    
    First, assume that $r$ and $s$ lie strictly on the left side of $\calR$.
    By~\ref{items:leaving_regions:left}, $x_0[W_L(u)]r[W]s[W_L(u)]u$ is left of $W_L(u)$, which is a contradiction.
    The argument when $r$ and $s$ lie strictly on the right side of $\calR$ is symmetric.
    Now, assume that $r$ lies strictly on the left side of $\calR$ and $s$ lies strictly on the right side of $\calR$.
    By~\ref{items:leaving_regions:left}, $x_0[W_L(u)]r[W]s[W_R(v)]v$ is left of $W_L(u)$, which is a contradiction with the fact that $u$ is left of $v$.
    The argument when $r$ lies strictly on the left side of $\calR$ and $s$ lies strictly on the right side of $\calR$ is symmetric.
    
    Finally, assume that $r = q$ and $s$ lies strictly on the left side of $\calR$.
    (Note that when $r = q$ and $s$ lies strictly on the right side of $\calR$, the proof is symmetric.) 
    Let $e_L^+$ and $e_R^+$ be the edges following $q$ in $W_L(u)$ and $W_R(v)$ respectively. 
    If $q = x_0$, let $e_L^- = e_R^- = e_{-\infty}$, otherwise, let  $e_L^-$ and $e_R^-$ be the edges preceding $q$ in $W_L(u)$ and $W_R(v)$ respectively. 
    We split the reasoning into cases depending on the position of $e$ in the $q$-ordering.

    Since $e$ is outside $\calR$, by~\ref{items:leaving_regions:x}, $e_R^+ \prec e \prec e_L^+$ in the $q$-ordering.
    If $e_L^- \prec e \prec e_L^+$ in the $q$-ordering, then $x_0[W_L(q)]q[W]s[W_L(u)]u$ is left of $W_L(u)$, which is a contradiction.
    If $e_R^+ \prec e \prec e_R^-$ in the $q$-ordering, then $x_0[W_R(q)]q[W]s[W_L(u)]u$ is right of $W_R(v)$, which is a contradiction with the fact that $u$ is left of $v$.
    Finally, assume that $e_R^- \prec e \prec e_L^-$ in the $q$-ordering.
    Then by~\ref{items:leaving_shadows:y}, $e$ lies in the terminal block $\calB$ of $\shadz(q)$.
    In particular, $q[W]s[W_L(u)]u$ lies in $\calB$, and so, $u \in \calB \subset \shadz(q)$, which contradicts~\Cref{prop:shad-disjoint-from-calR}.    
\end{proof}

\begin{figure}[tp]
  \begin{center}
    \includegraphics{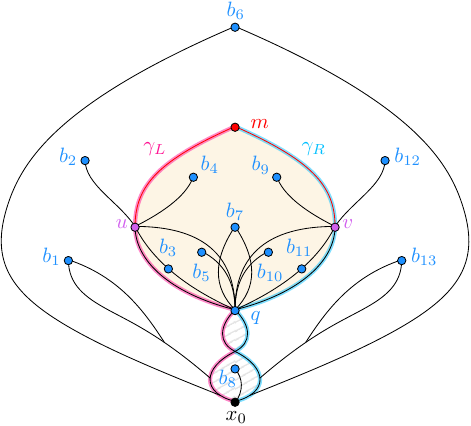}
  \end{center}
  \caption{
    Examples of elements in $B$ in different locations with respect to the region $\calR(a,u,v,U,V)$. 
  }
  \label{fig:characterization-regions}
\end{figure}

\begin{proposition}\label{prop:b_relative_to_region}
Let $(a,u,v,U,V,\calR,q,m,\gamma_L,\gamma_R)$ be a region tuple and let $b \in B$ with $b \notin \shadz(q) \setminus \set{q}$. 
Then
\begin{enumerate}
    \item $b \in \partial \calR$ if and only if $b$ lies in $q[W_L(u)]u \cup q[W_R(v)]v$. \label{prop:b_relative_to_region:item:boundary}
    \item $b \in \Int \calR$ if and only if $W_L(b)$ is right of $\gamma_L$ and $W_R(b)$ is left of $\gamma_R$. \label{prop:b_relative_to_region:item:inside}
    \item $b \notin \calR$ if and only if $W_L(b)$ is left of $\gamma_L$ or $W_R(b)$ is right of $\gamma_R$. \label{prop:b_relative_to_region:item:outside}
\end{enumerate}
\end{proposition}
\begin{proof}
Recall that the boundary of $\calR$ is the union of four paths: $q[W_L(u)]u$, $q[W_R(v)]v$, $u[U]m$, and $v[V]m$.
Since $b\in B$ and the two paths $U$, $V$ are exposed, item \ref{prop:b_relative_to_region:item:boundary} follows.

Note that exactly one of the following three conditions holds: the right side of \ref{prop:b_relative_to_region:item:inside}, the right side of \ref{prop:b_relative_to_region:item:outside}, $W_L(b)$ is contained in $\gamma_L$ or $W_R(b)$ is contained in $\gamma_R$. 
The last option is possible only if $b \in \partial \calB$ holds or $b\in \shadz(q) \setminus \set{q}$. 
However, we assumed $b\not\in \shadz(q) \setminus \set{q}$.
Therefore, in order to conclude the proof of the proposition, we will show the right-to-left direction of each of the equivalences: \ref{prop:b_relative_to_region:item:inside} and \ref{prop:b_relative_to_region:item:outside}. 


We proceed with the right-to-left implication in \ref{prop:b_relative_to_region:item:inside}.
Assume that $W_L(b)$ is right of $\gamma_L$ and $W_R(b)$ is left of $\gamma_R$.
Note that this does not hold for $b = q$, so $b \neq q$.
Since we have already proved~\ref{prop:b_relative_to_region:item:boundary}, we may additionally assume $b \notin \partial \calR$.
Our first claim is that $W_L(q)$ is a subpath of $W_L(b)$ and $W_R(q)$ is a subpath of $W_R(b)$.
Recall that $W_L(q)$ is a subpath of $\gamma_L$ and $W_R(q)$ is a subpath of $\gamma_R$. 
Therefore, we have
\begin{itemize}
    \item either $W_L(q)$ is a subpath of $W_L(b)$ or $W_L(q)$ is left of $W_L(b)$, and
    \item either $W_R(q)$ is a subpath of $W_R(b)$ or $W_R(q)$ is right of $W_R(b)$.
\end{itemize}
The situation, where in both bullets the first option occurs, is exactly the statement of our claim.
We will prove that this is the only possible situation.
If $W_L(q)$ is a subpath of $W_L(b)$, then in particular, $q \in \shadz(b)$, which by \cref{prop:paths_directions_in_shadows}.\ref{prop:paths_directions_in_shadows:right} yields that $W_R(q)$ is not right of $W_R(b)$, and so, $W_R(q)$ is a subpath of $W_R(b)$.
Symmetrically, if $W_R(q)$ is a subpath of $W_R(b)$, then $W_L(q)$ is a subpath of $W_L(b)$.
Finally, assume that in both bullets the second option occurs, that is, $W_L(q)$ is left of $W_L(b)$ and $W_R(q)$ is right of $W_R(b)$.
It follows that $(b,q)$ is an inside pair.
Therefore, by~\Cref{prop:inside_pair_implies_containment}, $b \in \shadz(q)$, which is a contradiction.
In summary, we proved that in both bullets the first option occurs.

Let $W$ be a witnessing path from $q$ to $b$ in $P$.
If all elements of $W$ are in $\calR$, then $b \in \calR$, and so, $b \in \Int \calR$ as desired. 
Now consider an element $w$ in $W$ in the boundary of $\calR$.
Since $b \notin \partial \calR$, we have $w \neq b$.
Let $e$ be the edge of $W$ following $w$.
We claim that $e$ lies in $\calR$, which will conclude the proof.
Since $w \in B$, $w$ lies in $W_L(u) \subset \gamma_L$ or in $W_R(v) \subset \gamma_R$.
If $w$ is an element of $\gamma_L$, then let $e_L^-,e_L^+$ be the edges immediately preceding and following $w$ in $\gamma_L$ (when $w = x_0$, we set $e_L^- = e_{-\infty}$).
If $w$ is an element of $\gamma_R$, then let $e_R^-,e_R^+$ be the edges immediately preceding and following $w$ in $\gamma_R$ (when $w = x_0$, we set $e_R^- = e_{-\infty}$).
First, suppose that $w = q$.
Recall that $e_R^+ \prec e_R^- \preccurlyeq e_L^- \prec e_L^+$ in the $q$-ordering (by~\ref{item:ordering-around-q}).
The path $W_L(u)$ is either left of $W_L(b)$ or is a subpath of $W_L(b)$.
The path $W_L(b)$ is either left of $W' = x_0[W_L(b)]q[W]b$ or is equal to $W'$.
It follows that $W_L(u)$ is either left of $W'$ or is a subpath of $W'$.
Symmetrically $W_R(v)$ is either right of $W'' = x_0[W_R(b)]q[W]b$ or is a subpath of $W''$.
In particular, $e_L^- \prec e_L^+ \preccurlyeq e$ and $e \preccurlyeq e_R^+ \prec e_R^-$ in the $q$-ordering.
Therefore, either $e_L^+ \preccurlyeq e \preccurlyeq e_R^+$ or $e_R^- \preccurlyeq e \preccurlyeq e_L^-$ in the $q$-ordering.
In the latter case, by \ref{items:leaving_shadows:y}, we obtain that the edge $e$ lies in $\shadz(q)$, 
and therefore, the whole path $q[W]b$ lies in $\shadz(q)$, which forces $b \in \shadz(q)$ and contradicts the assumption. 
It follows that $e_L^+ \preccurlyeq e \preccurlyeq e_R^+$ in the $q$-ordering, and so, by \ref{items:leaving_regions:x}, $e$ lies in $\calR$, as desired.
Next, assume that $w$ is strictly on the left side (the case, where $w$ is strictly on the right side is symmetric). 
The path $\gamma_L$ is left of $W_L(b)$, which is equal to or left of the path $x_0[W_L(u)]w[W]b$.
Since in all three paths the same edge, namely $e_L^-$, precedes $w$, we have $e_L^+ \preccurlyeq e \prec e_L^-$ in the $w$-ordering, and so, by \ref{items:leaving_regions:left}, $e$ lies in $\calR$ as desired.
This ends the proof of the right-to-left implication in~\ref{prop:b_relative_to_region:item:inside}.

Finally, we prove the right-to-left implication in \ref{prop:b_relative_to_region:item:outside}.
That is, if $W_L(b)$ is left of $\gamma_L$ or $W_R(b)$ is right of $\gamma_R$, then $b \notin \calR$.
We will prove that $W_L(b)$ left of $\gamma_L$ implies $b \notin \calR$.
The proof that $W_R(b)$ right of $\gamma_R$ implies $b \notin \calR$ is symmetric.

Assume that $W_L(b)$ is left of $\gamma_L$.
Let $w = \gce(W_L(b),\gamma_L)$ and let $e = ww'$ be the edge following $w$ in $W_L(b)$.
We claim that the interior of $e$ is disjoint from $\calR$.
Let $e_L^-,e_L^+$ be the edges immediately preceding and following $w$ in $\gamma_L$ (when $w = x_0$, we set $e_L^- = e_{-\infty}$).
In the case where $w < q$ in $P$, we have $w \in \shadz(q)$ and $w \neq q$.
Thus, by~\Cref{prop:shad-disjoint-from-calR}, $w \notin \calR$, hence, the interior of $e$ is disjoint from $\calR$ as desired.
Next, assume that $q \leq w$ in $P$.
Since $W_L(b)$ is left of $\gamma_L$, we have $e_L^- \prec e \prec e_L^+$ in the $q$-ordering.
Therefore, by \ref{items:leaving_regions:x} or \ref{items:leaving_regions:left}, the interior of $e$ is disjoint from $\calR$.
However, if $b \in \calR$, then by~\Cref{prop:path_does_not_leave_regions}, $w[W_L(b)]b$ lies in $\calR$, and so, $e$ lies in $\calR$.
This shows that $b \notin \calR$ as desired.
\end{proof}

For future reference, we state the following useful corollary.


\begin{corollary}
    \label{cor:sandwitch_b_in_region}
    Let $(a,u,v,U,V,\calR,q,m,\gamma_L,\gamma_R)$ be a region tuple and let $b \in B$.
\begin{enumerate}
    \item If $u$ is left of $b$ and $b$ is left of $v$, then, $b \in \Int \calR$. \label{cor:sandwitch_b_in_region:int}
    \item If $b$ is left of $u$ or $v$ is left of $b$, then $b \notin \calR$.\label{cor:sandwitch_b_in_region:out}
\end{enumerate}
\end{corollary}
\begin{proof}
    Since $q < u$ in $P$, we have $\shadz(q) \subset \shadz(u)$ (by \cref{prop:comparability_implies_shadow_containment}).
    Assume that $u$ is left of $b$.
    In particular, $b \notin \shadz(u)$, and so, $b \notin \shadz(q)$. 
    Since $W_L(u)$ is left of $W_L(b)$, we have that $\gamma_L$ is left of $W_L(b)$. 
    Since $b$ is left of $v$, we have that $W_R(b)$ is left of $\gamma_R$.
    Therefore, by \cref{prop:b_relative_to_region}.\ref{prop:b_relative_to_region:item:inside}, $b \in \Int \calR$, which completes the proof of~\ref{cor:sandwitch_b_in_region:int}.

    Similarly as above, if $b$ is left of $u$, $b \notin \shadz(u)$, and $b \notin \shadz(q)$; and if $v$ is left of $b$, then $b \notin \shadz(v)$, and $b \notin \shadz(q)$.
    If $b$ is left of $u$, then $W_L(b)$ is left of $\gamma_L$, and $b \notin \calR$ by~\cref{prop:b_relative_to_region}.\ref{prop:b_relative_to_region:item:outside}.
    Symmetrically, if $b$ is right of $v$, then $W_R(b)$ is right of $\gamma_R$, and $b \notin \calR$ by~\cref{prop:b_relative_to_region}.\ref{prop:b_relative_to_region:item:outside}.
    This gives~\ref{cor:sandwitch_b_in_region:out}.
\end{proof}

We conclude this section with a technical statement that will be used repetitively in the proof of~\Cref{lem:coloring} (Coloring Lemma).

\begin{proposition}
\label{prop:a-in-some-region}
    Let $(a,b),(a',b') \in I$ and $d \in B$ with $b$ left of $d$ and $d$ left of $b'$.
    \begin{enumerate}
        \myitem{(L)} \label{prop:a-in-some-region:L} Assume that $a < d$, $a \parallel b'$, and $a' \parallel d$ in $P$.
        Let $z' \in Z(a')$ with $z' \leq b$ in $P$.
        Let $U$ and $V$ be exposed witnessing paths from $a'$ to $z'$ and from $a'$ to $z_R(a')$ in $P$, respectively.
        Then, $z'$ is left of $z_R(a')$ and
            \[a \in \Int\calR(a',z',z_R(a'),U,V).\]
        \myitem{(R)} \label{prop:a-in-some-region:R} Assume that $a' < d$, $a' \parallel b$, and $a \parallel d$ in $P$.
        Let $z \in Z(a)$ with $z \leq b'$ in $P$.
        Let $U$ and $V$ be exposed witnessing paths from $a$ to $z_L(a)$ and from $a$ to $z$ in $P$, respectively.
        Then, $z_L(a)$ is left of $z$ and
            \[a' \in \Int\calR(a,z_L(a),z,U,V).\]
    \end{enumerate}
\end{proposition}

    \begin{figure}[tp]
      \begin{center}
        \includegraphics{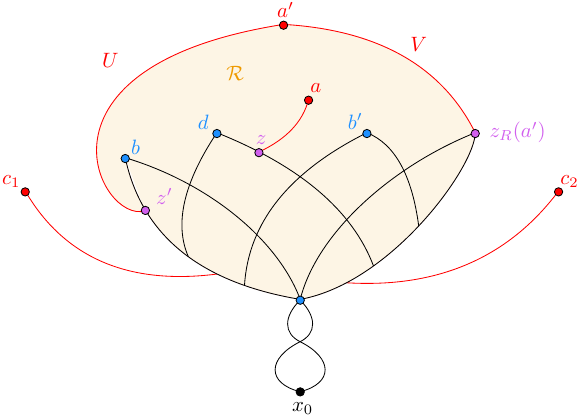}
      \end{center}
      \caption{
      An illustration of the statement of~\Cref{prop:a-in-some-region}.\ref{prop:a-in-some-region:L}.
      For simplicity, we do not always mark such elements.
      Note that the assumptions $a \parallel b$, $a \parallel b'$, and $a' \parallel d$ in $P$ are necessary.
      The necessity of the first two is depicted with elements $c_1$ and $c_2$ respectively.
      For the last, note that if $a' \leq d$ in $P$, and element from outside the region could send a comparability to $d$ through $V$.
      }
      \label{fig:prop-a-in-some-region}
    \end{figure}

\begin{proof}
    We prove only~\ref{prop:a-in-some-region:L} as the proof of~\ref{prop:a-in-some-region:R} is symmetric.
    See~\Cref{fig:prop-a-in-some-region}.
    Recall that $b$ is left of $d$, $d$ is left of $b'$, and the pairs $(a,b)$ and $(a',b')$ are dangerous (by~\ref{item:instance:dangerous}).
    Therefore, by~\Cref{prop:d_not_in_shadow_when_dangerous} (and transitivity -- \Cref{prop:left_porders_bs}), 
    \[a \notin \shadz(b')\text{, } a \notin \shadz(d)\text{, } a' \notin \shadz(b)\text{, and }a' \notin \shadz(d).\]
    In particular, $d \in Y(a)$, and so, we can fix $z \in Z(a)$ such that $z \leq d$ in $P$.
    Altogether, we have
    \[z',b \in Y(a') \text{ and } z,d \in Y(a).\]
    Since $\shadz(z') \subset \shadz(b)$ (by~\Cref{prop:comparability_implies_shadow_containment}), and $a \notin \shadz(b)$ (by~\ref{item:instance:not_in_shadow}), we also obtain
    \[a \notin \shadz(z').\]
    
    \begin{claim}\label{claim:z'-d-b'-left}
    $z'$ is left of $d$, $z'$ is left of $z$, and $z$ is left of $b'$.    
    \end{claim}
    \begin{proofclaim}
        First, we prove that $z'$ is left of $d$.
        By assumption, $a' \parallel d$ and $a' < z' \leq b$ in $P$.
        We have already proved that $a' \notin \shadz(d)$ and $z',b \in Y(a')$.
        Finally, $b$ is left of $d$, hence, by~\Cref{prop:left_is_preserved_from_y_to_z}.\ref{prop:left_is_preserved_from_y_to_z:left}, $z'$ is left of $d$.

        Next, we prove that $z'$ is left of $z$.
        Since $a \parallel b$ in $P$, we have $a \parallel z'$ in $P$.
        Moreover, $a < z \leq d$ in $P$.
        We have already proved that $a \notin \shadz(z')$, $z,d \in Y(a)$, and that $z'$ is left of $d$.
        Therefore, by~\Cref{prop:left_is_preserved_from_y_to_z}.\ref{prop:left_is_preserved_from_y_to_z:right}, $z'$ is left of $z$.

        Finally, we prove that $z$ is left of $b'$.
        By assumption, $a \parallel b'$ and $a < z \leq d$ in $P$.
        We have already proved that $a \notin \shadz(b')$ and $z,d \in Y(a)$.
        Finally, $d$ is left of $b'$, hence, by~\Cref{prop:left_is_preserved_from_y_to_z}.\ref{prop:left_is_preserved_from_y_to_z:left}, $z$ is left of $b'$.
        This completes the proof of the claim.
    \end{proofclaim}
    Since $b'$ is left of $z_R(a')$ (by~\Cref{prop:z_L_b_z_R}), \Cref{claim:z'-d-b'-left} yields that $z'$ is left of $z$ and $z$ is left of $z_R(a')$. 
    In particular, $z'$ is left of $z_R(a')$, and so, we can define $\calR = \calR(a',z',z_R(a'),U,V)$.
    Moreover, by~\Cref{cor:sandwitch_b_in_region}.\ref{cor:sandwitch_b_in_region:int}, $z \in \Int \calR$.
    Let $W$ be an exposed witnessing path from $a$ to $z$.
    Suppose to the contrary that $a \notin \Int \calR$.
    Then, $W$ intersects $\partial \calR$ in an element not in $B$, say $u$.
    It follows that $a' \leq u < z \leq d$ in $P$, which is a contradiction.
    This ends the proof.
\end{proof}

\subsection{Regular sequences}\label{ssec:regular}
The goal of the next two subsections is to classify all alternating cycles of size $2$ with pairs in $I$. 
Eventually, we will split them into four types, 
see~\cref{cor:characterize_regular_sacs}.

We say that a sequence $((a_1,b_1),\dots,(a_k,b_k))$ is \emph{regular} if $(a_i,b_i) \in I$ for every $i \in [k]$ and $b_i$ is left of $b_{i+1}$ for every $i \in [k-1]$. 
Since the \q{left of} relation is transitive, see~\cref{prop:left_porders_bs}, we note that in a regular sequence $((a_1,b_1),\dots,(a_k,b_k))$, the elements $b_1,\dots,b_k$ are linearly ordered by the \q{left of} relation.
Observe also that a subsequence of a regular sequence is regular.

We list four key properties that a regular sequence $((a_1,b_1),(a_2,b_2))$ may or may not satisfy. 
The first letter of the name of a property indicates if it concerns leftmost paths ($L$) or rightmost paths ($R$).
The ordering of the numbers $1$ and $2$ indicates the relative positions of $a_1$ and $a_2$ in the drawing, reading from left to right.
See \Cref{fig:types_of_2_sequences}. 
Properties labeled without a star are simple, while those with a star are more involved (and in fact imply respective simpler statements, see \Cref{obs:PL2_PR2_basic}).


\begin{enumerateNumLRinout}
    \myitem{$(L12)$} 
    $M_L(a_1)$ is left of $M_L(a_2)$. \label{RPP1} \label{Lin} \label{L12}
             \smallskip\smallskip
    \myitem{$(L21^\star)$} 
    \label{RPP2} \label{Lout} \label{L21}
    There exist $u \in Z(a_2)$ and an exposed witnessing path $U$ from $a_2$ to $u$ in $P$ such that
        \begin{itemize}
            \item $u$ is an element of both $W_L(z_L(a_1))$ and $W_L(b_1)$,
            \item $x_0[W_L(z_L(a_1))]u[U]a_2$ is left of $W_L(z_L(a_1))$.
        \end{itemize}
\end{enumerateNumLRinout}
\begin{enumerateNumLRinout}
    \myitem{$(R12)$} 
    $M_R(a_2)$ is right of $M_R(a_1)$. \label{RPP3} \label{Rin} \label{R12}
             \smallskip\smallskip
    \myitem{$(R21^\star)$} \label{RPP4} \label{Rout} \label{R21}
    There exist $v \in Z(a_1)$ and an exposed witnessing path $V$ from $a_1$ to $v$ in $P$ such that 
        \begin{itemize}
            \item $v$ is an element of both $W_R(z_R(a_2))$ and $W_R(b_2)$,
            \item $x_0[W_R(z_R(a_2))]v[V]a_1$ is right of $W_R(z_R(a_2))$.
        \end{itemize}
\end{enumerateNumLRinout}

\begin{figure}[tp]
  \begin{center}
    \includegraphics{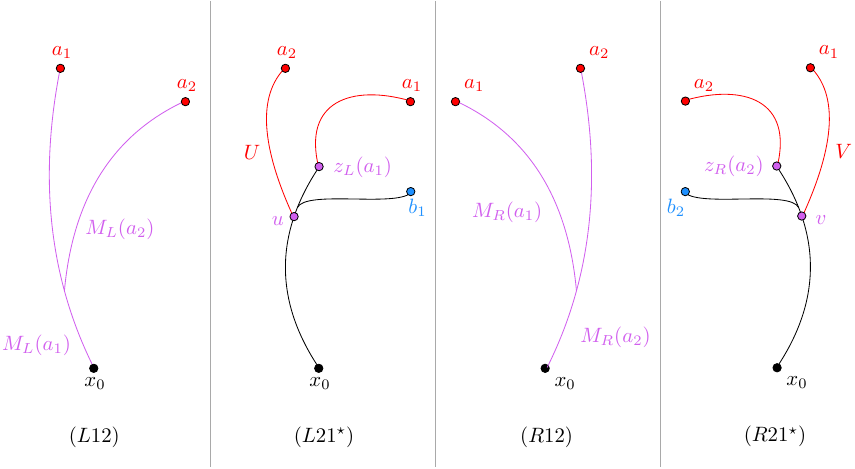}
  \end{center}
  \caption{
  In each part of the figure, a regular sequence $((a_1,b_1),(a_2,b_2))$ satisfies one of the four properties.
  }
  \label{fig:types_of_2_sequences}
\end{figure}


If a regular sequence $ \sigma = ((a_1,b_1),(a_2,b_2))$ satisfies \ref{RPP2}, then for all $u \in Z(a_2)$ and $U$ as in \ref{RPP2}, we say that $(u,U)$ \emph{witnesses} \ref{RPP2} for $\sigma$.
Sometimes, we omit $U$ and just say that $u$ witnesses \ref{RPP2} for $\sigma$.
Symmetrically, if $\sigma$ satisfies \ref{RPP4}, then for all $v \in Z(a_1)$ and $V$ as in \ref{RPP4}, we say that $(v,V)$ \emph{witnesses} \ref{RPP4} for $\sigma$.

Note that for a regular sequence $((a_1,b_1),(a_2,b_2))$, the satisfaction of each of the properties does not depend on both $b_1$ and $b_2$, i.e.\ for each of the properties, one or both of these elements can be replaced by other elements in $B$ so that the sequence is still regular, and then the property is still satisfied.
More precisely, we have the following.


\begin{obs}\label{obs:switch_bs_in_RPP}
  Let $\sigma=((a_1,b_1),(a_2,b_2))$ be regular, and let $b_1',b_2' \in B$.
  \begin{enumerate}
    \item If $\sigma$ satisfies \ref{Lin} and $\sigma'=((a_1,b'_1),(a_2,b'_2))$ is regular, then $\sigma'$ satisfies \ref{Lin}. \label{obs:switch_bs_in_RPP:2in1}
    \item If $\sigma$ satisfies \ref{Lout} and $\sigma'=((a_1,b_1),(a_2,b'_2))$ is regular, then $\sigma'$ satisfies \ref{Lout}. \label{obs:switch_bs_in_RPP:2out1}
    \item If $\sigma$ satisfies \ref{Rin} and $\sigma'=((a_1,b'_1),(a_2,b'_2))$ is regular, then $\sigma'$ satisfies \ref{Rin}.  \label{obs:switch_bs_in_RPP:1in2}
    \item If $\sigma$ satisfies \ref{Rout} and $\sigma'=((a_1,b'_1),(a_2,b_2))$ is regular, then $\sigma'$ satisfies \ref{Lout}.  \label{obs:switch_bs_in_RPP:1out2}
  \end{enumerate}
\end{obs}

In some sense, \ref{RPP2} complements \ref{RPP1} as if \ref{RPP2} holds then \ref{RPP1} cannot hold. The same applies to \ref{RPP4} and \ref{RPP3}. 
We highlight this in the next observation.

\begin{obs}\label{obs:PL2_PR2_basic}
    Let $\sigma=((a_1,b_1),(a_2,b_2))$ be regular. 
    \begin{enumerate}
        \myitem{$(L)$} If $\sigma$ satisfies \ref{RPP2}, then $a_2 < b_1$ in $P$ and $M_L(a_2)$ is left of $M_L(a_1)$. \label{obs:PL2_PR2_basic:left}
        \myitem{$(R)$} If $\sigma$ satisfies \ref{RPP4}, then $a_1 < b_2$ in $P$ and $M_R(a_1)$ is right of $M_R(a_2)$. \label{obs:PL2_PR2_basic:right}
    \end{enumerate}
\end{obs}


The following proposition shows that \ref{Lout} (and \ref{Rout}) implies a stronger property.
\begin{proposition} \label{prop:PL2_PR2_more}
    Let $\sigma=((a_1,b_1),(a_2,b_2))$ be regular. 
    \begin{enumerate}
        \myitem{$(L)$} 
        If $u\in Z(a_2)$ witnesses \ref{RPP2} for $\sigma$,  
        then for every $z\in Z(a_1)$ such that $z$ is left of $b_1$, the element $u$ lies in $W_L(z)$. \label{prop:PL2_PR2_more:left}
        \myitem{$(R)$} 
        If $v\in Z(a_1)$ witnesses \ref{RPP4} for $\sigma$,
        then for every $z\in Z(a_2)$ such that $z$ is right of~$b_2$, the element $v$ lies in $W_R(z)$. \label{prop:PL2_PR2_more:right}
    \end{enumerate}
\end{proposition}
\begin{proof}
    We prove only \ref{prop:PL2_PR2_more:left}.
    The proof of \ref{prop:PL2_PR2_more:right} is symmetric.
    Assume that $\sigma$ satisfies \ref{RPP2} which is witnessed by $u \in Z(a_2)$, and let $z \in Z(a_1)$ such that $z$ is left of $b_1$.
    Recall that $u$ lies in both $W_L(z_L(a_1))$ and $W_L(b_1)$.
    We consider cases of how the paths $W_L(z_L(a_1))$ and $W_L(z)$ relate to each other.

    If $W_L(z_L(a_1))$ is a subpath of $W_L(z)$, then the assertion holds trivially as $u$ lies in $W_L(z_L(a_1))$.
    If $W_L(z)$ is a subpath of $W_L(z_L(a_1))$, then it suffices to argue that $u < z$ in $P$.
    Indeed, otherwise $a_1 < z \leq u \leq b_1$ in $P$, which is a contradiction.
    It follows that $u$ lies in $W_L(z)$.
    
    Next assume that $W_L(z_L(a_1))$ is left of $W_L(z)$.
    Since $z$ is left of $b_1$, in particular, $W_L(z)$ is left of $W_L(b_1)$.
    Again, recall that $u$ lies in both $W_L(z_L(a_1))$ and $W_L(b_1)$.
    All this allows us to apply \cref{prop:sandwiched_paths} (the considered paths are pairwise $x_0$-consistent by~\Cref{prop:W-consistent}.\ref{prop:W-consistent:left}), to obtain that $u$ indeed lies in $W_L(z)$.

    Finally, assume that $W_L(z)$ is left of $W_L(z_L(a_1))$.
    Consider the path $M = x_0[W_L(z)]z[W]a_1$, where $W$ is an exposed witnessing path from $a_1$ to $z$ in $P$.
    Clearly, $M \in \calM(a_1)$.
    Since $x_0[M_L(a_1)]z_L(a_1) = W_L(z_L(a_1))$ by \cref{prop:paths_M_consistent}, the path $M$ is left of $M_L(a_1)$, which contradicts the definition of $M_L(a_1)$, and shows that this case does not occur.
\end{proof}

Next, we show that each of the four properties is transitive.

\begin{proposition}\label{prop:transitivity_RPP}
  Let $((a_1,b_1),(a_2,b_2),(a_3,b_3))$ be regular. 
  Let $\sigma_{12} = ((a_1,b_1),(a_{2},b_{2}))$, $\sigma_{23} = ((a_2,b_2),(a_{3},b_{3}))$, and $\sigma_{13} = ((a_1,b_1),(a_{3},b_{3}))$.
  \begin{enumerate}
  \item \label{prop:item:transitivity-Lin}
  If $\sigma_{12}$ and $\sigma_{23}$ satisfy \ref{Lin}, then
  $\sigma_{13}$ satisfies \ref{Lin}.
  \item \label{prop:item:transitivity-Lout}
  If $\sigma_{12}$ and $\sigma_{23}$ satisfy \ref{Lout}, then
  $\sigma_{13}$ satisfies \ref{Lout}.
  \item \label{prop:item:transitivity-Rin}
  If $\sigma_{12}$ and $\sigma_{23}$ satisfy \ref{Rin}, then
  $\sigma_{13}$ satisfies \ref{Rin}.
  \item \label{prop:item:transitivity-Rout}
  If $\sigma_{12}$ and $\sigma_{23}$ satisfy \ref{Rout}, then
  $\sigma_{13}$ satisfies \ref{Rout}.
  \end{enumerate}
\end{proposition}
\begin{proof}
The arguments for items \ref{prop:item:transitivity-Lin} and \ref{prop:item:transitivity-Rin} are symmetric so we only include a proof of the former. 
The same applies to items \ref{prop:item:transitivity-Lout} and \ref{prop:item:transitivity-Rout}.

We start with an argument for \ref{prop:item:transitivity-Lin}. 
Suppose that $\sigma_{12}$ and $\sigma_{23}$ satisfy \ref{Lin}. 
Therefore, $M_L(a_1)$ is left of $M_L(a_2)$ and $M_L(a_2)$ is left of $M_L(a_3)$. 
By~\cref{prop:u_e_is_poset}, this relation is transitive so $M_L(a_1)$ is left of $M_L(a_3)$, as desired.

Now we proceed with the proof~\ref{prop:item:transitivity-Lout}. 
Suppose that $\sigma_{12}$ and $\sigma_{23}$ satisfy \ref{Lout}. 
Let $u \in Z(a_2)$ and $U$ be an exposed witnessing path from $a_2$ to $u$ in $P$ such that $(u,U)$ witnesses \ref{Lout} for $\sigma_{12}$ and 
let $u' \in Z(a_3)$ and $U'$ be an exposed witnessing path from $a_3$ to $u'$ in $P$ such that $(u',U')$ witnesses \ref{Lout} for $\sigma_{23}$.
We claim that $(u',U')$ witnesses~\ref{Lout} for $\sigma_{13}$.

Since $(a_2,b_2)\in I$, $a_2 < u \leq b_1$ in $P$, and $b_1$ left of $b_2$, 
we can apply~\cref{cor:left_is_preserved_from_y_to_z}.\ref{cor:left_is_preserved_from_y_to_z:left} and conclude that $u$ is left of $b_2$. 
Now, since $\sigma_{23}$ is regular, 
$u'$ witnesses~\ref{Lout} for $\sigma_{23}$, 
$u\in Z(a_2)$, and $u$ left of $b_2$, we can apply~\cref{prop:PL2_PR2_more}.\ref{prop:PL2_PR2_more:left} to $\sigma_{23}$ and conclude that $u'$ lies in $W_L(u)$.
Since $u$ witnesses \ref{Lout} for $\sigma_{12}$, $u$ lies in both $W_L(z_L(a_1))$ and $W_L(b_1)$, and so, $u'$ lies in both $W_L(z_L(a_1))$ and $W_L(b_1)$.

Finally, to conclude that $(u',U')$ witnesses \ref{Lout} for $\sigma_{13}$ and end the proof, it suffices to show that $M' = x_0[W_L(z_L(a_1))]u'[U']a_3$ is left of $W_L(z_L(a_1))$.
First, note that $M' = x_0[W_L(z_L(a_2))]u'[U']a_3$, thus, by \ref{Lout} for $\sigma_{23}$, $M'$ is left of $W_L(z_L(a_2))$.
By~\cref{prop:paths_M_consistent}, $W_L(z_L(a_2))$ is a subpath of $M_L(a_2)$, hence, $W_L(z_L(a_2))$ is either a subpath of $M = x_0[W_L(z_L(a_1))]u[U]a_2$ or is left of $M$.
By \ref{Lout} for $\sigma_{12}$, $M$ is left of $W_L(z_L(a_1))$.
It follows that $M'$ is left of $W_L(z_L(a_1))$, which completes the proof of~\ref{prop:item:transitivity-Lout}.
\end{proof}

\subsection{Classifying alternating cycles of size 2}\label{ssec:classification}
Let $((a_1,b_1),(a_2,b_2))$ be an alternating cycle in $P$ with both pairs in $I$.
Note that by~\ref{item:instance:sacs} either $((a_1,b_1),(a_2,b_2))$ or $((a_2,b_2),(a_1,b_1))$ is regular. 
The goal of this subsection is to classify regular alternating cycles of size $2$ in $P$ into one of the four types.
To this end, we associate some regions to each regular alternating cycle of size $2$.
First, we need a simple property that is a corollary of~\Cref{prop:dangerous-implies-in-Y}.

\begin{corollary}\label{cor:regular-Y}
    Let $\sigma = ((a_1,b_1),(a_2,b_2))$ be a regular alternating cycle in $P$.
    Then, $b_2 \in Y(a_1)$ and $b_1 \in Y(a_2)$.
\end{corollary}

Let $\sigma=((a_1,b_1),(a_2,b_2))$ be a regular alternating cycle in $P$.
By~\Cref{cor:regular-Y}, $b_2 \in Y(a_1)$ and $b_1 \in Y(a_2)$.
Thus, we can fix $z_1 \in Z(a_1)$ and $z_2 \in Z(a_2)$ such that $z_1 \leq b_2$ and $z_2 \leq b_1$ in~$P$.
By~\Cref{prop:z_L_b_z_R}, $z_L(a_1)$ is left of $b_1$ and by~\Cref{cor:left_is_preserved_from_y_to_z}.\ref{cor:left_is_preserved_from_y_to_z:right}, $z_1$ is right of $b_1$.
Therefore, by transitivity, $z_L(a_1)$ is left of $z_1$.
Symmetrically, we obtain that $z_2$ is left of $z_R(a_2)$.
Let $W_1$ be an exposed witnessing path from $a_1$ to $z_1$ in $P$, and let $W_2$ be an exposed witnessing path from $a_2$ to $z_2$ in $P$.
Let
    \begin{align*}
        \calR_1 &= \calR(a_1,z_L(a_1),z_1,a_1[M_L(a_1)]z_L(a_1),W_1),\\
        \calR_2 &= \calR(a_2,z_2,z_R(a_2),W_2,a_2[M_R(a_2)]z_R(a_2)).
    \end{align*}
    
We say that $\calR_1$ is a \emph{left region} of $\sigma$ and $\calR_2$ is a \emph{right region} of $\sigma$.
Observe that by~\cref{cor:sandwitch_b_in_region}.\ref{cor:sandwitch_b_in_region:int}, 
\begin{equation}\label{eq:b_in_region}
    \text{$b_1 \in \Int \calR_1$ and $b_2 \in \Int\calR_2$.}
\end{equation}
In the next proposition, we connect relative positions of $a_1,a_2$ and $\calR_1,\calR_2$ with the properties \ref{Lin}, \ref{Lout}, \ref{Rin}, and \ref{Rout}.
See \cref{fig:2sacs}.

\begin{figure}[tp]
  \begin{center}
    \includegraphics{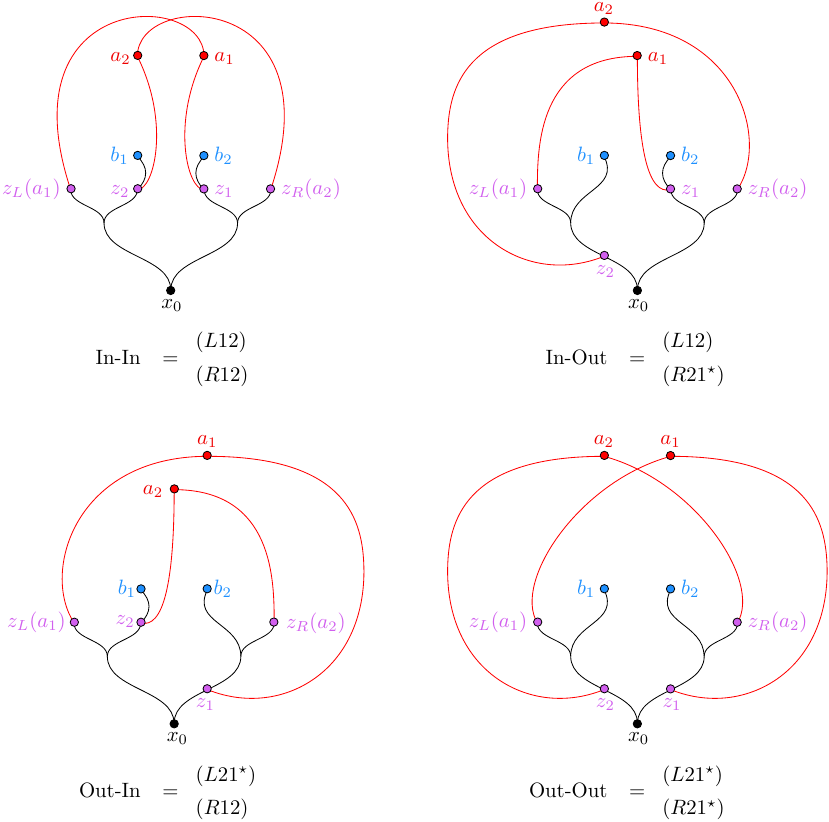}
  \end{center}
  \caption{
    In each part of the figure, we depict a left and a right region for a regular alternating cycle $((a_1,b_1),(a_2,b_2))$.
    We exhaust all possible combinations of the satisfaction of the statements $a_2 \in \calR_1$ and $a_1 \in \calR_2$.
    In this way, we obtain a classification of all alternating cycles of size $2$ in $P$ with both pairs in $I$, see the definitions after~\cref{cor:characterize_regular_sacs}.
  }
  \label{fig:2sacs}
\end{figure}

\begin{proposition}\label{prop:sac_four_types} 
  Let $\sigma=((a_1,b_1),(a_2,b_2))$ be a regular alternating cycle.
  Let $\calR_1$ be a left region of $\sigma$ and let $\calR_2$ be a right region of $\sigma$.
    \begin{enumerate}
    \item If $a_2 \in \Int \cgR_1$, then $\sigma$ 
      satisfies \ref{Lin}. \label{prop:sac_four_types:item:Lin}
    \item If $a_2 \notin \cgR_1$, then $\sigma$ 
      satisfies \ref{Lout}.\label{prop:sac_four_types:item:Lout}
    \item If $a_1 \in \Int \cgR_2$, then $\sigma$ satisfies \ref{Rin}.
\label{prop:sac_four_types:item:Rin}
    \item If $a_1 \notin \cgR_2$, then $\sigma$ 
      satisfies \ref{Rout}.\label{prop:sac_four_types:item:Rout}
  \end{enumerate}
\end{proposition}

\begin{proof}
    We only prove statements \ref{prop:sac_four_types:item:Lin} and \ref{prop:sac_four_types:item:Lout}.
    The proofs of the other statements are symmetric.
    Let $z_1 \in Z(a_1)$ such that $z_1 \leq b_2$ in $P$ and let $W_1$ be an exposed witnessing path from $a_1$ to $z_1$ in $P$ such that $\calR_1 = \calR(a_1,z_L(a_1),z_1,a_1[M_L(a_1)]z_L(a_1),W_1)$.
    Let $q$ be the lower-min of $\calR_1$, and let $m$ be the upper-min of $\calR_1$.
    Let $\gamma_L = x_0[W_L(z_L(a_1))]z_L(a_1)[M_L(a_1)]m$ and $\gamma_R = x_0[W_R(z_1)]z_1[W_1]m$.

    We start with an argument for~\ref{prop:sac_four_types:item:Lin}.
    Assume that $a_2 \in \Int \calR_1$.
    We shall prove that $\sigma$ satisfies \ref{Lin}, that is, $M_L(a_1)$ is left of $M_L(a_2)$.
    Since $a_2 \in \Int \calR_1$ and $x_0 \notin \Int \calR_1$, $M_L(a_2)$ contains an element in $\partial\calR_1$.
    Let $u$ be the first such element in $a_2[M_L(a_2)]x_0$.
    First, suppose that $u$ lies in $z_L(a_2)[M_L(a_2)]x_0$ and $u \neq z_L(a_2)$.
    In this case, $z_L(a_2) \in \Int\calR_1$, thus, by~\cref{prop:b_relative_to_region}.\ref{prop:b_relative_to_region:item:inside}, $\gamma_L$ is left of $W_L(z_L(a_2))$, which implies $M_L(a_1)$ left of $M_L(a_2)$ as desired.
    Next, suppose that $u$ lies in $a_2[M_L(a_2)]z_L(a_2)$.
    Note that each element $v$ of $\gamma_R$
    satisfies $v\leq z_1 \leq b_2$ in $P$.
    Therefore, $u$ does not lie in $\gamma_R$ as
    this would imply $a_2 \leq u \leq b_2$ in $P$, which is a contradiction.
    It follows that $u$ lies strictly on the left side of $\calR_1$.
    Since the last edge of $a_2[M_L(a_2)]u$ has its interior contained in $\Int \calR_1$, by \ref{items:leaving_regions:left} and \cref{prop:ML_consistent}.\ref{prop:ML_consistent:left}, $M_L(a_1)$ is left of $M_L(a_2)$.
    This completes the proof of~\ref{prop:sac_four_types:item:Lin}.

    Now, we prove~\ref{prop:sac_four_types:item:Lout}.
    Assume that $a_2 \notin \calR_1$.
    We shall prove that $\sigma$ satisfies \ref{Lout}.
    Since $a_2 < b_1$ in $P$, there is a witnessing path $W$ from $a_2$ to $b_1$ in $P$. 
    Note that $a_2 \notin \calR_1$ and $b_1 \in \Int\calR_1$ (by~\eqref{eq:b_in_region}), so $W$ must intersect $\partial \calR_1$. 
    Note that $W$ is disjoint from $\gamma_R$ as otherwise a common element $z$ of the two paths would certify $a_2\leq z\leq z_1\leq b_2$ in $P$, which is a contradiction. 
    Also, $W$ is disjoint from $m[M_L(a_1)]z_L(a_1)$ as otherwise a common element $z$ of the two paths would certify $a_1\leq m\leq z\leq b_1$ in $P$, which is a contradiction.
    Therefore, $W$ intersects $\partial \calR_1$ only in $W'$, where $W'$ is $q[W_L(z_L(a_1))]z_L(a_1)$ with the elements $q$ and $z_L(a_1)$ removed.

    Let $u$ be the first element of $a_2[W]b_1$ in $B$, 
    and let $v$ be the first element of $a_2[W]b_1$ in $\partial \calR_1$. 
    Thus, as we argued $v$ lies in $W'$. 
    Clearly, $u\leq v$ in $P$.
    We claim that $u = v$.
    Suppose to the contrary that $u \neq v$.
    In particular, since $a_2 \notin \calR_1$, we have $u \notin \calR_1$.
    Then, by \cref{prop:b_relative_to_region}.\ref{prop:b_relative_to_region:item:outside}, $W_L(u)$ is left $\gamma_L$ or $W_R(u)$ is right of $\gamma_R$.
    First, assume that $W_L(u)$ is left of $\gamma_L$.
    Since $v$ lies in $W'$, we have $v < z_L(a_1)$ in $P$, hence, $z_L(a_1)$ does not lie in $W_L(u)$.
    It follows that $W_L(u)$ is left of $W_L(z_L(a_1))$.
    However, the path $x_0[W_L(u)]u[W]v[W_L(z_L(a_1))]z_L(a_1)$ is now left of $W_L(z_L(a_1))$, which is a contradiction.
    Next, assume that $W_R(u)$ is right of $\gamma_R$.
    If $z_1$ lies in $W_R(u)$, then $z_1 \leq u < v < z_L(a_1)$ in $P$, which contradicts that $z_L(a_1)$ is left of $z_1$, thus, $z_1$ does not lie in $W_R(u)$.
    It follows that $W_R(u)$ is right of $W_R(z_1)$.
    Moreover, $x_0[W_R(u)]u[W]v[W_L(z_L(a_1))]z_L(a_1)$ is right of $W_R(z_1)$, which is right of $W_R(z_L(a_1))$.
    This yields existence of a witnessing path from $x_0$ to $z_L(a_1)$ right of $W_R(z_L(a_1))$ in $P$, which is a contradiction.
    We conclude that indeed $u = v$.
    
    Let $V = a_2[W]v$.
    We claim that $(v,V)$ witnesses \ref{Lout} for $\sigma$.
    As proved, $V$ is an exposed witnessing path from $a_2$ to $v$ in $P$, and so, $v \in Z(a_2)$. 
    Since $v$ is an element of $W'$, $v$ lies in $W_L(z_L(a_1))$.
    Let $w = \gce(W_L(z_L(a_1)),W_L(b_1))$.
    If $w < v$ in $P$, then since $W_L(z_L(a_1))$ is left of $W_L(b_1)$, we have $x_0[W_L(z_L(a_1))]v[W]b_1$ left of $W_L(b_1)$, which is false.
    Thus, $v \leq w$ in $P$, which implies that $v$ lies in $W_L(b_1)$.
    Finally, since $v$ is strictly on the left side of $\calR_1$, by~\ref{items:leaving_regions:left}, $x_0[W_L(z_L(a_1))]v[V]a_2$ is left of $W_L(z_L(a_1))$.
    This completes the proof that $\sigma$ satisfies \ref{Lout}.
\end{proof}

Let $\sigma=((a_1,b_1),(a_2,b_2))$ be a regular alternating cycle in $P$. Fix a left region $\calR_1$ of $\sigma$ and a right region $\calR_2$ of $\sigma$.
Since $a_1 \parallel a_2$ in $P$, we have $a_1 \notin \partial \calR_2$ and $a_2 \notin \partial \calR_1$.
Therefore, \cref{prop:sac_four_types} implies that $\sigma$ must satisfy either \ref{Lin} or \ref{Lout} and either \ref{Rin} or \ref{Rout}.
By \cref{obs:PL2_PR2_basic}, $\sigma$ cannot satisfy both \ref{Lin} and \ref{Lout} and symmetrically $\sigma$ cannot satisfy both \ref{Rin} and \ref{Rout}.
All this leads to the following conclusion.

\leqnomode
\begin{corollary} \label{cor:characterize_regular_sacs}
    Let $\sigma=((a_1,b_1),(a_2,b_2))$ be a regular alternating cycle in $P$ with both pairs in~$I$.
        \begin{align}
        \begin{split}
        \textrm{$\sigma$ satisfies \ref{Lin} } &\Longleftrightarrow \textrm{$a_2\in \Int \calR_1$ for every left region $\calR_1$ of $\sigma$}\\
        &\Longleftrightarrow \textrm{$a_2 \in \Int \calR_1$ for some left region $\calR_1$ of $\sigma$.}
        \end{split}
        \tag{i}\label{cor:characterize_regular_sacs:item:Lin}
        \\[1ex]
        \begin{split}
        \textrm{$\sigma$ satisfies \ref{Lout} } &\Longleftrightarrow \textrm{$a_2\not\in \calR_1$ for every left region $\calR_1$ of $\sigma$}\\
        &\Longleftrightarrow \textrm{$a_2 \not\in \calR_1$ for some left region $\calR_1$ of $\sigma$.}
        \end{split}
        \tag{ii}\label{cor:characterize_regular_sacs:item:Lout}
        \\[1ex]
        \begin{split}
        \textrm{$\sigma$ satisfies \ref{Rin} } &\Longleftrightarrow \textrm{$a_1\in \Int \calR_2$ for every right region $\calR_2$ of $\sigma$}\\
        &\Longleftrightarrow \textrm{$a_1 \in \Int \calR_2$ for some right region $\calR_2$ of $\sigma$.}
        \end{split}
        \tag{iii}\label{cor:characterize_regular_sacs:item:Rin}
        \\[1ex]
        \begin{split}
        \textrm{$\sigma$ satisfies \ref{Rout} } &\Longleftrightarrow \textrm{$a_1\not\in \calR_2$ for every right region $\calR_2$ of $\sigma$}\\
        &\Longleftrightarrow \textrm{$a_1 \not\in \calR_2$ for some right region $\calR_2$ of $\sigma$.}
        \end{split}
        \tag{iv}\label{cor:characterize_regular_sacs:item:Rout}
        \end{align}
\end{corollary}
\reqnomode
Therefore, all regular alternating cycles of size $2$ can be classified into four types.
We say that a regular alternating cycle $\sigma$ of size $2$ in $P$ with both pairs in $I$ is:
\begin{align*}
    \text{\emph{In-In}} &\text{ if $\sigma$ satisfies \ref{Rin} and \ref{Lin};}\\
    \text{\emph{In-Out}} &\text{ if $\sigma$ satisfies \ref{Rin} and \ref{Lout};}\\
    \text{\emph{Out-In}} &\text{ if $\sigma$ satisfies \ref{Rout} and \ref{Lin};}\\
    \text{\emph{Out-Out}} &\text{ if $\sigma$ satisfies \ref{Rout} and \ref{Lout}.}
\end{align*}
See again \Cref{fig:2sacs}.

\section{Auxiliary oriented graphs: definitions and the coloring lemma}
\label{sec:aux-definitions}

An \emph{orientation} of a graph $G$ is a function assigning to each edge $\set{u,v}$ of $G$ one of the pairs: $(u,v)$ or $(v,u)$. 
An \emph{oriented graph} is a graph with a fixed orientation. 
The \emph{vertex set} of an oriented graph $H$ is the vertex set of the underlying graph, while the \emph{edge set} of $H$ is the set of all $(u,v)$ such that $\{u,v\}$ is an edge in the underlying graph mapped to $(u,v)$ by the orientation.
We say that an edge $(u,v)$ of an oriented graph $H$ is an edge \emph{from $u$ to $v$} in $H$.

A \emph{directed path} is an oriented graph such that $\{v_0,\dots,v_k\}$ is its vertex set and $\{(v_{i-1},v_i) : i \in [k]\}$ is its edge set, where $v_0,\dots,v_k$ are pairwise distinct.
This directed path \emph{starts} in $v_0$ and \emph{ends} in $v_k$.
The number of vertices of a directed path is its \emph{order}.
A \emph{directed cycle} is an oriented graph with at least three vertices such that removing each of its edges gives a directed path.
An oriented graph is \emph{acyclic} if it contains no directed cycle.

When $H$ is an acyclic oriented graph, we define $\maxpath(H)$ as the maximum order of a directed path in~$H$.
Moreover, for a vertex $v$ of $H$, let $\maxsp(H,v)$ be the maximum order of a directed path in~$H$ starting in $v$.

In this section, we fix a maximal good instance $(P,x_0,G,e_{-\infty},I)$.
We study six auxiliary oriented graphs -- $\HOO$, $\HIIL$, $\HIIR$, $\HIILR$, $\HIO$, and $\HOI$ -- each with vertex set $I$ and edges given by carefully chosen regular alternating cycles of size $2$. 
All six oriented graphs are acyclic. 
Four of them -- 
$\HOO$, $\HIIL$, $\HIIR$, and $\HIILR$ -- 
have the maximum order of a directed path at most $\se_P(I)$. 
The graphs $\HIO$ and $\HOI$ do not admit such a strong bound in general, and instead, we present a much more subtle statement. 
We state each of these results (\cref{prop:HOO,prop:HIIL,prop:HIIR,prop:HIILR,prop:HOO,lemma:HIO}) in this section, however, we postpone the proofs to~\Cref{sec:aux-proofs}. 
These results give ground for the coloring $\kappa$ of $I$ in which $\kappa((a,b))$ for $(a,b)\in I$ is defined according to the order of the longest directed path starting or ending in $(a,b)$ in all six oriented graphs.
The number of colors used by $\kappa$ on $I$ will be in $\mathcal{O}(\se_P(I)^8)$. 
Finally, the key statement of this section, the Coloring Lemma, see~\cref{lem:coloring}, 
shows that pairs in $I$ with the same color under $\kappa$ form a reversible set in $P$. 
This completes the proof of~\cref{thm:maximal_instance_imply_dim_boundedness} up to the postponed proofs.

As mentioned above, $I$ is the vertex set of each of the six auxiliary oriented graphs defined in this section.
Moreover, from the definitions, 
it follows that each edge $((a_1,b_1),(a_2,b_2))$ in these graphs is a regular sequence. 
In particular, $b_1$ is left of $b_2$. 
Therefore, by the transitivity
of the \q{left of} relation (see~\Cref{prop:left_porders_bs}), 
each of the six oriented graphs is acyclic.

Let $\HOO$ be the oriented graph with the vertex set $I$ and $\sigma = ((a_1,b_1),(a_2,b_2))$ is an edge in~$\HOO$ if $\sigma$ is a regular Out-Out alternating cycle in $P$.

\begin{proposition} \label{prop:HOO} $\maxpath(\HOO) \leq \se_P(I)$.
\end{proposition}

Let $\HII$ be the oriented graph with the vertex set $I$ and $\sigma = ((a_1,b_1),(a_2,b_2))$ is an edge in~$\HII$ if $\sigma$ is a regular In-In alternating cycle in $P$.
The argument in the Coloring Lemma requires us to deal with certain supergraphs of $\HII$, namely $\HIIL$ and $\HIIR$.

Let $\HIIL$ be the oriented graph with the vertex set $I$ and $\sigma = ((a_1,b_1),(a_2,b_2))$ is an edge in~$\HIIL$ if
$\sigma$ is an edge in $\HII$ or there exists a \emph{witness} $t \in B$ for $\sigma$ such that
\begin{enumerateNumIIL}
    \item $(a_2,t)\in I$ and $((a_1,b_1),(a_2,t))$ is an edge in $\HII$, \label{items:HIIL-shifted-edgeHII}
    \item $t$ is left of $b_2$, and \label{items:HIIL-shifted-t-left-b}
    \item $a_1\parallel b_2$ in $P$. \label{items:HIIL-shifted-a-parallel-b}
\end{enumerateNumIIL}

\begin{proposition} \label{prop:HIIL} $\maxpath(\HIIL) \leq \se_P(I)$.
\end{proposition}

Let $\HIIR$ be the oriented graph with the vertex set $I$ and $\sigma = ((a_1,b_1),(a_2,b_2))$ is an edge in~$\HIIR$ if 
$\sigma$ is an edge in $\HII$ or there exists \emph{witness} $s \in B$ for $\sigma$ such that
\begin{enumerateNumIIR}
    \item $(a_1,s) \in I$ and $((a_1,s),(a_2,b_2))$ is an edge in $\HII$, \label{items:HIIR-shifted-edgeHII}
    \item $b_1$ is left of $s$, and \label{items:HIIR-shifted-t-left-b}
    \item $a_2\parallel b_1$ in $P$. \label{items:HIIR-shifted-a-parallel-b}
\end{enumerateNumIIR}

\begin{proposition} \label{prop:HIIR} $\maxpath(\HIIR) \leq \se_P(I)$.
\end{proposition}

Let $\sigma$ be an an edge in $\HIIL$ (in $\HIIR$, respectively). 
If $\sigma$ is an edge in $\HII$, then we say that $\sigma$ is a \emph{cycle} edge in $\HIIL$ (in $\HIIR$); otherwise, we say that $\sigma$ is a \emph{shifted} edge in $\HIIL$ (in $\HIIR$).
See~\Cref{fig:edges-HIIL-HIIR} for a schematic drawing of edges in $\HIIL$ and $\HIIR$.
See also~\Cref{fig:edges-HII-real} later in the paper.

\begin{figure}[tp]
  \begin{center}
    \includegraphics{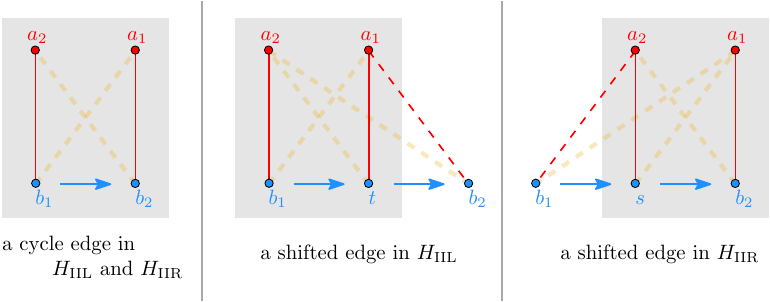}
  \end{center}
  \caption{
    Schematic drawings of edges in $\HIIL$ and $\HIIR$.
    Dashed lines indicate incomparabilities in $P$.
    With yellow lines, we denote incomparable pairs that are surely in $I$.
    Gray regions indicate In-In alternating cycles in~$P$.
    Blue arrows indicate the \q{left of} relation.
    We use the same drawing conventions in several forthcoming figures.
  }
  \label{fig:edges-HIIL-HIIR}
\end{figure}

Let $\HIILR$ be the oriented graph with the vertex set $I$ and $\sigma = ((a_1,b_1),(a_2,b_2))$ is an edge in~$\HIILR$ if there exists a \emph{witness} $(s,t)$ with $s,t \in B$ for $\sigma$ such that
\begin{enumerateNumIILR}
    \item $(a_1,s),(a_2,t)\in I$ and $((a_1,s),(a_2,t))$ is an edge in $\HII$, \label{items:HIILR-edgeHII}
    \item $b_1$ is left of $s$ and $t$ is left of $b_2$, and \label{items:HIILR-b-left-d-t-left-b}
    \item $a_1\parallel b_2$ and $a_2 \parallel b_1$ in $P$. \label{items:HIILR-a-parallel-b}
\end{enumerateNumIILR}

See~\Cref{fig:edges-HIILR} for a schematic drawing of edges in $\HIILR$.

\begin{proposition} $\maxpath(\HIILR) \leq \se_P(I)$.
    \label{prop:HIILR}
\end{proposition}

\begin{figure}[tp]
  \begin{center}
    \includegraphics{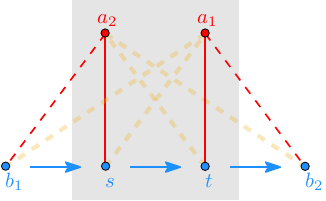}
  \end{center}
  \caption{
    Schematic drawing of an edge in $\HIILR$.
  }
  \label{fig:edges-HIILR}
\end{figure}

Let $\HIO$ be the oriented graph with the vertex set $I$ and $\sigma = ((a_1,b_1),(a_2,b_2))$ is an edge in~$\HIO$ if $\sigma$ is regular and satisfies \ref{Rin} and \ref{Lout}.

Let $\sigma = ((a_1,b_1),(a_2,b_2))$ be an edge in $\HIO$.
Note that by~\cref{obs:PL2_PR2_basic}.\ref{obs:PL2_PR2_basic:left}, $a_2 < b_1$ in $P$.
We say that $\sigma$ is a \emph{cycle} edge in $\HIO$ if $a_1 < b_2$ in $P$, in other words, if $\sigma$ is a regular In-Out alternating cycle in $P$.
We say that $\sigma$ is a \emph{shifted} edge in $\HIO$ if there exists a \emph{witness} $t \in B$ for $\sigma$ such that 
\begin{enumerateNumIO}
    \item $(a_2,t)\in I$ and $((a_1,b_1),(a_2,t))$ is a regular In-Out alternating cycle in $P$, \label{items:HIO-cycle}
    \item $t$ is left of $b_2$, and \label{items:HIO-t-left-b}
    \item $a_1\parallel b_2$ in $P$. \label{items:HIO-a-parallel-b}
\end{enumerateNumIO}
See~\cref{fig:edges-HIO} for a schematic drawing of edges in $\HIO$.
See also~\Cref{fig:edges-HIO-real} later in the paper.

\begin{figure}[tp]
  \begin{center}
    \includegraphics{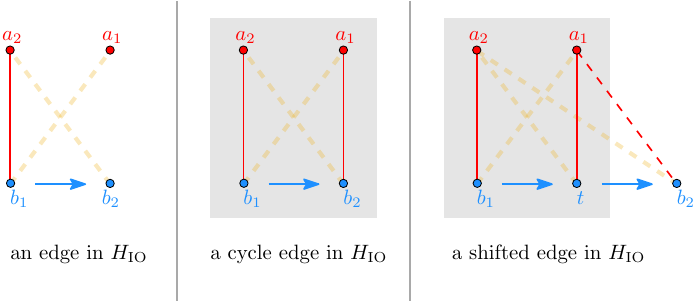}
  \end{center}
  \caption{
  Schematic drawings of edges in $\HIO$.
  Gray regions indicate In-Out alternating cycles in $P$.
  }
  \label{fig:edges-HIO}
\end{figure}

Let $\HOI$ be the directed graph with the vertex set $I$ and $\sigma = ((a_1,b_1),(a_2,b_2))$ is an edge in~$\HOI$ if $\sigma$ satisfies \ref{Lin} and \ref{Rout}.

Let $\sigma = ((a_1,b_1),(a_2,b_2))$ be an edge in $\HOI$.
Note that by~\cref{obs:PL2_PR2_basic}.\ref{obs:PL2_PR2_basic:right}, $a_1 < b_2$ in $P$.
We say that $\sigma$ is a \emph{cycle} edge in $\HOI$ if $a_2 < b_1$ in $P$, in other words, if $\sigma$ is a regular Out-In alternating cycle in $P$.
We say that $\sigma$ is a \emph{shifted} edge in $\HOI$ if there exists a \emph{witness} $t \in B$ for $\sigma$ such that
\begin{enumerateNumOI}
    \item $(a_1,t)\in I$ and $((a_1,t),(a_2,b_2))$ is a regular Out-In alternating cycle in $P$, \label{items:HOI-cycle}
    \item $b_1$ is left of $t$, and \label{items:HOI-b-left-t}
    \item $a_2\parallel b_1$ in $P$. \label{items:HOI-a-parallel-b}
\end{enumerateNumOI}

Let $H \in \{\HIO,\HOI\}$.
We assign a weight of $0$ or $1$ to each edge of $H$.
All cycle edges and shifted edges in $H$ are of weight $1$ and all the remaining edges in $H$ are of weight $0$.
The \emph{weight} of a path in $H$ is the sum of the weights of its edges.
Furthermore, for a vertex $v$ of $H$, $\maxsw(H,v)$ is the maximum weight of a directed path in $H$ starting in $v$ and $\maxew(H,v)$ is the maximum weight of a directed path in $H$ ending in $v$.

Contrary to the four previously defined oriented graphs, the maximum order of a directed path in $\HIO$ and $\HOI$ cannot be bounded in terms of $\se_P(I)$.
Also, the maximum weight of a directed path cannot be bounded in terms of $\se_P(I)$ (see~\Cref{fig:path-HIO} later in the paper).
Within~\cref{lemma:HIO}, we prove a weaker property that is good enough for our purposes in~\cref{lem:coloring}.

\vbox{
\begin{restatable}{lemma}{hmm}
\label{lemma:HIO}
Let $m = 2\se_P(I) \cdot (2\se_P(I)+6)$.
\begin{enumerate}
    \item \label{lemma:HIO:HIO} If $((a,b),(a',b'))$ is an edge of weight $1$ in $\HIO$, then 
    \[\maxsw(\HIO,(a,b)) \not\equiv \maxsw(\HIO,(a',b')) \bmod m.\]
    \item \label{lemma:HIO:HOI} If $((a,b),(a',b'))$ is an edge of weight $1$ in $\HOI$, then 
    \[\maxew(\HOI,(a,b)) \not\equiv \maxew(\HOI,(a',b')) \bmod m.\]
\end{enumerate}
\end{restatable}
}
We are ready to define the final coloring of $I$.
Let
\begin{align*}
s=\se_P(I) \ \text{ and } \ m=2s \cdot (2s+6).
\end{align*}
We define $\kappa$ to be a function such that for every $(a,b)\in I$,
\begin{align*}
\kappa((a,b))_1 &= \maxsp(\HOO,(a,b)),\\
\kappa((a,b))_2 &= \maxsp(\HIIL,(a,b)),\\
\kappa((a,b))_3 &= \maxsp(\HIIR,(a,b)),\\
\kappa((a,b))_4 &= \maxsp(\HIILR,(a,b)),\\
\kappa((a,b))_5 &= \maxsw(\HIO,(a,b))\bmod m,\\
\kappa((a,b))_6 &= \maxew(\HOI,(a,b))\bmod m,
\end{align*}
and
\[\kappa((a,b)) = (\kappa((a,b))_i)_{i \in [6]}.\]
Let $\Image(\kappa)$ be the image of $\kappa$. 
Altogether we have 
\begin{align*}
|\Image(\kappa)| &\leq \maxpath(\HOO)  \cdot \maxpath(\HIIL)\\
&\qquad\cdot \maxpath(\HIIR) \cdot \maxpath(\HIILR) \cdot m \cdot m\\
&\leq s^4\cdot (2s\cdot (2s+6))^2\\
&\leq 16s^6\cdot(s+3)^2 =\Oh\left(s^8\right),
\end{align*} 
where the second inequality follows from Propositions~\ref{prop:HOO} to \ref{prop:HIILR}.
The Coloring Lemma, proved below assuming~\cref{lemma:HIO}, states that incomparable pairs from $I$ with the same color under $\kappa$ form a reversible set in $P$.

Altogether, we obtain
\[
\dim_P(I) \leq |\Image(\kappa)| \leq 16\se_P(I)^6  \cdot (\se_P(I)+3)^2,
\]
which will complete the proof of~\Cref{thm:maximal_instance_imply_dim_boundedness}.
\begin{lemma}[Coloring lemma]\label{lem:coloring}
    For every $\xi\in\Image(\kappa)$, 
    the set $\set{(a,b)\in I\mid \kappa((a,b))=\xi}$ is reversible.
\end{lemma}
\begin{proof}[Proof assuming~\Cref{lemma:HIO}]
    Let $\xi\in\Image(\kappa)$ and $\xi=(\xi_i)_{i\in[6]}$. 
    Let $I_\xi=\set{(a,b)\in I\mid \kappa((a,b))=\xi}$. 
    We show that $I_\xi$ is reversible. 
    Suppose otherwise and 
    let $((a_1,b_1),\dots,(a_k,b_k))$ be a strict alternating cycle in~$P$ with all the pairs in $I_\xi$.
    The indices in $[k]$ are considered cyclically, e.g.\ $a_{k+1} = a_1$.

    \begin{claim}\label{claim:no-edge-in-aux-digraph-special}
        For all $\alpha,\beta \in [k]$ and $H \in \set{\HIO, \HOI}$, there is no edge of weight $1$ from $(a_\alpha,b_\alpha)$ to $(a_\beta,b_\beta)$ in $H$.
    \end{claim}
    \begin{proof}
        If there is an edge edge of weight $1$ from $(a_\alpha,b_\alpha)$ to $(a_\beta,b_\beta)$ in some $H \in \{\HIO,\HOI\}$, then by~\Cref{lemma:HIO}, either $\kappa((a_\alpha,b_\alpha))_5 \neq \kappa(a_\beta,b_\beta))_5$ or $\kappa((a_\alpha,b_\alpha))_6 \neq \kappa(a_\beta,b_\beta))_6$, in particular, $\kappa((a_\alpha,b_\alpha)) \neq \kappa((a_\beta,b_\beta))$, which is a contradiction.
    \end{proof}

    \begin{claim}\label{claim:no-edge-in-aux-digraph-normal}
        For all $\alpha,\beta \in [k]$ and $H \in \set{\HOO, \HIIL, \HIIR, \HIILR}$, there is no edge from $(a_\alpha,b_\alpha)$ to $(a_\beta,b_\beta)$ in $H$.
    \end{claim}
    \begin{proofclaim}
        Suppose otherwise, and let $\alpha,\beta \in [k]$ and $H \in \set{\HOO, \HIIL, \HIIR, \HIILR}$ be such that $((a_\alpha,b_\alpha),(a_\beta,b_\beta))$ is an edge in $H$. 
        Clearly,
        \[\maxsp(H, (a_\alpha,b_\alpha)) > \maxsp(H, (a_\beta,b_\beta)).\]
        Let $i\in[4]$ be such that $\kappa((a,b))_i = \maxsp(H,(a,b))$ for every $(a,b) \in I$.
        We obtain,
        \[\xi_i=\kappa((a_\alpha,b_\alpha))_i > \kappa((a_\beta,b_\beta))_i=\xi_i.\]
        This contradiction completes the proof of the claim.
    \end{proofclaim}

    By~\ref{item:instance:sacs}, the set $\set{b_1,\ldots,b_k}$ is linearly ordered by the \q{left of} relation.

    \begin{claim}\label{claim:k-neq-2}
        $k \neq 2$.
    \end{claim}
    \begin{proofclaim}
        Suppose that $k=2$. 
        Without loss of generality assume that $b_1$ is left of $b_2$.
        Then, $((a_1,b_1),(a_2,b_2))$ is a regular alternating cycle in $P$.
        Moreover, by~\Cref{cor:characterize_regular_sacs}, $((a_1,b_1),(a_2,b_2))$ is an alternating cycle of one of four types: 
        In-In, In-Out, Out-In, or Out-Out.
        It follows that $((a_1,b_1),(a_2,b_2))$ is an edge in $\HIIL$ (and $\HIIR$) or is an edge in $\HOO$ or is an edge of weight~$1$ in $\HIO$ or is an edge of weight~$1$ in $\HOI$.
        This contradicts~\Cref{claim:no-edge-in-aux-digraph-normal} or~\Cref{claim:no-edge-in-aux-digraph-special}.
    \end{proofclaim}

    Recall that pairs $(a_\alpha,b_\alpha)$ are in $I$, for all  $\alpha\in[k]$. 
    We prove that some pairs $(a_\alpha,b_\beta)$ for $\alpha,\beta\in[k]$ with $\alpha\neq \beta$ are also in $I$. 
    The next claim follows directly from~\Cref{prop:dangerous-implies-in-Y}.

    \begin{claim}\label{claim:not-in-shadow-in-Y}
        For all distinct $\alpha,\beta \in [k]$, we have
        $a_\alpha \notin \shadz(b_\beta)$ and $b_{\alpha+1} \in Y(a_\alpha)$.
    \end{claim}

    \begin{claim}\label{claim:pair-in-I}
        Let $\alpha,\beta \in [k]$.
        If $b_{\alpha+1}$ is left of $b_\beta$ and $b_\beta$ is left of $b_\alpha$, then $(a_\alpha,b_\beta) \in I$.
        Symmetrically, if $b_\alpha$ is left of $b_\beta$ and $b_\beta$ is left of $b_{\alpha+1}$, then $(a_\alpha,b_\beta) \in I$.
    \end{claim}
    \begin{proofclaim}
        Assume that $b_{\alpha+1}$ is left of $b_\beta$ and $b_\beta$ is left of $b_\alpha$.
        The proof in the case where $b_\alpha$ is left of $b_\beta$ and $b_\beta$ is left of $b_{\alpha+1}$ is symmetric.
        Note that $a_\alpha \in \pi_1(I)$ and $b_\beta \in \pi_2(I)$.
        Thus, by~\ref{item:instance:maximal}, it suffices to show that $a_\alpha \notin \shadz(b_\beta)$, and that $(a_\alpha,b_\beta)$ is dangerous.
        The former follows from~\Cref{claim:not-in-shadow-in-Y}.
        For the latter, note that $b_{\alpha+1}$ and $z_R(a_\alpha)$ witness that $(a_\alpha,b_\beta)$ is a dangerous pair.
        Indeed, $a_\alpha < b_{\alpha+1}$ and $a_\alpha < z_R(a_\alpha)$ in $P$; $b_{\alpha+1}$ is left of $b_\beta$ (by assumption); and $b_\beta$ is left of $z_R(a_\alpha)$ (by~\Cref{prop:z_L_b_z_R} and transitivity).
        Altogether, we obtain that $(a_\alpha,b_\beta) \in I$.
    \end{proofclaim}

    \begin{claim}\label{claim:no-1-2-k}
        There is no $\alpha \in [k]$ such that $b_{\alpha+1}$ is left of $b_{\alpha+2}$ and $b_{\alpha+2}$ is left of $b_\alpha$.
    \end{claim}

    \begin{figure}[tp]
      \begin{center}
        \includegraphics{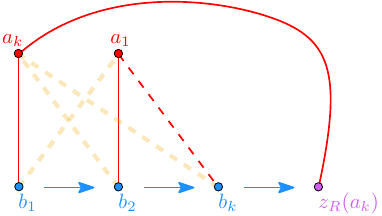}
      \end{center}
      \caption{
      Schematic drawing of the situation in~\Cref{claim:no-1-2-k}.
      In the proof, we show that $(a_k,b_2) \in I$.
      Next, we show that $((a_{1},b_{1}),(a_{k},b_{k}))$ is a shifted edge in $\HIIL$ or a shifted edge in $\HIO$, which is witnessed by $b_2$.
      }
      \label{fig:claim:no-1-2-k}
    \end{figure}
    
    \begin{proofclaim}
        Suppose to the contrary that there exists $\alpha \in [k]$ witnessing that the claim is false.
        Without loss of generality, assume that $\alpha = k$.
        Namely, $b_{1}$ is left of $b_{2}$ and $b_{2}$ is left of $b_k$. 
        See~\Cref{fig:claim:no-1-2-k}.
        We aim to show that $\sigma = ((a_{1},b_{1}),(a_{k},b_{k}))$ is a shifted edge in $\HIIL$ witnessed by~$b_2$ or a shifted edge in $\HIO$ witnessed by~$b_2$.

        Since $b_1$ is left of $b_2$ and $b_2$ is left of $b_k$, by~\Cref{claim:pair-in-I}, we obtain $(a_k,b_2) \in I$.
        Note that $\sigma' = ((a_1,b_1),(a_k,b_2))$ is a regular alternating cycle in $P$.
        Since $((a_1,b_1),\ldots,(a_k,b_k))$ is strict, we have $a_1 \parallel b_k$ in $P$ (as $k\neq2$ by~\Cref{claim:k-neq-2}).
        Next, since $b_2$ is left of $b_k$ and $a_1\parallel b_k$ in $P$, items~\ref{items:HIIL-shifted-t-left-b},~\ref{items:HIIL-shifted-a-parallel-b},~\ref{items:HIO-t-left-b}, and~\ref{items:HIO-a-parallel-b} are satisfied for $\sigma$ witnessed by $b_2$.
        We show that either~\ref{items:HIIL-shifted-edgeHII} or~\ref{items:HIO-cycle} is also satisfied for $\sigma$ witnessed by $b_2$.
        Namely, we show that $\sigma'$ is either an In-In alternating cycle or an In-Out alternating cycle.

       We prove that $\sigma'$ satisfies~\ref{Rin}.
       By~\Cref{claim:not-in-shadow-in-Y}, $b_1 \in Y(a_k)$, and so, we can fix $z_k \in Z(a_k)$ such that $z_k \leq b_{1}$ in $P$.
        We have $(a_{1},b_{1}),(a_k,b_k) \in I$, $b_{2} \in B$, $b_1$ left of $b_2$, $b_2$ left of $b_k$, and $a_1 < b_2$, $a_1 \parallel b_k$, $a_k \parallel b_2$ in $P$.
        Altogether, we can apply~\Cref{prop:a-in-some-region}.\ref{prop:a-in-some-region:L} to obtain that $z_k$ is left of $z_R(a_k)$ and $a_1 \in \Int \calR$ where $\calR = \calR(a_k,z_k,z_R(a_k),U,a_k[M_R(a_k)]z_R(a_k))$ and $U$ is an exposed witnessing path from $a_k$ to $z_k$ in $P$.
        Note that $\calR$ is a right region of $\sigma'$.
        Therefore, by~\Cref{prop:sac_four_types}.\ref{prop:sac_four_types:item:Rin}, $\sigma'$ satisfies \ref{Rin}.

        Finally, by~\Cref{cor:characterize_regular_sacs}, $\sigma'$ satisfies either~\ref{Lin} or~\ref{Lout}.
        In the former case, $\sigma'$ is an In-In alternating cycle and in the latter case, $\sigma'$ is an In-Out alternating cycle.
        As discussed before, this shows that $\sigma=((a_1,b_1),(a_k,b_k))$ is a shifted edge in $\HIIL$ or a shifted edge in $\HIO$.
        The former contradicts~\Cref{claim:no-edge-in-aux-digraph-normal} and the latter contradicts~\Cref{claim:no-edge-in-aux-digraph-special}.
    \end{proofclaim}

    The next claim is a symmetric statement to~\cref{claim:no-1-2-k}.
    \begin{claim}\label{claim:no-1-3-2}
        There is no $\alpha \in [k]$ such that $b_{\alpha}$ is left of $b_{\alpha+2}$ and $b_{\alpha+2}$ is left of $b_{\alpha+1}$.
    \end{claim}

    \begin{figure}[tp]
      \begin{center}
        \includegraphics{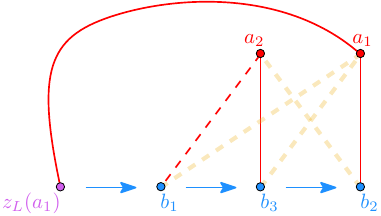}
      \end{center}
      \caption{
      Schematic drawings of the situation in~\Cref{claim:no-1-3-2}.
      In the proof, we show that $(a_1,b_3) \in I$.
      Next, we show that $((a_{1},b_{1}),(a_{2},b_{2}))$ is a shifted edge in $\HIIR$ or a shifted edge in $\HOI$, which is witnessed by $b_3$.
      }
      \label{fig:claim:no-1-3-2}
    \end{figure}

    \begin{proofclaim}
        Suppose to the contrary that there exists $\alpha \in [k]$ witnessing that the claim is false.
        Without loss of generality, assume that $\alpha = 1$.
        Namely, $b_{1}$ is left of $b_{3}$ and $b_{3}$ is left of $b_2$. 
        See~\Cref{fig:claim:no-1-3-2}.
        We aim to show that $\sigma = ((a_{1},b_{1}),(a_{2},b_{2}))$ is a shifted edge in $\HIIR$ witnessed by~$b_3$ or a shifted edge in $\HOI$ witnessed by~$b_3$.

        Since $b_1$ is left of $b_3$ and $b_3$ is left of $b_2$, by~\Cref{claim:pair-in-I}, we obtain $(a_1,b_3) \in I$.
        Note that $\sigma' = ((a_1,b_3),(a_2,b_2))$ is a regular alternating cycle in $P$.
        Since $((a_1,b_1),\ldots,(a_k,b_k))$ is strict, we have $a_2 \parallel b_1$ in $P$ (as $k\neq2$ by~\Cref{claim:k-neq-2}).
        Next, since $b_1$ is left of $b_3$ and $a_2\parallel b_1$ in $P$, items~\ref{items:HIIR-shifted-t-left-b},~\ref{items:HIIR-shifted-a-parallel-b},~\ref{items:HOI-b-left-t}, and~\ref{items:HOI-a-parallel-b} are satisfied for $\sigma$ witnessed by $b_3$.
        We show that either~\ref{items:HIIR-shifted-edgeHII} or~\ref{items:HOI-cycle} is also satisfied for $\sigma$ witnessed by $b_3$.
        Namely, we show that $\sigma'$ is either an In-In alternating cycle or an Out-In alternating cycle.

       We prove that $\sigma'$ satisfies~\ref{Lin}.
       By~\Cref{claim:not-in-shadow-in-Y}, $b_2 \in Y(a_1)$, and so, we can fix $z_1 \in Z(a_1)$ such that $z_1 \leq b_{2}$ in $P$.
        We have $(a_{1},b_{1}),(a_2,b_2) \in I$, $b_{3} \in B$, $b_1$ left of $b_3$, $b_3$ left of $b_2$, and $a_2 < b_3$, $a_2 \parallel b_1$, $a_1 \parallel b_3$ in $P$.
        Altogether, we can apply~\Cref{prop:a-in-some-region}.\ref{prop:a-in-some-region:R} to obtain that $z_L(a_1)$ is left of $z_1)$ and $a_2 \in \Int \calR$ where $\calR = \calR(a_1,z_L(a_1),z_1,a_1[M_L(a_1)]z_L(a_1),V)$ and $V$ is an exposed witnessing path from $a_1$ to $z_1$ in $P$.
        Note that $\calR$ is a left region of $\sigma'$.
        Therefore, by~\Cref{prop:sac_four_types}.\ref{prop:sac_four_types:item:Lin}, $\sigma'$ satisfies \ref{Lin}.

        Finally, by~\Cref{cor:characterize_regular_sacs}, $\sigma'$ satisfies either~\ref{Rin} or~\ref{Rout}.
        In the former case, $\sigma'$ is an In-In alternating cycle and in the latter case, $\sigma'$ is an Out-In alternating cycle.
        As discussed before, this shows that $\sigma=((a_1,b_1),(a_2,b_2))$ is a shifted edge in $\HIIR$ or a shifted edge in $\HOI$.
        The former contradicts~\Cref{claim:no-edge-in-aux-digraph-normal} and the latter contradicts~\Cref{claim:no-edge-in-aux-digraph-special}.
    \end{proofclaim}

    \begin{figure}[tp]
      \begin{center}
        \includegraphics{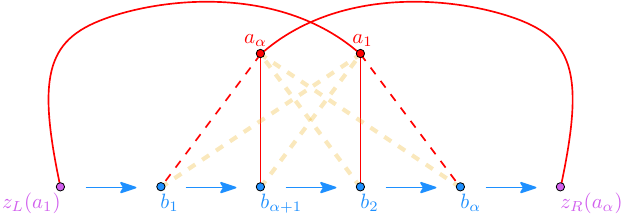}
      \end{center}
      \caption{
      Schematic drawing of the situation in~\Cref{claim:edge-in-HIILR}.
      In the proof, we show that $(a_1,b_{\alpha+1}),(a_\alpha,b_2) \in I$.
      Next, we show that $((a_{1},b_{1}),(a_{\alpha},b_{\alpha}))$ is an edge in $\HIILR$ which is witnessed by $(b_{\alpha+1},b_2)$.
      }
      \label{fig:wrap-up}
    \end{figure}

    We continue the proof, assuming without loss of generality that $b_1$ is left of $b_i$ for every $i \in [k]$ with $i \neq 1$.
\begin{claim}\label{claim:k-not-equal-3}
$k\neq3$.
\end{claim}
\begin{proofclaim}
    Suppose to the contrary that $k = 3$.
    Note that either $b_2$ is left of $b_3$ or $b_3$ is left of $b_2$.
    The former contradicts~\Cref{claim:no-1-2-k} and the latter contradicts~\Cref{claim:no-1-3-2}.
\end{proofclaim}
    Therefore, from now on, we assume $k \geq 4$.
    Note that by~\Cref{claim:no-1-3-2}, 
    $b_3$ is not left of $b_2$ and 
    by~\Cref{claim:no-1-2-k}, 
    $b_2$ is not left of $b_k$. 
    Hence, $b_1$ is left of $b_k$ and $b_k$ is left of $b_2$ and $b_2$ is left of $b_3$.
    It follows that there exists $\alpha \in \{3,\dots,k-1\}$ such that 
    \begin{align}\label{eq:b-ordering}
        \textrm{$b_1$ is left of $b_{\alpha+1}$, $b_{\alpha+1}$ left of $b_2$, and $b_2$ left of $b_\alpha$.}
    \end{align}
    \begin{claim}\label{claim:edge-in-HIILR}
        $((a_1,b_1),(a_\alpha,b_\alpha))$ is an edge in $\HIILR$.
    \end{claim}
    \begin{proofclaim}
        See~\Cref{fig:wrap-up}.
        We show that $(b_{\alpha+1},b_2)$ is a witness that $((a_1,b_1),(a_\alpha,b_\alpha))$ is an edge in $\HIILR$.
        Since $b_1$ is left of $b_{\alpha+1}$ and $b_2$ is left of $b_\alpha$,~\ref{items:HIILR-b-left-d-t-left-b} is satisfied.
        Since $\alpha \notin \{2,k\}$, we have $a_\alpha \parallel b_1$ and $a_1 \parallel b_\alpha$ in $P$, and thus,~\ref{items:HIILR-a-parallel-b} is satisfied.
        It remains to show that~\ref{items:HIILR-edgeHII} holds, that is, $(a_1,b_{\alpha+1}),(a_\alpha,b_2) \in I$ and $\sigma' = ((a_1,b_{\alpha+1}),(a_\alpha,b_2))$ is a regular In-In alternating cycle.
        The first part follows from~\eqref{eq:b-ordering} and~\Cref{claim:pair-in-I}.
        In particular, $\sigma'$ is a regular alternating cycle in $P$.
        We continue arguing that $\sigma'$ satisfies \ref{Lin} and \ref{Rin}.

        By~\Cref{claim:not-in-shadow-in-Y}, $b_{\alpha+1} \in Y(a_\alpha)$, and so, we can fix $z_\alpha \in Z(a_\alpha)$ with $z_\alpha \leq b_{\alpha+1}$ in $P$.
        We have $(a_{1},b_{\alpha+1}),(a_\alpha,b_\alpha) \in I$, $b_{2} \in B$ with $b_{\alpha+1}$ left of $b_2$ and $b_2$ left of $b_\alpha$, and $a_1 \leq b_2$ in $P$.
        Since $\alpha \notin \{1,2\}$, we have $a_1 \parallel b_\alpha$ and $a_\alpha \parallel b_2$ in $P$.
        Altogether, we can apply~\Cref{prop:a-in-some-region}.\ref{prop:a-in-some-region:L} to obtain that $z_\alpha$ is left of $z_R(a_\alpha)$ and $a_1 \in \Int \calR$ where $\calR = \calR(a_\alpha,z_\alpha,z_R(a_\alpha),U,a_\alpha[M_R(a_\alpha)]z_R(a_\alpha))$ and $U$ is an exposed witnessing path from $a_\alpha$ to $z_\alpha$ in $P$.
        Note that $\calR$ is a right region of $\sigma'$.
        Therefore, by~\Cref{prop:sac_four_types}.\ref{prop:sac_four_types:item:Rin}, $\sigma'$ satisfies \ref{Rin}.

        By~\Cref{claim:not-in-shadow-in-Y}, $b_{2} \in Y(a_1)$, and so, we can fix $z_1 \in Z(a_1)$ with $z_1 \leq b_2$ in $P$.
        We have $(a_1,b_1),(a_\alpha,b_2) \in I$, $b_{\alpha+1} \in B$ with $b_1$ left of $b_{\alpha+1}$ and $b_{\alpha+1}$ left of $b_2$, and $a_\alpha \leq b_{\alpha+1}$ in $P$.
        Since $\alpha \notin \{1,k\}$, we have $a_\alpha \parallel b_1$ and $a_1 \parallel b_{\alpha+1}$ in $P$.
        Altogether, we can apply~\Cref{prop:a-in-some-region}.\ref{prop:a-in-some-region:R} to obtain that $z_L(a_1)$ is left of $z_1$ and $a_\alpha \in \Int \calR$ where $\calR = \calR(a_1,z_L(a_1),z_1, a_1[M_L(a_1)]z_L(a_1),V)$ and $V$ is an exposed witnessing path from $a_1$ to $z_1$ in $P$.
        Note that $\calR$ is a left region of $\sigma'$.
        Therefore, by~\Cref{prop:sac_four_types}.\ref{prop:sac_four_types:item:Lin}, $\sigma'$ satisfies \ref{Lin}. 

        This completes the proof that $\sigma'=((a_1,b_{\alpha+1}),(a_{\alpha},b_2))$ is a regular In-In alternating cycle. Thus, $(b_{\alpha+1},b_2)$ is a witness that $((a_1,b_1),(a_{\alpha},b_{\alpha}))$ is an edge in $\HIILR$.
    \end{proofclaim}
    Claims~\ref{claim:edge-in-HIILR} and~\ref{claim:no-edge-in-aux-digraph-normal} contradict each other. 
    This completes the proof of the lemma.
\end{proof}

\section{Auxiliary oriented graphs: the proofs}
\label{sec:aux-proofs}
The goal of this section is to prove all the statements given without proofs in~\Cref{sec:aux-definitions} (\cref{prop:HOO,prop:HIIL,prop:HIIR,prop:HIILR,prop:HOO,lemma:HIO}).
To this end, we again fix a maximal good instance $(P,x_0,G,e_{-\infty},I)$ throughout the section.

\subsection{More on regions}
In this subsection, we collect technical statements on regions that are used later.

\begin{proposition} \label{prop:z_LR_are_outside}
        Let $(a,u,v,U,V,\calR,q,m,\gamma_L,\gamma_R)$ be a region tuple.
    \begin{enumerate}

      \myitem{$(L)$} $z_L(a) \notin \Int \calR$, and moreover, if $z_L(a) \not\leq u$ in~$P$, then  $z_L(a) \notin \calR$. \label{prop:z_LR_are_outside:left}
      \myitem{$(R)$} $z_R(a) \notin \Int \calR$, and moreover, if $z_R(a) \not\leq v$ in~$P$, then  $z_R(a) \notin \calR$. \label{prop:z_LR_are_outside:right}
  \end{enumerate}
\end{proposition}
\begin{proof}
    We prove only \ref{prop:z_LR_are_outside:left} as the proof of \ref{prop:z_LR_are_outside:right} is symmetric.
    Suppose to the contrary that $z_L(a) \in \Int \calR$.
    Then, by \cref{prop:b_relative_to_region}.\ref{prop:b_relative_to_region:item:inside}, $W_L(z_L(a))$ is right of $\gamma_L$.
    It follows that $W_L(z_L(a))$ is right of $M = x_0[W_L(u)]u[U]a$.
    Hence, by \cref{prop:paths_M_consistent}, $M_L(a)$ is right of $M$.
    Since $M \in \calM(a)$, this is a contradiction, and thus, we obtain $z_L(a) \notin \Int \calR$.

    Now, assume that $z_L(a) \not\leq u$ in~$P$, and suppose to the contrary that $z_L(a) \in \calR$.
    We already proved that $z_L(a) \notin \Int \calR$, hence, $z_L(a) \in \partial \calR$.
    Moreover, since $z_L(a) \not \leq u$ in~$P$, the element $z_L(a)$ does not lie in $W_L(u)$, which implies that $z_L(a)$ lies in $W_R(v)$.
    In particular, $z_L(a) \in \shadz(v)$.
    On the other hand, we claim that $z_L(a) \notin \shadz(u)$.
    Indeed, since $u \in Z(a)$, $a \notin \shadz(u)$, and so, by \cref{prop:properties_of_Z}.\ref{prop:properties_of_Z:item:not_interior_of_y}, $z_L(a)$ does not lie in the interior of $\shadz(u)$ (since $z_L(a) \not\leq u$ in~$P$, $z_L(a)$ is not on the boundary of $\shadz(u)$ either). 
    Therefore, since $u$ is left of $v$, by \cref{prop:in_one_shad_but_not_other_then_left}.\ref{prop:in_one_shad_but_not_other_then_left:right}, $z_L(a)$ is right of $u$.
    In particular, $M = x_0[W_L(u)]u[U]a$ is left of $M_L(a)$, which along with $M \in \calM(a)$ is a contradiction, and ends the proof.
\end{proof}

\begin{proposition}\label{prop:q-in-W_L(b)-and-W_R(b)}
Let $(a,u,v,U,V,\calR,q,m,\gamma_L,\gamma_R)$ be a region tuple.
Let $b \in B$ with $b \in \calR$.
Then, $q$ lies in both $W_L(b)$ and $W_R(b)$.
\end{proposition}
\begin{proof}
    We only prove that $q$ lies in $W_L(b)$ as the proof that $q$ lies in $W_R(b)$ is symmetric.
    Since $b \in \calR$ and $x_0 \notin \Int \calR$, the path $W_L(b)$ intersects $\partial \calR$, and so, it intersects $u[W_L(u)]q[W_R(v)]v$.
    Let $w$ be an element of $W_L(b)$ in $\partial \calR$.
    If $w$ lies in $u[W_L(u)]q$, then by~\cref{prop:W-consistent}.\ref{prop:W-consistent:left}, $x_0[W_L(u)]w = x_0[W_L(b)]w$, and so, $q$ lies in $W_L(b)$ as desired.
    Thus, assume that $w$ lies in $q[W_R(v)]v$ and $w \neq q$.
    Consider the path $W = x_0[W_L(q)]q[W_R(v)]w[W_L(b)]b$.
    If $W_L(b) = W$, then the assertion holds, thus, assume that $W_L(b) \neq W$.
    It follows that $W_L(b)$ is left of $W$.

    Next, we study the relation between $W_L(q)$ and $W_L(b)$.
    Since $q \leq b$ in~$P$, it is not possible that $W_L(b)$ is a strict subpath of $W_L(q)$.
    If $W_L(q)$ is a subpath of $W_L(b)$, then $q$ lies in $W_L(b)$, as desired. 
    If $W_L(q)$ is left of $W_L(b)$, then $W$ is left of $W_L(b)$, which is false.
    Finally, if $W_L(b)$ is left of $W_L(q)$, the path $x_0[W_L(b)]w[W_R(v)]v$ is left of $W_L(u) = x_0[W_L(q)]q[W_L(u)]u$, which contradicts the fact that $u$ is left of $v$.
    This ends the proof.
\end{proof}

\begin{corollary}\label{prop:uv_in_R_implies_path}
Let $(a,u,v,U,V,\calR,q,m,\gamma_L,\gamma_R)$ be a region tuple.
Let $u',v' \in B$ and let $q'$ be the maximal common element of $W_L(u')$ and $W_R(v')$ in~$P$.
If $u',v' \in \calR$, then, $u'[W_L(u')]q'[W_R(v')]v' \subset \calR$.
\end{corollary}
\begin{proof}
    Assume that $u',v' \in \calR$.
    By~\Cref{prop:q-in-W_L(b)-and-W_R(b)}, $q$ lies in both $W_L(u')$ and $W_R(v')$.
    In particular, $q \leq q'$ in~$P$ and by~\cref{prop:path_does_not_leave_regions}, both paths $q[W_L(u')]u'$ and $q[W_R(v')]v'$ are contained in $\calR$.
    This completes the proof.
\end{proof}

\begin{proposition}
\label{prop:path-leaving-the-right-side-of-a-region}
Let $(a,u,v,U,V,\calR,q,m,\gamma_L,\gamma_R)$ be a region tuple.
Let $d\in B$ with $d\not\in\calR$ and let $W$ be a witnessing path from $x_0$ to $d$ in~$P$ that intersects $\calR$. 
Let $w$ be the last element of $W$ in $\calR$.
    \begin{enumerate}
      \myitem{$(L)$} If $u$ is left of $d$, then $w$ lies in $\gamma_{R}$  and $\gamma_R$ is left of $x_0[W_R(v)]w[W]d$.\label{prop:path-leaving-the-right-side-of-a-region:L}
      \myitem{$(R)$} If $v$ is right of $d$, then $w$ lies in $\gamma_{L}$ and $\gamma_L$ is right of $x_0[W_L(u)]w[W]d$.\label{prop:path-leaving-the-right-side-of-a-region:R}
  \end{enumerate}
\end{proposition}

\begin{figure}[tp]
  \begin{center}
    \includegraphics{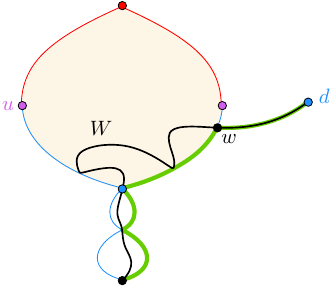}
  \end{center}
  \caption{
    An illustration of the statement of~\Cref{prop:path-leaving-the-right-side-of-a-region}.\ref{prop:path-leaving-the-right-side-of-a-region:L}.
    We show that $w$ lies in $\gamma_R$ and $\gamma_R$ is left of the path marked in green in the figure.
  }
  \label{fig:prop-path-leaving-region}
\end{figure}

\begin{proof}
We prove only~\ref{prop:path-leaving-the-right-side-of-a-region:L} as the proof of~\ref{prop:path-leaving-the-right-side-of-a-region:R} is symmetric.
See~\Cref{fig:prop-path-leaving-region} for an illustration of the statement.
Note that by definition $w \in B$ and $w \in \partial \calR$.
Let $e$ be the edge following $w$ in $W$.
Note that $e$ does not lie in $\calR$.
If $w$ is an element of $\gamma_L$, then let $e_L^-$, $e_L^+$ be the edges immediately preceding and following $w$ in $\gamma_L$ (when $w = x_0$, we set $e_L^- = e_{-\infty}$).
If $w$ is an element of $\gamma_R$, then let $e_R^-,e_R^+$ be the edges immediately preceding and following $w$ in $\gamma_R$ (when $w = x_0$, we set $e_R^- = e_{-\infty}$).

First, suppose that $w$ lies in $W_L(u)$ and $e_L^- \prec e \prec e_L^+$ in the $w$-ordering. 
This implies that $W' = x_0[W_L(u)]w[W]d$ is left of $\gamma_L$.
In particular, either $W_L(u)$ is a subpath of $W'$ or $W'$ is left of $W_L(u)$.
In both cases, there is a contradiction with the fact that $u$ is left of $d$.

Second, note that if $w$ lies in $W_R(v)$ and $e_R^+ \prec e \prec e_R^-$ in the $w$-ordering, then, $\gamma_R$ is left of $x_0[W_R(v)]w[W]d$, so the statement of the proposition holds.

Since $w \in B$ and $w\in\partial\calR$ we conclude that $w$ lies in  $u[W_L(u)]q[W_R(v)]v$.
If $w$ lies in $u[W_L(u)]q$ and $w \neq q$, then by~\ref{items:leaving_regions:left}, $e_L^- \prec e \prec e_L^+$ in the $w$-ordering, which leads to a contradiction as discussed above.
If $w$ lies in $q[W_R(v)]v$ and $w \neq q$, then by~\ref{items:leaving_regions:right}, $e_R^+ \prec e \prec e_R^-$ in the $w$-ordering, which yields the statement as discussed above.
Therefore, the only remaining case to consider is when $w = q$. 
In this case, all four edges $e_L^-$, $e_L^+$, $e_R^-$, $e_R^+$ are well-defined and $e \notin \{e_L^-, e_L^+, e_R^-, e_R^+\}$.
We have four possibilities: 
$e_L^- \prec e \prec e_L^+$, 
$e_L^+ \prec e \prec e_R^+$, 
$e_R^+ \prec e \prec e_R^-$, 
$e_R^- \prec e \prec e_L^-$ in the $w$-ordering.
The first case leads to a contradiction, while the third case yields the statement as discussed in a paragraph above. 
The second case contradicts~\ref{items:leaving_regions:x}.
Finally, in the fourth case, by \ref{items:leaving_shadows:y}, we obtain that the edge $e$ lies in $\shadz(q)$, and in particular, the whole path $q[W]d$ lies in $\shadz(q)$, which forces $d \in \shadz(q) \subset \shadz(u)$ and contradicts the assumption that $u$ is left of $d$. 
\end{proof}

\begin{proposition}
\label{prop:elements-not-in-shad-w}
    Let $(a,u,v,U,V,\calR,q,m,\gamma_L,\gamma_R)$ be a region tuple.
    Let $c \in A$ and $b \in B$ with $b,c \in \calR$.
    Let $z \in Z(c)$ and $W$ be an exposed witnessing path from $c$ to $z$ in~$P$ such that $W \subset \calR$.
    \begin{enumerate}
      \myitem{$(L)$} Let $w$ be an element of $W_R(v)$ with $q \leq w$ in~$P$.\label{prop:elements-not-in-shad-w:left}
    \begin{enumerate}
        \item If $b \notin \shadz(w)$, then either $W_L(w) $ is a subpath of $ W_L(b)$ or $W_L(b)$ is left of $W_L(w)$.\label{prop:elements-not-in-shad-w:left:i}
        \item If $c \notin \shadz(w)$, then $W_L(w)$ is not left of $W_L(z)$.\label{prop:elements-not-in-shad-w:left:ii}
        \item If $c \notin \shadz(w)$, then $\gce(W_L(u),W_L(w))$ lies in $W_L(z)$.\label{prop:elements-not-in-shad-w:left:iii}
    \end{enumerate}
      \myitem{$(R)$} Let $w$ be an element of $W_L(u)$ with $q \leq w$ in~$P$.
    \begin{enumerate}\label{prop:elements-not-in-shad-w:right}
        \item \mbox{If $b \notin \shadz(w)$, then either $W_R(w)$ is a subpath of $W_R(b)$ or $W_R(b)$ is right of $W_R(w)$.}\label{prop:elements-not-in-shad-w:right:i}
        \item If $c \notin \shadz(w)$, then $W_R(w)$ is not right of $W_R(z)$.\label{prop:elements-not-in-shad-w:right:ii}
        \item If $c \notin \shadz(w)$, then $\gce(W_R(v),W_R(w))$ lies in $W_R(z)$.\label{prop:elements-not-in-shad-w:right:iii}
    \end{enumerate}
    \end{enumerate}
\end{proposition}

\begin{figure}[tp]
  \begin{center}
    \includegraphics{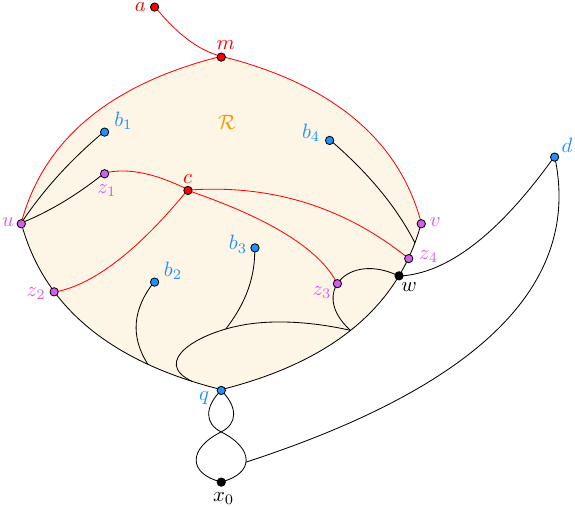}
  \end{center}
  \caption{
    An illustration of the statement of~\Cref{prop:elements-not-in-shad-w}.\ref{prop:elements-not-in-shad-w:left}.
    We depict several possibilities of how $z$ ($z_1$, $z_2$, $z_3$, and $z_4$) and $b$ ($b_1$, $b_2$, $b_3$, and $b_4$) can be positioned in the region.
    Later, in~\Cref{prop:shad-d-equals-shad-w}, we show that every element of $B$ common to $\calR$ and $\shadz(d)$ lies in $\shadz(w)$.
  }
  \label{fig:elements-not-in-shad-w}
\end{figure}

\begin{proof}
We prove only~\ref{prop:elements-not-in-shad-w:left} as the proof of~\ref{prop:elements-not-in-shad-w:right} is symmetric.
See~\Cref{fig:elements-not-in-shad-w} for an illustration of the statement.

Let $u' = \gce(W_L(u),W_L(w))$.
Note that by~\Cref{prop:q-in-W_L(b)-and-W_R(b)}, $q$ lies in $W_L(w)$, and so, $q \leq u'$ in~$P$.
Let
    \begin{align*}
        \gamma_1 &= u'[W_L(w)]w[\gamma_R]m[\gamma_L]u' \text{ and }\\
        \gamma_2 &= u'[W_L(u)]q[W_R(v)]w[W_L(w)]u'.
    \end{align*}
    Moreover, let $\Gamma_1$ be the region of $\gamma_1$ and $\Gamma_2$ be the region of $\gamma_2$.
    See~\Cref{fig:regions-into-gammas}.
    Note that all segments of $\gamma_1$ and $\gamma_2$ in the definition above, except $u'[W_L(w)]w$ are contained in $\partial \calR$. 
    Since $u'[W_L(w)]w \subset \calR$ (by~\Cref{prop:path_does_not_leave_regions}), we have $\gamma_1,\gamma_2 \subset \calR$, and so, by~\Cref{obs:region_containment}, $\Gamma_1,\Gamma_2 \subset \calR$.
    Observe also that $\Gamma_1 \cup \Gamma_2 = \calR$ and $\Gamma_1 \cap \Gamma_2 = u'[W_L(w)]w$.
    Note that $\gamma_2 = q[W_R(w)]w[W_L(w)]q$, and so, $\Gamma_2 \subset \shadz(w)$. 
    Moreover, since $q$ is a common element of $w$ we have that 
    $\shadz(w) = \shadz(q) \cup \Gamma_2$. 

\begin{figure}[tp]
  \begin{center}
    \includegraphics{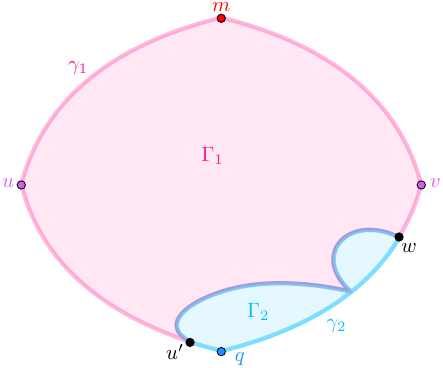}
  \end{center}
  \caption{
    The curves $\gamma_1$ and $\gamma_2$ and the regions $\Gamma_1$ and $\Gamma_2$.
  }
  \label{fig:regions-into-gammas}
\end{figure}

    First, we prove~\ref{prop:elements-not-in-shad-w:left:i}. 
    Assume that $b\not\in\shadz(w)$. 
    Since $b\in \calR$ and $b \notin \shadz(w)$, we have $b\in \Gamma_1$.
    We will prove that either $W_L(w)$ is a subpath of $W_L(b)$ or $W_L(b)$ is left of $W_L(w)$.
    
    We have $b \in \Gamma_1$ and $q \in \Gamma_2$.
    By~\Cref{prop:q-in-W_L(b)-and-W_R(b)}, $q$ lies in $W_L(b)$ and by~\Cref{prop:path_does_not_leave_regions}, $q[W_L(b)]b \subset \calR$.
    Therefore, $q[W_L(b)]b$ intersects $\Gamma_1 \cap \Gamma_2 = u'[W_L(w)]w$ -- let $w'$ be the maximal element of $P$ in this intersection.
    The paths $W_L(b)$ and $W_L(w)$ are $x_0$-consistent (by~\cref{prop:W-consistent}.\ref{prop:W-consistent:left}), and so, $w' = \gce(W_L(b),W_L(w))$. 
    
    If $w' = w$, then $W_L(w)$ is a subpath of $W_L(b)$ as desired. 
    Hence, assume $w' \neq w$.
    Note that $w' \in \shadz(w)$, thus, $w' \neq b$.
    Let $e$ be the edge following $w'$ in $W_L(b)$.
    Observe that $e$ lies in $\Gamma_1$.
    If $w' = q = x_0$, then let $e^- = e_{- \infty}$.
    Otherwise, let $e^-$ be the edge immediately preceding $w'$ in $W_L(w)$.
    Since $w' \neq w$, we can define $e^+$ to be the edge immediately following $w'$ in $W_L(w)$.
    We claim that 
    \begin{align}\label{eq:e-ee+}
        \text{$e^- \prec e \prec e^+$ in the $w'$-ordering.}
    \end{align}
    Note that~\cref{eq:e-ee+} and $w' = \gce(W_L(b),W_L(w))$ implies that $W_L(b)$ is left of $W_L(w)$ which will complete the proof of~\ref{prop:elements-not-in-shad-w:left:i}.
    
    Note that $u' < u$ in~$P$ as $u$ and $v$ are incomparable in~$P$.
    We define $f^+$ to be the edge immediately following $u'$ in $W_L(u)$.
    Finally, let $g^+$ be the first edge of $q[W_R(v)]v$.
    When $w' \neq u'$, $e^+$ follows $e^-$ in $\gamma_1$, and so, we obtain~\eqref{eq:e-ee+} by~\Cref{obs:when_in_gamma}.
    Now, suppose that $w' = u'$.
    In this case, $e^+$ follows $f^+$ in $\gamma_1$, and so, again by~\Cref{obs:when_in_gamma}, we have 
    \begin{align}\label{eq:f+ee+}
        \text{$f^+ \prec e \prec e^+$ in the $w'$-ordering.}
    \end{align}
    Since $w'[W_L(w)]w \subset \calR$, the edge $e^+$ lies in $\calR$.
    If $w' \neq q$, then by~\ref{items:leaving_regions:left}, we have $f^+ \prec e^+ \prec e^-$ in the $w'$-ordering, and combining with~\eqref{eq:f+ee+}, we obtain~\eqref{eq:e-ee+}.
    Finally, assume that $w' = u' = q$.
    Then, $e^-$ does not lie in $\calR$, and so, by~\ref{items:leaving_regions:x}, we have $g^+ \prec e^- \prec f^+$ in the $w'$-ordering and $f^+ \prec e^+ \prec g^+$ in the $w'$-ordering.
    It follows that $e^+ \prec e^- \prec f^+$ in the $w'$-ordering.
    Again combining with~\eqref{eq:f+ee+}, we obtain~\eqref{eq:e-ee+}.
    This completes the proof of~\eqref{eq:e-ee+} and as discussed of~\ref{prop:elements-not-in-shad-w:left:i}.

    For the proof of~\ref{prop:elements-not-in-shad-w:left:ii} and~\ref{prop:elements-not-in-shad-w:left:iii} assume that $c \notin \shadz(w)$.
    First, we claim that
    \begin{align}\label{eq:W-in-Gamma1}
        W \subset \Gamma_1.
    \end{align}
    Indeed, since $c \in \Gamma_1$, $W$ is an exposed path contained in $\calR$, $\Gamma_1 \subset \calR$, and the parts of the boundaries of $\Gamma_1$ and $\calR$ that are not in $B$ are the same, we have $W \subset \Gamma_1$, as desired.

    Now, we prove~\ref{prop:elements-not-in-shad-w:left:ii}.
    Since $c \in \calR$ and $c \notin \shadz(w)$, we have $c \in \Gamma_1$.
    Suppose to the contrary that $W_L(w)$ is left of $W_L(z)$.
    We claim that $z \notin \shadz(w)$.
    Since $W \subset \Gamma_1$ (by~\eqref{eq:W-in-Gamma1}), either $z \in \Int \Gamma_1$ or $z \in \partial \Gamma_1$.
    First, suppose that $z \in \Int \Gamma_1$.
    It follows that $z \notin \Gamma_2$.
    Additionally, by~\cref{prop:shad-disjoint-from-calR}, $z \notin \shadz(q)$.
    Recall that $\shadz(w) = \shadz(q) \cup \Gamma_2$.
    This implies $z \notin \shadz(w)$, so the claim holds in this case.
    Next, assume that $z \in \partial \Gamma_1$.
    If $z \in \shadz(w)$, then $z \in \partial\Gamma_1 \cap \shadz(w) = \partial\Gamma_1 \cap \Gamma_2 = u'[W_L(w)]w$, which yields that $z$ lies in $W_L(w)$, which contradicts $W_L(w)$ left of $W_L(z)$.
    This shows that indeed $z \notin \shadz(w)$.
    Now, we can apply~\ref{prop:elements-not-in-shad-w:left:i} for $b = z$.
    It follows that either $W_L(w)$ is a subpath of $W_L(z)$ or $W_L(z)$ is left of $W_L(w)$.
    In both cases, we obtain a contradiction, which completes the proof of~\ref{prop:elements-not-in-shad-w:left:ii}.

    Finally, we prove~\ref{prop:elements-not-in-shad-w:left:iii}, namely, we prove that $u'$ lies in $W_L(z)$.
    Since $W \subset \Gamma_1$ (by~\eqref{eq:W-in-Gamma1}), we have $z \in \Gamma_1$.
    On the other hand, $q \in \Gamma_2$.
    By~\Cref{prop:q-in-W_L(b)-and-W_R(b)}, $q$ lies in $W_L(z)$ and by~\Cref{prop:path_does_not_leave_regions}, $q[W_L(z)]z \subset \calR$.
    Therefore, $q[W_L(z)]z$ intersects $\Gamma_1 \cap \Gamma_2 = u'[W_L(w)]w$.
    In particular, by $x_0$-consistency of $W_L(w)$ and $W_L(z)$ (\Cref{prop:W-consistent}.\ref{prop:W-consistent:left}), $u'$ lies in $W_L(z)$, which completes the proof.
\end{proof}

\begin{proposition}\label{prop:shad-d-equals-shad-w}
    Let $(a,u,v,U,V,\calR,q,m,\gamma_L,\gamma_R)$ be a region tuple.
    Let $d \in B$ with $d \notin \calR$.
    \begin{enumerate}
      \myitem{$(L)$} Let $u$ be left of $d$ and let $w$ be the maximal common element of $W_L(d)$ and $W_R(v)$ in~$P$.
      For every $b \in B$ with $b \in \calR$ and $b \in \shadz(d)$, we have $b \in \shadz(w)$.
      Moreover, if such $b \in B$ exists, then $q \leq w$ in~$P$.\label{prop:shad-d-equals-shad-w:L}
      \myitem{$(R)$} Let $v$ be right of $d$ and $w$ be the maximal common element of $W_R(d)$ and $W_L(u)$ in~$P$.
      For every $b \in B$ with $b \in \calR$ and $b \in \shadz(d)$, we have $b \in \shadz(w)$.
      Moreover, if such $b \in B$ exists, then $q \leq w$ in~$P$.\label{prop:shad-d-equals-shad-w:R}
  \end{enumerate}
\end{proposition}
\begin{proof}
We prove only~\ref{prop:shad-d-equals-shad-w:L} as the proof of~\ref{prop:shad-d-equals-shad-w:R} is symmetric.
See again~\Cref{fig:elements-not-in-shad-w}.
Let $b \in B$ with $b \in \calR$ and $b \in \shadz(d)$.

First, we show the \q{moreover} part, that is, $q \leq w$ in~$P$.
Since $u$ is left of $d$, we have $W_L(u)$ left of $W_L(d)$.
Moreover, $W_L(q)$ is a subpath of $W_L(u)$, hence, either $W_L(q)$ is left of $W_L(d)$ or $W_L(q)$ is a subpath of $W_L(d)$.
In the latter case, $q$ lies in both $W_L(d)$ and $W_L(w)$, hence, $q \leq w$ in~$P$, as desired.
Thus, we may assume that $W_L(q)$ is left of $W_L(d)$.
Since $b \in \shadz(d)$, by~\Cref{prop:paths_directions_in_shadows}.\ref{prop:paths_directions_in_shadows:left}, $W_L(b)$ is not left of $W_L(d)$.
This translates into one of the three options:
$W_L(d)$ is a subpath of $W_L(b)$ or $W_L(b)$ is a subpath of $W_L(d)$ or $W_L(d)$ is left of $W_L(b)$. 
Since $b \in \calR$, by~\Cref{prop:q-in-W_L(b)-and-W_R(b)}, $q$ lies in $W_L(b)$, and so, $W_L(q)$ is a subpath of $W_L(b)$.
Each of the options \q{$W_L(d)$ is a subpath of $W_L(b)$} and \q{$W_L(d)$ is left of $W_L(b)$} implies that $W_L(q)$ is left of $W_L(b)$, which is false.  
Thus, the third option holds: $W_L(b)$ is a subpath of $W_L(d)$.
It follows that $W_L(q)$ is a subpath of $W_L(d)$, and so, $q \leq w$ in~$P$, as desired.

Now, we prove the main statement.
To this end, suppose to the contrary that $b \notin \shadz(w)$.
By~\Cref{prop:elements-not-in-shad-w}.\ref{prop:elements-not-in-shad-w:left}.\ref{prop:elements-not-in-shad-w:left:i}, either $w < b$ in~$P$ or $W_L(b)$ is left of $W_L(w)$.
In the latter case, we obtain $W_L(b)$ left of $W_L(d)$, which contradicts $b \in \shadz(d)$ by~\Cref{prop:paths_directions_in_shadows}.\ref{prop:paths_directions_in_shadows:left}.
Thus, assume that $w < b$ in~$P$ and let $W$ be a witnessing path from $w$ to $b$ in~$P$.

Note that $w \neq d$ as $w \in \calR$ and $d \notin \calR$.
Let $e^-$ and $e^+$ be the edges immediately preceding and following $w$ in $W_L(d)$ (when $w = x_0$, we set $e^- = e_{-\infty}$).
Let $f^-$ and $f^+$ be the edges immediately preceding and following $w$ in $\gamma_R$ (when $w = x_0$, we set $f^- = e_{-\infty}$).
Let $e$ be the edge of $W$ incident to $w$.
See~\Cref{fig:prop-shad-w-edges}.

\begin{figure}[tp]
  \begin{center}
    \includegraphics{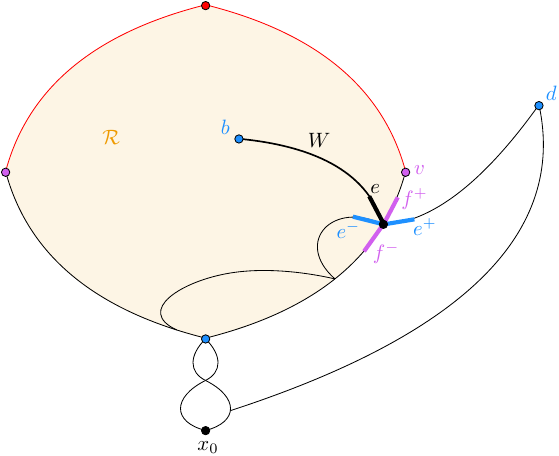}
  \end{center}
  \caption{
    We prove that the edges $e$, $e^-$, $e^+$, $f^-$, and $f^+$ are in the order depicted in the figure around $w$.
  }
  \label{fig:prop-shad-w-edges}
\end{figure}

Note that
\begin{align}
\label{eq:ordering:1} f^- \prec f^+ \prec e^+\ &\textrm{in the $w$-ordering}&&\textrm{by~\Cref{prop:path-leaving-the-right-side-of-a-region}.\ref{prop:path-leaving-the-right-side-of-a-region:L},}\\
\label{eq:ordering:2} f^- \prec e \preccurlyeq f^+\ &\textrm{in the $w$-ordering}&&\textrm{by~\ref{items:leaving_regions:left} or~\ref{items:leaving_regions:x}.}
\end{align}
Next, we claim that
\begin{align}
\label{eq:ordering:3} f^- \preccurlyeq e^- \prec f^+\ &\textrm{in the $w$-ordering.}
\end{align}
First, assume that $q < w$ in~$P$.
Then, since $q,w \in \calR$, by~\Cref{prop:path_does_not_leave_regions}, $q[W_L(w)]w \subset \calR$, and so, $e^- \subset \calR$.
Thus, by~\ref{items:leaving_regions:right},~\eqref{eq:ordering:3} follows.
Finally, assume that $q = w$.
Then, $f^+$ does not lie in $\shadz(q)$, and so, by~\ref{items:leaving_shadows:y},~\eqref{eq:ordering:3} follows.

By~\eqref{eq:ordering:2} and~\eqref{eq:ordering:3}, either $f^- \prec e \prec e^-$ in the $w$-ordering or $e^- \prec e \preccurlyeq f^+$ in the $w$-ordering.
The former implies (by~\ref{items:leaving_shadows:y}) that $e$ lies in $\shadz(w)$, which yields $b \in \shadz(w)$, which is a contradiction.
Therefore, the latter holds.
Combining it with~\cref{eq:ordering:1}, \cref{eq:ordering:2}, and \cref{eq:ordering:3}, we obtain $e^- \prec e \preccurlyeq f^+ \prec e^+$ in the $w$-ordering.
This implies that $x_0[W_L(w)]w[W]b$ is left of $W_L(d)$, and thus, $W_L(b)$ is left of $W_L(d)$.
However, this is a contradiction with $b \in \shadz(d)$ by~\Cref{prop:paths_directions_in_shadows}.\ref{prop:paths_directions_in_shadows:left}.
This completes the proof.
\end{proof}

\subsection{Out-Out oriented graphs}
In this subsection, we prove~\Cref{prop:HOO}, which states that $\maxpath(\HOO) \leq \se_P(I)$.
See an example of a path in $\HOO$ in \Cref{fig:path_HOO}.

\begin{proof}[Proof of~\cref{prop:HOO}]
    If $\maxpath(\HOO) \leq 1$, then there is nothing to prove as $\se_P(I)\geq1$.
    Thus, we can assume $\maxpath(\HOO)\geq 2$. 
    Let $((a_1,b_1),\dots,(a_n,b_n))$ be a path in $\HOO$ where $n$ is an integer with $n \geq 2$.
    By~\cref{prop:transitivity_RPP}, $((a_i,b_i),(a_{j},b_{j}))$ satisfies \ref{Lout} and \ref{Rout} for every $i,j \in [n]$ with $i < j$.
    Therefore, by \cref{obs:PL2_PR2_basic}, $a_{j} < b_i$ and $a_i < b_{j}$ in~$P$.
    It follows that $J = \{(a_i,b_i) : i \in [n]\}$ induces a standard example in~$P$.
    Since $J \subset I$ and $|J| = n$, we obtain $n \leq \se_P(I)$, as desired.
\end{proof}

\begin{figure}[tp]
  \begin{center}
    \includegraphics{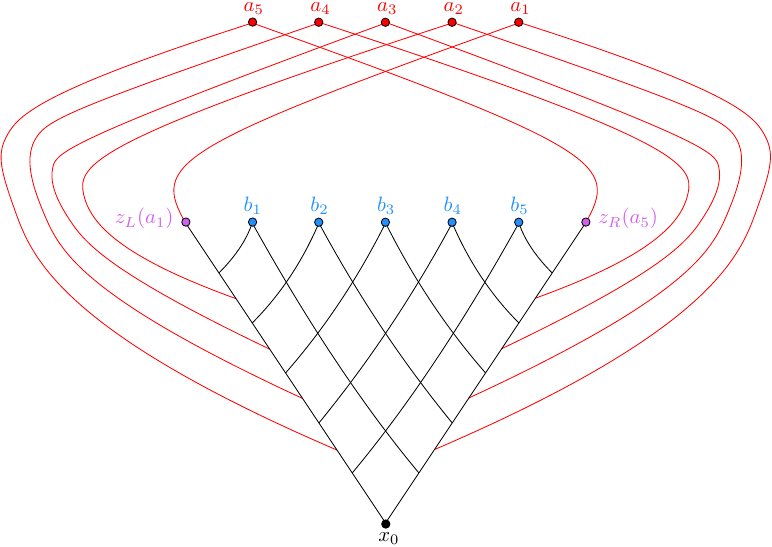}
  \end{center}
  \caption{
    $((a_1,b_1),(a_2,b_2),(a_3,b_3),(a_4,b_4),(a_5,b_5))$ is a path in $\HOO$.
  }
  \label{fig:path_HOO}
\end{figure}

\subsection{In-In oriented graphs}
In this subsection, we prove~\Cref{prop:HIIL,prop:HIIR,prop:HIILR}.
Namely, we prove that $\maxpath(H) \leq \se_P(I)$ for each $H \in \{ \HIIL, \HIIR, \HIILR\}$.
In fact, we will prove that $\maxpath(H)\leq\maxpath(\HII)$, and then that $\maxpath(\HII)\leq\se_P(I)$.
We begin with the latter, see~\Cref{prop:HII}.

For every $a \in A$, by~\cref{prop:z_L_b_z_R}, $z_L(a)$ is left of $z_R(a)$, and hence, we can define 
\[
  \cgR(a)=\cgR(a,z_L(a),z_R(a), z_L(a)[M_L(a)]a, z_R(a)[M_R(a)]a).
\]

\begin{proposition} \label{prop:In-In_the_same_region}
    Let $((a_1,b_1),(a_2,b_2))$ be a regular In-In alternating cycle. Then, $\calR(a_1) = \calR(a_2)$.
\end{proposition}
\begin{proof}
    Let $\sigma=((a_1,b_1),(a_2,b_2))$ and let $m_1$ and $m_2$ be the upper-mins of $\calR(a_1)$ and $\calR(a_2)$, respectively.
    First, we argue that
    \begin{align}\label{eq:m-lies-in-M}
        \text{$m_2$ lies in $M_L(a_1)$} \ \ \  \text{and} \ \ \ \text{$m_1$ lies in $M_R(a_2)$.}
    \end{align}
    We will prove only the first part of the statement: $m_2$ lies in $M_L(a_1)$, as the proof of the second part is symmetric.
    Let $u = \gce(M_L(a_1), M_L(a_2))$.
    Note that both $m_2$ and $u$ lie in $M_L(a_2)$. 
    There are two exclusive possibilities: 
    $x_0[M_L(a_2)]m_2$ is a subpath of $x_0[M_L(a_2)]u$ or
    $x_0[M_L(a_2)]u$ is a proper subpath of $x_0[M_L(a_2)]m_2$.
    We argue that the former always holds. 
    Assume to the contrary that $x_0[M_L(a_2)]u$ is a proper subpath of $x_0[M_L(a_2)]m_2$. 
    Since $\sigma$ satisfies~\ref{Lin}, $M_L(a_1)$ is left of $M_L(a_2)$. 
    Let $e$ be the edge following $u$ in $M_L(a_1)$ -- it exists because $a_1 \parallel a_2$ in~$P$. 
    
    We claim that 
    \begin{align}\label{eq:e-notin-R2}
        \text{$e$ does not lie in $\calR(a_2)$.}
    \end{align}
    Let $q_2$ be the lower-min of $\calR(a_2)$.
    Note that both $u$ and $q_2$ lie in $M_L(a_2)$.
    If $u < q_2$ in~$P$, then by~\cref{prop:shad-disjoint-from-calR}, $u \notin \calR(a_2)$, and so, $e$ does not lie in $\calR(a_2)$, as claimed.
    If $u = q_2$, then since $M_L(a_1)$ is left of $M_L(a_2)$, by~\ref{items:leaving_regions:x}, $e$ does not lie in $\calR(a_2)$, as claimed.
    Finally, otherwise, $u$ lies strictly on the left side of $\calR(a_2)$, and so, by~\ref{items:leaving_regions:left}, $e$ does not lie in $\calR(a_2)$, as claimed.
    This way, we obtain~\eqref{eq:e-notin-R2}.
    
    Since $\sigma$ is In-In, it satisfies~\ref{Rin} and since $\calR(a_2)$ is a right region of $\sigma$, by \cref{cor:characterize_regular_sacs}.\eqref{cor:characterize_regular_sacs:item:Rin}, we have $a_1 \in \Int \calR(a_2)$. 
    Since $e$ does not lie in $\calR(a_2)$ (by~\eqref{eq:e-notin-R2}) and $a_1\in\Int \calR(a_2)$, the path $u[M_L(a_1)]a_1$ intersects $\partial \calR(a_2)$ in an element distinct from $u$, let $v$ be such an element.
    Since the paths $M_L(a_1)$ and $M_L(a_2)$ are $x_0$-consistent (by \cref{prop:ML_consistent}.\ref{prop:ML_consistent:left}), $v$ does not lie in $M_L(a_2)$, and hence, it lies in $M_R(a_2)$.
    Let $M = x_0[M_L(a_2)]u[M_L(a_1)]v[M_R(a_2)]a_2$ and note that $M$ is left of $M_L(a_2)$.
    To conclude~\eqref{eq:m-lies-in-M}, it suffices to show that
    \begin{align}\label{eq:M-in-calM}
        M \in \calM(a_2).
    \end{align}
    Indeed, if $M \in \calM(a_2)$, then $M$ being left of $M_L(a_2)$ contradicts the definition of $M_L(a_2)$ and completes the proof of~\eqref{eq:m-lies-in-M}.
    
    We split the proof of~\cref{eq:M-in-calM} into two cases: 
    $v \in B$ and $v\not\in B$. 
    First, suppose that $v\in B$. 
    Then, $M = x_0[M_L(a_1)]v[M_R(a_2)]z_R(a_2)[M_R(a_2)]a_2$.
    Since $z_R(a_2) \in Z(a_2)$, $x_0[M_L(a_1)]v[M_R(a_2)]z_R(a_2)$ is a witnessing path in~$P$, and $z_R(a_2)[M_R(a_2)]a_2$ is an exposed witnessing path in~$P$, $M \in \calM(a_2)$.
    Next, suppose that $v \notin B$.
    We have $M = x_0[M_L(a_1)]z_L(a_1)[M_L(a_1)]v[M_R(a_2)]a_2$.
    Note that $x_0[M_L(a_1)]z_L(a_1)$ is a witnessing path in~$P$ and $z_L(a_1)[M_L(a_1)]v[M_R(a_2)]a_2$ is an exposed witnessing path in~$P$.
    To obtain that $M \in \calM(a_2)$, it suffices to show that $z_L(a_1) \in Z(a_2)$.
    It is clear that there is an exposed witnessing path from $a_2$ to $z_L(a_1)$ in~$P$ (the path $a_2[M_R(a_2)]v[M_L(a_1)]z_L(a_1)$).
    Thus, it remains to prove that $a_2 \notin \shadz(z_L(a_1))$.
    Observe that none of the elements of the paths $a_1[M_L(a_1)]v$ and $a_2[M_R(a_2)]v$ are in $B$.
    Therefore, these paths do not intersect $\partial \shadz(z_L(a_1))$.
    In particular, either all $a_1$, $v$, and $a_2$ are in $\shadz(z_L(a_1))$ or none of them.
    Since $z_L(a_1) \in Z(a_1)$ by definition, the latter holds.
    We conclude that $a_2 \notin \shadz(z_L(a_1))$, and so, $z_L(a_1) \in Z(a_2)$, and so, $M \in \calM(a_2)$, and finally,~\eqref{eq:M-in-calM} holds.
    This completes the proof of~\eqref{eq:m-lies-in-M}.

    Since $M_L(a_1),M_L(a_2)$ and $M_R(a_1),M_R(a_2)$ are $x_0$-consistent (by~\Cref{prop:ML_consistent}), we obtain
    \begin{align} \label{eq:claim:in:In-In_the_same_region}
        x_0[M_L(a_2)]m_2 = x_0[M_L(a_1)]m_2\quad \textrm{and}\quad x_0[M_R(a_1)]m_1 = x_0[M_R(a_2)]m_1.
    \end{align}
    Since $m_1$ and $m_2$ lie in $M_L(a_1)$, $m_1$ and $m_2$ are comparable in~$P$.

    Suppose first that $m_1 < m_2$ in~$P$.
    Recall that $m_2$ and $m_1$ lie in $M_R(a_2)$. 
    Since $m_1<m_2$ in~$P$ we conclude that $m_2$ lies in $x_0[M_R(a_2)]m_1$, 
    which by \eqref{eq:claim:in:In-In_the_same_region} implies that $m_2$ lies in $M_R(a_1)$. 
    All this together means that $m_2$ is a common element of $M_L(a_1)$ and $M_R(a_1)$. 
    However, this and $m_1<m_2$ in~$P$ contradicts the choice of~$m_1$. 
    Similarly, the assumption of $m_2<m_1$ in~$P$ also leads to a contradiction. 
    We conclude that $m_1 = m_2$, and so, $x_0[M_L(a_1)]m_1 =x_0[M_L(a_2)]m_2$ and $x_0[M_R(a_1)]m_1 = x_0[M_R(a_2)]m_2$.
    Thus, $\calR(a_1) = \calR(a_2)$.
\end{proof}

Now, we prove the transitivity of regular In-In alternating cycles, which will give~\Cref{prop:HII}, that is, $\maxpath(\HII) \leq \se_P(I)$.
See examples of paths in $\HII$ in \Cref{fig:path_HII}.

\begin{figure}[tp]
     \centering
     \begin{subfigure}[b]{1\textwidth}
         \centering
         \includegraphics{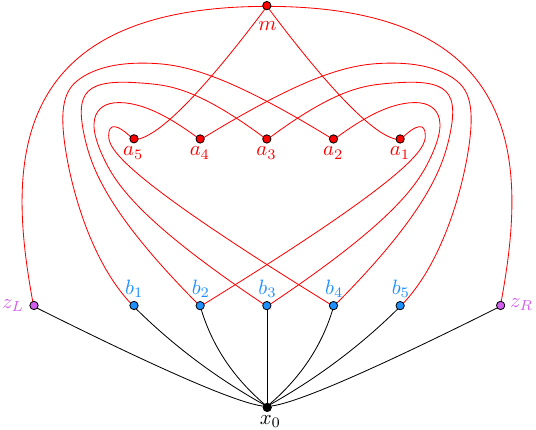}
     \end{subfigure}
     \begin{subfigure}[t]{1\textwidth}
         \centering
         \includegraphics{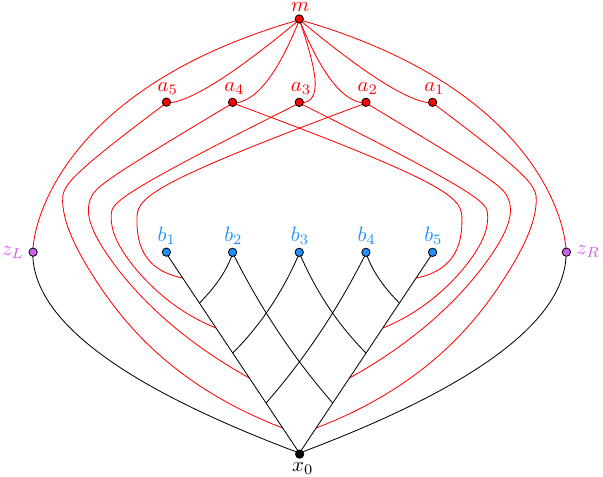}
     \end{subfigure}
  \caption{
    In both figures $((a_1,b_1),(a_2,b_2),(a_3,b_3),(a_4,b_4),(a_5,b_5))$ is a path in $\HII$.
    We also have $\calR(a_1) = \calR(a_2) = \calR(a_3) = \calR(a_4) = \calR(a_5)$ in both figures -- recall~\cref{prop:In-In_the_same_region}.
    In the proof of~\cref{prop:In-In-cycle-transitive}, there are two cases: $a_1 \in \calS$ and $a_1 \notin \calS$.
    Both are possible, we depict the former in the top part of the figure and the latter in the bottom part.
    }
  \label{fig:path_HII}
\end{figure}

\begin{proposition}\label{prop:In-In-cycle-transitive}
    Let $((a_1,b_1),(a_2,b_2))$ and $((a_2,b_2),(a_3,b_3))$ be regular In-In alternating cycles. 
    Then, $((a_1,b_1),(a_3,b_3))$ is a regular  In-In alternating cycle.
\end{proposition}
\begin{proof} 
    Let $\sigma=((a_1,b_1),(a_3,b_3))$.
    Since $b_1$ is left of $b_2$ and $b_2$ is left of $b_3$, $b_1$ is left of $b_3$ (by~\Cref{prop:left_porders_bs}), hence, $\sigma$ is regular.
    By \cref{prop:transitivity_RPP}.\ref{prop:item:transitivity-Lin} and \ref{prop:item:transitivity-Rin}, $\sigma$ satisfies \ref{Lin} and \ref{Rin}.
    Therefore, it suffices to prove that $\sigma$ is an alternating cycle, that is, $a_1 < b_3$ and $a_3 < b_1$ in~$P$.
    We only prove that $a_1 < b_3$ in~$P$, the proof of the other inequality is symmetric.

    By~\Cref{cor:regular-Y}, $b_1,b_3 \in Y(a_2)$, thus, we can fix $z_1,z_3 \in Z(a_2)$ such that $z_1 \leq b_1$ and $z_3 \leq b_3$ in~$P$.
    Let $W_1$ and $W_3$ be exposed witnessing paths in~$P$ from $a_2$ to $z_1$ and $z_3$ respectively.
    Since $b_1$ is left of $b_2$ and $b_2$ is left of $b_3$, by \cref{cor:left_is_preserved_from_y_to_z}, $z_1$ is left of $b_2$ and $b_2$ is left of $z_3$.
    In particular, $z_1$ is left of $z_3$, and so, we can define the region $\calS = \calR(a_2,z_1,z_3,W_1,W_3)$.
    We consider two cases: either $a_1 \in \calS$ or $a_1 \notin \calS$.
    Note that neither of the cases leads to a contradiction, see~\cref{fig:in-in-transitivity}.

\begin{figure}[tp]
  \begin{center}
    \includegraphics{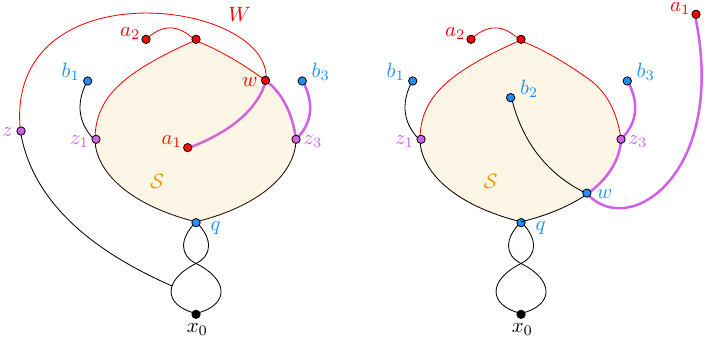}
  \end{center}
  \caption{
    An illustration of the proof of~\Cref{prop:In-In-cycle-transitive}.
    On the left, we depict the case where $a_1 \in \calS$ and on the right, we depict the case where $a_1 \notin \calS$.
    We mark with a purple curve the final comparability $a_1 < b_3$ in $P$.
  }
  \label{fig:in-in-transitivity}
\end{figure}

    First, assume that $a_1 \in \calS$.
    By \cref{prop:In-In_the_same_region}, $\calR(a_1) = \calR(a_2) = \calR(a_3)$, and so, $z_L(a_1) = z_L(a_2) = z_L(a_3) =: z$.
    Note that $a_1,a_2,a_3 \leq z$ in~$P$, hence, $z \not\leq b_1,b_3$ in~$P$, and thus, $z \not\leq z_1,z_3$ in~$P$.
    Therefore, by \cref{prop:z_LR_are_outside}, $z \notin \calS$.
    Let $W$ be an exposed witnessing path from $a_1$ to $z$ in~$P$.
    Since $a_1 \in \calS$ and $z \notin \calS$, there is an element $w$ of $W$ in $\partial \calS$.
    Since $W$ is an exposed path and $a_1 \parallel b_1$ in~$P$, $w$ lies in $W_3$.
    It follows that $a_1 < w < z_3 \leq b_3$ in~$P$, as desired.

    Finally, assume that $a_1 \notin \calS$.
    Since $z_1$ is left of $b_2$ and $b_2$ is left of $z_3$, by \cref{cor:sandwitch_b_in_region}.\ref{cor:sandwitch_b_in_region:int}, we have $b_2 \in \Int \calS$.
    Let $W$ be a witnessing path from $a_1$ to $b_2$ in~$P$.
    Since $a_1 \notin \calS$ and $b_2 \in \calS$, there is an element $w$ of $W$ in $\partial \calS$.
    Since $a_1 \parallel b_1$ and $a_2 \parallel b_2$ in~$P$, $w$ lies in $W_R(z_3)$.
    It follows that $a_1 < w \leq z_3 \leq b_3$ in~$P$.
\end{proof}

\begin{proposition} \label{prop:HII} $\maxpath(\HII) \leq \se_P(I)$.
\end{proposition}
\begin{proof}
    If $\maxpath(\HII) \leq 1$, then there is nothing to prove as $\se_P(I)\geq1$.
    Thus, we can assume $\maxpath(\HII)\geq 2$. 
    Let $((a_1,b_1),\dots,(a_n,b_n))$ be a path in $\HII$ where $n$ is an integer with $n \geq 2$.
    By~\cref{prop:In-In-cycle-transitive}, $((a_i,b_i),(a_{j},b_{j}))$ is an alternating cycle for every $i,j \in [n]$ with $i < j$.
    It follows that $J = \{(a_i,b_i) : i \in [n]\}$ induces a standard example in~$P$.
    Since $J \subset I$ and $|J| = n$, we obtain $n \leq \se_P(I)$, as desired.
\end{proof}

Next, we show that every directed path in $\HIIL$ and $\HIIR$ can be \q{transformed} into a directed path in $\HII$.
We gave a schematic illustration of edges of these directed graphs in~\Cref{fig:edges-HIIL-HIIR}.
See~\Cref{fig:edges-HII-real} for a more precise illustration.
See also~\Cref{fig:path-HIIL} for an example of a path in $\HIIL$ consisting of only shifted edges.

\begin{figure}[tp]
  \begin{center}
    \includegraphics{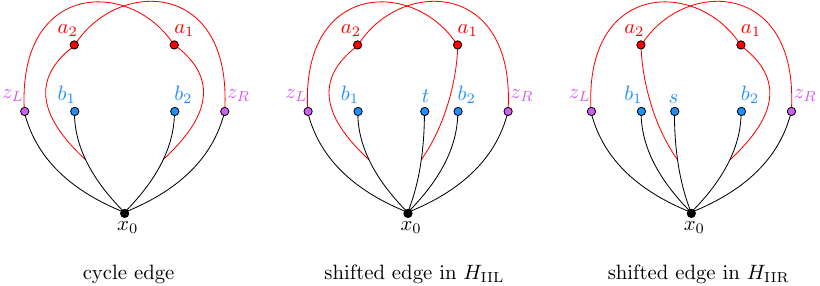}
  \end{center}
  \caption{
    Edges in $\HIIL$ and $\HIIR$.
    Note that in all the three cases $z_L = z_L(a_1) = z_L(a_2)$ and $z_R = z_R(a_1) = z_R(a_2)$ as shown in~\Cref{prop:In-In_the_same_region}.
  }
  \label{fig:edges-HII-real}
\end{figure}

\begin{figure}[tp]
  \begin{center}
    \includegraphics{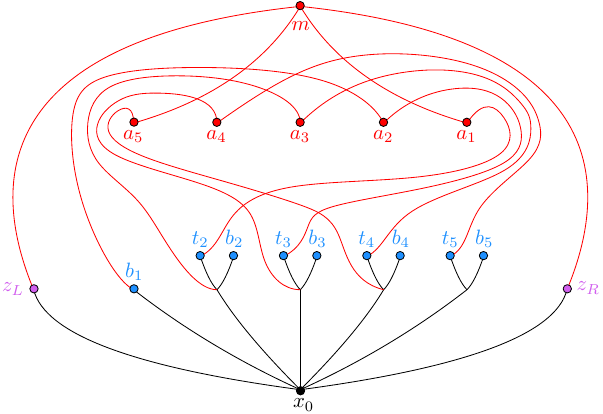}
  \end{center}
  \caption{
  $((a_1,b_1),(a_2,b_2),(a_3,b_3),(a_4,b_4),(a_5,b_5))$ is a path in $\HIIL$.
  All the edges are shifted edges.
    For every $i \in \{2,3,4,5\}$, $t_{i}$ is a witness for $((a_{i-1},b_{i-1}),(a_{i},b_{i}))$.
    Note that $\{(a_1,b_1),(a_2,t_2),(a_3,t_3),(a_4,t_4),(a_5,t_5)\}$ induces a standard example of order $5$ in~$P$.
  }
  \label{fig:path-HIIL}
\end{figure}

\begin{proposition}\label{prop:HIILR-switch-two}
\hfill
  \begin{enumerate}
      \myitem{$(L)$}  Let $((a_1,b_1),(a_2,b_2))$ be a shifted edge in $\HIIL$ with a witness $t \in B$ and let $((a_2,b_2),(a_3,b_3))$ be a cycle edge in $\HIIL$.
    Then,
  $((a_2,t),(a_3,b_3))$ is a cycle edge in $\HIIL$. \label{prop:HIILR-switch-two-left}
    \myitem{$(R)$} Let $((a_1,b_1),(a_2,b_2))$ be a cycle edge in $\HIIR$ and let $((a_2,b_2),(a_3,b_3))$ be a shifted edge in $\HIIR$ with a witness $s \in B$.
    Then,
  $((a_1,b_1),(a_2,s))$ is a cycle edge in $\HIIR$. \label{prop:HIILR-switch-two-right}
  \end{enumerate}
\end{proposition}
\begin{proof}
    We only prove~\ref{prop:HIILR-switch-two-left} as the proof of~\ref{prop:HIILR-switch-two-right} is symmetric.
    See~\cref{fig:proof-HIIL} for an illustration of the proof.
    Let $\sigma = ((a_2,t),(a_3,b_3))$. 
    By~\ref{items:HIIL-shifted-edgeHII}, $(a_2,t)\in I$.
    Since $((a_1,b_1),(a_2,t))$ and $((a_2,b_2),(a_3,b_3))$ are edges in $\HII$, in particular $b_1$ is left of $t$ and $b_2$ is left of $b_3$. 
    By~\ref{items:HIIL-shifted-t-left-b}, $t$ is left of $b_2$. Altogether,
    \[
    \textrm{$b_1$ is left of $t$, $t$ is left of $b_2$, and $b_2$ is left of $b_3$.}
    \]
    In particular, $\sigma$ is regular.
    Since $((a_2,b_2),(a_3,b_3))$ is an In-In alternating cycle in~$P$, by~\cref{obs:switch_bs_in_RPP}.\ref{obs:switch_bs_in_RPP:2in1} and \ref{obs:switch_bs_in_RPP:1in2}, $\sigma$ satisfies \ref{Lin} and \ref{Rin}. 
    It remains to argue that $\sigma$ is an alternating cycle in~$P$, that is, $a_2 < b_3$ and $a_3 < t$ in~$P$. 
    The former follows from the fact that $((a_2,b_2),(a_3,b_3))$ is an alternating cycle in~$P$.
    Therefore, to conclude the proof it suffices to show that $a_3 < t$ in~$P$.
    Suppose to the contrary that $a_3 \parallel t$~in~$P$.

    \begin{figure}[tp]
      \begin{center}
        \includegraphics{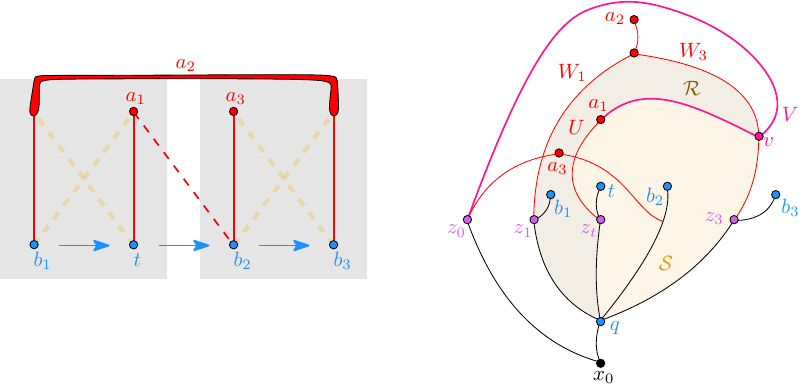}
      \end{center}
      \caption{
        An illustration of the proof of~\cref{prop:HIILR-switch-two}.
        On the left-hand side, we show a schematic summary of the assumptions (drawing conventions are the same as in~\cref{fig:edges-HIIL-HIIR}).
        On the right-hand side, we illustrate the objects occurring in the proof of the proposition.
      }
      \label{fig:proof-HIIL}
    \end{figure}

    
    By~\cref{cor:regular-Y}, $b_1,b_3 \in Y(a_2)$ and $t \in Y(a_1)$.
    Thus, we can fix $z_1,z_3 \in Z(a_2)$  and $z_t \in Z(a_1)$ such that $a_2< z_1 \leq b_1$, $a_2<z_3 \leq b_3$, and $a_1<z_t \leq t$ in~$P$.

    \begin{claim}\label{claim:prop:HIILR-switch-two:z-left-right}
        $b_1$ is left of $z_t$, 
        $z_1$ is left of $z_t$, $z_t$ is left of $b_2$, and $b_2$ is left of $z_3$.
    \end{claim}
    \begin{proofclaim}
    Since $(a_1,b_1) \in I$ and $a_1 < z_t \leq t$ in~$P$, by~\cref{cor:left_is_preserved_from_y_to_z}.\ref{cor:left_is_preserved_from_y_to_z:right}, $b_1$ is left of $z_t$.
    Since $(a_2,b_2) \in I$ and $a_2 < z_3 \leq b_3$ in~$P$, by~\cref{cor:left_is_preserved_from_y_to_z}.\ref{cor:left_is_preserved_from_y_to_z:right}, $b_2$ is left of $z_3$.
     
    We have $(a_1,b_1) \in I$ and $b_1$ left of $b_2$, thus, by~\cref{prop:dangerous-implies-in-Y}, $a_1 \notin \shadz(b_2)$.
    Additionally, $a_1 \parallel b_2$ in~$P$ (by~\ref{items:HIIL-shifted-a-parallel-b}) and $t$ is left of $b_2$.
    Therefore, by~\cref{prop:left_is_preserved_from_y_to_z}.\ref{prop:left_is_preserved_from_y_to_z:right}, $z_t$ is left of $b_2$.

    It remains to show that $z_1$ is left of $z_t$.
    Since $a_2 \parallel t$ and $z_t \leq t$ in~$P$, we have $a_2 \parallel z_t$ in~$P$.
    Since $(a_2,t) \in I$, we have $a_2 \notin \shadz(t)$ by~\ref{item:instance:not_in_shadow}.
    By~\cref{prop:comparability_implies_shadow_containment}, $\shadz(z_t) \subset \shadz(t)$, and so, $a_2 \notin \shadz(z_t)$.
    Moreover, $z_1,b_1 \in Y(a_2)$.
    Thus, \cref{prop:left_is_preserved_from_y_to_z}.\ref{prop:left_is_preserved_from_y_to_z:left} gives that $z_1$ is left of~$z_t$.
    \qedhere

    \end{proofclaim}

    Let $W_1$ be an exposed witnessing path from $a_2$ to $z_1$ in~$P$ and let $W_3$ be an exposed witnessing path from $a_2$ to $z_3$ in~$P$.
    By~\cref{claim:prop:HIILR-switch-two:z-left-right} we have $z_1$ is left of $z_3$, thus the region $\calR = \calR(a_2,z_1,z_3,W_1,W_3)$ is well defined.

    \begin{claim}\label{claim:prop:HIILR-switch-two:a_1_in_IntR}
        $a_1 \in \Int \calR$.
    \end{claim}
    \begin{proofclaim}
        Since $z_1$ is left of $z_t$ and $z_t$ is left of $z_3$ (by~\cref{claim:prop:HIILR-switch-two:z-left-right})~\cref{cor:sandwitch_b_in_region}.\ref{cor:sandwitch_b_in_region:int} implies $z_t\in\Int \calR$.
        Suppose to the contrary that $a_1 \notin \Int \calR$ and let $U$ be an exposed witnessing path from $a_1$ to $z_t$ in~$P$.
        Since $z_t \in \Int \calR$ and $a_1 \notin \Int \calR$, there exists an element $u$ of $U$ with $u \in \partial \calR$.
        Since $U$ is an exposed path, $u \notin B$.
        If $u$ lies on the left side of $\calR$, then $a_1 \leq u < z_1 \leq b_1$ in~$P$, which is false.
        If $u$ lies on the right side of $\calR$, then $a_2 \leq u < z_t \leq t$ in~$P$, which is also false.
    \end{proofclaim}

    By~\ref{items:HIIL-shifted-edgeHII}, $((a_1,b_1),(a_2,t))$ is an edge in $\HII$ and by assumption $((a_2,b_2),(a_3,b_3))$ is an edge in $\HII$.
    Hence, by~\cref{prop:In-In_the_same_region}, $\calR(a_1) = \calR(a_2) = \calR(a_3)$.
    Let $z_0 = z_L(a_1) = z_L(a_2) = z_L(a_3)$.

    \begin{claim}\label{claim:prop:HIILR-switch-two:z_0_notin_R}
        $z_0 \notin \calR$.
    \end{claim}
    \begin{proofclaim}
        Note that $z_0 \not\leq z_1$ in $P$.
        Indeed, otherwise, $a_1 < z_L(a_1) = z_0 \leq z_1 \leq b_1$ in $P$.
        Since $z_0 = z_L(a_2)$, by~\Cref{prop:z_LR_are_outside}.\ref{prop:z_LR_are_outside:left}, $z_0 \notin \calR$, as desired.
    \end{proofclaim}

    Let $U$ and $V$ be exposed witnessing paths in $P$ from $a_1$ to $z_t$ and from $a_1$ to $z_0$, respectively. 
    By~\cref{claim:prop:HIILR-switch-two:z_0_notin_R,claim:prop:HIILR-switch-two:a_1_in_IntR}, $a_1 \in \Int \calR$ $z_0 \notin \calR$, hence, $V$ intersects $\partial \calR$.
    Since $V$ is exposed and $a_1 \parallel z_1$ in~$P$, $V$ intersects $W_3$.
    Let $v$ be an element of this intersection.
    Note that $v \notin B$.

    \begin{claim}\label{claim:z3-inZa1}
        $z_3 \in Z(a_1)$.
    \end{claim}
    \begin{proofclaim}
    The path $a_1[V]v[W_3]z_3$ is an exposed witnessing path in $P$, thus, it suffices to show that $a_1 \notin \shadz(z_3)$.
    Observe that none of the elements of the paths $a_2[W_3]v$ and $a_1[V]v$ are in $B$.
    Therefore, these paths do not intersect $\partial \shadz(z_3)$.
    In particular, either all $a_1$, $v$, and $a_2$ are in $\shadz(z_3)$ or none of them.
    Since $z_3 \in Z(a_2)$, the latter holds.
    We conclude that $a_1 \notin \shadz(z_3)$, and so, $z_3 \in Z(a_1)$.
    \end{proofclaim}

    Since $z_t$ is left of $z_3$ (by~\Cref{claim:prop:HIILR-switch-two:z-left-right}), $z_t \in Z(a_1)$, and $z_3 \in Z(a_1)$ (by~\Cref{claim:z3-inZa1}), we can define the region 
    \[\calS = \calR(a_1,z_t,z_3,U, a_1[V]v[W_3]z_3).\]

Since $z_t$ is left of $b_2$ and $b_2$ is left of $z_3$ (by~\cref{claim:prop:HIILR-switch-two:z-left-right}), by~\cref{cor:sandwitch_b_in_region}.\ref{cor:sandwitch_b_in_region:int}, $b_2 \in \Int \calS$.
Since $z_0$ is left of $z_t$ (by~\cref{claim:prop:HIILR-switch-two:z-left-right} and~\Cref{prop:left_porders_bs}), by~\cref{cor:sandwitch_b_in_region}.\ref{cor:sandwitch_b_in_region:out}, $z_0 \notin \calS$.
We consider two cases based on the location of $a_3$. 

First, assume that $a_3 \in \Int \calS$.\footnote{In fact, with more careful analysis, one can show that $a_3 \in \Int \calS$ leads to a contradiction.}
Consider a witnessing path $W$ from $a_3$ to $z_0$ in~$P$.
Since $a_3 \in \Int \calS$ and $z_0 \notin \calS$, the path $W$ has an element $w$ in $\partial\calS$.
If $w$ lies on the right side of $\calS$, then $a_3 \leq w \leq z_3 \leq b_3$ in~$P$, which is false.
Therefore, $w$ lies on the left side of $\calS$, which implies $a_3 \leq w \leq z_t \leq t$ in~$P$, as desired.

Finally, assume that $a_3 \notin \Int \calS$.
Consider a witnessing path $W$ from $a_3$ to $b_2$ in~$P$. 
Since $a_3\not\in\Int \calS$ and $b_2\in\Int\calS$, the path $W$ has an element $w$ in $\partial\calS$.
If $w$ lies on the right side of $\calS$, then $a_3\leq w \leq z_3\leq b_3$ in~$P$, which is false.
Therefore, $w$ lies on the left side of $\calS$, which implies $a_3\leq w \leq z_t \leq t$ in~$P$. 
This completes the proof.
\end{proof}

\begin{corollary}
    \label{cor:HIIL-HIIR-paths}
\hfill
  \begin{enumerate}
      \myitem{$(L)$}  For all positive integers $n$ and for each path $((a_1,b_1),\dots ,(a_n,b_n))$ in $\HIIL$, there exist $t_2,\dots,t_n \in B$ such that $((a_1,b_1),(a_2,t_2),\dots,(a_n,t_n))$ is a path in $\HII$. \label{cor:HIIL-HIIR-paths-left}
    \myitem{$(R)$} For all positive integers $n$ and for each path $((a_1,b_1),\dots ,(a_n,b_n))$ in $\HIIR$, there exist $s_1,\dots,s_{n-1} \in B$ such that $((a_1,s_1),\dots,(a_{n-1},s_{n-1}),(a_n,b_n))$ is a path in $\HII$. \label{cor:HIIL-HIIR-paths-right}
  \end{enumerate}
\end{corollary}
\begin{proof}
    We only prove~\ref{cor:HIIL-HIIR-paths-left} as the proof of~\ref{cor:HIIL-HIIR-paths-right} is symmetric.
    We proceed by induction on $n$.
    When $n=1$, the statement is vacuously satisfied.
    Suppose that $n \geq 2$ and let $((a_1,b_1),\dots ,(a_n,b_n))$ be a path in $\HIIL$.
    By induction, there exist $t_3, \dots t_n$ such that $((a_2,b_2),(a_3,t_3),\dots,(a_n,t_n))$ is a path in $\HII$. 
    If $((a_1,b_1),(a_2,b_2))$ is a cycle edge in $\HIIL$, then it suffices to set $t_2 = b_2$.
    Thus, assume that $((a_1,b_1),(a_2,b_2))$ is a shifted edge in $\HIIL$, say witnessed by $t \in B$.
    We claim that the assertion is satisfied for $t_2 = t$.
    By~\ref{items:HIIL-shifted-edgeHII}, $((a_1,b_1),(a_2,t))$ is an edge in $\HII$.
    Hence, when $n = 2$, the statement holds.
    When $n \geq 3$, by~\cref{prop:HIILR-switch-two}.\ref{prop:HIILR-switch-two-left}, $((a_2,t),(a_3,t_3))$ is an edge in $\HII$.
    This ends the proof.
\end{proof}

\begin{proof}[Proof of \Cref{prop:HIIL,prop:HIIR}]
    \Cref{cor:HIIL-HIIR-paths} implies that 
    \[\maxpath(\HIIL) \leq \maxpath(\HII) \ \text{ and } \ \maxpath(\HIIR) \leq \maxpath(\HII).\] 
    \Cref{prop:HII} states that $\maxpath(\HII) \leq \se_P(I)$, and so, we obtain
    \[\maxpath(\HIIL) \leq \se_P(I) \ \text{ and } \ \maxpath(\HIIR) \leq \se_P(I).\qedhere\]
\end{proof}

In the last part of this subsection, we study $\HIILR$.
In~\cref{fig:path-HIILR}, we show an example of a path in $\HIILR$.
Also, recall a schematic drawing of an edge in $\HIILR$ in~\Cref{fig:edges-HIILR}.
We aim to prove~\Cref{prop:HIILR}, namely, that $\maxpath(\HIILR) \leq \se_P(I)$.

\begin{figure}[tp]
  \begin{center}
    \includegraphics{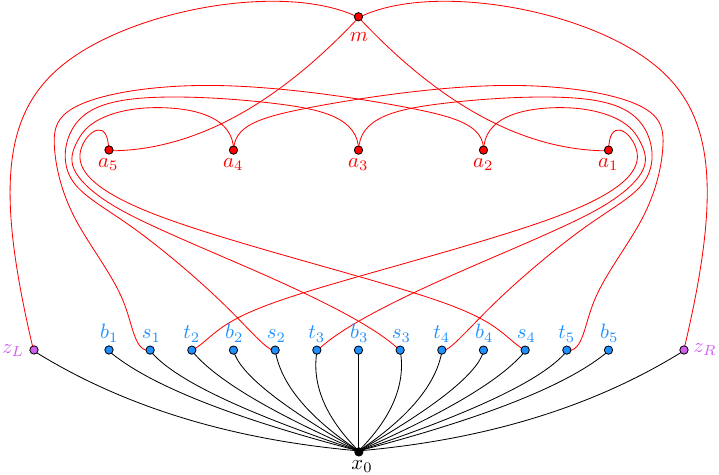}
  \end{center}
  \caption{
  $((a_1,b_1),(a_2,b_2),(a_3,b_3),(a_4,b_4),(a_5,b_5))$ is a path in $\HIILR$.
    For every $i \in [4]$, $(s_{i},t_{i+1})$ is a witness for $((a_{i},b_{i}),(a_{i+1},b_{i+1}))$.
    Note that $\{(a_1,s_1),(a_2,s_2),(a_3,s_3),(a_4,s_4)\}$ induces a standard example of order $4$ in~$P$.
  }
  \label{fig:path-HIILR}
\end{figure}

\begin{proposition}\label{prop:change_to_s_HIILR}
    Let $\sigma_{12} = ((a_1,b_1),(a_2,b_2))$ and $\sigma_{23} = ((a_2,b_2),(a_3,b_3))$ be two edges in $\HIILR$ witnessed by $(s_1,t_2)$ and $(s_2,t_3)$ respectively.
    Then, $\sigma = ((a_1,s_1),(a_2,s_2))$ is an edge in $\HII$.
\end{proposition}
\begin{proof}
See~\Cref{fig:proof-HIILR} for an illustration of the proof. 
By~\ref{items:HIILR-edgeHII}, $(a_1,s_1) \in I$ and $(a_2,s_2) \in I$.
By~\cref{obs:switch_bs_in_RPP}.\ref{obs:switch_bs_in_RPP:2in1} and \ref{obs:switch_bs_in_RPP:1in2}, $((a_1,s_1),(a_2,s_2))$ satisfies \ref{Lin} and \ref{Rin}.
By~\ref{items:HIILR-edgeHII} and~\ref{items:HIILR-b-left-d-t-left-b}, for each $i \in [2]$,
    \[\text{$b_{i}$ is left of $s_{i}$, $s_{i}$ is left of $t_{i+1}$, and $t_{i+1}$ is left of $b_{i+1}$.}\]
In particular, $s_1$ is left of $s_{2}$.
By~\ref{items:HIILR-edgeHII}, $a_{2} < s_1$ in~$P$.
Therefore, it suffices to show that $a_1 < s_2$ in~$P$.


Let $i \in [2]$.
By~\Cref{cor:regular-Y}, $t_{i+1} \in Y(a_{i})$ and  $s_{i} \in Y(a_{i+1})$.
Thus, we can fix $u_{i+1} \in Z(a_{i+1})$ and $v_{i} \in Z(a_i)$ such that $u_{i+1} \leq s_{i}$ and  $v_i \leq t_{i+1}$ in~$P$.
Additionally, let $u_1 = z_L(a_1)$ and $v_3 = z_R(a_3)$.
For every $i \in [3]$, let $U_{i}$ and $V_{i}$ be exposed witnessing paths in~$P$ from $a_{i}$ to $u_{i}$ and from $a_i$ to $v_i$ respectively.
By~\Cref{prop:z_L_b_z_R},
\begin{align}\label{eq:u-b-b-v}
    \text{$u_1$ is left of $b_1$ and $b_3$ is left of $v_3$.}
\end{align}

\begin{claim}\label{claim:HIILR-left-right}
    $b_i$ is left of $u_{i+1}$, 
    $u_{i+1}$ is left of $v_i$, and 
    $v_i$ is left of $b_{i+1}$,
    for every $i \in [2]$.
\end{claim}
\begin{proofclaim}
    Let $i\in[2]$.
    Since $(a_i,b_i),(a_{i+1},b_{i+1}) \in I$ and $b_i$ is left $b_{i+1}$, by~\Cref{prop:dangerous-implies-in-Y},
        \[\text{$a_{i} \notin \shadz(b_{i+1})$ and $a_{i+1} \notin \shadz(b_i)$.} \]
    Since $a_{i+1} \parallel b_i$ in $P$ (by~\ref{items:HIILR-a-parallel-b}), $u_{i+1},s_{i} \in Y(a_{i+1})$, $u_{i+1} \leq s_{i}$ in~$P$, and $b_{i}$ is left of $s_i$, by~\cref{prop:left_is_preserved_from_y_to_z}.\ref{prop:left_is_preserved_from_y_to_z:right}, $b_i$ is left of $u_{i+1}$.
    Since $a_i \parallel b_{i+1}$ in $P$ (by~\ref{items:HIILR-a-parallel-b}), $v_{i},t_{i+1} \in Y(a_{i})$, $v_{i} \leq t_{i+1}$ in~$P$, and $t_{i+1}$ is left of $b_{i+1}$, by~\cref{prop:left_is_preserved_from_y_to_z}.\ref{prop:left_is_preserved_from_y_to_z:left}, $v_{i}$ is left of $b_{i+1}$.

    It remains to show that $u_{i+1}$ is left of $v_i$.
    Since $(a_{i+1},t_{i+1}) \in I$, $a_{i+1} < u_{i+1} \leq s_i$ in~$P$, and $s_{i}$ is left of $t_{i+1}$ by~\cref{cor:left_is_preserved_from_y_to_z}.\ref{cor:left_is_preserved_from_y_to_z:left}, $u_{i+1}$ is left of $t_{i+1}$.
    Since $a_i \parallel s_i$ and $u_{i+1} \leq s_{i}$ in~$P$, we have $a_{i} \parallel u_{i+1}$ in~$P$.
    Since $(a_i,s_i) \in I$, we have $a_i \notin \shadz(s_i)$.
    By~\cref{prop:comparability_implies_shadow_containment}, $\shadz(u_{i+1}) \subset \shadz(s_{i})$, and so, $a_{i} \notin \shadz(u_{i+1})$.
    Moreover, $v_{i},t_{i+1} \in Y(a_{i})$.
    It follows that we can apply~\cref{prop:left_is_preserved_from_y_to_z}.\ref{prop:left_is_preserved_from_y_to_z:right} to obtain that $u_{i+1}$ is left of~$v_{i}$.
\end{proofclaim}

\begin{figure}[tp]
  \begin{center}
    \includegraphics{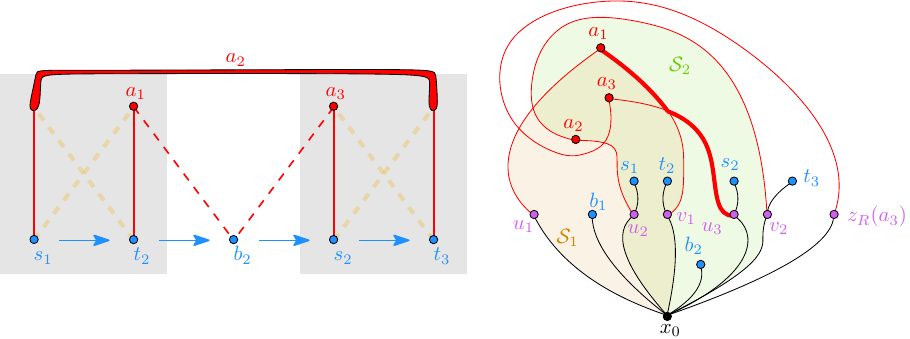}
  \end{center}
  \caption{
        An illustration of the proof of~\cref{prop:change_to_s_HIILR}.
        On the left-hand side, we show a schematic summary of the assumptions (drawing conventions are the same as in~\cref{fig:edges-HIIL-HIIR}).
        On the right-hand side, we illustrate the objects occurring in the proof of the proposition.
  }
  \label{fig:proof-HIILR}
\end{figure}

Since $u_1$ is left of $b_1$ (by~\eqref{eq:u-b-b-v}) and $b_1$ is left of $v_1$ (by~\cref{claim:HIILR-left-right}), we obtain that $u_1$ is left of $v_1$.
Moreover, by~\cref{claim:HIILR-left-right}, $u_2$ is left of $v_2$.
Thus, for each $i\in[2]$ we can define regions
\[
    \calS_i =\calR(a_i,u_i,v_i,U_i,V_i).
\]
\begin{claim}\label{claim:HIILR-1}
    $a_{i+1} \in \Int \calS_i$ and $U_{i+1} \subset \Int \calS_i$ for every $i \in [2]$.
\end{claim}
\begin{proofclaim}
    Let $i \in [2]$.
    By~\eqref{eq:u-b-b-v} and~\cref{claim:HIILR-left-right}, $u_i$ is left of $u_{i+1}$ and $u_{i+1}$ is left of $v_i$.
    Thus, by~\cref{cor:sandwitch_b_in_region}.\ref{cor:sandwitch_b_in_region:int}, $u_{i+1} \in \Int \calS_i$.
    Note that to prove the claim it suffices to show that $U_{i+1} \subset \Int \calS_i$.
    Since $u_{i+1} \in \Int \calS_i$, it is enough to show that $U_{i+1}$ is disjoint from $\partial\calS_i$.
    Suppose to the contrary that $w$ is an element in the intersection of $U_{i+1}$ and $\partial\calS_i$.
    Since $U_{i+1}$ is an exposed path, we have $w \notin B$, which implies that $w$ lies in $U_i \cup V_i$.
    This yields $a_i \leq w < u_{i+1} \leq s_i$ in~$P$, which is false.
\end{proofclaim}

\begin{claim}\label{claim:HIILR-2}
    $a_{3} \in \Int \calS_1$.
\end{claim}
\begin{proofclaim}
    Note that $v_{2}$ is left of $b_{3}$ (by~\cref{claim:HIILR-left-right}), and $b_{3}$ is left of $v_3$ (by~\cref{eq:u-b-b-v}).
    Hence, $v_{2}$ is left of $v_3$ (by~\cref{prop:left_porders_bs}), and so by~\cref{cor:sandwitch_b_in_region}.\ref{cor:sandwitch_b_in_region:out}, $v_3 \notin \calS_{2}$.
    We have $a_{3} \in \Int \calS_{2}$ (by~\cref{claim:HIILR-1}) and $v_3 \notin \calS_{2}$, hence, $V_3$ must intersect $\partial \calS_{2}$, say in $w$.
    Since $V_3$ is exposed, $w$ lies either in $U_{2}$ or in $V_{2}$.
    The latter is not possible as otherwise, $a_{3} \leq w < v_{2} \leq t_{3}$ in~$P$.
    Thus, $w$ lies in $U_{2}$.
    By~\cref{claim:HIILR-1}, $U_{2} \subset \Int \calS_1$, hence, $w \in \Int \calS_1$.
    We argue that $a_3[V_3]w \subset \Int \calS_1$, which suffices to conclude the proof of the claim.
    Otherwise, there exists $w'$ in $a_3[V_3]w$ such that $w'\in\partial \calS_1$.
    Moreover, $w' \notin B$ as $V_3$ is exposed.
    It follows that $a_1 \leq w' \leq w < u_{2}\leq s_1$ in~$P$, which is a contradiction.
\end{proofclaim}

Recall that we needed to prove that $a_1 < s_{2}$ in~$P$.
Since $v_1$ is left of $u_{3}$ (by \cref{claim:HIILR-left-right}), by~\cref{cor:sandwitch_b_in_region}.\ref{cor:sandwitch_b_in_region:out}, $u_{3} \notin \calS_{1}$.
However, by~\cref{claim:HIILR-2}, $a_{3} \in \Int \calS_1$, hence, $U_{3}$ intersects $\partial \calS_1$, say in $w$.
Since $U_{3}$ is exposed, $w$ lies in $U_1 \cup V_1$.
In particular, $a_1 \leq w < u_{3} \leq s_{2}$ in~$P$.
\end{proof}

\begin{proof}[Proof of~\Cref{prop:HIILR}]
If $\maxpath(\HIILR) \leq 1$, then the assertion is clear.
Thus, we can assume that $\maxpath(\HIILR) \geq 2$.
Let $((a_1,b_1),\dots,(a_n,b_n))$ be a path in $\HIILR$ where $n$ is an integer with $n \geq 2$.
For each $i\in[n-1]$, let $(s_i, t_{i+1})$ with $s_i,t_{i+1} \in B$ be a witness for the edge $((a_i,b_i),(a_{i+1},b_{i+1}))$ in $\HIILR$. 
By~\cref{prop:change_to_s_HIILR}, for each $i \in [n-2]$, $((a_i,s_i),(a_{i+1},s_{i+1}))$ is an edge in $\HII$.
Moreover, $((a_{n-1},s_{n-1}),(a_n,t_n))$ is an edge in $\HII$ by~\ref{items:HIILR-edgeHII}.
Therefore, $((a_1,s_1),\dots,(a_{n-1},s_{n-1}),(a_n,t_n))$ is a path in $\HII$, and so, by~\cref{prop:HII}, $\maxpath(\HIILR) \leq \maxpath(\HII) \leq \se_P(I)$, which completes the proof.
\end{proof}

\subsection{In-Out and Out-In oriented graphs}
In this subsection, we prove~\Cref{lemma:HIO}, which concerns $\HIO$ and $\HOI$.
We restate the lemma below.
Note that we give a proof only in the case of $\HIO$, and the proof for $\HOI$ is symmetric.
In~\Cref{fig:edges-HIO}, we gave a schematic drawing of edges in $\HIO$
See~\Cref{fig:edges-HIO-real} for a more precise illustration.
In~\Cref{fig:path-HIO}, we give an example of a path in $\HIO$.

\begin{figure}[tp]
  \begin{center}
    \includegraphics{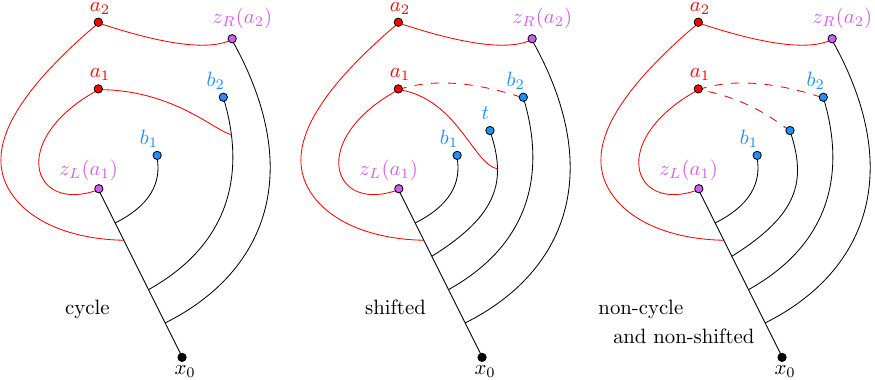}
  \end{center}
  \caption{
    $((a_1,b_1),(a_2,b_2))$ are edges in $\HIO$.
  }
  \label{fig:edges-HIO-real}
\end{figure}

\begin{figure}[tp]
  \begin{center}
    \includegraphics{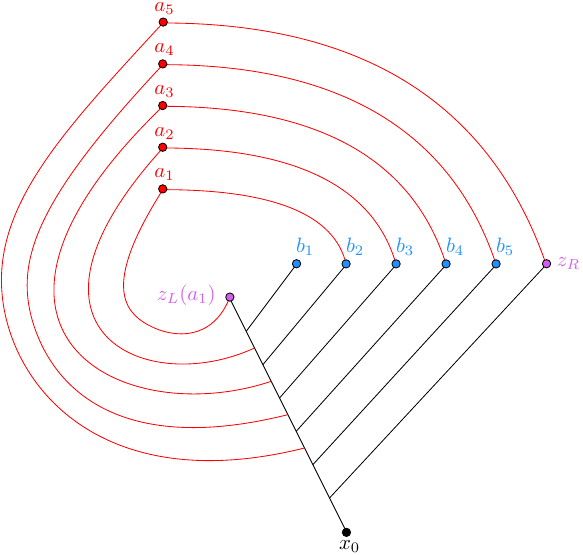}
  \end{center}
  \caption{
  $((a_1,b_1),(a_2,b_2),(a_3,b_3),(a_4,b_4),a_5,b_5))$ is a path in $\HIO$. 
  Note that in this poset there is no $S_3$.
  }
  \label{fig:path-HIO}
\end{figure}

\vbox{
\hmm*
}

\begin{proof}
We only prove the statement for $\HIO$ (i.e.~\ref{lemma:HIO:HIO}), the proof for $\HOI$ (i.e.~\ref{lemma:HIO:HOI}) is symmetric.
Let $((a,b),(a',b'))$ be an edge of weight $1$ in $\HIO$.
Clearly,
    \[\maxsw(\HIO,(a,b)) > \maxsw(\HIO,(a,b)).\]
It suffices to show that
    \begin{align}\label{eq:lem:HIO:old}
        \maxsw(\HIO,(a',b')) > \maxsw(\HIO,(a,b)) - m.
    \end{align}
See~\Cref{fig:lemma-HIO} for a high-level idea of the proof.

\begin{figure}[tp]
  \begin{center}
    \includegraphics{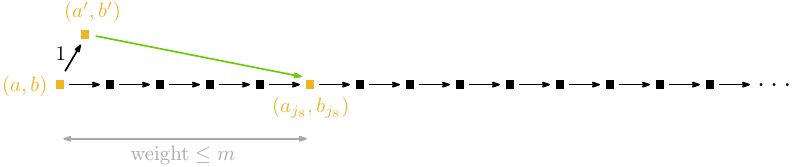}
  \end{center}
  \caption{
    We depict a high-level idea of the proof of~\Cref{lemma:HIO}.
    The arrows indicate edges in $\HIO$.
    We fix a heavy path starting in $(a,b)$.
    The main goal is to find a pair on this path that is \q{quite close in the path} to $(a,b)$, and there is an edge (in green) from $(a',b')$ to this pair.
    In the figure, this pair is called $(a_{j_8},b_{j_8})$, which is the notation following the proof.
  }
  \label{fig:lemma-HIO}
\end{figure}

If $\maxsw(\HIO,(a,b)) \leq m - 1$, then the assertion is clear. 
Thus, assume that $\maxsw(\HIO,(a,b)) \geq m$.
Let 
\[\text{$((a_1,b_1),\dots,(a_n,b_n))$ be a path in $\HIO$}\]
witnessing $\maxsw(\HIO,(a,b))$, that is, $(a_1,b_1) = (a,b)$ and the number of edges of weight one in this path is $\maxsw(\HIO,(a,b))$. 
Note that $n$ can be arbitrarily large.
For each $i \in [n-1]$, let $\sigma_i = ((a_i,b_i),(a_{i+1},b_{i+1}))$ and if $\sigma_i$ is a shifted edge in $\HIO$, then let $t_{i+1} \in B$ be a witness for $\sigma_i$.

For convenience, we define
    \[\spine = W_L(z_L(a_1)).\]
For every $d \in B$, define
    \[\pi(d) = \gce(\spine,W_L(d)).\]
Note that the element $\pi(d)$ has an equivalent description when $d$ is right of $z_L(a_1)$, see~\cref{claim:HIO:equivalent_pi_def}.

We inductively construct elements $u_1,\dots,u_{n-1}$ and witnessing paths $U_1,\dots,U_{n-1}$ in~$P$ such that for every $i \in [n-1]$,
\begin{enumorig}[label=(u\arabic*)]
    \item $u_1 = z_L(a_1)$, \label{items-u:u_1}
    \item $u_i \in Z(a_i)$,\label{items-u:Z}
    \item $u_i$ lies in $\spine$,\label{items-u:W_L}
    \item $u_i$ is left of $b_i$,\label{items-u:left-d}
    \item $u_{i} \leq \pi(b_{i-1})$ in~$P$ if $i > 1$,\label{items-u:u<pi}
    \item $U_i$ is an exposed witnessing path from $a_i$ to $u_i$ in~$P$,\label{items-u:path}
    \item $x_0[\spine]u_i[U_i]a_i$ is left of $\spine$ if $i > 1$.\label{items-u:path-left} 
\end{enumorig}
See \cref{fig:elements-u}.
Let $i \in [n-1]$.
Since $u_i$ and $\pi(b_i)$ both lie in $\spine$ (by~\ref{items-u:u_1} and~\ref{items-u:W_L}), the elements $u_i$ and $\pi(b_i)$ are comparable in~$P$. 
However, $u_i\leq \pi(b_i)$ in~$P$ leads to a contradiction: $a_i< u_i \leq \pi(b_i)\leq b_i$ in $P$ (the first comparability follows from~\ref{items-u:Z}). 
Thus,~\ref{items-u:u_1}--\ref{items-u:W_L} imply that for every $i \in [n-1]$,
\begin{enumorig}[label=(u\arabic*), resume]
    \item $\pi(b_i) < u_i$ in~$P$.\label{items-u:pi<u}
\end{enumorig}

\begin{figure}[tp]
  \begin{center}
    \includegraphics{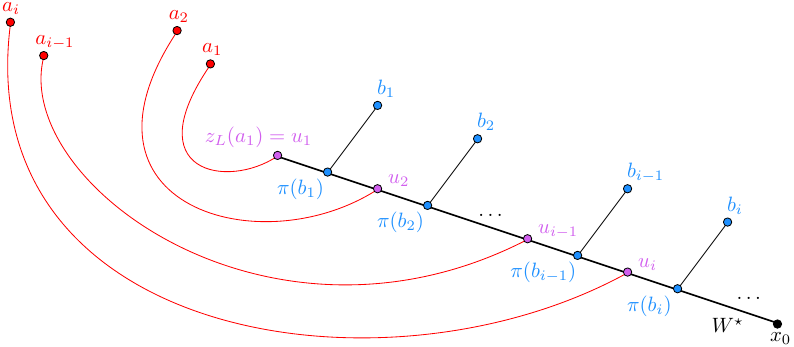}
  \end{center}
  \caption{
  An illustration of items~\ref{items-u:u_1}--\ref{items-u:pi<u}.
  }
  \label{fig:elements-u}
\end{figure}

We proceed with the construction of $u_1,\ldots, u_{n-1}$ and $U_1,\ldots,U_{n-1}$. 
First, let $u_1 = z_L(a_1)$ and let $U_1$ be an exposed witnessing path from $a_1$ to $u_1$ in~$P$.
Items~\ref{items-u:u_1}, \ref{items-u:Z}, \ref{items-u:W_L}, and \ref{items-u:path} are clearly satisfied for $i=1$ and items~\ref{items-u:u<pi} and~\ref{items-u:path-left} are vacuously true for $i=1$.
Since $(a_1,b_1) \in I$, by~\cref{prop:z_L_b_z_R}, $u_1 = z_L(a_1)$ is left of $b_1$, so~\ref{items-u:left-d} also holds.

Let $i \in [n-1]$ with $i > 1$ and assume that $u_1,\dots,u_{i-1}, U_1,\dots,U_{i-1}$ are already defined and satisfy all the items.
Let $u_i \in Z(a_{i})$ and let $U_i$ be an exposed witnessing path from $a_i$ to $u_i$ in~$P$ such that $(u_i,U_i)$ witnesses~\ref{Lout} for $\sigma_{i-1}$.
Immediately from this definition items~\ref{items-u:Z} and~\ref{items-u:path} hold.

Since $u_{i-1}\in Z(a_{i-1})$ (by~\ref{items-u:Z}), $u_{i-1}$ is left of $b_{i-1}$ (by~\ref{items-u:left-d}),
and $\sigma_{i-1}$ satisfies~\ref{Lout}, \cref{prop:PL2_PR2_more}.\ref{prop:PL2_PR2_more:left} implies that $u_i$ lies in $W_L(u_{i-1})$. 
Since $u_{i-1}$ lies in $\spine$ (by~\ref{items-u:W_L}), $W_L(u_{i-1})$ is a subpath of $\spine$. 
In particular, $u_i$ lies in $\spine$ so~\ref{items-u:W_L} holds for $i$.
Since $(a_i,b_i) \in I$, $a_i \leq u_i \leq b_{i-1}$ in~$P$, and $b_{i-1}$ is left of $b_i$, by~\cref{cor:left_is_preserved_from_y_to_z}.\ref{cor:left_is_preserved_from_y_to_z:left}, we obtain that $u_i$ is left of $b_i$, thus,~\ref{items-u:left-d} holds for $i$.


Recall that $(u_i,U_i)$ witnesses~\ref{Lout} for $\sigma_{i-1}$.
Since $u_i$ lies in $\spine$ (by~\ref{items-u:W_L}) and $u_i$ lies in $W_L(b_{i-1})$, we obtain that $u_i \leq \pi(b_{i-1})$ in~$P$, and so,~\ref{items-u:u<pi} holds for $i$. 
The element $u_i$ also lies in $W_L(z_L(a_{i-1}))$.
Thus, $x_0[W_L(z_L(a_{i-1}))]u_i = x_0[\spine]u_i$.
Finally, $M = x_0[W_L(z_L(a_{i-1}))]u_i[U_i]a_i = x_0[\spine]u_i[U_i]a_i$ is left of $W_L(z_L(a_{i-1}))$ (again by~\ref{Lout}).
Item~\ref{items-u:path-left} is equivalent to saying that $M$ is left of $\spine$.
Thus, we obtained~\ref{items-u:path-left} for $i$ when $i = 2$.
Next, we assume that $i > 2$.
Since $u_{i-1} \in Z(a_{i-1})$ (by~\ref{items-u:Z}), $W_L(z_L(a_{i-1}))$ is either left of $M' = x_0[W_L(u_{i-1})]  u_{i-1}[U_{i-1}]a_{i-1} = x_0[\spine]u_{i-1}[U_{i-1}]a_{i-1}$ or is a subpath of $M'$. 
If $W_L(z_L(a_{i-1}))$ is left of $M'$, then by transitivity, $M$ is left of $M'$.
Similarly, if $W_L(z_L(a_{i-1}))$ is a subpath of $M'$, then $M$ is left of $M'$ as well.
Item~\ref{items-u:path-left} for $i$ follows now from the fact that by~\ref{items-u:path-left} for $i-1$ (recall $i > 2$), $M'$ is left of $\spine$.
This completes the construction.

\begin{claim}\label{claim:HIO:u}
    For all $i,j \in [n-1]$ with $i < j$,
    \begin{enumerate}
        \item $u_j < u_i$ in~$P$, \label{claim:HIO:u:u_j<u_i}
        \item $W_L(u_j)$ is a proper subpath of $W_L(u_i)$, \label{claim:HIO:u:WL-supath-WL}
        \item $a_i \parallel u_j$ in~$P$. \label{claim:HIO:u:a_i-parallel-u_j}
    \end{enumerate}
\end{claim}
\begin{proofclaim}
    Let $i,j \in [n-1]$ with $i < j$.
    By~\ref{items-u:u<pi} and~\ref{items-u:pi<u}, $u_j \leq \pi(b_{j-1}) < u_{j-1}$ in~$P$, thus,~\ref{claim:HIO:u:u_j<u_i} follows from a simple induction.
    Item~\ref{claim:HIO:u:WL-supath-WL} follows from the fact that both $u_i$ and $u_j$ lie in $\spine$ (by~\ref{items-u:W_L}) and~\ref{claim:HIO:u:u_j<u_i}.
    For the proof of~\ref{claim:HIO:u:a_i-parallel-u_j} suppose to the contrary that $a_i < u_j$ or $u_j \leq a_i$ in~$P$.
    Note that the latter can not hold as $a_i \notin B$ and $u_j \in B$.
    Thus, we assume $a_i < u_j$ in~$P$.
    Then, $a_i < u_j \leq u_{i+1} \leq \pi(b_{i}) \leq b_i$ in~$P$, where the second inequality follows from~\ref{claim:HIO:u:u_j<u_i} and the third inequality follows from~\ref{items-u:pi<u}.
    Clearly, $a_i<b_i$ in~$P$ is a contradiction, which completes the proof.
\end{proofclaim}


The next claim helps us to define a handful of useful regions.

\begin{claim}\label{claim:HIO:Y}
    For all $i \in [n-1]$,
    \begin{enumerate}
        \item $a_i \notin \shadz(b_{i+1})$, \label{claim:HIO:Y:notin}
        \item $b_{i+1} \in Y(a_i)$ if $\sigma_i$ is a cycle edge, \label{claim:HIO:Y:cycle}
        \item $t_{i+1}\in Y(a_i)$ if $\sigma_i$ is a shifted edge. \label{claim:HIO:Y:matching}
    \end{enumerate}
\end{claim}
\begin{proofclaim}
    Let $i \in [n-1]$.
    Recall that $(a_i,b_i) \in I$.
    The element $b_i$ is left of $b_{i+1}$ and when $\sigma_i$ is a shifted edge, $b_i$ is left of $t_{i+1}$ (by~\ref{items:HIO-cycle}).
    Therefore, by~\cref{prop:dangerous-implies-in-Y}, $a_i \notin \shadz(b_{i+1})$ and when $\sigma_i$ is a shifted edge, $a_i \notin \shadz(t_{i+1})$.
    Items~\ref{claim:HIO:Y:cycle} and~\ref{claim:HIO:Y:matching} follow immediately since $a_i< b_{i+1}$ in~$P$ if $\sigma_i$ is a cycle edge and $a_i< t_{i+1}$ in~$P$ if $\sigma_i$ is a shifted edge.
\end{proofclaim}

Let
\[
E=\set{i \in [n-1] : \textrm{ $\sigma_i$ is of weight $1$}}.
\]
\Cref{claim:HIO:Y} allows for the following definition.
For each $i \in E$, let $v_i$ be an element of $P$ such that
\[
v_i \in Z(a_i) \text{ and }\begin{cases}
v_i \leq b_{i+1} \text{ in~$P$}&\textrm{if $\sigma_i$ is a cycle edge},\\
v_i \leq t_{i+1} \text{ in~$P$}&\textrm{if $\sigma_i$ is a shifted edge}.\\
\end{cases}
\]
Additionally, for every $i \in E$, let $V_i$ be an exposed witnessing path from $a_i$ to $v_i$ in~$P$.
See an illustration in~\Cref{fig:v-elements}.

\begin{claim}\label{claim:HIO:v}
For all $i \in E$,
\begin{enumerate}
    \item $b_i$ is left of $v_i$, \label{claim:HIO:v:b_left_v}
    \item $u_i$ is left of $v_i$, \label{claim:HIO:v:u_left_v}
    \item $v_i$ is left of $b_{i+1}$ if $\sigma_i$ is a shifted edge, \label{claim:HIO:v:v_left_b}
    \item $a_{i+1} \parallel v_i$ in~$P$ and $a_{i+1} \notin \shadz(v_i)$. \label{claim:HIO:v:a_parallel_v}
\end{enumerate}

\end{claim}
\begin{proofclaim}
Let $i \in E$.
If $\sigma_i$ is a cycle edge, then since $(a_i,b_i) \in I$, $a_i < v_i \leq b_{i+1}$ in~$P$, and $b_i$ is left of $b_{i+1}$, by~\cref{cor:left_is_preserved_from_y_to_z}.\ref{cor:left_is_preserved_from_y_to_z:right}, $b_i$ is left of $v_i$.
Assume that $\sigma_i$ is a shifted edge.
Since $(a_i,b_i) \in I$, $a_i < v_i \leq t_{i+1}$ in~$P$, and $b_i$ is left of $t_{i+1}$ (by~\ref{items:HIO-cycle}), by~\cref{cor:left_is_preserved_from_y_to_z}.\ref{cor:left_is_preserved_from_y_to_z:right}, $b_i$ is left of $v_i$.
This completes the proof of~\ref{claim:HIO:v:b_left_v}, and also~\ref{claim:HIO:v:u_left_v} since $u_i$ is left of $b_i$ by~\ref{items-u:left-d}.

For the proof of~\ref{claim:HIO:v:v_left_b}, assume that $\sigma_i$ is a shifted edge.
By~\ref{items:HIO-a-parallel-b}, $a_i \parallel b_{i+1}$ in~$P$ and by~\cref{claim:HIO:Y}.\ref{claim:HIO:Y:notin}, $a_i \notin \shadz(b_{i+1})$.
By~\cref{claim:HIO:Y}.\ref{claim:HIO:Y:matching}, $v_i,t_{i+1} \in Y(a_i)$.
Altogether, by~\cref{prop:left_is_preserved_from_y_to_z}.\ref{cor:left_is_preserved_from_y_to_z:left}, $v_i$ is left of $b_{i+1}$, which implies~\ref{claim:HIO:v:v_left_b}.

Let $d = b_{i+1}$ when $\sigma_i$ is a cycle edge and $d = t_{i+1}$ when $\sigma_i$ is a shifted edge.
Note that $(a_{i+1},d) \in I$ (by~\ref{items:HIO-cycle}) and $v_i \leq d$ in~$P$.
In particular, $a_{i+1} \parallel d$ in~$P$ and $a_{i+1} \notin \shadz(d)$ (by~\ref{item:instance:not_in_shadow}).
Since $a_{i+1} \notin B$ and $v_i \in B$, we have $v_i \not\leq a_{i+1}$ in~$P$.
Since $v_i \leq d$ in $P$, we also have $a_{i+1} \not\leq v_i$ in $P$.
Therefore, $a_{i+1} \parallel v_i$ in $P$.
Since $\shadz(v_i) \subset \shadz(d)$ (by~\cref{prop:comparability_implies_shadow_containment}), we also obtain $a_{i+1} \notin \shadz(v_i)$.
This completes the proof of~\ref{claim:HIO:v:a_parallel_v}.
\end{proofclaim}

The following is a simple corollary of \Cref{claim:HIO:v}.
\begin{claim}\label{claim:HIO:left-right}
    For all $i \in [n-1]$ and $j \in E$ with $i \leq j$,
    \begin{enumerate}
        \item $u_i$ is left of $b_j$,\label{claim:HIO:left-right:u_left_b}
        \item $b_i$ is left of $v_j$,\label{claim:HIO:left-right:b_left_v}
        \item $u_i$ is left of $v_j$.\label{claim:HIO:left-right:u_left_v}
    \end{enumerate}
\end{claim}
\begin{proofclaim}
    Let $i \in [n-1]$ and $j \in E$ with $i \leq j$.
    By~\ref{items-u:left-d}, $u_i$ is left of $b_i$ and by~\cref{claim:HIO:v}.\ref{claim:HIO:v:b_left_v}, $b_j$ is left of $v_j$.
    Additionally, $b_i$ is left of or is equal to $b_j$.
    Now, the transitivity of the \q{left of} relation (\cref{prop:left_porders_bs}) gives all the items.
\end{proofclaim}

By~\cref{claim:HIO:v}.\ref{claim:HIO:v:u_left_v}, $u_i$ is left of $v_i$ for every $i \in E$.
Moreover, by~\ref{items-u:Z} and the definition of $v_i$, we have $u_i,v_i \in Z(a_i)$.
Therefore, for every $i \in E$, we can define the region
    \[\calR_i = \calR(a_i,u_i,v_i,U_i,V_i).\]
Next, for each $i \in E$, let $q_i$ and $m_i$ be the lower-min and the upper-min of $\calR_i$ respectively, and let $\gamma_{L,i} = x_0[W_L(u_i)]u_i[U_i]m_i$ and $\gamma_{R,i} = x_0[W_R(v_i)]v_i[V_i]m_i$.
See~\Cref{fig:v-elements} again.

\begin{figure}[tp]
  \begin{center}
    \includegraphics{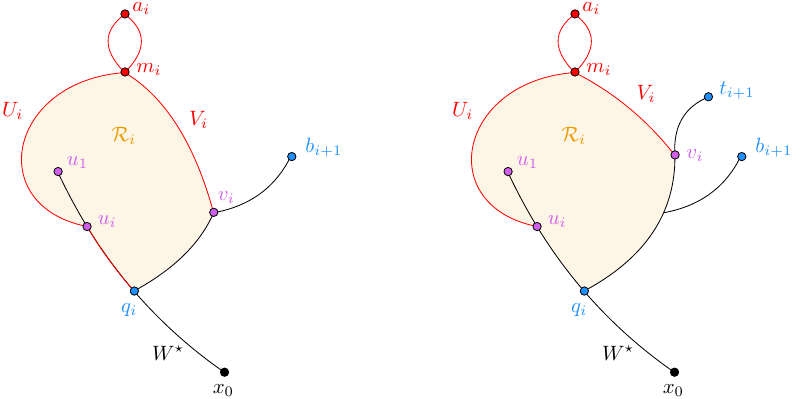}
  \end{center}
  \caption{
    The region $\calR_i$: on the left when $\sigma_i$ is a cycle edge, and on the right when $\sigma_i$ is a shifted edge.
    We prove in~\Cref{claim:HIO:u_i-in-R_j} that $u_i[\spine]u_1$ indeed lies in $\calR_i$ as depicted.
  }
  \label{fig:v-elements}
\end{figure}

\begin{claim}\label{claim:HIO:u_i-in-R_j}
    For all $j \in E$, we have $u_{j}[\spine]u_1\subseteq \Int\calR_j \cup \{u_j\}$. 
    In particular, for all $i\in[n-1]$ and $j \in E$ with $i<j$, we have $u_i\in\Int\calR_j$. 
\end{claim}
\begin{proofclaim}
    Let $j \in E$.
    Note that $q_j < u_j < u_1$ in $P$ (by~\cref{claim:HIO:u}.\ref{claim:HIO:u:u_j<u_i}).
    Since $u_j \in \calR_j$, by~\Cref{prop:shad-disjoint-from-calR}, $u_j \notin \shadz(q_j)$.
    In particular, $u_j < u_1$ and $u_j \not\leq q_j$ in~$P$, thus, by~\Cref{obs:equivalence_for_shadows}, we have $u_1 \notin \shadz(q_j)$.
    We argue that $\gamma_{L,j}$ is left of $\spine = W_L(u_1)$ and $W_R(u_1)$ is left of $\gamma_{R,j}$.
    This will show that $u_1 \in \Int \calR_j$ by~\cref{prop:b_relative_to_region}.\ref{prop:b_relative_to_region:item:inside}.
    The path $\gamma_{L,j}$ is left of $\spine$ by~\ref{items-u:path-left}.
    By~\cref{claim:HIO:left-right}.\ref{claim:HIO:left-right:u_left_v}, $u_1$ is left of $v_j$, thus, $W_R(u_1)$ is left of $W_R(v_j)$, and so, $W_R(u_1)$ is left of $\gamma_{R,j}$.
    As noted before, we obtain that $u_1 \in \Int \calR_j$.
    Consider the path $u_j[\spine]u_1$.
    This path is disjoint from $q_j[\gamma_{R,j}]v_j$ as $u_j \parallel v_j$ in~$P$ (by~\Cref{claim:HIO:left-right}.\ref{claim:HIO:left-right:u_left_v}). 
    Moreover, this path does not intersect $q_j[\gamma_{L,j}]u_j$
    in any element other than $u_j$ as this would yield a directed cycle in~$P$.
    Since all elements of $u_j[\spine]u_1$ are in $B$, it follows that $W \subset \Int \calR_j \cup \{u_j\}$, as desired.
    The \q{in particular} statement follows since for all $i \in [n-1]$ with $i < j$, the element $u_i$ lies in $u_j[\spine]u_1$ and $u_i \neq u_j$ since $u_i$ lies in $\spine$ (by~\ref{items-u:W_L}) and $u_j<u_i$ in $P$ (by~\cref{claim:HIO:u}.\ref{claim:HIO:u:u_j<u_i}).
    \end{proofclaim}

In the next claim, we show that for all $i,j \in [n-1]$ with $i < j$, if $\sigma_i$ is of weight $1$, then either $v_j \in \Int \calR_i$ or $v_j \notin \calR_i$, in other words, we show that $v_j\not\in\partial\calR_i$.
Later, we use this to define a $2$-coloring of such pairs of indices.
This coloring will enable us to break down the problem into two subproblems with stronger assumptions.

\begin{claim}\label{claim:HIO:classification}
    For all $i,j \in E$ with $i < j$,
    \begin{enumerate}
        \item $v_j \not\leq v_i$ in~$P$; \label{claim:HIO:classification:v_j_notleq_v_i}
        \item if $\sigma_i$ is a shifted edge, then $v_j \notin \calR_i$;\label{claim:HIO:classification:matching}
        \item if $\sigma_i$ is a cycle edge, then $v_j \notin \partial\calR_i$; \label{claim:HIO:classification:cycle-partial} 
        \item if $\sigma_i$ is a cycle edge and $v_j \in \Int \calR_i$, then $v_i$ lies in $W_R(v_j)$; \label{claim:HIO:classification:cycle-int-w}
        \item if $\sigma_i$ is a cycle edge and $v_j \in \Int \calR_i$, then $v_i$ lies in $W_R(b_{i+1})$. \label{claim:HIO:classification:cycle-int-b}
    \end{enumerate}
\end{claim}
\begin{proofclaim}
    See~\cref{fig:vj-in-or-out} for an illustration.
    Let $i,j \in E$ with $i < j$.
    Suppose to the contrary that $v_j \leq v_i$ in~$P$.
    Let $d_{i+1} = b_{i+1}$ when $\sigma_i$ is a cycle edge and $d_{i+1} = t_{i+1}$ when $\sigma_i$ is a shifted edge. 
    If $i + 1 = j$ then  
    $a_{i+1} < v_{i+1} \leq v_i \leq d_{i+1}$ in~$P$, which is a clear contradiction. 
    Thus, we assume $i+1<j$.
    Now we have that $d_{i+1}$ is left of $b_j$ (by~\ref{items:HIO-t-left-b} when $d_{i+1} = t_{i+1}$).
    Since $v_j \leq v_i$ in~$P$, we have, $v_j \in \shadz(v_i) \subset \shadz(d_{i+1})$ (by~\cref{prop:comparability_implies_shadow_containment}).
    By~\cref{claim:HIO:v}.\ref{claim:HIO:v:b_left_v}, $b_j$ is left of $v_j$, thus, $v_j \notin \shadz(b_j)$. 
    Therefore, we can apply~\cref{prop:in_one_shad_but_not_other_then_left}.\ref{prop:in_one_shad_but_not_other_then_left:left}, to obtain that $v_j$ is left of $b_j$, which is again a contradiction that completes the proof of~\ref{claim:HIO:classification:v_j_notleq_v_i}.

    For the proof of~\ref{claim:HIO:classification:matching} observe that if $\sigma_i$ is a shifted edge, then $v_i$ is left of $b_{i+1}$ (by~\cref{claim:HIO:v}.\ref{claim:HIO:v:v_left_b}), $b_{i+1}$ is either left of $b_j$ or equal to $b_j$, and $b_j$ is left of $v_j$ (by~\cref{claim:HIO:v}.\ref{claim:HIO:v:b_left_v}), hence, $v_i$ is left of $v_j$.
    Therefore, by~\cref{cor:sandwitch_b_in_region}.\ref{cor:sandwitch_b_in_region:out}, $v_j \notin \calR_i$.

    From now on, assume that $\sigma_i$ is a cycle edge.
    All elements of $\partial\calR_i$ that are in $B$ are in $W_L(u_i) \cup W_R(v_i)$.
    Since $u_i \parallel v_j$ in~$P$ (by~\cref{claim:HIO:left-right}.\ref{claim:HIO:left-right:u_left_v}), $v_j$ does not lie in $W_L(u_i)$.
    By~\ref{claim:HIO:classification:v_j_notleq_v_i}, $v_j$ does not lie in $W_R(v_i)$.
    This completes the proof of~\ref{claim:HIO:classification:cycle-partial}.

    \begin{figure}[tp]
      \begin{center}
        \includegraphics{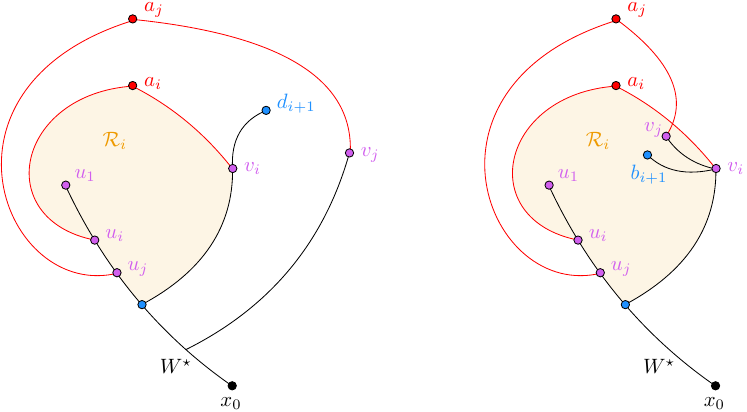}
      \end{center}
      \caption{
        An illustration of the statement of~\Cref{claim:HIO:classification}.
        Note that both cases $v_j \notin \calR_i$ and $v_j \in \calR_i$ are possible.
        However, the latter is possible only in the case of a cycle edge. 
      }
      \label{fig:vj-in-or-out}
    \end{figure}

    From now on, assume also that $v_j \in \Int \calR_i$.
    Let $w = \gce(W_R(v_j),\gamma_{R,i})$.
    In particular, $w$ lies in $W_R(v_i)$.
    If $w = v_i$, then the assertion holds.
    Thus, suppose to the contrary that $w < v_i$ in~$P$.
    Since $v_j \in \Int \calR_i$, by~\cref{prop:b_relative_to_region}.\ref{prop:b_relative_to_region:item:inside}, $W_R(v_j)$ is left of $\gamma_{R,i}$, and so, of $W_R(v_i)$.
    Let $W$ be a witnessing path from $v_i$ to $b_{i+1}$ in~$P$.
    Note that $x_0[W_R(v_i)]v_i[W]b_{i+1}$ is either left of or equal to $W_R(b_{i+1})$.
    It follows that $W_R(v_j)$ is left of $W_R(b_{i+1})$.
    However, $b_{i+1}$ is left of $v_j$ by~\cref{claim:HIO:left-right}.\ref{claim:HIO:left-right:b_left_v}, thus, we obtained a contradiction.
    This proves~\ref{claim:HIO:classification:cycle-int-w}.

    The path $x_0[W_R(v_i)]v_i[W]b_{i+1}$ is either left of or equal to $W_R(b_{i+1})$.
    Observe that either $v_i$ lies in $W_R(b_{i+1})$ or $W_R(v_i)$ is left of $W_R(b_{i+1})$.
    Thus, suppose the latter.
    Recall that $b_{i+1}$ is left of $v_j$, and in particular, $W_R(b_{i+1})$ is left of $W_R(v_j)$.
    We obtain that $W_R(v_i)$ is left of $W_R(v_j)$, which is false by~\ref{claim:HIO:classification:cycle-int-w}.
    This completes the proof of~\ref{claim:HIO:classification:cycle-int-b}.
\end{proofclaim}

Recall that $|E| \geq m$ and let $E_0$ be the $m$ least integers of $E$.
Let $i,j \in E_0$ with $i < j$.
By~\cref{claim:HIO:classification}, either $v_j \in \Int \calR_i$ or $v_j \notin \calR_i$. 
See~\Cref{fig:vj-in-or-out} again.
Let
\[
\phi(i,j) = \begin{cases}
                \colorin &\textrm{if $v_j \in \Int \calR_i$,}\\
                \colorout &\textrm{if $v_j \notin \calR_i$.} 
            \end{cases}
\]
The mapping $\phi$ can be seen as a $2$-coloring  of $\binom{E_0}{2}$.
Consider two auxiliary digraphs $\Hin$ and $\Hout$ with vertex sets $E_0$.
For all $i,j \in E_0$, $(i,j)$ is an edge in $\Hin$ if $i < j$ and $\phi(i,j) = \colorin$, and  $(i,j)$ is an edge in $\Hout$ if $i < j$ and $\phi(i,j) = \colorout$.
Note that $\Hin$ and $\Hout$ are acyclic.
Next, for every $i \in E_0$, let
    \[\phi'(i) = (\maxsp(\Hin, i), \maxsp(\Hout, i)).\]
We claim that $\phi'$ is injective.
Indeed, let $i,j \in E_0$ with $i < j$.
If $\phi(i,j) = \colorin$, then for every path $j_1\cdots j_m$ in $\Hin$ with $j_1 = j$, $i j_1 \cdots j_m$ is also a path in $\Hin$.
It follows that $\maxsp(\Hin, i) > \maxsp(\Hin, j)$.
Analogously, if $\phi(i,j) = \colorout$, we obtain $\maxsp(\Hout, i) > \maxsp(\Hout, j)$.
This shows that $\phi'$ is injective.

Since $\phi'$ is an injective function from $E_0$ to $[\maxpath(\Hin)]\times[\maxpath(\Hout)]$ we have 
\begin{align*}
    |E_0| \leq \maxpath(\Hin) \cdot \maxpath(\Hout). 
\end{align*}
In the next claim, we show that $\maxpath(\Hin) \leq \se_P(I)$.
This together with $m = |E_0|$ and the inequality above imply
\begin{align}\label{eq:max-path-Hout}
     2(2\se_P(I)+6) = m \slash \se_P(I) \leq \maxpath(\Hout).
\end{align}

\begin{figure}[tp]
  \begin{center}
    \includegraphics{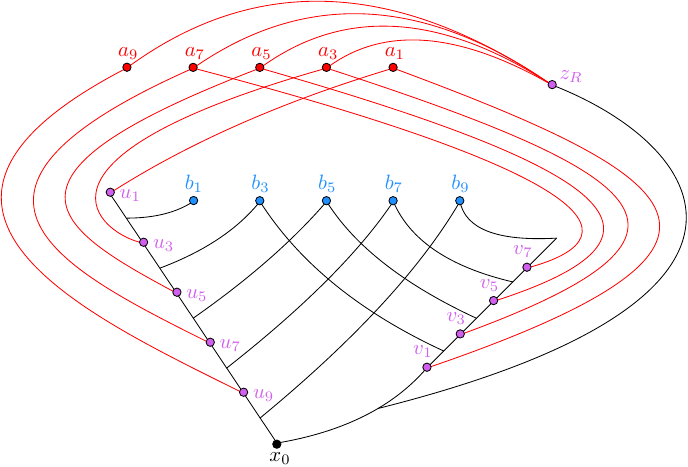}
  \end{center}
  \caption{
  $13579$ is a path in $\Hin$.
  We do not draw other elements of the poset to keep the figure clearer.
  Note that $\{(a_i,b_i) : i \in \{1,3,5,7,9\}\}$ induces a standard example in $P$.
  }
  \label{fig:Hin-path}
\end{figure}

\begin{claim}\label{claim:HIO:maxpath_out_leq_se}
    $\maxpath(\Hin) \leq \se_P(I)$.
\end{claim}
\begin{proofclaim}
    Let $i_1\cdots i_p$ be a path in $\Hin$. 
    An example of such a path is shown in~\cref{fig:Hin-path}.
    We argue that $\set{(a_{i_j},b_{i_j}) : j\in[p]}$ induces a standard example in~$P$, i.e., $a_{i_k} < b_{i_\ell}$ in~$P$ for all distinct $k,\ell\in[p]$.
    This will show that $p \leq \se_P(I)$
    By~\cref{claim:HIO:u}.\ref{claim:HIO:u:u_j<u_i} and~\ref{items-u:u<pi}, for all $k,\ell \in [p]$ with $k < \ell$, we have $a_{i_\ell} < u_{i_\ell} \leq u_{i_k+1} \leq \pi(b_{i_k}) \leq b_{i_k}$ in~$P$.
    It remains to argue that 
    $a_{i_k} < b_{i_\ell}$ in~$P$ for all $k,\ell \in [p]$ with $k < \ell$.

    Let $j \in [p-1]$. 
    Since $(i_j,i_{j+1})$ is an edge in $\Hin$, we have $v_{i_{j+1}} \in \Int \calR_{i_j}$.
    By~\cref{claim:HIO:classification}.\ref{claim:HIO:classification:matching} and since $i_j \in E$ (that is, $\sigma_{i_j}$ is of weight $1$), $\sigma_{i_j}$ is a cycle edge.
    Therefore, by~\cref{claim:HIO:classification}.\ref{claim:HIO:classification:cycle-int-w}, $v_{i_j}$ lies in $W_R(v_{i_{j+1}})$.
    In other words, $W_R(v_{i_j})$ is a subpath of $W_R(v_{i_{j+1}})$.
    Now, let $k,\ell \in [p]$ with $k < \ell$.
    By transitivity of \q{being a subpath} relation, $W_R(v_{i_k})$ is a subpath of $W_R(v_{i_{\ell}})$, and in particular, $v_{i_k}$ lies in $W_R(v_{i_{\ell}})$.
    By~\cref{claim:HIO:classification}.\ref{claim:HIO:classification:cycle-int-b}, $v_{i_k}$ lies in $W_R(b_{i_k+1})$ and $v_{i_\ell}$ lies in $W_R(b_{i_\ell+1})$.
    If $i_k+1 = i_\ell$, we obtain $a_{i_k} < v_{i_k} \leq b_{i_k+1} = b_{i_\ell}$ in~$P$, which yields the assertion, thus, suppose that $i_k+1 < i_\ell$. 
    Since $v_{i_k}$ lies in $W_R(v_{i_{\ell}})$ and $v_{i_\ell}$ lies in $W_R(b_{i_\ell+1})$, $v_{i_k}$ lies in $W_R(b_{i_\ell+1})$.
    Consider the paths $W_R(b_{i_k+1})$, $W_R(b_{i_\ell})$, and $W_R(b_{i_\ell+1})$.
    They are pairwise $x_0$-consistent (by~\cref{prop:W-consistent}.\ref{prop:W-consistent:right}).
    Moreover, since $b_{i_k+1}$ is left of $b_{i_\ell}$ and $b_{i_\ell}$ is left of $b_{i_\ell+1}$, we have that $W_R(b_{i_k+1})$ is left of $W_R(b_{i_\ell})$ and $W_R(b_{i_\ell})$ is left of $W_R(b_{i_\ell+1})$.
    We argued that $v_{i_k}$ lies in $W_R(b_{i_k+1})$ and $W_R(b_{i_\ell+1})$.
    Therefore, by~\cref{prop:sandwiched_paths}, $v_{i_k}$ lies in $W_R(b_{i_\ell})$.
    This implies that $a_{i_k} < v_{i_k} \leq b_{i_\ell}$ in~$P$, which ends the proof.
\end{proofclaim}

In the next part of the proof, we assign a region to each edge in $\Hout$.
These regions are helpful in studying the structure of paths in $\Hout$.
Let $(i,j)$ be an edge in $\Hout$.\footnote{In fact, for many statements and definitions we only need $i,j \in E$ with $i < j$. However, there is no point in such generality. We mark the moment when we use the fact that $(i,j)$ is an edge in $\Hout$ with another footnote.}
In the next paragraph, we define $v_{i,j}$ and a witnessing path $V_{i,j}$ from $a_{i+1}$ to $v_{i,j}$ in~$P$.
We split the definition into cases, where the case with a higher number excludes cases with lower numbers.
See~\Cref{fig:def-v-ij}.

Case 1: $j = i+1$.

We set $v_{i,j} = v_{i+1} = v_j$ and $V_{i,j} = V_{i+1}=V_{j}$.

Case 2: $U_{i+1} \cap V_j \neq \emptyset$.

Let $v$ be the first element of $u_{i+1}[U_{i+1}]a_{i+1}$ in $V_j$.
We set $v_{i,j} = v_j$ and $V_{i,j} = a_{i+1}[U_{i+1}]v[V_{j}]v_j$.

Case 3: $a_{i+1}[M_R(a_{i+1})]z_R(a_{i+1}) \cap V_j \neq \emptyset$.

Let $v$ be the first element of $a_{i+1}[M_R(a_{i+1})]z_R(a_{i+1})$ in $V_j$.
We set $v_{i,j} = v_j$ and $V_{i,j} = a_{i+1}[M_R(a_{i+1})]v[V_{j}]v_j$.

Case 4: otherwise.

We set $v_{i,j} = z_R(a_{i+1})$ and $V_{i,j} = a_{i+1}[M_R(a_{i+1})]z_R(a_{i+1})$.

This completes the definition of $v_{i,j}$ and $V_{i,j}$.
Note that $v_{i,j} \in \{v_j,z_R(a_{i+1})\}$ and $V_{i,j}$ is an exposed witnessing path from $a_{i+1}$ to $v_{i,j}$ in~$P$.
Moreover, we have the following.
\begin{claim}\label{claim:HIO:u_left_v}
    Let $(i,j)$ be an edge in $\Hout$.
    \begin{enumerate}
        \item $u_i$ is left of $v_{i,j}$, \label{claim:HIO:u_left_v:u_i}
        \item $u_{i+1}$ is left of $v_{i,j}$, \label{claim:HIO:u_left_v:u_i+1}
        \item $v_{i,j} \in Z(a_{i+1})$. \label{claim:HIO:u_left_v:Z}
    \end{enumerate}
\end{claim}
\begin{proofclaim}
    For the proofs of~\ref{claim:HIO:u_left_v:u_i} and~\ref{claim:HIO:u_left_v:u_i+1}, let $k \in \{i,i+1\}$.
    We have $u_{k}$ left of $v_j$ by~\cref{claim:HIO:left-right}.\ref{claim:HIO:left-right:u_left_v}.
    Furthermore, $b_{k}$ is left of or equal to $b_{i+1}$ and $b_{i+1}$ is left of $z_R(a_{i+1})$ (by~\cref{prop:z_L_b_z_R}). Since $u_k$ is left of $b_k$ (by~\ref{items-u:left-d}), $u_{k}$ is left of $z_R(a_{i+1})$, as desired.
    For the proof of~\ref{claim:HIO:u_left_v:Z}, note that $z_R(a_{i+1}) \in Z(a_{i+1})$ by definition, and so, we assume that $v_{i,j} = v_j$.
    Since $(a_{i+1},b_{i+1}) \in I$ and $b_{i+1}$ is left of $v_j$ (by~\cref{claim:HIO:left-right}.\ref{claim:HIO:left-right:b_left_v}), by~\cref{prop:dangerous-implies-in-Y}, $a_{i+1} \notin \shadz(v_j)$.
    Finally, recall that $V_{i,j}$ is an exposed witnessing path from $a_{i+1}$ to $v_{j}$ in~$P$, and so, $v_{j} \in Z(a_{i+1})$, as desired.
\end{proofclaim}
Let $(i,j)$ be an edge in $\Hout$.
Recall that $u_{i+1}$ is left of $v_{i,j}$ (by~\cref{claim:HIO:u_left_v}.\ref{claim:HIO:u_left_v:u_i+1}), 
$u_{i+1}\in Z(a_{i+1})$ (by \ref{items-u:Z}), and 
$v_{i,j} \in Z(a_{i+1})$ (by~\cref{claim:HIO:u_left_v}.\ref{claim:HIO:u_left_v:Z}). 
Therefore, we can define 
    \[\calR_{i,j} = \calR(a_{i+1},u_{i+1},v_{i,j},U_{i+1},V_{i,j}).\]
See \Cref{fig:def-v-ij} again.
Let $q_{i,j}$ and $m_{i,j}$ be the lower-min and the upper-min of $\calR_{i,j}$.
The goal of the next part of this section is to show that if $(i,j)$ is an edge in $\Hout$, then $\calR_i \subset \calR_{i,j} \subset \calR_j$.

\begin{figure}[tp]
  \begin{center}
    \includegraphics{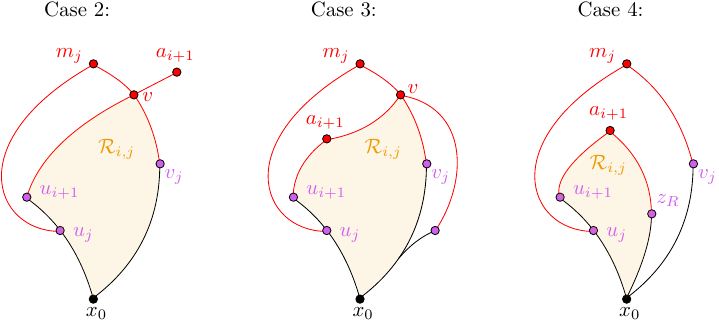}
  \end{center}
  \caption{
  Illustration of some cases of the definition of $\calR_{i,j}$.
  }
  \label{fig:def-v-ij}
\end{figure}

\begin{claim}\label{claim:HIO:R_ij_subset_R_j}
    Let $(i,j)$ be an edge in $\Hout$.
    Then,
        \[\calR_{i,j} \subset \calR_j.\]
\end{claim}
\begin{proofclaim}
    If $j=i+1$, then $\calR_{i,j}=\calR_j$ and the claim holds.
    Thus, we assume $j > i+1$. 
    In order to prove the claim, we show that $\partial \calR_{i,j} \subset \calR_j$, which suffices by~\cref{obs:region_containment}.
    Furthermore, if $u_{i+1}[U_{i+1}]m_{i,j}[V_{i,j}]v_{i,j} \subset \calR_j$, then in particular, $u_{i+1},v_{i,j} \in \cgR_j$, so by~\cref{prop:uv_in_R_implies_path}, $u_{i+1}[W_L(u_{i+1})]q_{i,j}[W_R(v_{i,j})]v_{i,j} \subset \cgR_j$, and therefore, $\partial\calR_{i,j} \subset \calR_j$.
    Let 
    \[M = u_{i+1}[U_{i+1}]m_{i,j}[V_{i,j}]v_{i,j}.\]
    As discussed, it suffices to argue that 
        \[M \subset \calR_j.\]
    We split the rest of the proof into cases depending on the applied case in the definition of $\calR_{i,j}$.
    Recall that we assumed $j>i+1$ so Case 1 does not hold.

    Case 2.
    We have
        \[M = u_{i+1}[U_{i+1}]v[V_j]v_j,\]
    where $v$ is the first element of $u_{i+1}[U_{i+1}]a_{i+1}$ in $V_j$.
    Note that $m_j < v$ in~$P$ as otherwise, we have $a_{i+1}<v\leq m_j < u_j$ in~$P$, which contradicts~\cref{claim:HIO:u}.\ref{claim:HIO:u:a_i-parallel-u_j}.
        Thus, $v[V_j]v_j \subset \partial \calR_j$, and so, it suffices to show that
    \[
    \text{$u_{i+1}[U_{i+1}]v \subset\calR_j$.}
    \]
    By~\cref{claim:HIO:u_i-in-R_j}, $u_{i+1} \in \Int \calR_j$, thus, all we need to show is that 
    $u_{i+1}[U_{i+1}]v$ intersects with $\partial\calR_j$ only in $v$. 
    By~\cref{claim:HIO:u}.\ref{claim:HIO:u:a_i-parallel-u_j}, $a_{i+1} \parallel u_j$ in~$P$, hence, $U_{i+1}$ is disjoint from $U_j$.
    Since $U_{i+1}$ is an exposed witnessing path in $P$ and $u_{i+1} \in \Int \calR_j$, $U_{i+1}$ is disjoint from $u_j[W_L(u_j)]q_j[W_R(v_j)]v_j$.
    This and the definition of $v$ imply that the only common element of $u_{i+1}[U_{i+1}]v$ and $\partial\calR_j$ is $v$, as desired.

    Case 3.
    We have
        \[M = u_{i+1}[U_{i+1}]m[M_R(a_{i+1})]v[V_j]v_j,\]
    where $v$ is the first element of $a_{i+1}[M_R(a_{i+1})]z_R(a_{i+1})$ in $V_j$ and $m$ is the last common element of $a_{i+1}[M_R(a_{i+1})]z_R(a_{i+1})$ and $a_{i+1}[U_{i+1}]u_{i+1}$.
    Note that $m_j < v$ in~$P$ as otherwise, we have $a_{i+1}<v\leq m_j < u_j$ in~$P$, which contradicts~\cref{claim:HIO:u}.\ref{claim:HIO:u:a_i-parallel-u_j}.
    Thus, $v[V_j]v_j \subset \partial \calR_j$, and so, it suffices to show that
    \[
    \text{$W = u_{i+1}[U_{i+1}]m[M_R(a_{i+1})]v\subset\calR_j$.}
    \]
    By~\cref{claim:HIO:u_i-in-R_j}, $u_{i+1} \in \Int \calR_j$, thus, all we need to show is that 
    $W$ intersects with $\partial\calR_j$ only in~$v$. 
    By~\cref{claim:HIO:u}.\ref{claim:HIO:u:a_i-parallel-u_j}, $a_{i+1} \parallel u_j$ in~$P$, hence, $W$ is disjoint from $U_j$.
    Since the only element of $W$ in $B$ is $u_{i+1}$ and $u_{i+1} \in \Int \calR_j$, $W$ is disjoint from $u_j[W_L(u_j)]q_j[W_R(v_j)]v_j$.
    The path $[U_{i+1}]$ is disjoint from $V_j$ as Case 2 does not hold.
    All this and the definition of $v$ imply that the only common element of $W$ and $\partial\calR_j$ is $v$, as desired.

    Case 4.
    We have
        \[M = u_{i+1}[U_{i+1}]m[M_R(a_{i+1})]z_R(a_{i+1}),\]
    where $m$ is the maximal in~$P$ common element of $a_{i+1}[M_R(a_{i+1})]z_R(a_{i+1})$ and $U_{i+1}$.
    By~\cref{claim:HIO:u_i-in-R_j}, $u_{i+1} \in \Int \calR_j$, thus, all we need to show is that 
    $M$ intersects with $\partial\calR_j$ only in $z_R(a_{i+1})$. 
    Since Cases 2 and 3 do not hold and by~\cref{claim:HIO:u}.\ref{claim:HIO:u:a_i-parallel-u_j}, $M$ is disjoint from $u_j[U_j]m_j[V_j]v_j$.
    It follows that $M$ can intersect $\partial \calR_j$ only in $u_j[W_L(u_j)]q_j[W_R(v_j)]v_j$.
    Since $M$ is the union of two exposed witnessing paths in $P$ and $u_{i+1} \in \Int \calR_j$, this intersection may occur only in $z_R(a_{i+1})$, as desired.
    This completes the proof.
\end{proofclaim}

\begin{claim} \label{claim:HIO:R_i_subset_R_j_helpers}
    Let $(i,j)$ be an edge in $\Hout$.
    We have
    \begin{enumerate}
        \item $\gamma_{R,i}$ is left of $W_R(v_{i,j})$,\label{claim:HIO:R_i_subset_R_j_helpers:gamma_left_W_R}
        \item $u_i \in \Int\calR_{i,j}$.\label{claim:HIO:R_i_subset_R_j_helpers:first_segment}
    \end{enumerate}
\end{claim}
\begin{proofclaim}
    In order to prove~\ref{claim:HIO:R_i_subset_R_j_helpers:gamma_left_W_R}, recall that $v_{i,j} \in \{v_j,z_R(a_{i+1})\}$.
    First, assume that $v_{i,j} = v_j$.
    By definition of edges in $\Hout$, $v_j \notin \calR_i$.\footnote{Note that this is the only place, in this section of the proof, where we call the definition of edges in $\Hout$.}
    Thus, by~\cref{prop:b_relative_to_region}.\ref{prop:b_relative_to_region:item:outside}, either $W_L(v_j)$ is left of $\gamma_{L,i}$ or $\gamma_{R,i}$ is left of $W_R(v_j)$.
    Since the latter is the desired statement, let us show that the former can not hold.
    By~\cref{claim:HIO:u}.\ref{claim:HIO:u:WL-supath-WL}, $W_L(u_j)$ is a subpath of $W_L(u_i)$.
    Moreover, recall that $W_L(u_i)$ is a subpath of $\gamma_{L,i}$.
    It follows that if $W_L(v_j)$ is left of $\gamma_{L,i}$, then either $W_L(v_j)$ is left of $W_L(u_j)$ or $W_L(u_j)$ is a subpath of $W_L(v_j)$.
    Both are false, since $u_j$ is left of $v_j$ (by~\cref{claim:HIO:v}.\ref{claim:HIO:v:u_left_v}), and so, $W_L(u_j)$ is left of $W_L(v_j)$.
    This completes the proof of~\ref{claim:HIO:R_i_subset_R_j_helpers:gamma_left_W_R} in the case, where $v_{i,j} = v_j$.
    
    Next, assume that $v_{i,j} = z_R(a_{i+1})$.
    Since $a_{i+1} \parallel v_i$ in~$P$ (by~\cref{claim:HIO:v}.\ref{claim:HIO:v:a_parallel_v}), $z_R(a_{i+1})$ does not lie in $W_R(v_i)$, and in particular, does not lie in $\gamma_{R,i}$.
    Therefore, either $\gamma_{R,i}$ is left of $W_R(z_R(a_{i+1}))$ or $W_R(z_R(a_{i+1}))$ is left of $\gamma_{R,i}$.
    Since the former is the desired statement, let us show that the latter can not hold.
    Suppose to the contrary that $W_R(z_R(a_{i+1}))$ is left of $\gamma_{R,i}$.
    The path $\gamma_{R,i}$ is either left of $M_R(a_i)$ or is a subpath of $M_R(a_i)$.
    It follows that $W_R(z_R(a_{i+1}))$ is left of $M_R(a_i)$.
    In particular, $M_R(a_{i+1})$ is left of $M_R(a_i)$, however, this is a contradiction with~\ref{Rin} for $\sigma_i$, and so, the proof of~\ref{claim:HIO:R_i_subset_R_j_helpers:gamma_left_W_R} is completed.


    Let $\gamma_L = x_0[W_L(u_{i+1})]u_{i+1}[U_{i+1}]m_{i,j}$ and $\gamma_R = x_0[W_R(v_{i,j})]v_{i,j}[V_{i,j}]m_{i,j}$.
    We argue that $\gamma_{L}$ is left of $W_L(u_i)$ and $W_R(u_i)$ is left of $\gamma_{R}$.
    This will show that $u_i \in \Int \calR_{i,j}$ by~\cref{prop:b_relative_to_region}.\ref{prop:b_relative_to_region:item:inside}.
    By~\Cref{claim:HIO:u}.\ref{claim:HIO:u:WL-supath-WL}, $W_L(u_{i+1})$ is a proper subpath of $W_L(u_i)$ and $W_L(u_i)$ is a subpath of $\spine$.
    Hence, by~\ref{items-u:path-left}, $\gamma_{L}$ is left of $W_L(u_i)$.
    By~\cref{claim:HIO:u_left_v}.\ref{claim:HIO:u_left_v:u_i}, $u_{i}$ is left of $v_{i,j}$, thus, $W_R(u_{i})$ is left of $W_R(v_{i,j})$, and so, $W_R(u_i)$ is left of $\gamma_{R}$. 
    As discussed, this ends the proof of~\ref{claim:HIO:R_i_subset_R_j_helpers:first_segment}.
\end{proofclaim}

\begin{claim}\label{claim:HIO:R_i_subset_R_ij}
    Let $(i,j)$ be an edge in $\Hout$.
    Then,
        \[\calR_{i} \subset \calR_{i,j}.\]
\end{claim}
\begin{proofclaim}
    In order to prove the claim, we show that $\partial \calR_i \subset \calR_{i,j}$, which suffices by~\cref{obs:region_containment}.
    Moreover, if $u_i[U_i]m_i[V_i]v_i \subset \calR_{i,j}$, then by~\cref{prop:uv_in_R_implies_path}, $\partial \calR_{i} \subset \calR_{i,j}$. Thus, we only argue that
    \begin{equation}\label{eq:U_iV_i_subset_R_ij}
        u_i[U_i]m_i[V_i]v_i \subset \calR_{i,j}.
    \end{equation}
    By~\cref{claim:HIO:u}.\ref{claim:HIO:u:a_i-parallel-u_j}, we have $a_i \parallel u_{i+1}$ in~$P$ and by~\cref{claim:HIO:v}.\ref{claim:HIO:v:a_parallel_v}, $a_{i+1} \parallel v_i$ in~$P$.
    In particular,
    \begin{equation}\label{eq:disjoint_paths}
        U_{i+1} \cap (U_i \cup V_i) = \emptyset \text{ and } V_i \cap (U_{i+1} \cup V_{i,j}) = \emptyset.
    \end{equation}
    By~\cref{claim:HIO:R_i_subset_R_j_helpers}.\ref{claim:HIO:R_i_subset_R_j_helpers:first_segment}, $u_i \in\Int\calR_{i,j}$.
    Note that we will use this fact extensively.
    
    First, assume that $u_i[U_i]m_i$ and $v_{i,j}[V_{i,j}]m_{i,j}$ are disjoint.
    Since $U_i$ and $V_i$ are exposed witnessing paths in~$P$, by~\eqref{eq:disjoint_paths}, $u_i[U_i]m_i[V_i]v_i$ can intersect $\partial \calR_{i,j}$ only in $v_i$.
    This implies~\eqref{eq:U_iV_i_subset_R_ij}, which concludes the proof in this case.

\begin{figure}[tp]
  \begin{center}
    \includegraphics{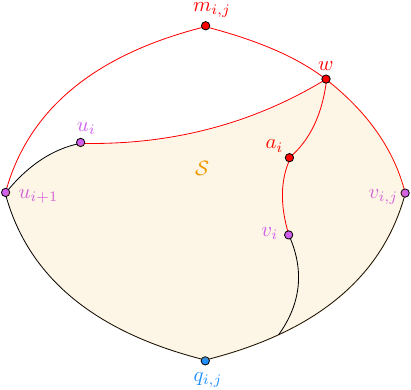}
  \end{center}
  \caption{
    Region $\calR_{i,j}$ along with objects appearing in the proof of~\Cref{claim:HIO:R_i_subset_R_ij}.
  }
  \label{fig:def-R-subset-R-proof}
\end{figure}

    For the rest of the proof, we assume that $u_i[U_i]m_i$ and $v_{i,j}[V_{i,j}]m_{i,j}$ intersect.
    See an illustration in~\cref{fig:def-R-subset-R-proof}.
    Let $w$ be the first element of $u_i[U_i]m_i$ in $v_{i,j}[V_{i,j}]m_{i,j}$.
    In order to define a new auxiliary region we argue that 
    \begin{equation}\label{eq:u_in_Z}
        u_i \in Z(a_{i+1}).
    \end{equation}
    We have $u_i$ left of $b_i$ (by~\ref{items-u:left-d}) and $b_i$ left of $b_{i+1}$.
    Now, since $(a_{i+1},b_{i+1}) \in I$ and $u_i$ is left of $b_{i+1}$, by~\cref{prop:dangerous-implies-in-Y}, $a_{i+1} \notin \shadz(u_i)$.
    Moreover, $a_{i+1}[V_{i,j}]w[U_i]u_i$ is an exposed witnessing path from $a_{i+1}$ to $u_i$ in~$P$, as desired. 
    This completes the proof of~\eqref{eq:u_in_Z}. 

    Since $u_i$ is left of $v_{i,j}$ (by~\cref{claim:HIO:u_left_v}.\ref{claim:HIO:u_left_v:u_i}), $u_i \in Z(a_{i+1})$ (by~\eqref{eq:u_in_Z}), and $v_{i,j} \in Z(a_{i+1})$ (by~\cref{claim:HIO:u_left_v}.\ref{claim:HIO:u_left_v:Z}) we can define
        \[\calS = \calR(a_{i+1},u_i,v_{i,j}, a_{i+1}[V_{i,j}]w[U_i]u_i, a_{i+1}[V_{i,j}]v_{i,j}).\]
    Note that by definition, $w$ is the upper-min of $\calS$.
    Since $u_i \in \Int \calR_{i,j}$, also by the definition of $w$ and~\eqref{eq:disjoint_paths}, we have $u_i[U_i]w \subset \calR_{i,j}$.
    On the other hand, $w[V_{i,j}]v_{i,j} \subset \partial \calR_{i,j} \subset \calR_{i,j}$.
    Since $u_i,v_{i,j} \in \calR_{i,j}$,~\cref{prop:uv_in_R_implies_path} gives $\partial \calS \subset \calR_{i,j}$, and subsequently,~\Cref{obs:region_containment} gives 
    \begin{equation*} \label{eq:S_subset_R}
        \calS \subset \calR_{i,j}.
    \end{equation*}
    As noted above, $u_i[U_i]w \subset \calR_{i,j}$.
    In particular, if $w[U_i]m_i[V_i]v_i \subset \calS$, then $w[U_i]m_i[V_i]v_i \subset \calR_{i,j}$, and so,~\eqref{eq:U_iV_i_subset_R_ij} follows.
    Hence, all we need to prove that is 
    \begin{equation}\label{eq:wUmVv_in_S}
        w[U_i]m_i[V_i]v_i \subset \calS.
    \end{equation}
    First, we show that
    \begin{equation}\label{eq:v_i_in_S}
        v_i \in \calS.
    \end{equation}
    Let $q_{\calS}$ be the lower-min of $\calS$.\footnote{In fact, one can prove that $q_{\calS} = q_{i,j}$ but we do not need this equality.}
    If $v_i \in \partial \calS$, then~\eqref{eq:v_i_in_S} clearly holds, thus, assume that $v_i \notin \partial\calS$.
    Since $q_{\calS} < u_i$ in~$P$, ~\cref{prop:comparability_implies_shadow_containment} implies that $\shadz(q_{\calS})\subset \shadz(u_i)$. 
    Since $u_i$ is left of $v_i$, $v_i \notin \shadz(u_i)$, and so, $v_i \notin \shadz(q_{\calS})$.
    We assumed~$v_i \notin \partial\calS$, thus,~\cref{claim:HIO:R_i_subset_R_j_helpers}.\ref{claim:HIO:R_i_subset_R_j_helpers:gamma_left_W_R} implies that $W_R(v_{i})$ is left of $W_R(v_{i,j})$.
    On the other hand, $W_L(u_{i})$ is left of $W_L(v_i)$ by~\cref{claim:HIO:v}.\ref{claim:HIO:v:u_left_v}.
    Therefore, by~\cref{prop:b_relative_to_region}.\ref{prop:b_relative_to_region:item:inside}, $v_i \in \Int \calS$, which concludes the proof of~\eqref{eq:v_i_in_S}.

    Let $e$ be the first edge of $v_i[V_i]m_i$.
    Next, we show that
    \begin{equation}\label{eq:interior_of_e_in_S}
        \text{the interior of $e$ lies in $\calS$.}
    \end{equation}
    Recall that $v_i \in \calS$ (by~\ref{eq:v_i_in_S}).
    If $v_i \in \Int \calS$, then~\eqref{eq:interior_of_e_in_S} clearly follows, hence, assume that $v_i \in \partial \calS$.
    In other words, $v_i$ is an element of $u_i[W_L(u_i)]q_{\calS}[W_R(v_{i,j})]v_{i,j}$.
    Since $u_i$ is left of $v_i$, $v_i$ does not lie in $W_L(u_i)$, and so, $v_i$ lies in $W_R(v_{i,j})$ and $v_i \neq q_{\calS}$.
    In other words, $v_i$ lies strictly on the right side of $\calS$.
    Therefore, by~\cref{claim:HIO:R_i_subset_R_j_helpers}.\ref{claim:HIO:R_i_subset_R_j_helpers:gamma_left_W_R} and~\ref{items:leaving_regions:right}, the edge $e$ lies in $\calS$ and~\eqref{eq:interior_of_e_in_S} follows.

    Finally, we are ready to argue~\eqref{eq:wUmVv_in_S}.
    By~\eqref{eq:v_i_in_S} and~\eqref{eq:interior_of_e_in_S}, it suffices to show that $W = v_i[V_i]m_i[U_i]w$ intersects $\partial \calS$ at most in the endpoints.
    Since $U_i$ and $V_i$ are exposed witnessing paths in~$P$, the only element in $B$ of $W$ is $v_i$, and so, it suffices to show that $W$ intersects $u_i[U_i]w[V_{i,j}]v_{i,j}$ only in $w$.
    The path $m_i[U_i]w$ does not intersect $u_i[U_i]w[V_{i,j}]v_{i,j}$ in an element distinct from $w$ as this yields a directed cycle in~$P$.
    The path $v_i[V_i]m_i$ is disjoint from $V_{i,j}$ by~\eqref{eq:disjoint_paths}.
    Finally, $v_i[V_i]m_i$ does not intersect $u_i[U_i]w$ in an element distinct from $w$ as this contradicts the definition of $m_i$.
    We showed~\eqref{eq:wUmVv_in_S}, and so the proof of the claim is complete.
\end{proofclaim}

Next, we develop two claims about paths leaving regions: \q{going in} claim and \q{going out} claim (see~\cref{fig:going-in-out}).

\begin{figure}[tp]
  \begin{center}
    \includegraphics{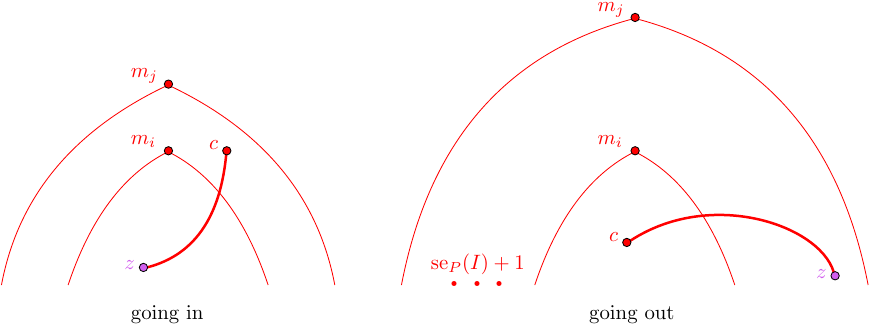}
  \end{center}
  \caption{
    \q{Going in} claim (\Cref{claim:going-in}) and \q{going out} claim (\Cref{claim:going-out}) state that exposed witnessing paths can not escape too far while going in or out, respectively.
    In the figure, we show a schematic drawing of the statements.
  }
  \label{fig:going-in-out}
\end{figure}

\begin{claim}[Going in]\label{claim:going-in}
    Let $(i,j)$ be an edge of $\Hout$.
    Let $c \in A$ be such that $c \parallel u_i$ in~$P$ and let $z \in Z(c)$.
    If $z \in \Int\calR_i$, then $c \in \Int \calR_j$.
\end{claim}
\begin{proofclaim}
    Assume that $z \in \Int \calR_i$.
    By~\cref{claim:HIO:R_ij_subset_R_j} and \cref{claim:HIO:R_i_subset_R_ij}, we have
    \[\calR_i\subseteq \calR_{i,j} \subset \calR_j.
    \]
    In particular, it suffices to show that $c \in \Int \calR_{i,j}$.
    Suppose to the contrary that $c \notin \Int \calR_{i,j}$.
    Let $W$ be an exposed witnessing path from $c$ to $z$ in~$P$.
    Since $\calR_i \subset \calR_{i,j}$, $z\in\Int \calR_i$ and $c\not\in\Int\calR_{i,j}$, 
    there exist elements $w_i$ and $w_{i,j}$ of $W$ such that 
    \[w_i\in\partial\calR_i,\ w_{i,j}\in\partial\calR_{i,j},\ \textrm{and
    $c \leq w_{i,j} \leq w_i < z$ in~$P$}.
    \]
    The last inequality is strict as $z$ lies in the interior of $\calR_i$.
    In particular, since $W$ is an exposed witnessing path in~$P$, we have $w_i,w_{i,j} \notin B$. 
    Hence, $w_i\in\partial\calR_i$ implies that $w_i$ lies in $u_i[U_i]m_i[V_i]v_i$ and $w_{i,j}\in\partial\calR_{i,j}$ implies that $w_{i,j}$ lies in $u_{i+1}[U_{i+1}]m_{i,j}[V_{i,j}]v_{i,j}$.
    Recall that we assumed $c \parallel u_i$ in~$P$.
    It follows that $w_i$ does not lie in $U_i$. 
    Also by~\cref{claim:HIO:u}.\ref{claim:HIO:u:u_j<u_i}, $u_{i+1} < u_i$ in~$P$.
    In particular, $c \not\leq u_{i+1}$ in~$P$ and since $c\in A$ and $u_{i+1}\in B$ we also have $u_{i+1}\not\leq c$ in~$P$. 
    Hence $c\parallel u_{i+1}$ in~$P$. 
    It follows that $w_{i,j}$ does not lie in $U_{i+1}$.
    Therefore, $w_i$ lies in $V_i$ and $w_{i,j}$ lies in $V_{i,j}$.
    We obtain
        \[a_{i+1} \leq w_{i,j} \leq w_i \leq v_i \text{ in~$P$}.\]
    This contradicts~\cref{claim:HIO:v}.\ref{claim:HIO:v:a_parallel_v} and completes the proof.
\end{proofclaim}

The next series of claims leads to the \q{going out} claim, which is crucial in the final steps of the proof.
At the very beginning of the proof of~\cref{lemma:HIO}, for each $d \in B$, we defined $\pi(d) = \gce(\spine,W_L(d))$.
This element has an equivalent description given that $u_1$ is left of~$d$.

\begin{claim}\label{claim:HIO:equivalent_pi_def}
    Let $d \in B$ such that $u_1$ is left of $d$.
    Then,
        \[\pi(d) = \text{the maximal element of $\spine$ that is less than $d$ in~$P$}.\]
\end{claim}
\begin{proofclaim}
    Let $\pi'(d)$ be the maximal element of $\spine$ that is less than $d$ in~$P$.
    The aim is to prove that $\pi(d) = \pi'(d)$.
    Clearly, both $\pi(d)$ and $\pi'(d)$ lie in $\spine$ and $\pi(d) \leq \pi'(d)$ in~$P$.
    Suppose to the contrary that $\pi(d) < \pi'(d)$ in~$P$.
    Let $W$ be a witnessing path from $\pi'(d)$ to $d$ in~$P$.
    Consider the path $W' = x_0[\spine]\pi'(d)[W]d$.
    Since $u_1$ is left of $d$ and $\pi(d) = \gce(W_L(u_1),W_L(d))$, we obtain that $W'$ is left of $W_L(d)$, which is a contradiction.
\end{proofclaim}

\begin{claim} \label{claim:HIO:b_notin_region_then_pi_b_leq_pi_v} 
    Let $i \in E$ and $d \in B$ such that $u_1$ is left of $d$, $u_i \parallel d$ in~$P$, and $d \notin \calR_i$.
    Then, $\pi(d) \leq \pi(v_i)$ in~$P$.
\end{claim}
\begin{proofclaim}
    Recall that $W_L(u_i)$ is a subpath of $\spine$ (by~\ref{items-u:W_L}).
    Note that since $u_i \parallel d$ in~$P$, $\pi(d)$ does not lie in $u_i[\spine]u_1$.
    By~\cref{claim:HIO:equivalent_pi_def}, $q_i \leq \pi(v_i)$ in~$P$.
    Thus, if $\pi(d) \leq q_i$ in~$P$, the assertion holds, and so, we assume that $q_i < \pi(d) < u_i$ in~$P$.
    In particular, $\pi(d)$ lies strictly on the left side of $\calR_i$.
    Since $\spine$ is left of $W_L(d)$, the first edge of $\pi(d)[W_L(d)]d$ lies in $\calR_i$ (by~\ref{items:leaving_regions:left}).
    However, $d \notin \calR_i$, and so, $\pi(d)[W_L(d)]d$ intersects $\partial \calR_i$ in an element distinct from $\pi(d)$.
    Since $\pi(d)[W_L(d)]d$ has all elements in $B$, this path must intersect $u_i[W_L(u_i)]q_i[W_R(v_i)]v_i$.
    The paths $W_L(u_i)$ and $W_L(d)$ are $x_0$-consistent (by~\cref{prop:W-consistent}.\ref{prop:W-consistent:left}), thus, $\pi(d)[W_L(d)]d$ does not intersect $W_L(u_i)$ in an element distinct from $\pi(d)$.
    It follows that $\pi(d)[W_L(d)]d$ intersects $W_R(v_i)$.
    In particular, $\pi(d) < v_i$ in~$P$.
    By~\cref{claim:HIO:equivalent_pi_def}, this implies that $\pi(d) \leq \pi(v_i)$ in~$P$, which ends the proof.
\end{proofclaim}

\begin{claim}\label{claim:HOI:two_edges_Hout_pi}
    Let $i,j,k \in [n-1]$ such that $i < j < k$ and $(i,j)$ is an edge of $\Hout$.
    Then,
    \begin{enumerate}
        \item $b_k \notin \cgR_i$,\label{claim:HOI:two_edges_Hout_pi:b_k-notin-R_i}
        \item $\pi(b_k) \leq \pi(v_i)$ in~$P$, \label{claim:HOI:two_edges_Hout_pi:pi_b_k_pi_v_i}
        \item $a_{k+1} < v_i$ in~$P$. \label{claim:HOI:two_edges_Hout_pi:a_v}
    \end{enumerate}
\end{claim}
\begin{proofclaim}
    We start with the proof of~\ref{claim:HOI:two_edges_Hout_pi:b_k-notin-R_i}. 
    We are going to call~\cref{prop:b_relative_to_region} twice for the region $\calR_i$ and elements $b_k$ and $v_j$.
    To make these calls simpler, we note beforehand that
    \begin{equation}
    b_k, v_j \not\in\shadz(q_i).
    \label{eq:elements-not-in-shadz-qi}
    \end{equation}
    Since $q_i < u_i$ in $P$, 
    we have $\shadz(q_i) \subset \shadz(u_i)$ (by~\Cref{prop:comparability_implies_shadow_containment}).
    Since $u_i$ left of $b_k$ (by~\cref{claim:HIO:left-right}.\ref{claim:HIO:left-right:u_left_b}) and $u_i$ left of $v_j$ (by~\cref{claim:HIO:left-right}.\ref{claim:HIO:left-right:u_left_v}),~\eqref{eq:elements-not-in-shadz-qi} holds.

    We will argue that
    \begin{align}\label{eq:gamma-WR}
        \text{$\gamma_{R,i}$ left of $W_R(b_{j+1}).$}
    \end{align}
    Note that~\eqref{eq:gamma-WR} quickly implies~\ref{claim:HOI:two_edges_Hout_pi:b_k-notin-R_i}. 
    Indeed, $W_R(b_{j+1})$ is left of or equal to $W_R(b_k)$ (since $j+1\leq k$), hence,~\eqref{eq:gamma-WR} implies that $\gamma_{R,i}$ is left of $W_R(b_k)$. Thus, by~\cref{prop:b_relative_to_region}.\ref{prop:b_relative_to_region:item:outside}, $b_k \notin \calR_i$, as desired.

    Now, we prove~\eqref{eq:gamma-WR}.
    Since $u_i$ is left of $v_j$ (by~\Cref{claim:HIO:left-right}.\ref{claim:HIO:left-right:u_left_v}) and $v_j \notin \calR_i$ (as $(i,j)$ is an edge in $\Hout$), by~\Cref{prop:b_relative_to_region}.\ref{prop:b_relative_to_region:item:outside}, 
    \begin{align}\label{eq:gamma-WRv}
        \text{$\gamma_{R,i}$ is left of $W_R(v_j)$.}
    \end{align}
    Note that $j \in E$ as $(i,j)$ is an edge of $\Hout$.
    In other words, $\sigma_j$ is an edge of weight $1$.
    If $\sigma_j$ is a shifted edge, then $v_j$ is left of $b_{j+1}$ (by~\Cref{claim:HIO:v}.\ref{claim:HIO:v:v_left_b}), and so, $W_R(v_j)$ is left of $W_R(b_{j+1})$.
    This and~\eqref{eq:gamma-WRv} imply~\eqref{eq:gamma-WR}.
    If $\sigma_j$ is a cycle edge, then $v_j \leq b_{j+1}$ in $P$, and so, either $W_R(v_j)$ is a subpath of $W_R(b_{j+1})$ or $W_R(v_j)$ is left of $W_R(b_{j+1})$ (by~\Cref{prop:shortcuts}.\ref{prop:item:shortcuts-right-tree}).
    This and~\eqref{eq:gamma-WRv} again imply~\eqref{eq:gamma-WR}.
    This completes the proof of~\ref{claim:HOI:two_edges_Hout_pi:b_k-notin-R_i}.

    For the proof of ~\ref{claim:HOI:two_edges_Hout_pi:pi_b_k_pi_v_i}, note that 
    $b_k \in B$, $u_1$ is left of $b_k$ and $u_i \parallel b_k$ in~$P$ (by~\cref{claim:HIO:left-right}.\ref{claim:HIO:left-right:u_left_b}). Therefore,~\ref{claim:HOI:two_edges_Hout_pi:pi_b_k_pi_v_i} follows from~\cref{claim:HIO:b_notin_region_then_pi_b_leq_pi_v}.

    Finally, we have
        \[a_{k+1} < u_{k+1} \leq \pi(b_k) \leq \pi(v_i) < v_i \text{ in~$P$}.\]
    The inequalities follow respectively from: \ref{items-u:Z}, \ref{items-u:u<pi}, \ref{claim:HOI:two_edges_Hout_pi:pi_b_k_pi_v_i}, and the definition of $\pi$.
    This way, we obtain~\ref{claim:HOI:two_edges_Hout_pi:a_v}.
\end{proofclaim}

 In~\Cref{claim:HOI:two_edges_Hout_pi}, we showed some properties of elements with indices $i,k \in [n-1]$ with $i<k$ such that there exists $j \in [n-1]$ with
 $i<j<k$ and $(i,j)\in\Hout$. 
 We would like to have these properties for every pair of indices.
 To this end, we fix a long path in $\Hout$ and focus on every other element of the path. 

Fix a path in $\Hout$ witnessing $\maxpath(\Hout)$, and let $E_1$ be its vertex set.
\cref{claim:HIO:R_i_subset_R_ij,claim:HIO:R_ij_subset_R_j} immediately imply the following.

\begin{claim}\label{claim:region-i-in-region-j}
Let $i,j\in E_1$ with $i<j$. Then, $\calR_i \subseteq \calR_j$.
\end{claim}
Next, let $E_2$ be every other element of $E_1$ starting from the first one. 
In particular, by~\eqref{eq:max-path-Hout}
\begin{align}
    |E_2| &\geq |E_1| \slash 2 = \maxpath(\Hout) \slash 2 \geq 2\se_P(I) + 6.\label{eq:E_2}
\end{align}
From \Cref{claim:HOI:two_edges_Hout_pi} we obtain the following.

\begin{claim}\label{claim:HIO:cor-a_leq_v}
    Let $i,j \in E_2$ with $i < j$.
    Then,
    \begin{enumerate}
        \item $b_j \notin \calR_i$, \label{claim:HOI:cor-a_leq_v:b_k-notin-R_i}
        \item $\pi(b_j) \leq \pi(v_i)$ in~$P$, \label{claim:HOI:cor-a_leq_v:pi_b_k_pi_v_i}
        \item $a_{j+1} < v_i$ in~$P$. \label{claim:HOI:cor-a_leq_v:a_v}
    \end{enumerate}
\end{claim}

Finally, we are ready to state and prove the \q{going out} claim.

\begin{claim}[Going out]\label{claim:going-out}
    Let $i,j \in E_2$ with $i < j$ such that $|\{k \in E_2 : i \leq k \leq j\}| \geq \se_P(I) + 1$.
    Let $c \in A$, $z \in Z(c)$, and let $W$ be an exposed witnessing path from $c$ to $z$ in~$P$. 
    If $c \in \calR_i$, then $W \subset \calR_j$ (in particular, $z \in \calR_j$).
\end{claim}
\begin{proofclaim}
    Assume that $c \in \calR_i$. 
    Suppose to the contrary that $W \not\subset \calR_j$.
    Let $X = \{k \in E_2 : i \leq k \leq j\}$ and for each $k \in X$, let $d_{k+1} = b_{k+1}$ when $\sigma_k$ is a cycle edge and $d_{k+1} = t_{k+1}$ when $\sigma_k$ is a shifted edge.
    Recall that $(a_{k+1},d_{k+1}) \in I$ for each $k \in X$ (by~\ref{items:HIO-cycle} when $\sigma_k$ is a shifted edge).
    We show that for all $k,\ell \in X$ with $k < \ell$, we have
        \[a_{k+1} < d_{\ell+1} \text{ and } a_{\ell+1} < d_{k+1} \text{ in } P.\]
    Therefore, $\{(a_{k+1},d_{k+1}) : k\in X\}$ induces a standard example of order $|X|$ in~$P$ with all the incomparable pairs in $I$, which is a contradiction as $|X| \geq \se_P(I)+1$.
    
    Fix some $k,\ell \in X$ with $k < \ell$.
    First, note that by~\cref{claim:HIO:cor-a_leq_v}.\ref{claim:HOI:cor-a_leq_v:a_v}, $a_{\ell+1} < v_k \leq d_{k+1}$ in~$P$.
    The rest of the proof is devoted to arguing that $a_{k+1} < d_{\ell+1}$ in~$P$.
    Since $k,\ell \in E_2$, we can fix $k' \in E_1$ such that $k < k' < \ell$ and $(k,k')$ is an edge of $\Hout$.
    By~\cref{claim:HIO:R_i_subset_R_ij,claim:HIO:R_ij_subset_R_j,claim:region-i-in-region-j},
        \[\calR_i \subset \calR_k \subset \calR_{k,k'} \subset \calR_{k'} \subset \calR_{\ell} \subset \calR_j.\]
    Since $c \in \calR_i$, $W \not\subset \calR_j$, and $W$ is exposed, there exist elements $w_{k,k'}$, $w_{\ell}$, and $w_j$ of $W$ not in $B$ such that
        \[w_{k,k'}\in\partial\calR_{k,k'},\ w_{\ell}\in\partial\calR_{\ell},\ w_j \in \partial \calR_j, \ \textrm{and
        $c \leq w_{k,k'} \leq w_\ell \leq w_j < z$ in~$P$}.
        \]
    The element $w_{k,k'}$ lies in $u_{k+1}[U_{k+1}]m_{k,k'}[V_{k,k'}]v_{k,k'}$ and the element $w_{\ell}$ lies in $u_{\ell}[U_{\ell}]m_{\ell}[V_{\ell}]v_{\ell}$.
    The former yields $a_{k+1} \leq m_{k,k'} \leq w_{k,k'}$ in~$P$.
    Since $k,\ell \in E_2$ and $k < \ell$, we have $k+1 < \ell$.
    If $w_\ell$ lies in $U_\ell$, then $a_{k+1} \leq w_{k,k'} \leq w_\ell \leq u_\ell \leq u_{k+2} \leq b_{k+1}$ in $P$ (the last two comparabilities follow from~\cref{claim:HIO:u}.\ref{claim:HIO:u:u_j<u_i} and~\ref{items-u:u<pi} respectively), which is a false.
    Thus, $w_\ell$ lies in $V_\ell$.
    This yields $w_\ell < v_\ell$ in~$P$.
    Altogether, we obtain
        \[a_{k+1} \leq w_{k,k'} \leq w_\ell < v_\ell \leq d_{\ell+1} \text{ in } P,\]
    as desired. 
\end{proofclaim}

We have $|E_2| \geq 2\se_P(I)+6 = 1 + (\se_P(I) + 1) + 1 + (\se_P(I) + 1) + 2$ (by~\eqref{eq:E_2}).
Hence, we can define the following numbers.
Let $j_1,\dots,j_8 \in E_2$ be such that
\begin{align*}
    &j_1 < j_2 < j_3 < j_4 < j_5 < j_6 < j_7 < j_8, \\
    &|\{k \in E_2 : j_2 \leq k \leq j_3\}| \geq \se_P(I)+1, \\
    &|\{k \in E_2 : j_5 \leq k \leq j_6\}| \geq \se_P(I)+1. 
\end{align*}
The plan is to show 
that $((a',b'),(a_{j_8},b_{j_8}))$ is an edge in $\HIO$.
This will conclude the proof of~\Cref{lemma:HIO}, 
see~\eqref{eq:final-HIO}.
Recall that $(a,b) = (a_1,b_1)$.

\begin{claim}\label{claim:a_in_Int_R_j_1}
    $a \in \calR_{j_2}$.
\end{claim}
\begin{proofclaim}
    Let $\ell \in E_1$ be such that $j_1 < \ell < j_2$ and $(\ell,j_2)$ is an edge in $\Hout$.
    In particular, $1 < \ell$, and so, by~\Cref{claim:HIO:u_i-in-R_j}, $u_1 \in \Int \calR_{\ell}$.
    Additionally, by~\Cref{claim:HIO:u}.\ref{claim:HIO:u:a_i-parallel-u_j}, $a \parallel u_\ell$ in~$P$.
    Since $u_1 \in Z(a)$ (by~\ref{items-u:Z}), we can apply~\Cref{claim:going-in} to obtain $a \in \Int \calR_{j_2}$.
\end{proofclaim}

\begin{figure}[tp]
  \begin{center}
    \includegraphics{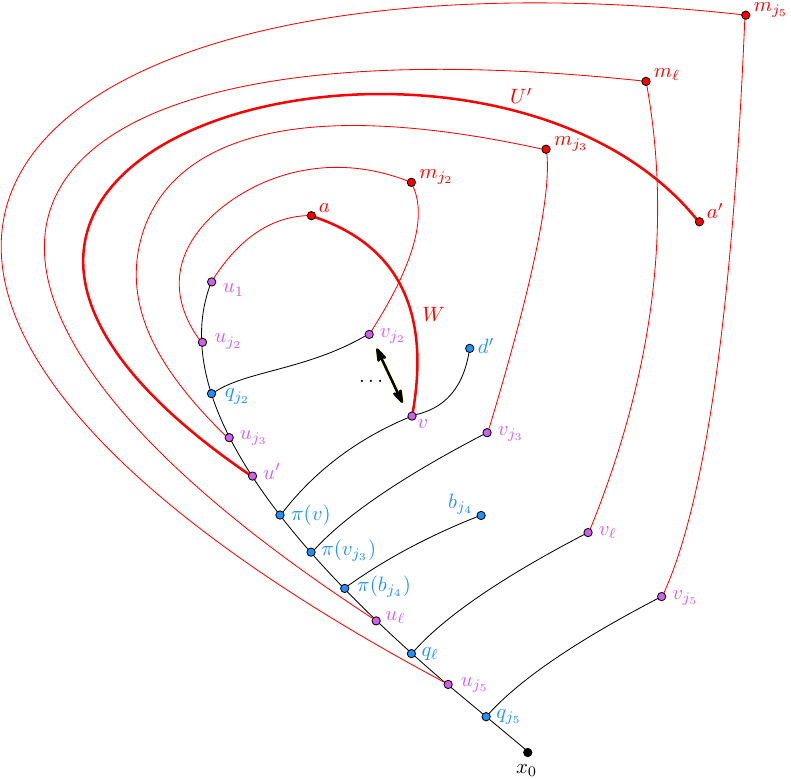}
  \end{center}
  \caption{
    An illustration of the proof of~\cref{claim:a'}.
    We denote by $d'$ either $b'$ or $t'$, depending on the case.
  }
  \label{fig:a'}
\end{figure}


Let $u' \in Z(a')$ and let $U'$ be an exposed witnessing path from $a'$ to $u'$ in~$P$ be such that $(u',U')$ witnesses~\ref{Lout} for $((a_1,b_1),(a',b'))$.
Note that $u'$ lies in $\spine$.

\begin{claim}\hfill\label{claim:a'}
    \begin{enumerate}
        \item $a' \in \calR_{j_5}$.\label{claim:a':in-region}
        \item $u_{j_5}$ lies in $W_L(u')$. \label{claim:a':uj5}
    \end{enumerate}
\end{claim}
\begin{proofclaim}
    The proof is illustrated in~\Cref{fig:a'}.
    Recall that $((a,b),(a',b'))$ is an edge of weight $1$ in $\HIO$.
    If it is a shifted edge, let $t' \in B$ be a witness for the edge.
    Let $v \in B$ be such that
    \[
        v \in Z(a) \text{ and }\begin{cases}
        v \leq b' \text{ in~$P$}&\textrm{if $((a,b),(a',b'))$ is a cycle edge},\\
        v \leq t' \text{ in~$P$}&\textrm{if $((a,b),(a',b'))$ is a shifted edge}.\\
        \end{cases}
    \]
    By~\cref{claim:a_in_Int_R_j_1}, we have $a \in \calR_{j_2}$. 
    Let $W$ be an exposed witnessing path from $a$ to $v$ in~$P$. 
    Since $|\{k \in E_2 : j_2 \leq k \leq j_3\}| \geq \se_P(I)+1$, by~\Cref{claim:going-out}, $W\subseteq \calR_{j_3}$.
    Moreover, since $(a,b) \in I$ and $b$ is left of $v_{j_3}$ (by~\Cref{claim:HIO:left-right}.\ref{claim:HIO:left-right:b_left_v}), by~\Cref{prop:dangerous-implies-in-Y}, $a \notin \shadz(v_{j_3})$.    
    Since $W\subseteq \calR_{j_3}$ and $a \notin \shadz(v_{j_3})$, by~\Cref{prop:elements-not-in-shad-w}.\ref{prop:elements-not-in-shad-w:left}.\ref{prop:elements-not-in-shad-w:left:iii}, we obtain that $\gce(W_L(u_{j_3}),W_L(v_{j_3}))$ lies in $W_L(v)$.
    Since $W_L(u_{j_3})$ is a subpath of $\spine$ (by~\cref{claim:HIO:u}.\ref{claim:HIO:u:WL-supath-WL}) and $u_{j_3} \parallel v_{j_3}$ in $P$ (by~\Cref{claim:HIO:v}.\ref{claim:HIO:v:u_left_v}), we have $\pi(v_{j_3}) = \gce(W_L(u_{j_3}),W_L(v_{j_3}))$, and so,
    \begin{equation}\label{eq:application-of-going-out}
        \pi(v_{j_3}) \leq \pi(v) \text{ in } P.
    \end{equation}
    Let $\ell \in E_1$ be such that $j_4 < \ell < j_5$ and $(\ell,j_5)$ is an edge in $\Hout$. 
    We have
    \begin{equation} \label{eq:u-leq-u-ell}
        u_\ell \leq u_{j_4+1} \leq \pi(b_{j_4}) \leq \pi(v_{j_3}) \leq \pi(v) < u' \text{ in } P,
    \end{equation}
    where the first comparability follows from~\cref{claim:HIO:u}.\ref{claim:HIO:u:u_j<u_i}, the second follows from~\ref{items-u:u<pi}, the third follows from~\Cref{claim:HIO:cor-a_leq_v}.\ref{claim:HOI:cor-a_leq_v:pi_b_k_pi_v_i}, the fourth follows from~\eqref{eq:application-of-going-out}, and the fifth follows from the facts that $u'$ and $\pi(v)$ lie in $\spine$ and $a' \parallel v$ in~$P$ (by~\ref{items:HIO-cycle} in the case of a shifted edge).
    Note that this already implies~\ref{claim:a':uj5} since $u_{j_5} \leq u_\ell < u'$ in $P$ (by~\eqref{claim:HIO:u}.\ref{claim:HIO:u:u_j<u_i} and~\eqref{eq:u-leq-u-ell}) and all of these elements lie in $\spine$ (by~\ref{items-u:W_L}).
    
    By~\Cref{claim:HIO:u_i-in-R_j}, $u_\ell[\spine]u_1 \subset \Int \calR_\ell \cup \{u_\ell\}$. 
    By~\eqref{eq:u-leq-u-ell}, $u'$ lies in  $u_\ell[\spine]u_1$ and $u' \neq u_\ell$, so we conclude that $u' \in \Int \calR_\ell$.
    Since $u_\ell \leq \pi(v)$ in~$P$ (by~\eqref{eq:u-leq-u-ell}), we have $a' \parallel u_\ell$ in~$P$.
    Therefore, by~\Cref{claim:going-in}, $a' \in \Int \calR_{j_5}$.
    This completes the proof of~\ref{claim:a':in-region} and the claim.
\end{proofclaim}

\begin{claim}\label{claim:a'-to-z-path}
    Let $z \in Z(a')$ and let $W$ be a witnessing path from $a'$ to $z$ in $P$.
    Then, $W \subset \calR_{j_6}$.
    In particular, $z \in \calR_{j_6}$.
\end{claim}
\begin{proofclaim}
    By~\Cref{claim:a'}.\ref{claim:a':in-region}, $a' \in \calR_{j_5}$. 
    Since $j_5,j_6\in E_2$ and 
    $|\set{k\in E_2: j_5\leq k \leq j_6}|\geq \se_P(I)+1$, \Cref{claim:going-out} 
    implies that $a'\in \calR_{j6}$.
\end{proofclaim}

\begin{claim}\label{claim:b'-in-R-j6}
    $b' \in \Int \calR_{j_6}$.
\end{claim}
\begin{proofclaim}
    We have $q_{j_6} < u_{j_6} < u_2 \leq \pi(b) \leq b$ in $P$ by~\Cref{claim:HIO:u}.\ref{claim:HIO:u:u_j<u_i} and~\ref{items-u:u<pi}.
    Thus, $\shadz(q_{j_6}) \subset \shadz(b)$ (by~\Cref{prop:comparability_implies_shadow_containment}), and so, since $b$ is left of $b'$, we obtain that $b' \notin \shadz(q_{j_6})$.
    Therefore, to show that $b' \in \Int \calR_{j_6}$, it suffices (by~\Cref{prop:b_relative_to_region}.\ref{prop:b_relative_to_region:item:inside}) to prove that $\gamma_{L,j_6}$ is left of $W_L(b')$ and $W_R(b')$ is left of $\gamma_{R,j_6}$.
    We only prove the first statement, the proof of the second one is symmetric.
    By~\Cref{claim:a'-to-z-path}, we have $z_L(a') \in \calR_{j_6}$.
    By~\Cref{prop:z_L_b_z_R}, $z_L(a')$ is left of $b'$.
    Since $z_L(a') \in \calR_{j_6}$, by~\Cref{prop:b_relative_to_region}.\ref{prop:b_relative_to_region:item:outside}, either $\gamma_{L,j_6}$ is left of $W_L(z_L(a'))$ or $W_L(z_L(a'))$ is a subpath of $\gamma_{L,j_6}$.
    In both cases, it follows that $\gamma_{L,j_6}$ is left of $W_L(b')$ as $W_L(z_L(a'))$ is left of $W_L(b')$.
    This completes the proof of the claim.
\end{proofclaim}

\begin{claim}\label{claim:uj6}
$u_{j_6}$ lies in $W_L(z_L(a'))$.
\end{claim}
\begin{proofclaim}
    Let $U'' = z_L(a')[M_L(a')]a'$ and note that $U'' \subset \calR_{j_6}$ by~\Cref{claim:a'-to-z-path}.
    In particular, $z_L(a') \in \calR_{j_6}$, and so, by~\Cref{prop:b_relative_to_region}, $W_L(z_L(a'))$ is not left of $W_L(u_{j_6})$.
    If $W_L(u_{j_6})$ is a subpath of $W_L(z_L(a'))$, then the assertion follows.
    We will show that all the other cases lead to a contradiction.
    Since $u_{j_6} < u_{j_5}$ in $P$ (by~\cref{claim:a'}.\ref{claim:a':uj5}) and $W_L(u_{j_5})$ is a subpath of $W_L(u')$ (by~\cref{claim:a'}.\ref{claim:a':uj5}), we obtain that
    \begin{align}\label{eq:uj6-subpath-u'}
        \text{$W_L(u_{j_6})$ is a subpath of $W_L(u')$}.
    \end{align}
    If $W_L(u_{j_6})$ is left of $W_L(z_L(a'))$, then by~\eqref{eq:uj6-subpath-u'}, $W_L(u')$ is left of $W_L(z_L(a'))$, which contradicts the definition of $M_L(a')$.
    Finally, suppose that $W_L(z_L(a'))$ is a proper subpath of $W_L(u_{j_6})$.
    Note that $q_{j_6} < z_L(a')$ as otherwise $a' < z_L(a') \leq q_{j_6} \leq b'$ in~$P$ (by~\Cref{claim:b'-in-R-j6} and~\cref{prop:q-in-W_L(b)-and-W_R(b)}), which is false. 
    Altogether, we have $q_{j_6} < z_L(a') < u_{j_6}$ in~$P$.
    Let $e^-$ and $e^+$ be the edges of $W_L(u_{j_6})$ immediately preceding and immediately following $z_L(a')$, respectively. 
    Let $f$ be the edge of $U''$ incident to $z_L(a')$.
    Since $U'' \subset \calR_{j_6}$, by~\ref{items:leaving_regions:left}, $e^+ \prec f \prec e^-$ in the $z_L(a')$-ordering.
    It follows that $W_L(u_{j_6})$ is left of $M_L(a')$.
    Furthermore, by~\eqref{eq:uj6-subpath-u'}, $W_L(u')$ is left of $M_L(a')$, and so, $x_0[W_L(u')]u'[U']a'$ is left of $M_L(a')$, which is false.
    This contradiction ends the proof.
\end{proofclaim}

\begin{claim}\label{claim:b'-left-bj7}
    $b'$ is left of $b_{j_7}$.
\end{claim}
\begin{proofclaim}
    The proof is illustrated in~\Cref{fig:b'-left-bj7}.
    First, we prove that
    \begin{align}\label{eq:shad-containment-b'-bj7}
        b' \notin \shadz(b_{j_7}).
    \end{align}
    Let $w$ be the maximal common element of $W_R(v_{j_6})$ and $W_L(b_{j_7})$ in~$P$.
    Note that $u' \notin \shadz(w)$ as otherwise by~\cref{claim:a'}.\ref{claim:a':uj5} and~\Cref{prop:comparability_implies_shadow_containment} applied multiple times, 
    $u_{j_6} \in \shadz(u_{j_5}) \subset \shadz(u') \subset \shadz(w) \subset \shadz(v_{j_6})$, which contradicts~\Cref{claim:HIO:v}.\ref{claim:HIO:v:u_left_v}.
    If $a' \in \shadz(w)$, then an exposed witnessing path from $a'$ to $u'$ in~$P$ intersects $\partial\shadz(w)$ in an element of $B$, and so, exactly in $u'$, which is a contradiction since $u' \notin \shadz(w)$.
    Therefore, $a' \notin \shadz(w)$.
    Let $W$ be an exposed witnessing path from $a'$ to $z_R(a')$ in $P$.
    By~\Cref{claim:a'-to-z-path}, we have $W \subset \calR_{j_6}$. 
    Therefore, we can apply~\Cref{prop:elements-not-in-shad-w}.\ref{prop:elements-not-in-shad-w:left}.\ref{prop:elements-not-in-shad-w:left:ii} to obtain that $W_L(w)$ is not left of $W_L(z_R(a'))$.
    It follows that one of the following options holds:
    (a) $W_L(z_R(a'))$ is left of $W_L(w)$, 
    (b) $W_L(z_R(a'))$ is a subpath of $W_L(w)$,
    (c) $W_L(w)$ is a subpath of $W_L(z_R(a'))$.    

    Recall that the goal is to prove~\eqref{eq:shad-containment-b'-bj7}.
    Suppose to the contrary that $b' \in \shadz(b_{j_7})$.
    By~\Cref{claim:b'-in-R-j6}, $b' \in \calR_{j_6}$.
    By~\Cref{claim:HIO:left-right}.\ref{claim:HIO:left-right:u_left_b}, $u_{j_6}$ is left of $b_{j_7}$.
    By~\Cref{claim:HIO:cor-a_leq_v}.\ref{claim:HOI:cor-a_leq_v:b_k-notin-R_i}, $b_{j_7} \notin \calR_{j_6}$.
    It follows that we can apply~\Cref{prop:shad-d-equals-shad-w}.\ref{prop:shad-d-equals-shad-w:L} to obtain that $q_{j_6} \leq w$ in~$P$ and $b' \in \shadz(w)$.

    If $W_L(w)$ is a subpath of $W_L(z_R(a'))$ (option (c)), then $b' \in \shadz(w) \subset \shadz(z_R(a'))$ (by~\Cref{prop:comparability_implies_shadow_containment}), which contradicts $b'$ being left of $z_R(a')$ ($(a',b') \in I$ and~\Cref{prop:z_L_b_z_R}).
    Therefore, we assume that $W_L(z_R(a'))$ is left of $W_L(w)$ (option (a)) or $W_L(z_R(a'))$ is a subpath of $W_L(w)$ (option (b)).
    Since $W_L(b')$ is left of $W_L(z_R(a'))$ (by~\Cref{prop:z_L_b_z_R}), we obtain that $W_L(b')$ is left of $W_L(w)$.
    Recall that the supposition that we made (i.e.\ $b' \in \shadz(b_{j_7})$) implies $b' \in \shadz(w)$.
    However,~\Cref{prop:paths_directions_in_shadows}.\ref{prop:paths_directions_in_shadows:left} yields that in this case $W_L(b')$ is not left of $W_L(w)$, which is a final contradiction that ends the proof of~\eqref{eq:shad-containment-b'-bj7}.
    
    \begin{figure}[tp]
      \begin{center}
        \includegraphics{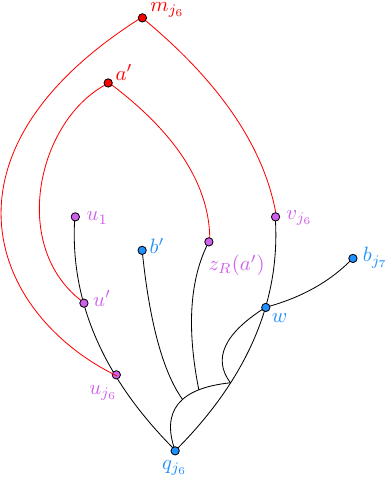}
      \end{center}
      \caption{
        An illustration of the proof of~\cref{claim:b'-left-bj7}.
      }
      \label{fig:b'-left-bj7}
    \end{figure}

    Next, we argue that
    \begin{align}\label{eq:W_Rb'-left-W_Rbj7}
        W_R(b') \text{ is left of } W_R(b_{j_7}).
    \end{align}
    Since $b' \in \Int \calR_{j_6}$ (by~\Cref{claim:b'-in-R-j6}), by~\Cref{prop:b_relative_to_region}.\ref{prop:b_relative_to_region:item:inside}, $W_R(b')$ is left of $\gamma_{R,j_6}$.
    Since $b_{j_7} \notin \calR_{j_6}$ (\Cref{claim:HIO:cor-a_leq_v}.\ref{claim:HOI:cor-a_leq_v:b_k-notin-R_i}), by~\Cref{prop:b_relative_to_region}.\ref{prop:b_relative_to_region:item:outside}, $\gamma_{R,j_6}$ is left of $W_R(b_{j_7})$.
    Altogether, we obtain $W_R(b')$ is left of $W_R(b_{j_7})$, hence,~\eqref{eq:W_Rb'-left-W_Rbj7} follows.
    Finally,~\eqref{eq:W_Rb'-left-W_Rbj7} and~\Cref{prop:paths_directions_in_shadows}.\ref{prop:paths_directions_in_shadows:right} imply that $b_{j_7} \notin \shadz(b')$.
    Furthermore, by~\Cref{cor:notin_shads_implies_left_or_right} implies that $b'$ is left of $b_{j_7}$.
\end{proofclaim}

\begin{claim}\label{claim:final-HIO-Rin} $((a',b'),(a_{j_8},b_{j_8}))$ satisfies \ref{Rin}.
\end{claim}
\begin{proofclaim}
We need to prove that $M_R(a')$ is left of $M_R(a_{j_8})$.
To this end, we show that
\begin{align}\label{eq:M_R(a')-left-gamma}
    M_R(a') \text{ is left of } \gamma_{R,j_6}.
\end{align}
Let $W'$ be an exposed witnessing path from $a'$ to $z_R(a')$ in~$P$.
By~\Cref{claim:a'-to-z-path}, we have $W' \subset \calR_{j_6}$.
In particular, $z_R(a') \in \calR_{j_6}$.
If $z_R(a') \in \Int \calR_{j_6}$, then by~\Cref{prop:b_relative_to_region}.\ref{prop:b_relative_to_region:item:inside}, $W_R(a')$ is left of $\gamma_{R,j_6}$, which implies~\eqref{eq:M_R(a')-left-gamma}.
Thus, we can assume that $z_R(a') \in \partial \calR_{j_6}$.
If $z_R(a')$ lies in $q_{j_6}[W_L(u_{j_6})]u_{j_6}$, then $z_R(a') \leq u_{j_6} \leq z_L(a')$ in $P$ (by~\Cref{claim:uj6}), which contradicts $z_L(a')$ being left of $z_R(a')$ (by~\Cref{prop:z_L_b_z_R}).
It follows that $z_R(a')$ lies strictly on the right side of $\calR_{j_6}$.
Recall that $W' \subset \calR_{j_6}$.
Finally, by~\ref{items:leaving_regions:right}, $x_0[W_R(v_{j_6})]z_R(a')[W']a'$ is left of $\gamma_{R,j_6}$, which again gives~\eqref{eq:M_R(a')-left-gamma}.

By~\Cref{claim:HIO:cor-a_leq_v}.\ref{claim:HOI:cor-a_leq_v:b_k-notin-R_i}, $b_{j_8} \notin \calR_{j_6}$.
It follows that by~\Cref{prop:b_relative_to_region}.\ref{prop:b_relative_to_region:item:outside}, $\gamma_{R,j_6}$ is left of $W_R(b_{j_8})$.
This and~\eqref{eq:M_R(a')-left-gamma} gives that $M_R(a')$ is left of $W_R(b_{j_8})$.
Finally, $W_R(b_{j_8})$ is left of $W_R(z_R(a_{j_8}))$ by~\Cref{prop:z_L_b_z_R}, thus, $M_R(a')$ is left of $M_R(a_{j_8})$.
\end{proofclaim}

\begin{claim}\label{claim:final-HIO-Lout} $((a',b'),(a_{j_8},b_{j_8}))$ satisfies \ref{Lout}.
\end{claim}
\begin{proofclaim}
    Let $u = u_{j_8}$ and $U = U_{j_8}$.
    We claim that $(u,U)$ witnesses~\ref{Lout} for $((a',b'),(a_{j_8},b_{j_8}))$.
    Note that $u \in Z(a_{j_8})$ (by~\ref{items-u:Z}) and $U$ is an exposed witnessing path from $a_{j_8}$ to $u$ in~$P$ (by~\ref{items-u:path}).
    Recall that by~\ref{items-u:W_L}, $u$ and $u_{j_6}$ lie in $\spine$.
    First, we show that
    \begin{align}\label{pi-bj7-leq-pib'}
        \pi(b_{j_7}) \leq \pi(b') \text{ in } P.
    \end{align}
    By~\Cref{prop:z_L_b_z_R}, $\spine = W_L(z_L(a))$ is left of $W_L(b)$.
    Thus, since $W_L(b)$ is left of $W_L(b')$, we have $\spine$ left of $W_L(b')$.
    By~\Cref{claim:b'-left-bj7}, $W_L(b')$ is left of $W_L(b_{j_7})$.
    Note that by definition, $\pi(b_{j_7})$ lies in both $\spine$ and $W_L(b_{j_7})$.
    The paths $\spine$, $W_L(b')$, and $W_L(b_{j_7})$ are pairwise $x_0$-consistent (by~\Cref{prop:W-consistent}.\ref{prop:W-consistent:left}).
    Therefore, by~\Cref{prop:sandwiched_paths},  $\pi(b_{j_7})$ also lies in $W_L(b')$, and so, $\pi(b_{j_7}) \leq \pi(b')$ in~$P$, which yields~\eqref{pi-bj7-leq-pib'}.

    To get the assertion, we have to prove that $u$ lies in both $W_L(b')$ and $W_L(z_L(a'))$, and that $M = x_0[W_L(u)]u[U]a_{j_8}$ is left of $W_L(z_L(a'))$.
    Note that by~\ref{items-u:u<pi}, \cref{claim:HIO:u}.\ref{claim:HIO:u:u_j<u_i}, and~\eqref{pi-bj7-leq-pib'} we have $u < u_{j_7+1} \leq \pi(b_{j_7}) \leq \pi(b')$ in~$P$. 
    This implies that $u$ lies in $W_L(b')$, as desired.
    By~\Cref{claim:uj6}, $u_{j_6}$ lies in $W_L(z_L(a'))$.
    Since $u < u_{j_6}$ in~$P$ (by~\Cref{claim:HIO:u}.\ref{claim:HIO:u:u_j<u_i}), we obtain that $u$ lies in $W_L(z_L(a'))$, as desired.
    Finally, by~\ref{items-u:path-left}, $M$ is left of $\spine$.
    Since $W_L(u)$ is a proper subpath of $W_L(u_{j_6})$ (by~\Cref{claim:HIO:u}.\ref{claim:HIO:u:WL-supath-WL}), $W_L(u_{j_6})$ is a subpath of $\spine$, and $W_L(u_{j_6})$ is a subpath of $W_L(z_L(a'))$ (by~\Cref{claim:uj6}), we obtain that $M$ is left of $W_L(z_L(a'))$, as desired.
    This ends the proof of the claim.
\end{proofclaim}

By definition, $(a',b'),(a_{j_8},b_{j_8}) \in I$. 
Moreover, by~\Cref{claim:b'-left-bj7} we have $b'$ is left of $b_{j_7}$ which is left of $b_{j_8}$, and therefore $((a',b'),(a_{j_8},b_{j_8}))$ is a regular sequence. 
Finally,~\cref{claim:final-HIO-Rin,claim:final-HIO-Lout} imply that $((a',b'),(a_{j_8},b_{j_8}))$ is an edge in $\HIO$.
The following concludes the proof~\eqref{eq:lem:HIO:old}, and so, of~\cref{lemma:HIO}.
\begin{equation}
\begin{aligned} 
    \maxsw(\HIO, (a',b')) 
    &\geq|\set{i : i\in\set{j_8,\ldots,n-1} \textrm{ and $\sigma_i$ is of weight $1$}}|\\
    &= |E| - |E \cap [j_8-1]|\\
    &= |E| - |E \cap [j_8]|+1\\
    &\geq |E| - |E_0| + 1\\
    &>\maxsw(\HIO,(a,b)) - m.
\end{aligned}
\label{eq:final-HIO}
\end{equation}
The first inequality follows from the fact that $((a',b'),(a_{j_8},b_{j_8}),(a_{j_8+1},b_{j_8+1}),\ldots,(a_n,b_n))$ is a path in $\HIO$. 
The third and fourth lines follow by $j_8\in E_2\subseteq E_0$ and therefore $E\cap [j_8]\subseteq E_0$.
\end{proof}

\section{Proof of the main result}
\label{sec:proof}
In this section, we wrap up the proof of the main theorem. 
In fact, we prove a more general statement (\Cref{thm:technical}), which immediately
implies that the class of posets that are subposets of posets with planar cover graphs is also dim-bounded (\Cref{cor:subposets-dim-bounded}).
Note that a subposet of a poset with a planar cover graph does not necessarily have a planar cover graph, e.g.\ standard examples. 
Finally, note that~\Cref{thm:technical} implies~\Cref{thm:cover-graph_se} by taking $I = \Inc(P)$.
\begin{theorem}
\label{thm:technical}
For every poset $P$ with a planar cover graph and
for every $I\subseteq\Inc(P)$, 
\[\dim_P(I)\leq 64s^6 \cdot (s+3)^2 + 12,\]
where $s = \se_P(\Inc(P) \cap (\pi_1(I) \times \pi_2(I))$.
\end{theorem}
\begin{proof}
    Let $P$ be a poset with a planar cover graph and let $I \subset \Inc(P)$.
    By~\Cref{cor:poset-to-instance}, there exists an instance $(P',x_0',G',e_{-\infty}',I')$ such that $P' - x_0'$ is a convex subposet of $P$ and $I' \subset I$, or $P' - x_0'$ is a convex subposet of $P^{-1}$ and $I' \subset I^{-1}$, and
    \begin{align*}
        \dim_P(I) &\leq 2 \dim_{P'}(I').\\
        \intertext{
        In particular, since $P'-x_0'$ is a subposet of $P$  and $I'\subseteq I$
    or $P'-x_0'$ is a subposet of $P^{-1}$ and $I'\subseteq I^{-1}$  
    we have
        }
        \se_{P'}(\Inc(P') \cap (\pi_1(I') \times \pi_2(I')) &\leq \se_{P}(\Inc(P) \cap (\pi_1(I) \times \pi_2(I)).
    \end{align*}
    By \cref{cor:interface}, there exists a good instance $(Q,x_0,G,e_{-\infty},J)$ such that $Q$ is a convex subposet of $P'$, $J \subset I'$, and 
    \begin{align*}
        \dim_{P'}(I') &\leq 2\dim_{Q}(J) + 6.\\
        \intertext{In particular, since $Q$ is a subposet of $P'$ and $J \subset I'$}
        \se_{Q}(\Inc(Q) \cap (\pi_1(J) \times \pi_2(J)) &\leq \se_{P'}(\Inc(P') \cap (\pi_1(I') \times \pi_2(I')).
    \end{align*}
    By~\cref{prop:good_instance_to_maximal_good_instance}, there exists $J \subset J^+ \subset \Inc(Q) \cap (\pi_1(J) \times \pi_2(J))$ such that $(Q,x_0,G,e_{-\infty},J^+)$ is a maximal good instance.
    In particular, 
    \begin{align*}
        \dim_{Q}(J) &\leq \dim_{Q}(J^+) \text{ and }\\
        \se_{Q}(J^+) &\leq \se_{Q}(\Inc(Q) \cap (\pi_1(J) \times \pi_2(J)).
    \end{align*}
    By~\Cref{thm:maximal_instance_imply_dim_boundedness},
    \[ \dim_{Q}(J^+) \leq 16\se_{Q}(J^+)^6 \cdot (\se_{Q}(J^+)+3)^2. \]
    Observe that
        \begin{align*}
            \se_{Q}(J^+) & \leq \se_{Q}(\Inc(Q) \cap (\pi_1(J) \times \pi_2(J)) \\
            & \leq \se_{P'}(\Inc(P') \cap (\pi_1(I') \times \pi_2(I'))\\
            & \leq \se_{P}(\Inc(P) \cap (\pi_1(I) \times \pi_2(I)).
        \end{align*}
    Summarizing, we obtain
    \begin{align*}
    \dim_P(I) &\leq 2 \dim_{P'}(I') \\
            &\leq 4\dim_{Q}(J) + 12\\
            &\leq 4\dim_{Q}(J^+) + 12\\
            &\leq 4(16\se_{Q}(J^+)^6 \cdot (\se_{Q}(J^+)+3)^2) + 12\\
            &\leq 64s^6 \cdot (s+3)^2 + 12,
    \end{align*}
    where $s = \se(\Inc(P) \cap (\pi_1(I) \times \pi_2(I))$.
\end{proof}

\begin{corollary} \label{cor:subposets-dim-bounded}
The class of posets that are subposets of posets with planar cover graphs is $\dim$-bounded.
\end{corollary}
\begin{proof}
    Let $f(s) = 64s^6 \cdot (s+3)^2 + 12$ for every positive integer $s$. 
    Let $P$ be a poset with a planar cover graph and 
    let $Q$ be a subposet of $P$. 
    Let $I \subset \Inc(P)$ be the set of all incomparable pairs $(a,b)$ in $P$ such that both $a$ and $b$ are elements of $Q$.
    By definition, $I = \Inc(P) \cap (\pi_1(I) \times \pi_2(I))$.
    In particular, $I = \Inc(Q)$, $\dim_Q(I) = \dim_P(I)$, and $\se_Q(I) = \se_P(I)$.
    Therefore, by~\Cref{thm:technical},
    \[\dim(Q) = \dim_Q(I) = \dim_P(I) \leq f(\se_P(I)) = f(\se_Q(I)) = f(\se(Q)).\qedhere\]
\end{proof}

\bibliographystyle{plain}
\bibliography{bibliography}

\end{document}